\documentclass[11pt]{article}
\usepackage[a4paper, total={6in, 8in}]{geometry}
\usepackage{amssymb,amsmath,amsthm,tikz,multirow, mathdots}
\usepackage{enumerate, enumitem} 
\usepackage[hidelinks, hyperindex]{hyperref}
\hyperref[sec:function]{}
\usepackage [autostyle, english = british]{csquotes}
\usepackage [english]{babel}
\usepackage{mathtools}
\usepackage{bm}
\usepackage{authblk}
\usepackage{wasysym} 
\usepackage{scalerel}
\usepackage{caption}
\usepackage{subcaption}

\usepackage{comment} 

\usetikzlibrary{calc,arrows, arrows.meta, math} 
\usetikzlibrary{shapes.geometric}

\newtheorem{theorem}{Theorem}

\newtheorem*{theorem*}{Theorem}

\theoremstyle{definition}
\newtheorem*{definition*}{Definition}
\newtheorem*{case*}{Case}
\newtheorem*{subcase*}{Subcase}
\newtheorem*{subsubcase*}{Subsubcase}

\theoremstyle{plain}
\newtheorem{thm}{Theorem}[section]
\newtheorem{lem}[thm]{Lemma}
\newtheorem{prop}[thm]{Proposition}

\theoremstyle{definition}

\newtheorem{expl}{Example}[section]

\theoremstyle{remark}
\newtheorem{remark}[theorem]{Remark}

\numberwithin{equation}{section}

\colorlet{HMintBlue}{cyan!50!teal} 
\colorlet{HGold}{orange!75!brown!75!black!80}

\newcommand{\AVC}{\text{AVC}} 

\newcommand{\quotes}[1]{``#1''} 
\newcommand{\bvert}{\vrule width 2pt}

\newcommand{\arcThroughThreePoints}[4][]{
\coordinate (middle1) at ($(#2)!.5!(#3)$);
\coordinate (middle2) at ($(#3)!.5!(#4)$);
\coordinate (aux1) at ($(middle1)!1!90:(#3)$);
\coordinate (aux2) at ($(middle2)!1!90:(#4)$);
\coordinate (center) at ($(intersection of middle1--aux1 and middle2--aux2)$);
\draw[#1] 
 let \p1=($(#2)-(center)$),
      \p2=($(#4)-(center)$),
      \n0={veclen(\p1)},       
      \n1={atan2(\y1,\x1)}, 
      \n2={atan2(\y2,\x2)},
      \n3={\n2>\n1?\n2:\n2+360}
    in (#2) arc(\n1:\n3:\n0);
}

\def\GetOne(#1,#2,#3,#4,#5,#6,#7,#8)#9{\def#9{#1}}
\def\GetTwo(#1,#2,#3,#4,#5,#6,#7,#8)#9{\def#9{#2}}
\def\GetThree(#1,#2,#3,#4,#5,#6,#7,#8)#9{\def#9{#3}}
\def\GetFour(#1,#2,#3,#4,#5,#6,#7,#8)#9{\def#9{#4}}
\def\GetFive(#1,#2,#3,#4,#5,#6,#7,#8)#9{\def#9{#5}}
\def\GetSix(#1,#2,#3,#4,#5,#6,#7,#8)#9{\def#9{#6}}
\def\GetSeven(#1,#2,#3,#4,#5,#6,#7,#8)#9{\def#9{#7}}
\def\GetEight(#1,#2,#3,#4,#5,#6,#7,#8)#9{\def#9{#8}}

\begin{document}
\title{Dihedral Tilings of the Sphere by Kites and Regular Polygons}

\author[1]{Robert Barish \footnote{\tt email: rbarish@ims.u-tokyo.ac.jp.}}
\author[2]{Kam Hang Cheng \footnote{\tt email: keroc@ust.hk.}}
\author[3]{Hoi Ping Luk \footnote{\tt email: hoi@connect.ust.hk; hoiping@ntis.zcu.cz.}}

\affil[1]{Institute of Medical Science, University of Tokyo, Japan.}
\affil[2]{Department of Mathematics, The Hong Kong University of Science \& Technology, Hong Kong.}
\affil[3]{KMA, NTIS, Z\'apado\v{c}esk\'a univerzita v Plzni, Czech Republic.}

\maketitle

\abstract{In this article, we study the edge-to-edge dihedral tilings of the sphere by kites and regular $m$-gons with $m\ge 4$. All such tilings have been identified and fully classified, using various combinatorial and geometric tools. New phenomena are observed among the tilings. \\


}

\section{Introduction}
\label{Sec-Intro}

The study of spherical tilings dates back to Plato and Archimedes. A {\em spherical tiling} (or a tiling of the sphere) is a covering of the sphere, without holes and overlaps, by {\em tiles} each of which is a spherical polygon (an $m$-gon) with $m$ edges formed by geodesic arcs. A {\em vertex} is where the edges meet, and the {\em degree} of a vertex is the number of edges that meet at the vertex. Naturally, we assume both $m$ and the degree of each vertex to be at least $3$, and that every polygon is simple (i.e., has no self-intersections). An {\em edge-to-edge} tiling has no vertex lying in the interior of an edge. Often being compared with polyhedra, edge-to-edge spherical tilings are also referred to as spherical polyhedra, because of the commonality in their underlying graphs. We study edge-to-edge spherical tilings, and for brevity, the word {\em tilings} in this article always refers to edge-to-edge spherical tilings. Standard notions in graph theory and the study of polyhedra are used without introduction.

Mathematical objects with suitable notions of uniformity or regularity set the foundation for theories as well as their applications. For example, a {\em regular} $m$-gon has all $m$ edges of the same length and all $m$ angles of the same size. As in the case of polyhedra, the five Platonic solids are the only {\em regular} convex polyhedra in $\mathbb{R}^3$: each Platonic solid has a single orbit of the faces, the edges, and the vertices under the action of the solid's symmetry group. Such properties on the faces, the edges, and the vertices are respectively called {\em face-, edge-} and {\em vertex-transitivity}. Note that the faces of a Platonic solid are necessarily congruent regular polygons. A polyhedron is called {\em uniform} if its faces are regular polygons and it is vertex-transitive. The non-regular uniform convex polyhedra are the thirteen Archimedean solids and the infinite families of prisms and antiprisms, classified by Johannes Kepler in his Harmonices Mundi. The non-uniform convex polyhedra with regular polygonal faces are the Johnson-Zalgaller solids \cite{joh,zal}. Note that a convex polyhedron from the list of Archimedean solids, prisms, antiprisms or Johnson-Zalgaller solids has exactly one congruence class of edge lengths and more than one congruence class of faces. More generally, without a prescribed requirement on the symmetry group action and regularity in the faces, one may study polyhedra in terms of the congruence classes of the faces and that of the edge lengths. The same applies to spherical tilings. 

A congruence class of tiles in a tiling is called a {\em prototile}. In the classification of edge-to-edge spherical tilings, those with a single prototile, i.e. all the tiles are congruent, are determined recently \cite{awy, ay, cly, cly2, gsy, lc, ua, wy, wy2}. Such tilings are called {\em monohedral}. Among them, the ones with regular tiles share the same underlying graphs with the Platonic solids. Spherical tilings in this sense of regularity are scarce. One may next pursue a study of tilings with more than one prototile. Spherical tilings with regular polygons as prototiles are fully classified, edge-to-edge case in \cite{lnp} and non-edge-to-edge case in \cite{aehj}. Naturally, one direction of classification is to take regular or equilateral polygons as prototiles \cite{cl,luk1}, in which there is a single edge length. A more general direction is to investigate prototiles having more than one edge length. In retrospect, a majority of the monohedral edge-to-edge spherical tilings have exactly two edge lengths; and interestingly some of the tilings admit a concave prototile or a prototile with an angle exactly equal to $\pi$, so that the tilings may also be regarded as non-edge-to-edge \cite[Figure 2]{cl}. 

A {\em dihedral} tiling has exactly two prototiles. In this article, we study dihedral edge-to-edge spherical tilings with exactly two edge lengths. An early study classified such tilings with a regular polygon and a quadrilateral having equal opposite edges as prototiles \cite{luk2}. Naturally, we consider one prototile to be a regular $m$-gon ($m\ge4$) and the other to be a kite, see Figure \ref{Fig-kite-mgon}. A kite has edge combination $x^2y^2$ and angle combination $\beta\gamma^2\delta$. A regular $m$-gon has edge combination $x^m$ for $m\ge4$, where $x,y$ denote edges of distinct edge lengths, and angle combination $\alpha^m$. The prototiles being simple implies that $x,y \in (0,\pi)$. The methods used in this article may also be applied to the dihedral tilings by kites and equilateral triangles ($m=3$). As will become transparent in Section \ref{Sec-Toolbox-Strategy}, adaptations will be required; hence, these cases will be handled in a future work.

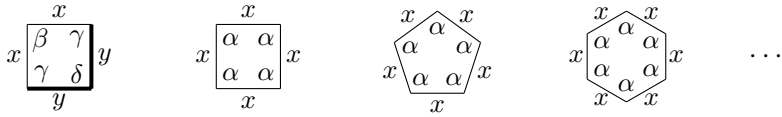
\begin{figure}[h!] 
\centering
\begin{tikzpicture}
\tikzmath{
\s=2.5;
\r=0.6;
\th=360/4;
\ph=360/5;
\ps=360/6;
\x=\r*cos(0.5*\th);
\R = sqrt(\x^2+(3*\x)^2);
\aR = acos(3*\x/\R);
}

\begin{scope}[] 
\foreach \a in {0,...,3} {

\draw[rotate=\th*\a]
	(0.5*\th:\r) -- (1.5*\th:\r)
;
}

\foreach \a in {0,1} {
\node at (\th+\th*\a: 1*\r) {\small $x$}; 
}

\foreach \b in {2,3} {
\node at (\th+\th*\b: 1*\r) {\small $y$}; 
}

\draw[line width=1.5]
	(0.5*\th:\r) -- (3.5*\th:\r)
	(2.5*\th:\r) -- (3.5*\th:\r)
;

\node at (0.5*\th: 0.55*\r) {\small $\gamma$};
\node at (1.55*\th: 0.55*\r) {\small $\beta$};
\node at (2.5*\th: 0.55*\r) {\small $\gamma$};
\node at (3.5*\th: 0.55*\r) {\small $\delta$};
\end{scope}

\begin{scope}[xshift=\s cm] 
\foreach \a in {0,...,3} {
\tikzset{rotate=\a*\th}
\draw[]
	(0.5*\th:\r) -- (1.5*\th:\r)
;
\node at (0.5*\th: 0.55*\r) {\small $\alpha$};
\node at (\th: 1*\r) {\small $x$};
}

\end{scope}

\begin{scope}[xshift=2*\s cm] 
\foreach \a in {0,...,4} {
\tikzset{rotate=\a*\ph}

\draw[]
	(90:\r) -- (90+\ph:\r)
;
\node at (90: 0.625*\r) {\small $\alpha$};
\node at (90-0.5*\ph: 1.1*\r) {\small $x$};
}

\end{scope}

\begin{scope}[xshift=3*\s cm] 
\foreach \a in {0,...,5} {
\tikzset{rotate=\a*\ps}

\draw[]
	(90:\r) -- (90+\ps:\r)
;

\node at (90: 0.625*\r) {\small $\alpha$};
\node at (90-0.5*\ps: 1.1*\r) {\small $x$};
}

\end{scope}

\node at (3.75*\s, 0) {$\cdots$};

\end{tikzpicture}
\caption{Kite and regular $m$-gon where $m\ge4$}
\label{Fig-kite-mgon}
\end{figure}

We apply a m\'elange of combinatorial and geometric tools and obtain a rich collection of tilings. This collection includes three infinite families and some finite ones, which is a phenomenon comparable to the counterparts in the monohedral tiling classification. Apart from the sophisticated features comparable to the monohedral case, our findings also exhibit new phenomena and standout features absent in previous endeavours. One such feature is the pair of finite families generated by a sequence of flips of a cluster of tiles. Another one is an infinite subfamily in a family, which consists of odd tilings. Here an odd tiling is one that consists of an odd number of tiles. The parity phenomenon in the monohedral counterparts is highlighted in \cite{luk3}: an edge-to-edge spherical tiling by congruent polygons always has an even number of tiles. This phenomenon agrees with Gr\"unbaum's theorem \cite{gru} on polyhedra in $\mathbb{R}^3$ with congruent faces. A study on the existence of an infinite family of tilings by odd numbers of $m$-gons can be seen in \cite{luk4}, in which the family can be derived from our result via edge reduction. 

The contents of this article are organised as follows. The main result is explained in Section \ref{Sec-Main}. The combinatoric and geometric tools are established in Section \ref{Sec-Toolbox-Strategy}. The classification of the tilings is conducted in Section \ref{Sec-Tilings}. Further discussion is given in Section \ref{Sec-Discussion}.

\section{Main Result}
\label{Sec-Main}

As a prelude to the main result, we shall first recall concepts from graph theory and polyhedra. They are standard and hence are used without introduction. They unify an exposition of our main result by providing a description of the structures of the tilings as well as establishing a connection between the geometries and combinatorics. Essentially, the primal structure of each tiling can be identified with the graph of a polyhedron. This is consistent with the derivation of Archimedean solids from Platonic solids via operations in terms of Conway polyhedron notations. 

The Platonic solids are the tetrahedron $\mathcal{T}$, the cube $\mathcal{C}$, the octahedron $\mathcal{O}$, the dodecahedron $\mathcal{D}$ and the icosahedron $\mathcal{I}$. The Archimedean solids can be derived from the Platonic solids. For example, $t$ for truncation and $a$ for rectification (or ambo), and $t\mathcal{O}, a\mathcal{O}$ denote the truncated octahedron and the rectified octahedron, respectively. 

We call a remaining facet resulting from an operation on a facet of a (Platonic) solid a  {\em legacy face}. Similarly, we call a facet created by an operation a {\em derived face}. For example, in $t\mathcal{O}$ a hexagon is a legacy face while a triangle is a derived face. 

The appearance of kites as tiles is explained in terms of (nuanced) subdivisions of certain subsets of the faces, such as the legacy faces or derived faces. These subdivisions, with the exception of a triangular subdivision, are collectively called {\em kite subdivisions}. The various ways that subdivisions take place reflect the complexity of the classification and they will be explained following the main theorem. A kite subdivision $K$ is a special case of quadrilateral subdivision $Q$. Similarly, the appearance of squares in some tilings can be explained in terms of square subdivisions of the thickened $1$-skeleton of the octahedron or its flip modification $F$. Triangular subdivisions, quadrilateral subdivisions and the flip modification of the octahedron in spherical tilings are explained in detail in \cite{cly}. A more compact description may be given in terms of ribbon graphs (see definition in \cite{emm}): the thickening of $1$-skeleton of a tiling can be regarded as the ribbon graph of the embedded graph of the tiling. The other (rather unexpected) source of squares is what we refer to as the ``$2$-edged multigraphs'' of the cube and the dodecahedron. Here a {\em $2$-edged multigraph} of a graph means that the multigraph is obtained by replacing each edge of the graph by a pair of parallel edges.

Our main result is summarised in the theorem below. 

\begin{theorem*} The edge-to-edge dihedral spherical tilings by kites and regular $m$-gons with $m\ge4$ are
\begin{enumerate}[label={\Roman*}.]
\item the earth map type --- four infinite families:
\begin{enumerate}[start=1,label={\small (EM\arabic*)}]
\item \label{Label:EMT-Herschel} an infinite family of Herschel tilings;
\item \label{Label:EMT-Antiprisms} an infinite family of tilings via a kite subdivision of the polar polygons in the antiprisms;
\item \label{Label:EMT-Prisms} an infinite family of tilings via a kite subdivision of the polar polygons in the prisms, and their flip modification;
\item \label{Label:EMT-Subdiv-Equator-Prisms} an infinite even family of tilings via a kite subdivision of the equatorial quadrilaterals of prisms;
\end{enumerate}
\item the Platonic type --- two finite families and fourteen isolated tilings:
\begin{enumerate}[start=1,label={\small (P\arabic*)}]
\item \label{Label:trunc-Octa} a finite family of tilings via a kite subdivision of the legacy faces in the truncated octahedron;
\item \label{Label:trunc-Icosa} a finite family of tilings via a kite subdivision of the legacy faces in the truncated icosahedron;
\item \label{Label:polar-subdiv-cube} a tiling via a kite subdivision of two polar faces of the cube;
\item \label{Label:trunc-Platonic} five tilings via a kite subdivision of the derived faces in the truncated Platonic solids; 
\item \label{Label:rect-Platonic} five tilings via a kite subdivision of the legacy faces in the rectified Platonic solids; 
\item \label{Label:multigr-Cube-Dodeca} two tilings via a kite subdivision of the double-edged multigraphs of the cube and the dodecahedron;
\item \label{Label:tri-subdiv-thick-Octa} two tilings via a triangular subdivision of the faces and a square subdivision of the thickened graphs of the octahedron and its flip modification;
\item \label{Label:quad-subdiv-thick-Octa} two tilings via a kite subdivision of the faces and a square subdivision of the thickened graphs of the octahedron and its flip modification; 
\end{enumerate}
\item the Johnson-Zalgaller type: a kite subdivision of the square pyramid $\mathcal{J}_1$.
\end{enumerate} 
\end{theorem*}

In light of \ref{Label:trunc-Octa} and \ref{Label:trunc-Icosa}, we obtain the following result (see Proposition \ref{Prop-tP-families}).

\begin{prop}\label{Prop-Platonic-finite-families} The truncated Platonic solids give rise to finite families of dihedral tilings via kite subdivisions and their flip modifications.
\end{prop}

For brevity, from now on we often reference the tilings by their labels in the main theorem. We remark that the tilings, \ref{Label:trunc-Octa}, \ref{Label:trunc-Icosa} and \ref{Label:rect-Platonic}, may as well be regarded as of Archimedean type because of their underlying structures.

We use plane drawings to display the distinguishing features of the tilings: the infinite families of earth map type \ref{Label:EMT-Herschel}, \ref{Label:EMT-Antiprisms}, \ref{Label:EMT-Prisms}, and \ref{Label:EMT-Subdiv-Equator-Prisms} in Figures \ref{Fig-EMT-Herschel}, \ref{Fig-EMT-Antiprisms}, \ref{Fig-EMT-Prisms}, and \ref{Fig-EMT-Prisms-Subdiv-Equatorial}; the canonical seeds for generating the two finite families in Figure \ref{Fig-Finite-Families-tOcta-tIcosa}; the tilings of Platonic type \ref{Label:trunc-Octa}--\ref{Label:quad-subdiv-thick-Octa} in Figures \ref{Fig-polar-subdiv-cube}, \ref{Fig-kite-subdiv-trunc-Platonic}, \ref{Fig-Platonic-Archimedean}, \ref{Fig-Platonic-multigraph-Cube-Dodeca}, and \ref{Fig-Platonic-thicken-Octa}. The only tiling of Johnson-Zalgaller type is in Figure \ref{Fig-Tiling-Subdiv-J1}. 

In a figure, the arrowhead of a dart represents an end vertex (not on display) of the corresponding edge. For example, in Figures \ref{Fig-EMT-Herschel}, \ref{Fig-EMT-Antiprisms} and \ref{Fig-EMT-Prisms}, the upward darts in a drawing join a single vertex (the ``north pole") and the downward ones join another (the ``south pole"). In Figures \ref{Fig-EMT-Antiprisms}, \ref{Fig-EMT-Prisms}, \ref{Fig-polar-subdiv-cube}, \ref{Fig-kite-subdiv-trunc-Platonic}, \ref{Fig-Platonic-Archimedean}, \ref{Fig-Platonic-multigraph-Cube-Dodeca}, \ref{Fig-skeleton-op} and \ref{Fig-Tiling-Subdiv-J1}, the outward darts in a drawing join a single vertex (which may also be regarded as the north pole or the south pole).

We have also produced interactive $3$D rendering of the tilings in the third author's GeoGebra page: \url{www.geogebra.org/m/mpfra5q7}.

An in-depth explanation of each tiling is given as follows.

\subsection{Earth map type}

Tilings of earth map type (or simply {\em earth map tilings}) in Figures  \ref{Fig-EMT-Herschel}, \ref{Fig-EMT-Antiprisms}, \ref{Fig-EMT-Prisms}, \ref{Fig-EMT-Prisms-3D}, \ref{Fig-EMT-Prisms-Subdiv-Equatorial} resemble the earth map and hence the name. Notably, the poles of earth map tilings are the vertices with negative combinatorial curvature (see definition in \cite{hig}). A tiling is formed by gluing copies of {\em a time zone} (shaded) and the leftmost and rightmost boundaries are eventually identified. This can be done for any arbitrarily large number of time zones, giving rise to an infinite family in each case.

In Figure \ref{Fig-EMT-Herschel}, the first drawing shows three time zones of \ref{Label:EMT-Herschel} with the first one shaded; the second shows an alternative composition of a time zone -- a lune consisting of two kites and two equal halves of the rhombus. The infinite family is obtained by gluing $d\ge3$ copies of a time zone. Among them, the minimum member has its $1$-skeleton identified as the embedded Herschel graph in the third drawing. Combinatorially, for a fixed $m$ and a spherical tiling by $m$-gons, the relation between non-Hamiltonicity and the parity in the number of tiles is studied in \cite{luk4}. Given the non-Hamiltonian property of the Herschel graph \cite{bj}, we call them the {\em Herschel tilings}. A member in \ref{Label:EMT-Herschel} with an odd number of time zones has an odd number of $4$-gons. Geometrically, the family can be reduced to an equilateral family in \cite{luk4} via an {\em edge-reduction} $x=y$. 

\begin{figure}[h!] 
\centering
\begin{subfigure}{\linewidth}
\centering
\begin{tikzpicture}[>=latex]
\tikzmath{
\s=1;
\r=0.325;
\th=360/4;
\x=\r*cos(0.5*\th);
\R = sqrt(\x^2+(3*\x)^2);
\aR = acos(3*\x/\R);
\tz=3;
\tzz=\tz-1;
}

\begin{scope}[scale=1]

\node at (-4*\r, 0) {\small  \parbox[c]{2cm}{Equatorial \\ view:}};

\fill[gray!40]
	(\th:2*\r) -- (\th:\r) -- (0:\r) -- (2*\r, -\r) -- (2*\r, -2*\r) -- (4*\r, -2*\r) -- (4*\r, -\r) -- (3*\r, 0) -- (2*\r, \r) -- (2*\r, 2*\r) -- cycle
;

\foreach \a in {0,...,\tz} {
\tikzset{shift={(2*\a*\r,0)}}

\draw[]
	(\th:\r) -- (0:\r)
;

\draw[shift={(0:\r)}]
	(0:0) -- (-0.5*\th:2*\x) 
	(-0.5*\th:2*\x) -- ([shift={(-0.5*\th:2*\x)}]270:\r)
;
}

\foreach \a in {0,...,\tzz} {
\tikzset{shift={(2*\a*\r,0)}}

\draw[shift={(0:\r)}]
	(0:0) -- (0.5*\th:2*\x)
	(-0.5*\th:2*\x) -- ([shift={(-0.5*\th:2*\x)}]0.5*\th:2*\x)
;
}

\foreach \a in {0,...,\tz} {
\tikzset{shift={(2*\a*\r,0)}}

\draw[arrows = {-Latex[scale=0.45]},
line width=1.5]
	(\th:\r) -- (\th:2.6*\r)
;

\draw[arrows = {-Latex[scale=0.45]}, line width=1.5, shift={(0:\r)}]
	(-0.5*\th:2*\x) -- ([shift={(-0.5*\th:2*\x)}]270:1.6*\r)
;
}

\node at (\tz*2*\r+4*\r, 0.25*\r) {\LARGE $\cdots$};

\end{scope}

\begin{scope}[xshift=5.1*\s cm]

\fill[gray!40]
	(-\r, 2*\r) -- (-\r,-2*\r) -- (\r,-2*\r) --(\r,2*\r) 
;

\foreach \aa in {-1,1} {
\tikzset{xshift=\aa*\r cm, xscale=\aa}

\foreach \a in {1,2} {

\draw[rotate=\a*\th]
	(0:\r) -- (\th:\r)
;

}

\foreach \a in {0,3} {

\draw[rotate=\a*\th, gray!50]
	(0:\r) -- (\th:\r)
;
}

\foreach \a in {0,1} {
\tikzset{rotate=\a*2*\th}

\draw[arrows = {-Latex[scale=0.45]}, line width=1.5]
	(0,\r) -- (0,2.6*\r)
;

\draw[dotted]
	(0,\r) -- (0,-\r)
;
}
}

\end{scope}

\begin{scope}[xshift= 9*\s cm]

\node at (-4*\r, 0) {\small  \parbox[c]{2cm}{Polar \\ view:}};

\fill[white]
	(4*\x,0) -- (0,4*\x)  -- (-4*\x,0)  -- (0,-4*\x) -- cycle
;

\foreach \a in {0,1} {

\draw[rotate=\a*180, orange!75!brown!75!black!80, line width=1.5]
	(90-0.5*\th:\r) -- (90:2*\x)
	(90+0.5*\th:\r) -- (90:2*\x)
	(90:2*\x) -- (90:4*\x)
;
}

\foreach \a in {0,1} { 
\tikzset{rotate=\a*180}

\fill[teal!70!blue]
	(4*\x,0) -- (0,4*\x)  -- (-4*\x,0)  -- (0,-4*\x) -- (0,-4.8*\x) -- (-4.8*\x,0) -- (0,4.8*\x) -- (4.8*\x,0) -- cycle
;
}

\foreach \aa in {-1,1} {
\tikzset{xscale=\aa}
\fill[teal!70!blue]
	(0,0) -- (90-0.5*\th:\r) -- (4*\x,0) -- (90-1.5*\th:\r) -- cycle
;
}

\foreach \a in {0,...,3} {
\draw[rotate=\a*90]
	(0,0) -- (90-0.5*\th:\r)
	(0:4*\x) -- (90:4*\x)
;
}

\foreach \a in {0,1} {
\draw[rotate=\a*180]
	(90-0.5*\th:\r) -- (0:4*\x)
	(90-1.5*\th:\r) -- (0:4*\x)
;
}

\end{scope}
\end{tikzpicture}
\end{subfigure}
\begin{subfigure}{\linewidth}
\centering
\begin{tikzpicture}
\tikzmath{
\XS=2.5;
}

\begin{scope}[] 
\node [inner sep=0] (image) at (0,0) 
            {\includegraphics[height=2cm]{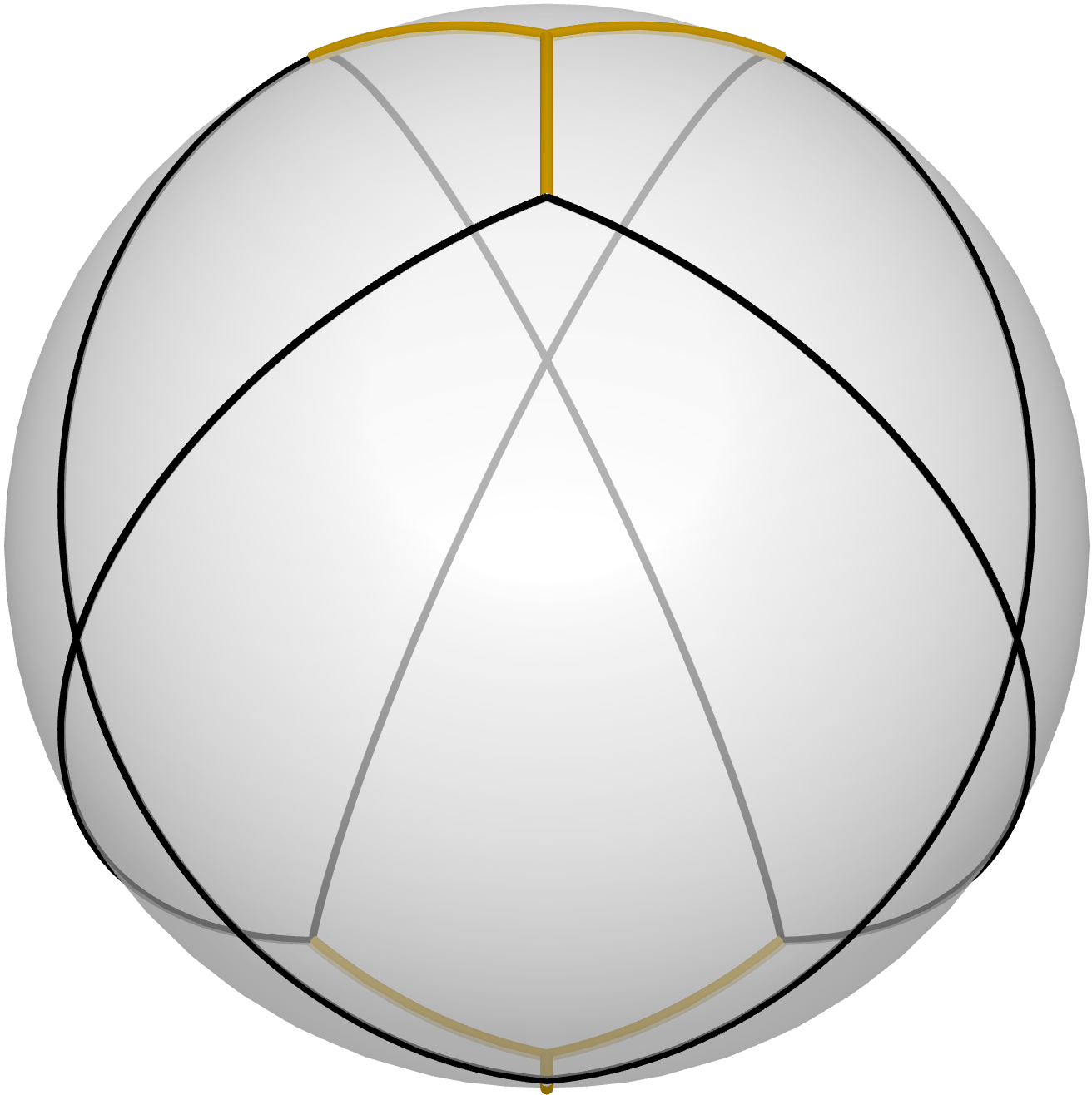}};
\end{scope}

\begin{scope}[xshift=\XS cm] 
\node [inner sep=0] (image) at (0,0) 
            {\includegraphics[height=2cm]{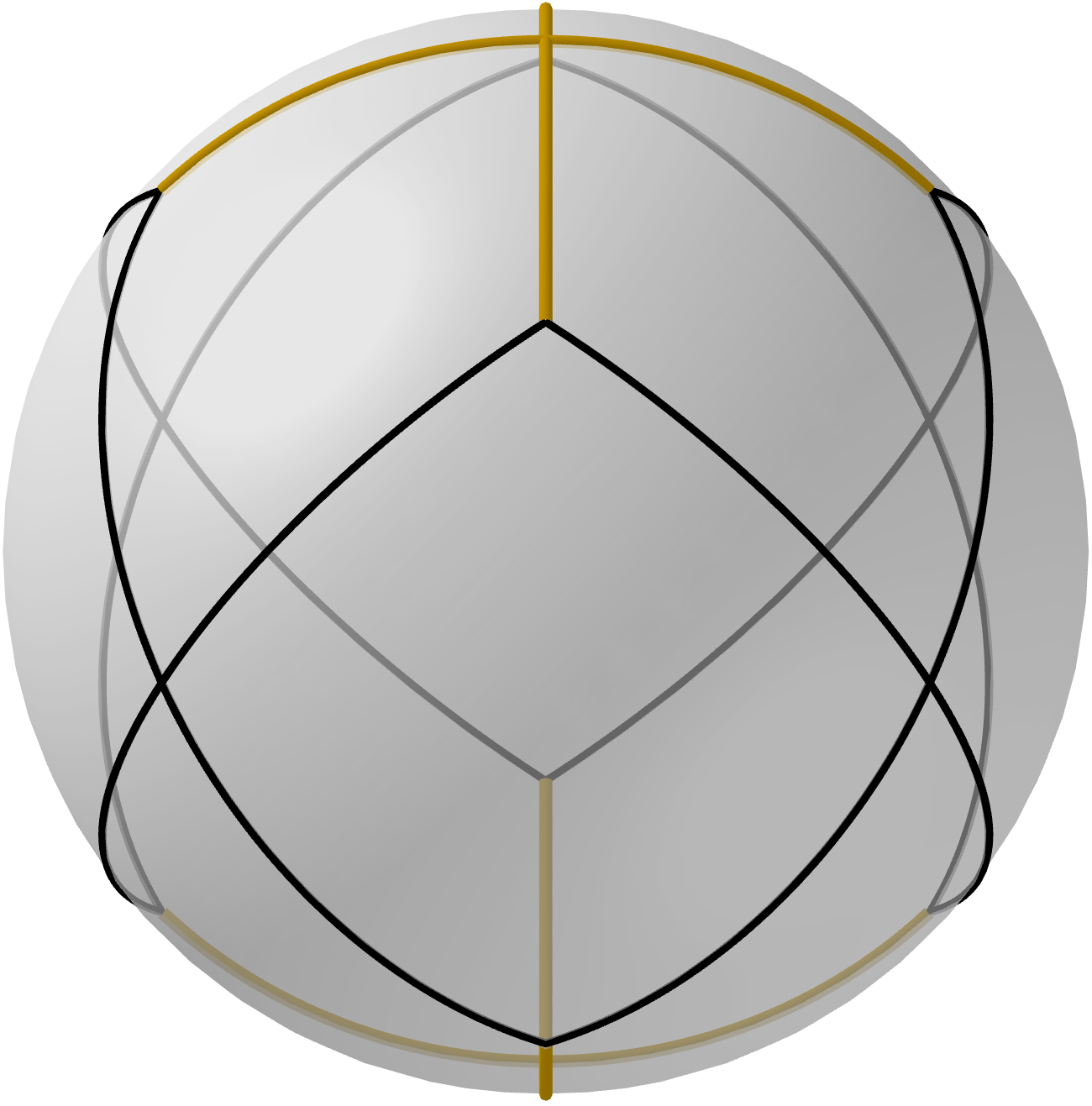}};
\end{scope}

\begin{scope}[xshift=2*\XS cm] 
\node [inner sep=0] (image) at (0,0) 
            {\includegraphics[height=2cm]{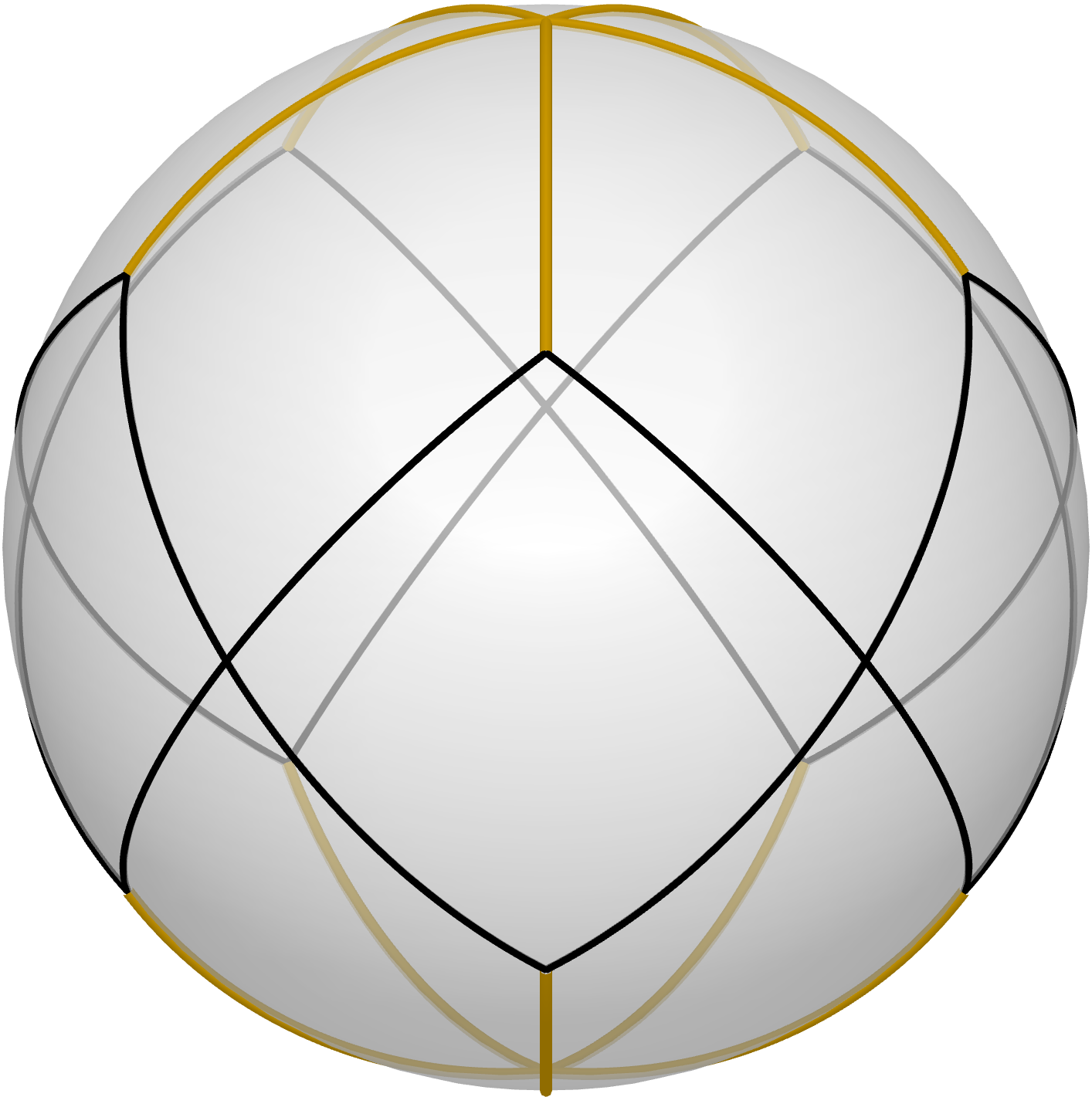}};
\end{scope}
\end{tikzpicture}
\end{subfigure}
\caption{The first row -- $3$ time zones of \ref{Label:EMT-Herschel}, an alternative composition of a time zone, the Herschel graph; the second row -- the first three tilings of the family}
\label{Fig-EMT-Herschel}
\end{figure}
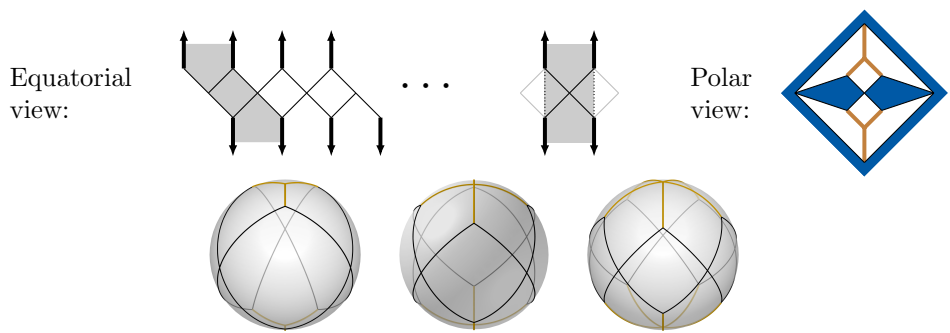

In Figure \ref{Fig-EMT-Antiprisms}, the first row shows three time zones of \ref{Label:EMT-Antiprisms} and an alternative composition of a time zone. The infinite family is again obtained by gluing $d$ copies of a time zone for $d\ge4$. When $d=3$, the tiling is geometrically impossible (to be discussed after Figure \ref{Fig-a4-a2b2-Tiling-alga2-al3be-ded}). Alternatively, the tilings for $d\ge4$ can be obtained via a kite subdivision of the two polar regular $d$-gons in the family of antiprisms, see the second row of Figure \ref{Fig-EMT-Antiprisms} for $d=4,5$ and $6$.


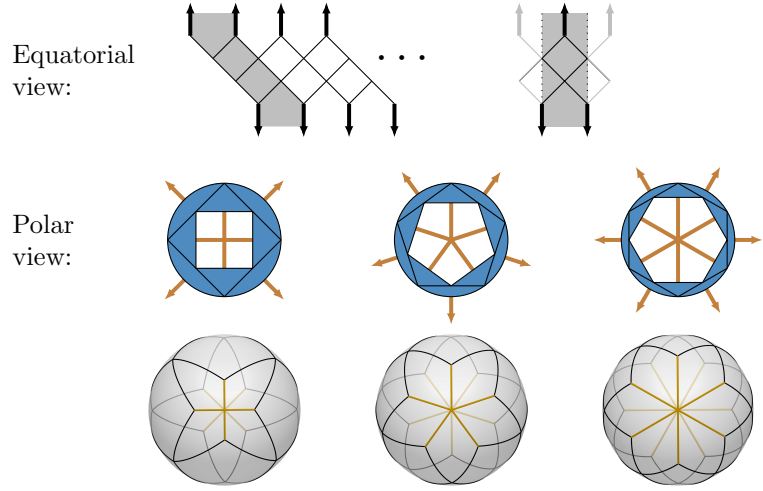
\begin{figure}[h!] 
\centering
\begin{tikzpicture}[>=latex]

\tikzmath{
\XS=4.35;
\YS=2.25;
\r=0.3;
\th=360/4;
\x=\r*cos(0.5*\th);
\R = sqrt(\x^2+(3*\x)^2);
\aR = acos(3*\x/\R);
\tz=3;
\tzz=\tz-1;
}

\begin{scope}

\node at (-6*\r, 0) {\small  \parbox[c]{2cm}{Equatorial \\ view:}};

\begin{scope}[] 

\fill[gray!50]
	(-1.5*\r,2.5*\r) --  (1.5*\th:3*\x)  -- (3.5*\th:3*\x) -- (1.5*\r, -2.5*\r) -- (1.5*\r, -2.5*\r) 
	-- (3.5*\r, -2.5*\r) -- (3.5*\r, -1.5*\r) -- (0.5*\r, 1.5*\r) -- (0.5*\r, 1.5*\r) -- (0.5*\r, 2.5*\r) -- cycle
;

\foreach \a in {0,...,\tz} {
\tikzset{shift={(2*\a*\r,0)}}

\draw[]
	(1.5*\th:\x) -- (1.5*\th:3*\x)
	(1.5*\th:3*\x) -- (-1.5*\r, 2.5*\r)
	(1.5*\th:\x) -- (3.5*\th:\x)
	(3.5*\th:\x) -- (3.5*\th:3*\x)
;
}

\foreach \a in {0,...,\tzz} {

\tikzset{shift={(2*\a*\r,0)}}

\draw[]
	(-0.5*\r, 0.5*\r) -- (0.5*\r, 1.5*\r)
	(0.5*\r, -0.5*\r) -- (1.5*\r, 0.5*\r)
	(1.5*\r, -1.5*\r) -- (2.5*\r, -0.5*\r)
;
}

\foreach \a in {0,...,\tz} {
\tikzset{shift={(2*\a*\r,0)}}

\draw[arrows = {-Latex[scale=0.45]}, line width=1.5]
	(1.5*\th:3*\x) -- (-1.5*\r, 3*\r)
;

\draw[arrows = {-Latex[scale=0.45]}, line width=1.5]
	(3.5*\th:3*\x) -- (1.5*\r, -3*\r)
;
}

\node at (\tz*2*\r+2*\r,0.5*\r) {\LARGE $\cdots$};

\end{scope}

\begin{scope}[xshift=\XS cm]

\fill[gray!50]
	(-0.5*\r, 2.5*\r) -- (-0.5*\r, -2.5*\r) -- (1.5*\r, -2.5*\r) -- (1.5*\r, 2.5*\r)
;

\draw[]
	(2.5*\r, -0.5*\r) -- (0.5*\r, 1.5*\r)
	(-1.5*\r, -0.5*\r) -- (0.5*\r, 1.5*\r)
	(0.5*\r, -0.5*\r) -- (-0.5*\r, 0.5*\r)
	(0.5*\r, -0.5*\r) -- (-0.5*\r, -1.5*\r)
	(0.5*\r, -0.5*\r) -- (1.5*\r, 0.5*\r)
	(0.5*\r, -0.5*\r) -- (1.5*\r, -1.5*\r)
;

\draw[gray!50]
	(-0.5*\r, 0.5*\r) -- (-1.5*\r, 1.5*\r)
	(-0.5*\r, 0.5*\r) -- (-1.5*\r, -0.5*\r)
	(-0.5*\r, -1.5*\r) -- (-1.5*\r, -0.5*\r)
	(1.5*\r, 0.5*\r) -- (2.5*\r, -0.5*\r)
	(1.5*\r, -1.5*\r) -- (2.5*\r, -0.5*\r)
	(1.5*\r, 0.5*\r) -- (2.5*\r, 1.5*\r)
;

\draw[arrows = {-Latex[scale=0.45]}, gray!50, line width=1.5]
	(-1.5*\r, 1.5*\r) -- (-1.5*\r, 3*\r)
;

\draw[arrows = {-Latex[scale=0.45]}, gray!50, line width=1.5]
	(2.5*\r, 1.5*\r) -- (2.5*\r, 3*\r)
;

\draw[arrows = {-Latex[scale=0.45]}, line width=1.5]
	(0.5*\r, 1.5*\r) -- (0.5*\r, 3*\r)
;

\draw[arrows = {-Latex[scale=0.45]}, line width=1.5]
	(1.5*\r, -1.5*\r) --  (1.5*\r, -3*\r)
;

\draw[arrows = {-Latex[scale=0.45]}, line width=1.5]
	(-0.5*\r, -1.5*\r) --  (-0.5*\r, -3*\r)
;

\draw[dotted]
	(-0.5*\r, -1.5*\r) -- (-0.5*\r, 2.5*\r)
	(1.5*\r, -1.5*\r) -- (1.5*\r, 2.5*\r) 
;

\end{scope} 

\end{scope} 

\begin{scope}[yshift=-\YS cm]

\node at (-6*\r, 0) {\small \parbox[c]{2cm}{Polar \\view:}};

\tikzmath{
\xs=3;
}

\foreach \n in {0,1,2} {
\tikzset{xshift=\xs*\n cm}

\tikzmath{
\r=0.75;
\m=\n+4;
\mm = \m-1;
\th=360/\m;
}

\fill[teal!70!blue!70] (0,0) circle (\r);

\foreach \a in {0,...,\mm} {
\tikzset{rotate=\a*\th}

\fill[white] ($(90:\r) !1/2! (90-\th:\r)$) -- ($(90:\r) !1/2! (90+\th:\r)$) -- ($(90+\th:\r) !1/2! (90+2*\th:\r)$) -- ($(90+2*\th:\r) !1/2! (90+3*\th:\r)$);

}

\foreach \a in {0,...,\mm} {
\tikzset{rotate=\a*\th}

\draw[arrows = {-Latex[scale=0.45]}, orange!75!brown!75!black!80, line width=1.5]
	(0,0) -- ($($(90:\r) !1/2! (90-\th:\r)$) !1/2! ($(90:\r) !1/2! (90+\th:\r)$)$)
	(90-0.5*\th:\r) --(90-0.5*\th:1.5*\r) 
;

\draw[]
	(90:\r) -- (90+\th:\r)
	($(90:\r) !1/2! (90-\th:\r)$) -- ($(90:\r) !1/2! (90+\th:\r)$) 
;
}

\draw[] (0,0) circle (\r);
}

\end{scope} 

\begin{scope}[yshift=-2*\YS cm]

\tikzmath{
\s=1;
\xs=3;
}

\begin{scope}[] 
\node [inner sep=0] (image) at (0,0) 
            {\includegraphics[height=2cm]{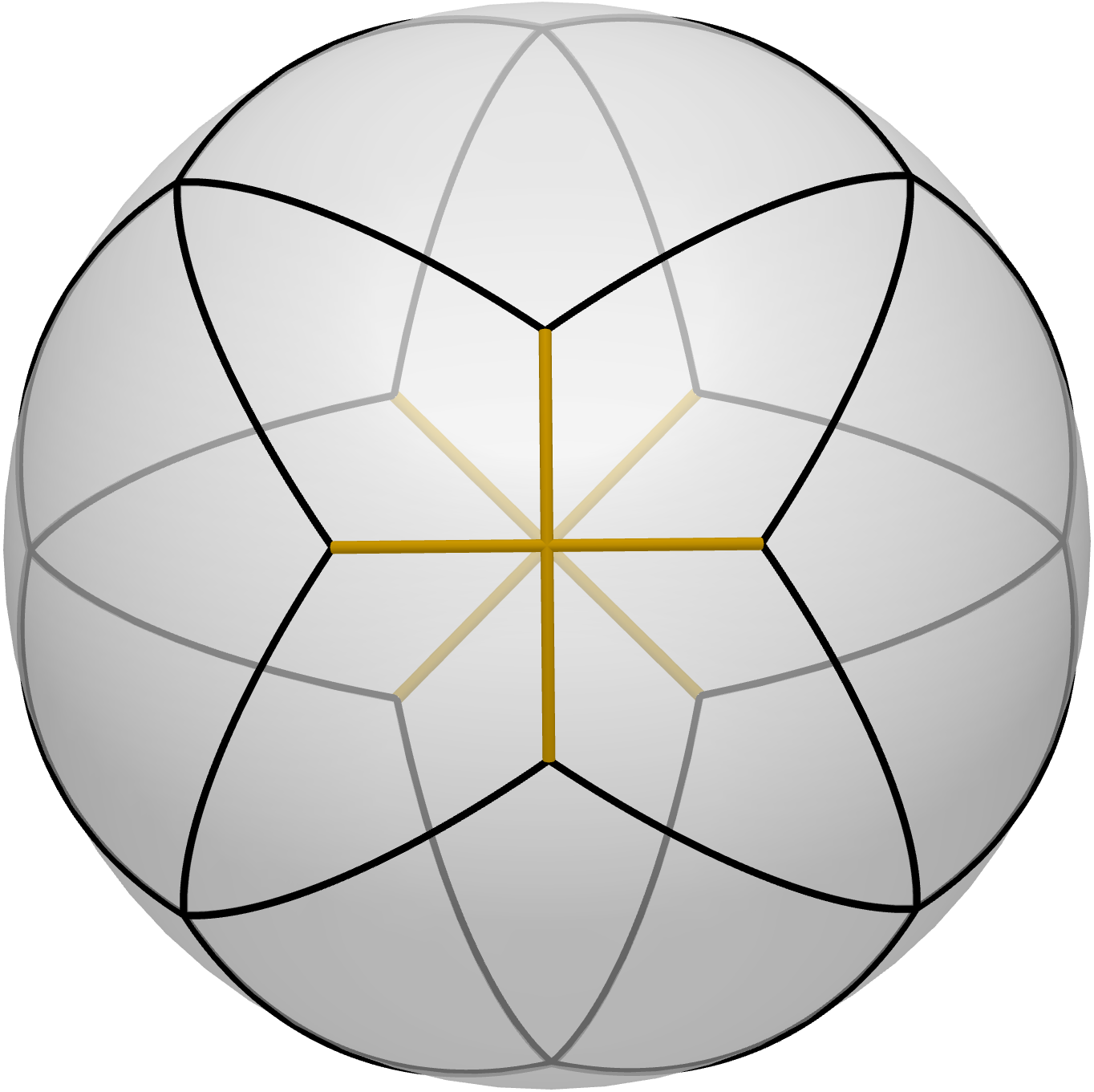}};
\end{scope}

\begin{scope}[xshift=\xs cm] 
\node [inner sep=0] (image) at (0,0) 
            {\includegraphics[height=2cm]{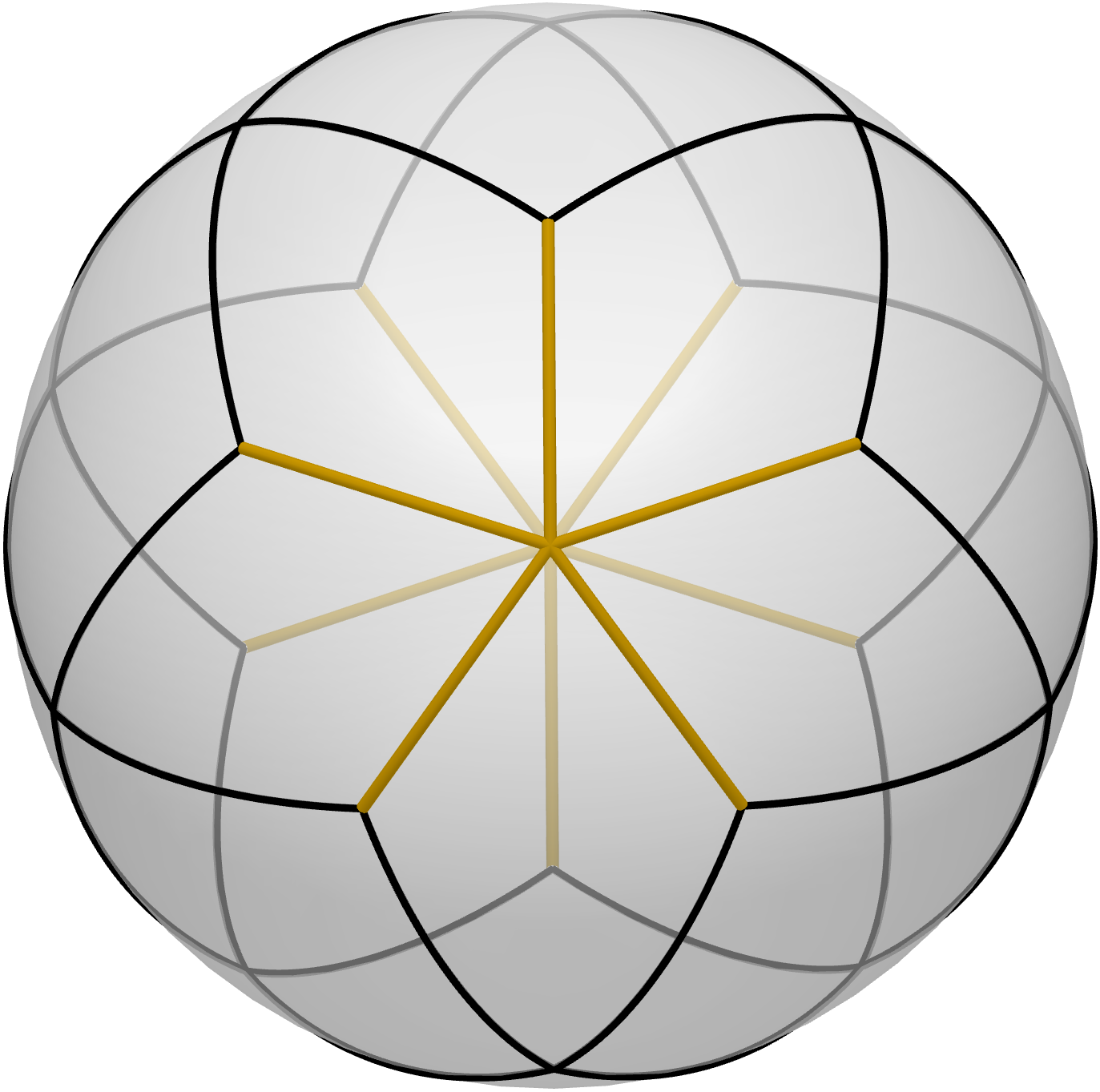}};
\end{scope}

\begin{scope}[xshift=2*\xs cm] 
\node [inner sep=0] (image) at (0,0) 
            {\includegraphics[height=2cm]{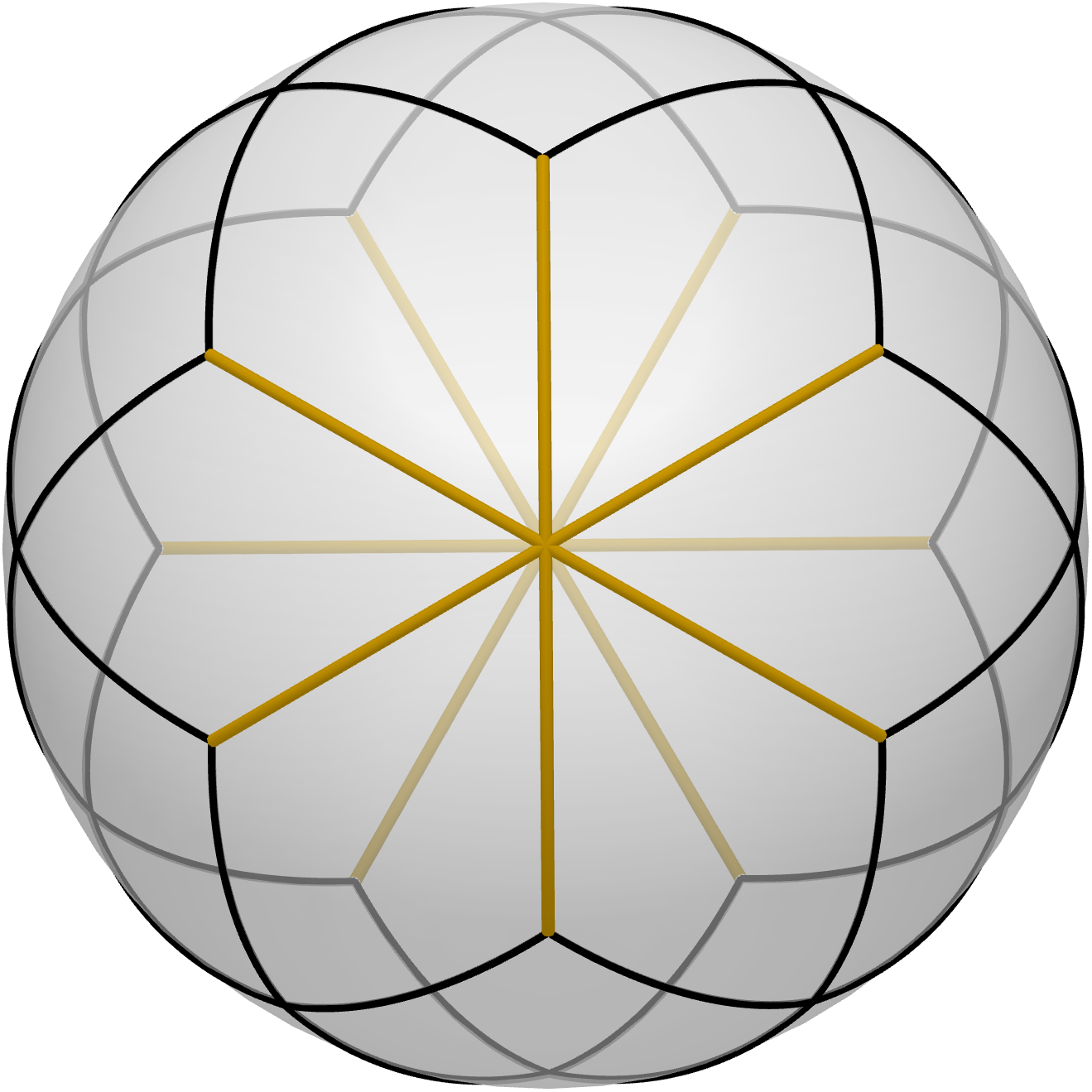}};
\end{scope}

\end{scope} 

\end{tikzpicture}
\caption{\ref{Label:EMT-Antiprisms} via a kite subdivision of the polar polygons in antiprisms}
\label{Fig-EMT-Antiprisms}
\end{figure}

In Figure \ref{Fig-EMT-Prisms}, the first row shows two time zones of \ref{Label:EMT-Prisms} and an alternative composition of a time zone. The infinite family is again obtained by gluing $d$ copies of a time zone for $d\ge3$. The modification flips the bottom part (or equivalently the top part) of the tiling which results in the time zones in the second picture of the first row. The tilings with three time zones are drawn from a polar view in the second row, where the three kites in a kite subdivision are incident to a pole. The modifications are equivalent to flipping one subdivided polar polygon. The first three members and their corresponding modification in $3$D are shown in Figure \ref{Fig-EMT-Prisms-3D}.


\begin{figure}[h!] 
\centering
\begin{tikzpicture}[>=latex]

\tikzmath{
\s=0.75;
\ys=2.25;
\r=0.325;
\th=360/4;
\x=\r*cos(0.5*\th);
}

\begin{scope}

\node at (-5*\r, 0) {\small  \parbox[c]{2cm}{Equatorial \\ view:}};

\begin{scope}

\fill[gray!50]
	(-2*\x, 3*\x) -- (-2*\x, -3*\x) -- (2*\x, -3*\x) -- (2*\x, 3*\x)
;

\foreach \xs in {0,1} {

\tikzset{xshift=4*\xs*\x cm}

\foreach \aa in {-1,1} {

\tikzset{shift={(\aa*\x,0)}}

\foreach \a in {0,1,2,3} {

\tikzset{rotate=\a*\th}

\draw[]
	(0.5*\th:\r) -- (1.5*\th:\r)
;


}
}

\foreach \aa in {-1,1} {

\tikzset{xscale=\aa}

\draw[arrows = {-Latex[scale=0.45]}, line width=1.5]
	(2*\x,\x) -- (2*\x,3.65*\x)
;

\draw[arrows = {-Latex[scale=0.45]}, line width=1.5]
	(2*\x,-\x) -- (2*\x,-3.65*\x)
;

}



}



\node at (8.25*\x,0) {\LARGE $\cdots$};


\end{scope}

\begin{scope}[xshift=5*\s cm]

\fill[gray!50]
	(-2*\x, 3*\x) -- (-2*\x, -\x) -- (0, -\x) -- (0, -3*\x) -- (4*\x, -3*\x) -- (4*\x, -\x) -- (2*\x, -\x) -- (2*\x, 3*\x)
;

\foreach \xs in {0,1} {

\tikzset{xshift=4*\xs*\x cm}

\foreach \aa in {-1,1} {

\tikzset{shift={(\aa*\x,0)}}

\foreach \a in {0,1,2,3} {

\tikzset{rotate=\a*\th}

\draw[]
	(0.5*\th:\r) -- (1.5*\th:\r)
;


}
}

\foreach \aa in {-1,1} {
\tikzset{xscale=\aa}

\draw[arrows = {-Latex[scale=0.45]}, line width=1.5]
	(2*\x,\x) -- (2*\x,3.65*\x)
;

\draw[arrows = {-Latex[scale=0.45]}, line width=1.5, opacity=0.25]
	(4*\x,-\x) -- (4*\x,-3.65*\x)
;

\draw[opacity=0.25]
	(2*\x,-\x) -- (4*\x,-\x)
;
}

\draw[arrows = {-Latex[scale=0.45]}, line width=1.5]
	(0, -\x) -- (0, -3.65*\x)
;


}






\node at (8.75*\x,0) {\LARGE $\cdots$};

\end{scope}

\end{scope}

\begin{scope}[yshift=-\ys cm]

\tikzmath{
\th=360/3;
}

\node at (-5*\r, 0) {\small  \parbox[c]{2cm}{Polar \\ view:}};

\begin{scope}[xshift=0.6*\s cm]

\draw[fill=teal!70!blue!70] (0,0) circle (2.5*\r);

\draw[fill=white] (0,0) circle (\r);

\foreach \a in {0,1,2} {
\tikzset{rotate=\a*\th}

\draw[]
	(0,\r) -- (0,2.5*\r)
	(90-0.5*\th:\r) -- (90-0.5*\th:2.5*\r)
;

\draw[arrows = {-Latex[scale=0.45]}, orange!75!brown!75!black!80, line width=1.5]
	(0,0) -- (0,\r)
	(0,2.5*\r) -- (0,3.75*\r)
;
}


\end{scope}

\begin{scope}[xshift=5.5*\s cm]

\draw[fill=teal!70!blue!70] (0,0) circle (2.5*\r);

\draw[fill=white] (0,0) circle (\r);

\foreach \a in {0,1,2} {

\tikzset{rotate=\a*\th}

\draw[]
	(0,\r) -- (0,2.5*\r)
	(90-0.5*\th:\r) -- (90-0.5*\th:2.5*\r)
;

\draw[arrows = {-Latex[scale=0.45]}, orange!75!brown!75!black!80, line width=1.5]
	(0,0) -- (0,\r)
	(90-0.5*\th:2.5*\r) -- (90-0.5*\th:3.75*\r)
;

}


\end{scope}

\end{scope}

\end{tikzpicture}
\caption{\ref{Label:EMT-Prisms} via a kite subdivision of the polar polygons in the prisms and the flip}
\label{Fig-EMT-Prisms}
\end{figure}
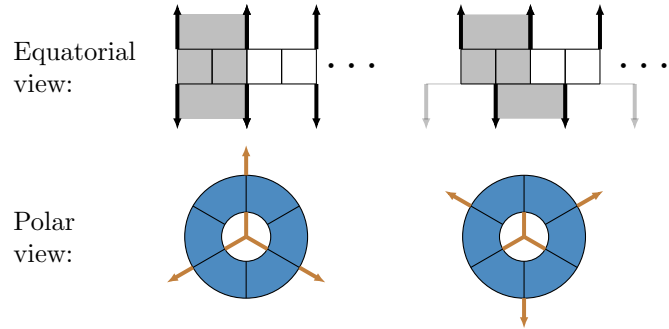


\begin{figure}[h!]
\centering
\begin{subfigure}{\linewidth}
\centering
\begin{tikzpicture}
\tikzmath{
\XS=2.5;
}

\begin{scope}[] 
\node [inner sep=0] (image) at (0,0) 
            {\includegraphics[height=2cm]{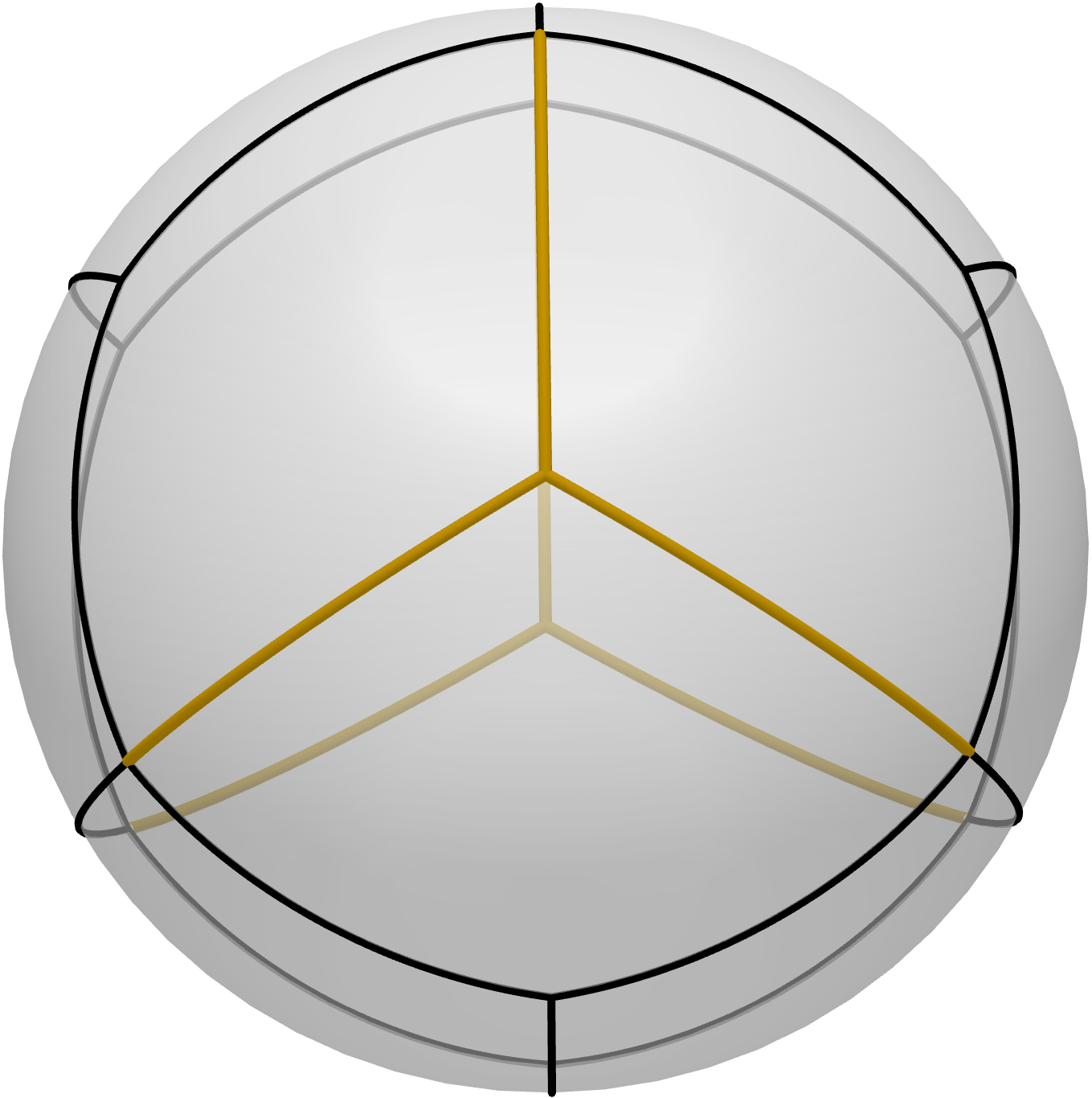}};
\end{scope}

\begin{scope}[xshift=\XS cm] 
\node [inner sep=0] (image) at (0,0) 
            {\includegraphics[height=2cm]{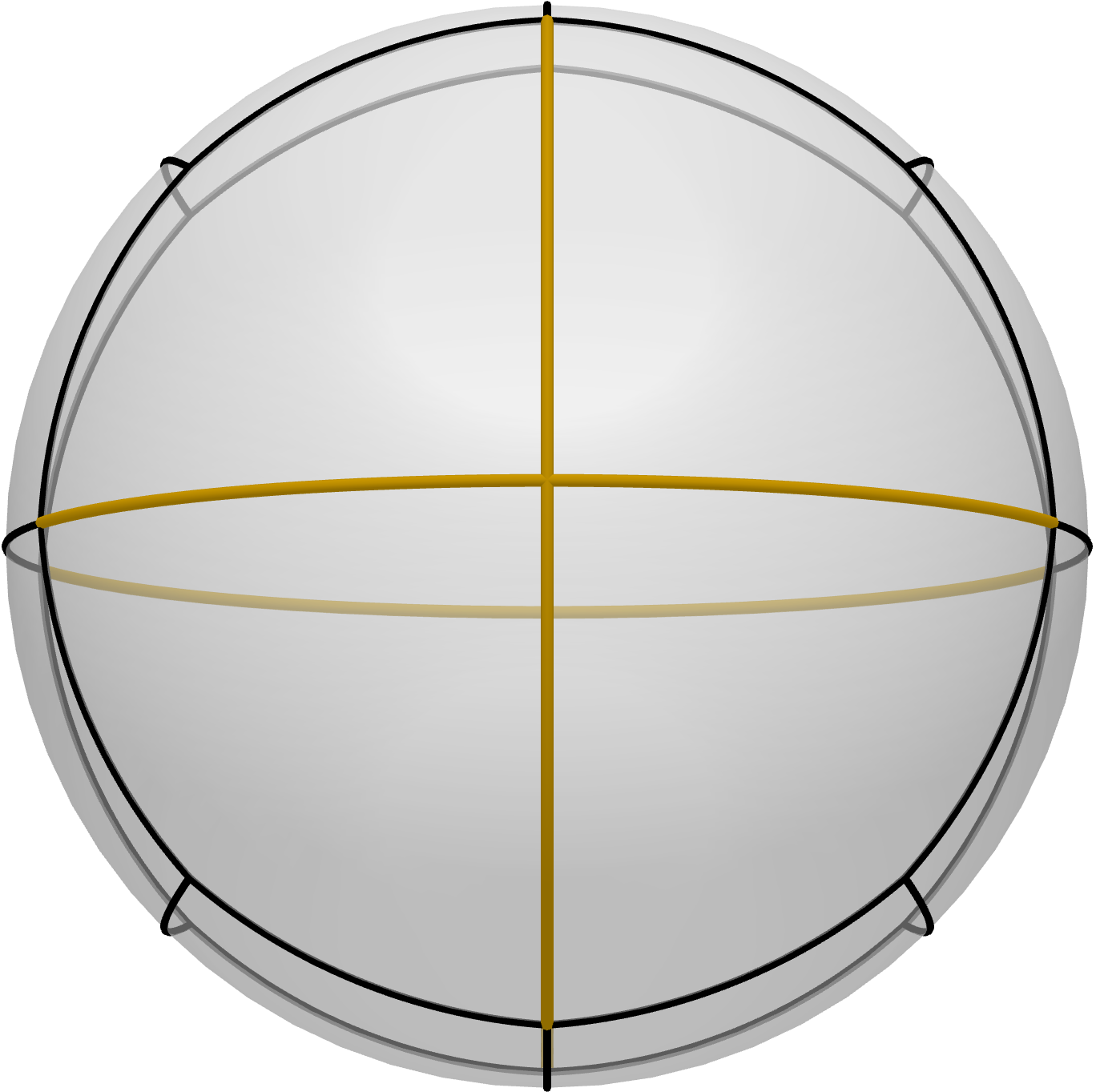}};
\end{scope}

\begin{scope}[xshift=2*\XS cm] 
\node [inner sep=0] (image) at (0,0) 
            {\includegraphics[height=2cm]{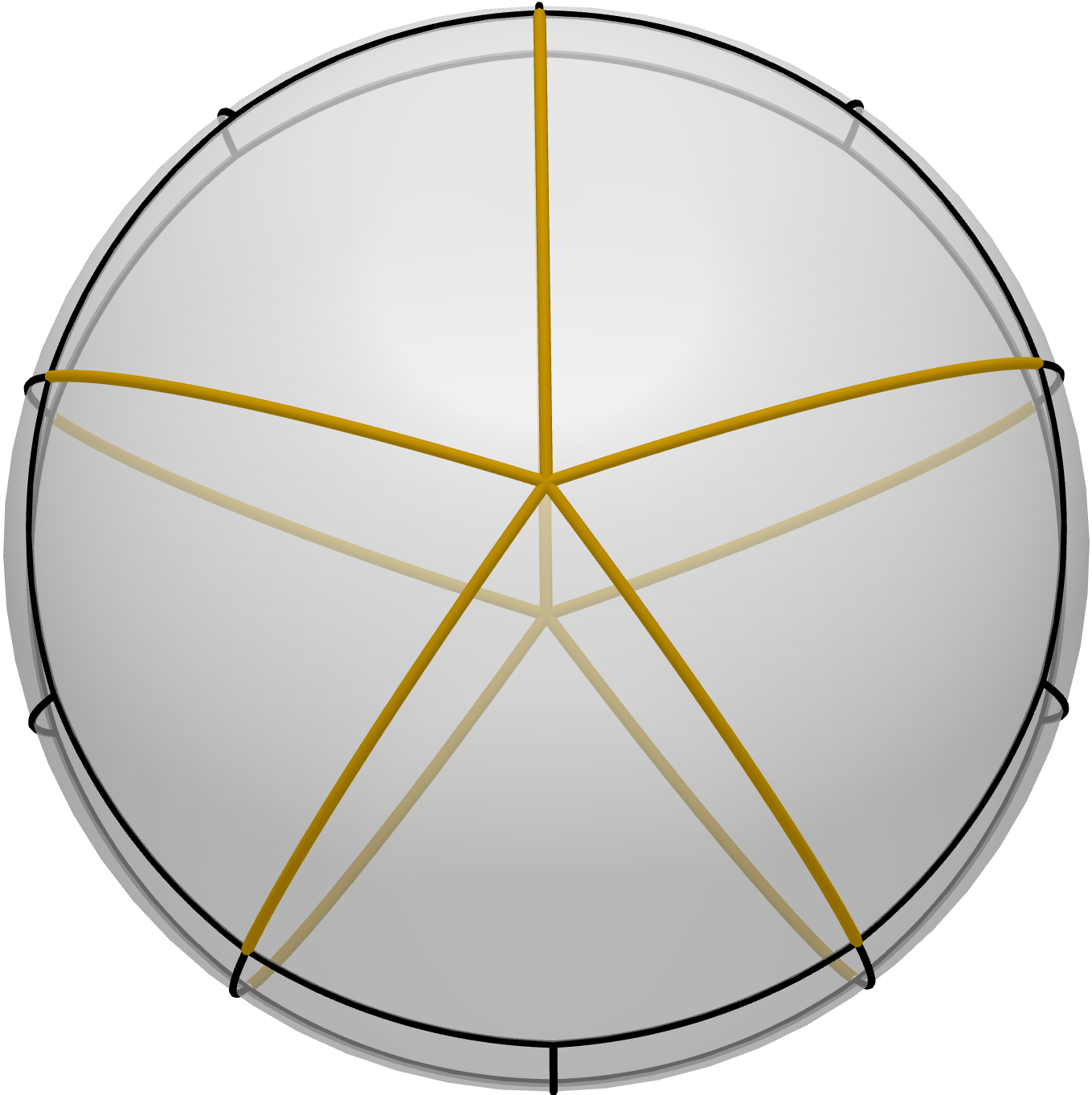}};
\end{scope}

\end{tikzpicture}
\caption{\ref{Label:EMT-Prisms} via subdividing prisms}
\end{subfigure}
\begin{subfigure}{\linewidth}
\centering
\begin{tikzpicture}
\tikzmath{
\XS=2.5;
}

\begin{scope}[] 
\node [inner sep=0] (image) at (0,0) 
            {\includegraphics[height=2cm]{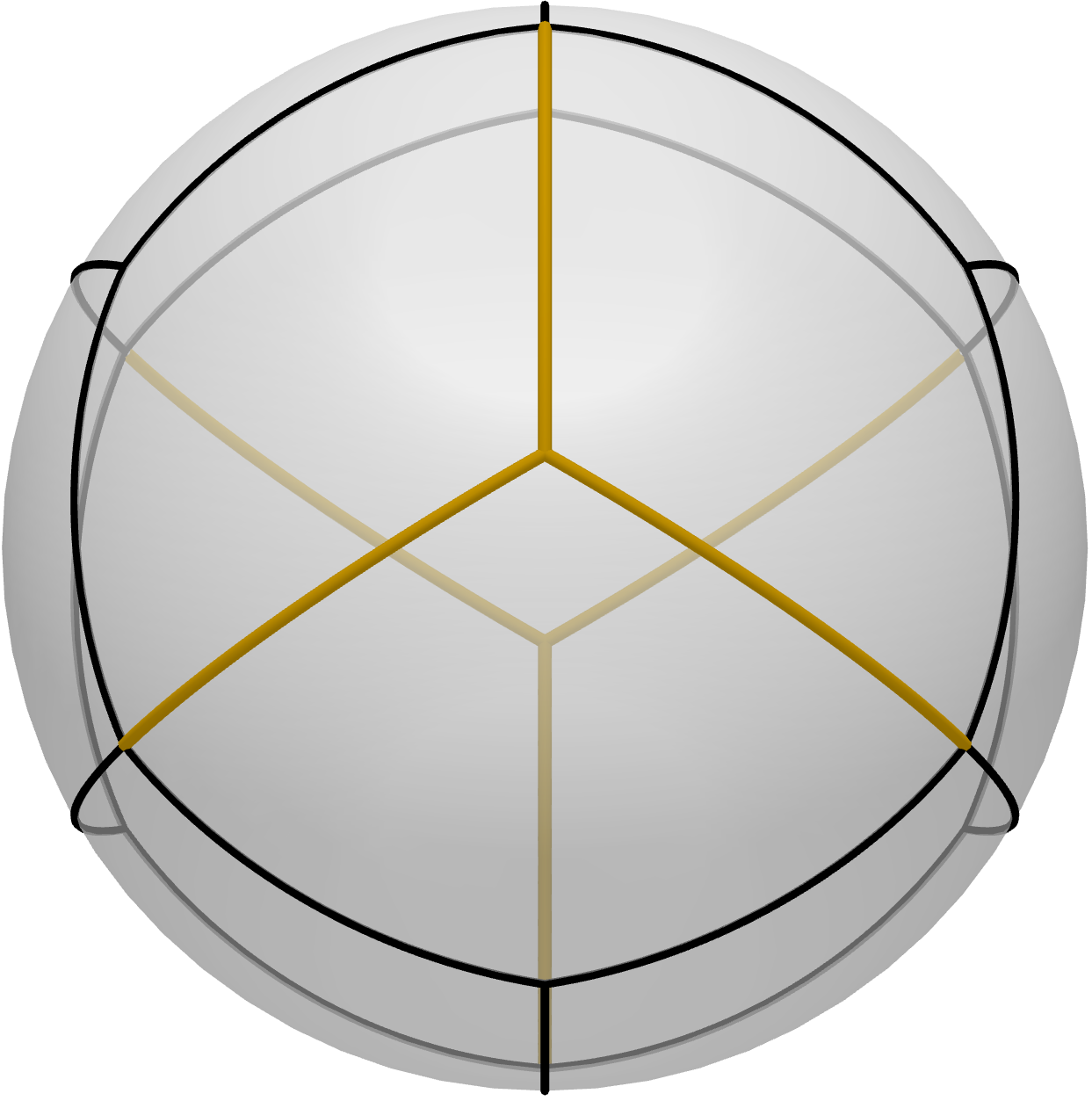}};
\end{scope}

\begin{scope}[xshift=\XS cm] 
\node [inner sep=0] (image) at (0,0) 
            {\includegraphics[height=2cm]{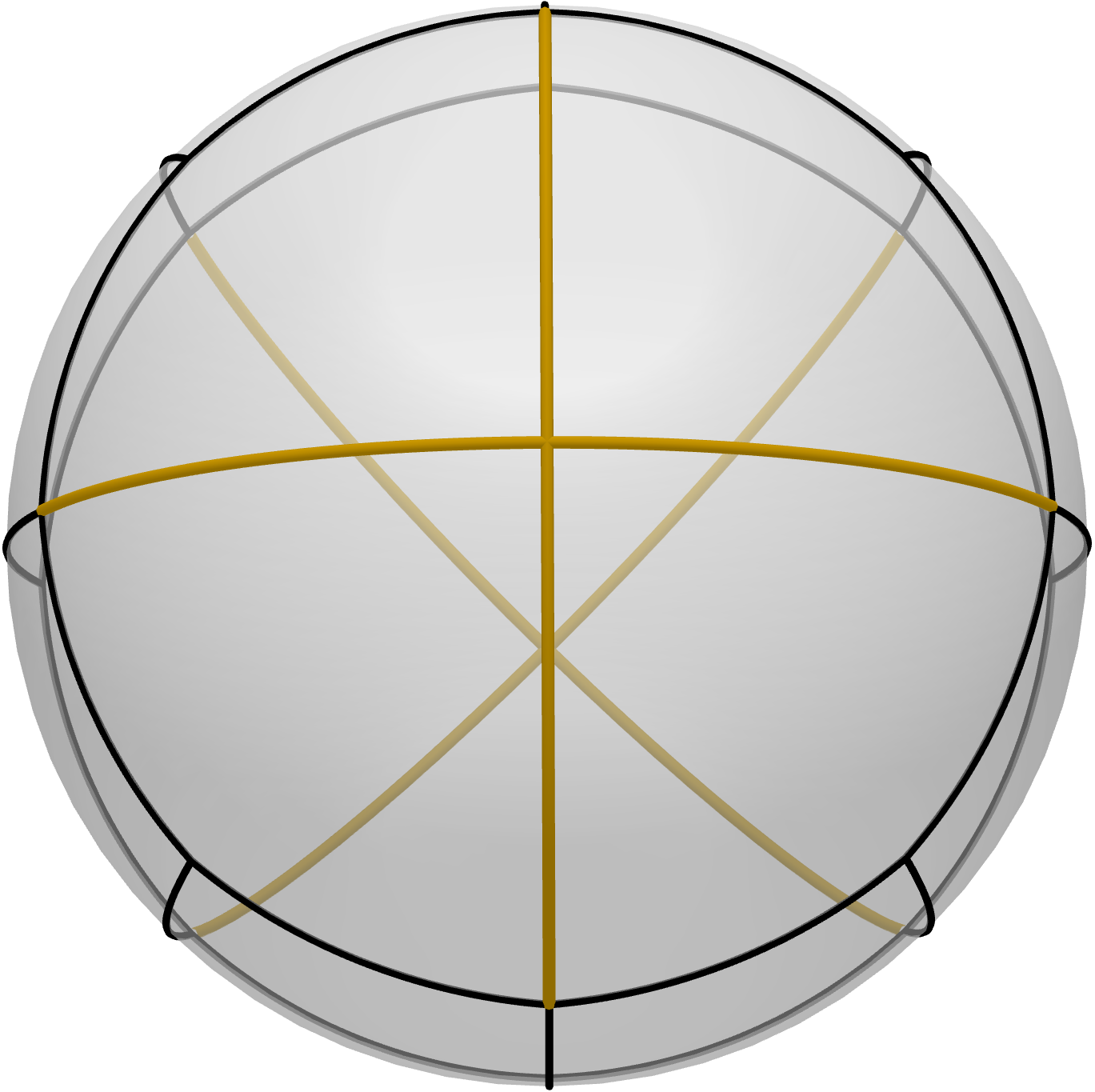}};
\end{scope}

\begin{scope}[xshift=2*\XS cm] 
\node [inner sep=0] (image) at (0,0) 
            {\includegraphics[height=2cm]{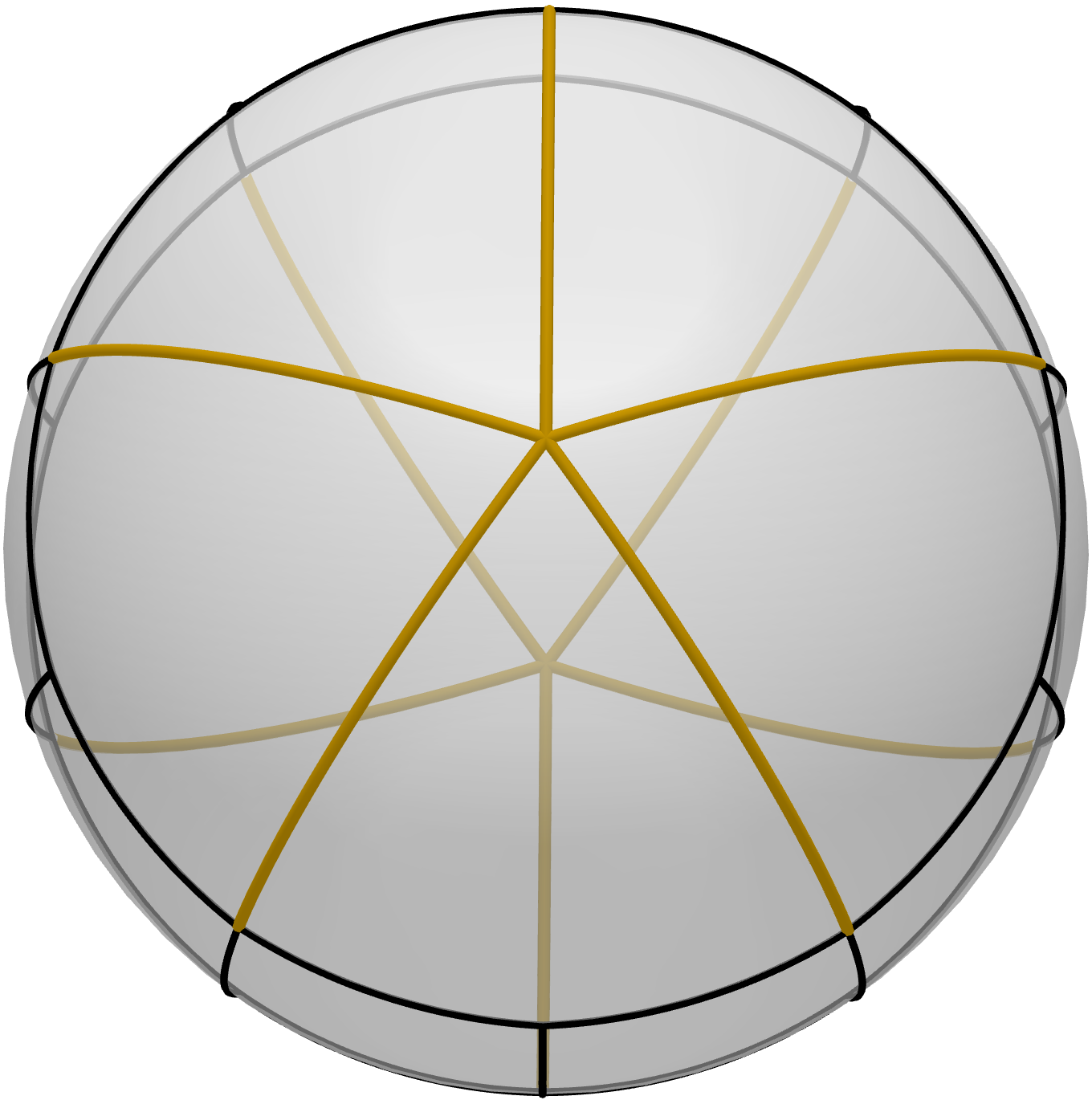}};
\end{scope}

\end{tikzpicture}
\caption{\ref{Label:EMT-Prisms} via flipping subdivided prisms}
\end{subfigure}
\caption{The first three tilings in \ref{Label:EMT-Prisms} and their flip in $3$D}
\label{Fig-EMT-Prisms-3D}
\end{figure}


\begin{figure}[h!]
\centering
\begin{subfigure}[t]{0.325\linewidth}
\centering
\begin{tikzpicture}
\tikzmath{
\r=0.35;
\n=6;
\nn=\n-1;
\nnn=\n/2-1;
\th=360/\n;
\XS=2.1;
\R=2.8*\r;
}

\fill[teal!75!blue]
	(90:\R) -- (90+\th:\R) -- (90+2*\th:\R) -- (90+3*\th:\R) -- (90+4*\th:\R) -- (90+5*\th:\R) -- cycle
;

\fill[white]
	(90:2.5*\r) -- (90+\th:2.5*\r) -- (90+2*\th:2.5*\r) -- (90+3*\th:2.5*\r) -- (90+4*\th:2.5*\r) -- (90+5*\th:2.5*\r) -- cycle
;

\fill[teal!75!blue]
	(90:\r) -- (90+\th:\r) -- (90+2*\th:\r) -- (90+3*\th:\r) -- (90+4*\th:\r) -- (90+5*\th:\r) -- cycle
;

\foreach \a in {0,...,\nn} {
\tikzset{rotate=\a*\th}
\draw[]
	(90:\r) -- (90+\th:\r)
	(90:2.5*\r) -- (90+\th:2.5*\r) 
;

\draw[HGold, line width=1.5]
	(90:1.75*\r) -- (90+\th:1.75*\r) 
;
}

\foreach \a in {0,...,\nnn} {
\tikzset{rotate=2*\a*\th}
\draw[]
	(90:\r) -- (90:1.75*\r) 
	(270:1.75*\r) -- (270:2.5*\r) 
;
\draw[HGold, line width=1.5]
	(270:\r) -- (270:1.75*\r) 
	(90:1.75*\r) -- (90:2.5*\r) 
;
}

\begin{scope}[] 
\node [inner sep=0] (image) at (\XS,0) 
            {\includegraphics[height=2cm]{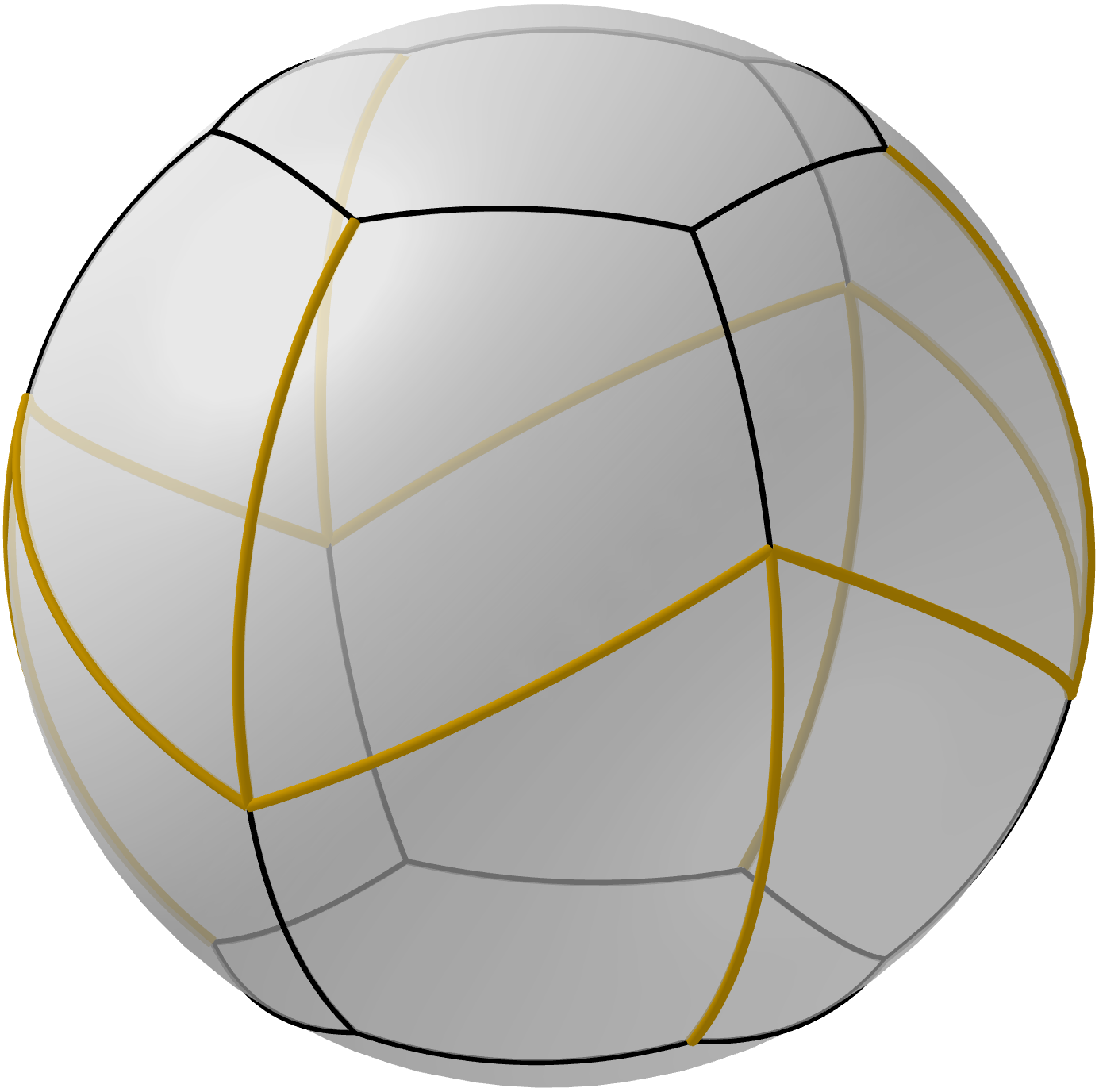}};
\end{scope}

\end{tikzpicture}
\caption{$m=6$}
\end{subfigure}
\begin{subfigure}[t]{0.325\linewidth}
\centering
\begin{tikzpicture}
\tikzmath{
\r=0.35;
\n=8;
\nn=\n-1;
\nnn=\n/2-1;
\th=360/\n;
\XS=2.2;
\R=2.8*\r;
}

\fill[teal!75!blue]
	(90:\R) -- (90+\th:\R) -- (90+2*\th:\R) -- (90+3*\th:\R) -- (90+4*\th:\R) -- (90+5*\th:\R) -- (90+6*\th:\R) -- (90+7*\th:\R) -- cycle
;

\fill[white]
	(90:2.5*\r) -- (90+\th:2.5*\r) -- (90+2*\th:2.5*\r) -- (90+3*\th:2.5*\r) -- (90+4*\th:2.5*\r) -- (90+5*\th:2.5*\r) -- (90+6*\th:2.5*\r) -- (90+7*\th:2.5*\r) -- cycle
;

\fill[teal!75!blue]
	(90:\r) -- (90+\th:\r) -- (90+2*\th:\r) -- (90+3*\th:\r) -- (90+4*\th:\r) -- (90+5*\th:\r) -- (90+6*\th:\r) -- (90+7*\th:\r) -- cycle
;

\foreach \a in {0,...,\nn} {
\tikzset{rotate=\a*\th}
\draw[]
	(90:\r) -- (90+\th:\r)
	(90:2.5*\r) -- (90+\th:2.5*\r) 
;

\draw[HGold, line width=1.5]
	(90:1.75*\r) -- (90+\th:1.75*\r) 
;
}

\foreach \a in {0,...,\nnn} {
\tikzset{rotate=2*\a*\th}
\draw[]
	(90:\r) -- (90:1.75*\r) 
	(90+\th:1.75*\r) -- (90+\th:2.5*\r) 
;
\draw[HGold, line width=1.5]
	(90:1.75*\r) -- (90:2.5*\r) 
	(270+\th:\r) -- (270+\th:1.75*\r) 
;
}

\begin{scope}[] 
\node [inner sep=0] (image) at (\XS,0) 
            {\includegraphics[height=2cm]{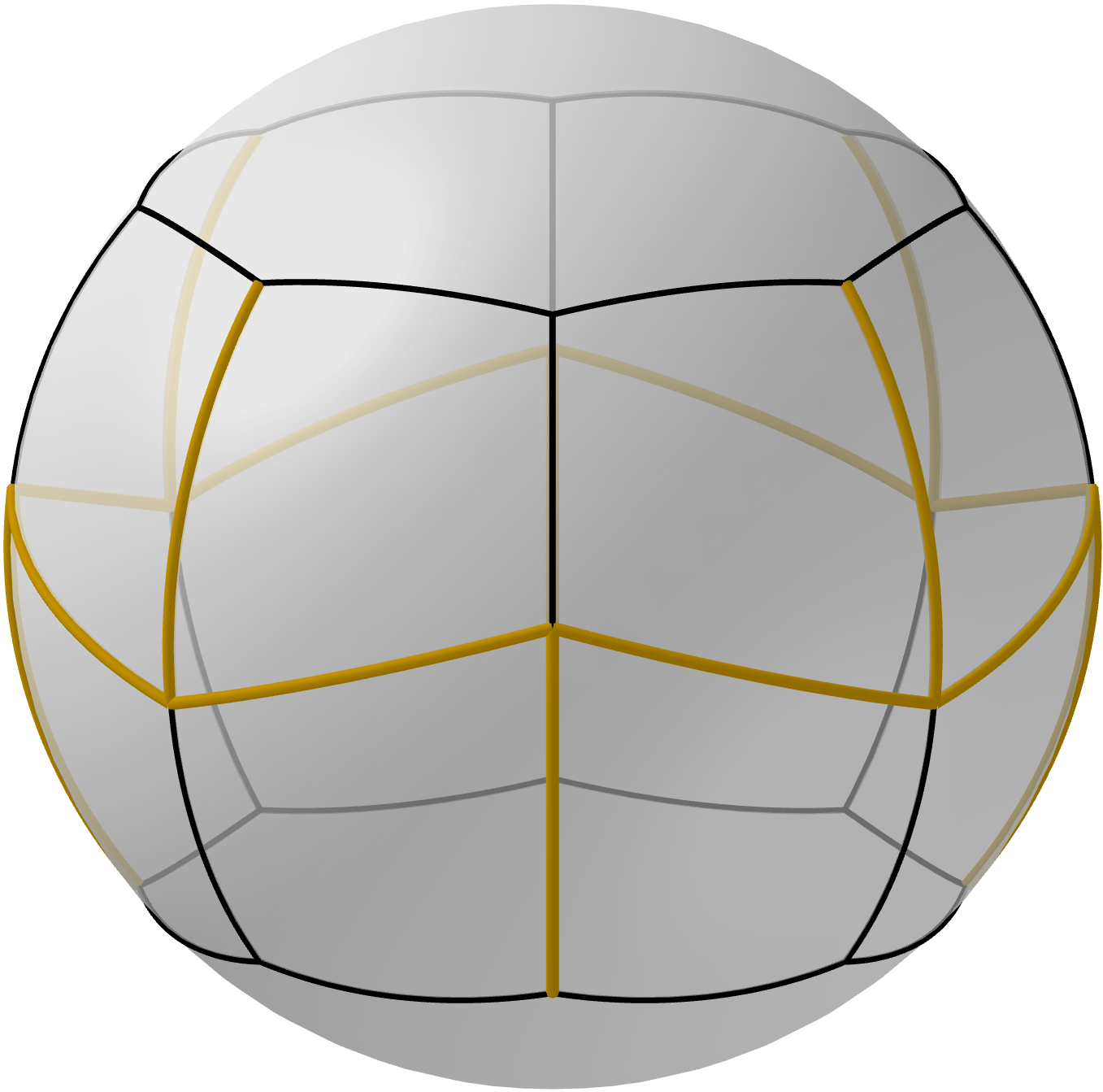}};
\end{scope}

\end{tikzpicture}
\caption{$m=8$}
\end{subfigure}
\begin{subfigure}[t]{0.325\linewidth}
\centering
\begin{tikzpicture}
\tikzmath{
\r=0.35;
\n=10;
\nn=\n-1;
\nnn=\n/2-1;
\th=360/\n;
\XS=2.15;
\R=2.8*\r;
}

\fill[teal!75!blue]
	(90:\R) -- (90+\th:\R) -- (90+2*\th:\R) -- (90+3*\th:\R) -- (90+4*\th:\R) -- (90+5*\th:\R) -- (90+6*\th:\R) -- (90+7*\th:\R) -- (90+8*\th:\R) -- (90+9*\th:\R) -- cycle
;

\fill[white]
	(90:2.5*\r) -- (90+\th:2.5*\r) -- (90+2*\th:2.5*\r) -- (90+3*\th:2.5*\r) -- (90+4*\th:2.5*\r) -- (90+5*\th:2.5*\r) -- (90+6*\th:2.5*\r) -- (90+7*\th:2.5*\r) -- (90+8*\th:2.5*\r) -- (90+9*\th:2.5*\r)  -- cycle
;

\fill[teal!75!blue]
	(90:\r) -- (90+\th:\r) -- (90+2*\th:\r) -- (90+3*\th:\r) -- (90+4*\th:\r) -- (90+5*\th:\r) -- (90+6*\th:\r) -- (90+7*\th:\r) -- (90+8*\th:\r) -- (90+9*\th:\r) -- cycle
;

\foreach \a in {0,...,\nn} {
\tikzset{rotate=\a*\th}
\draw[]
	(90:\r) -- (90+\th:\r)
	(90:2.5*\r) -- (90+\th:2.5*\r) 
;

\draw[HGold, line width=1.5]
	(90:1.75*\r) -- (90+\th:1.75*\r) 
;
}

\foreach \a in {0,...,\nnn} {
\tikzset{rotate=2*\a*\th}
\draw[]
	(90:\r) -- (90:1.75*\r) 
	(270:1.75*\r) -- (270:2.5*\r) 
;
\draw[HGold, line width=1.5]
	(90:1.75*\r) -- (90:2.5*\r) 
	(270:\r) -- (270:1.75*\r) 
;
}

\begin{scope}[] 
\node [inner sep=0] (image) at (\XS,0) 
            {\includegraphics[height=2cm]{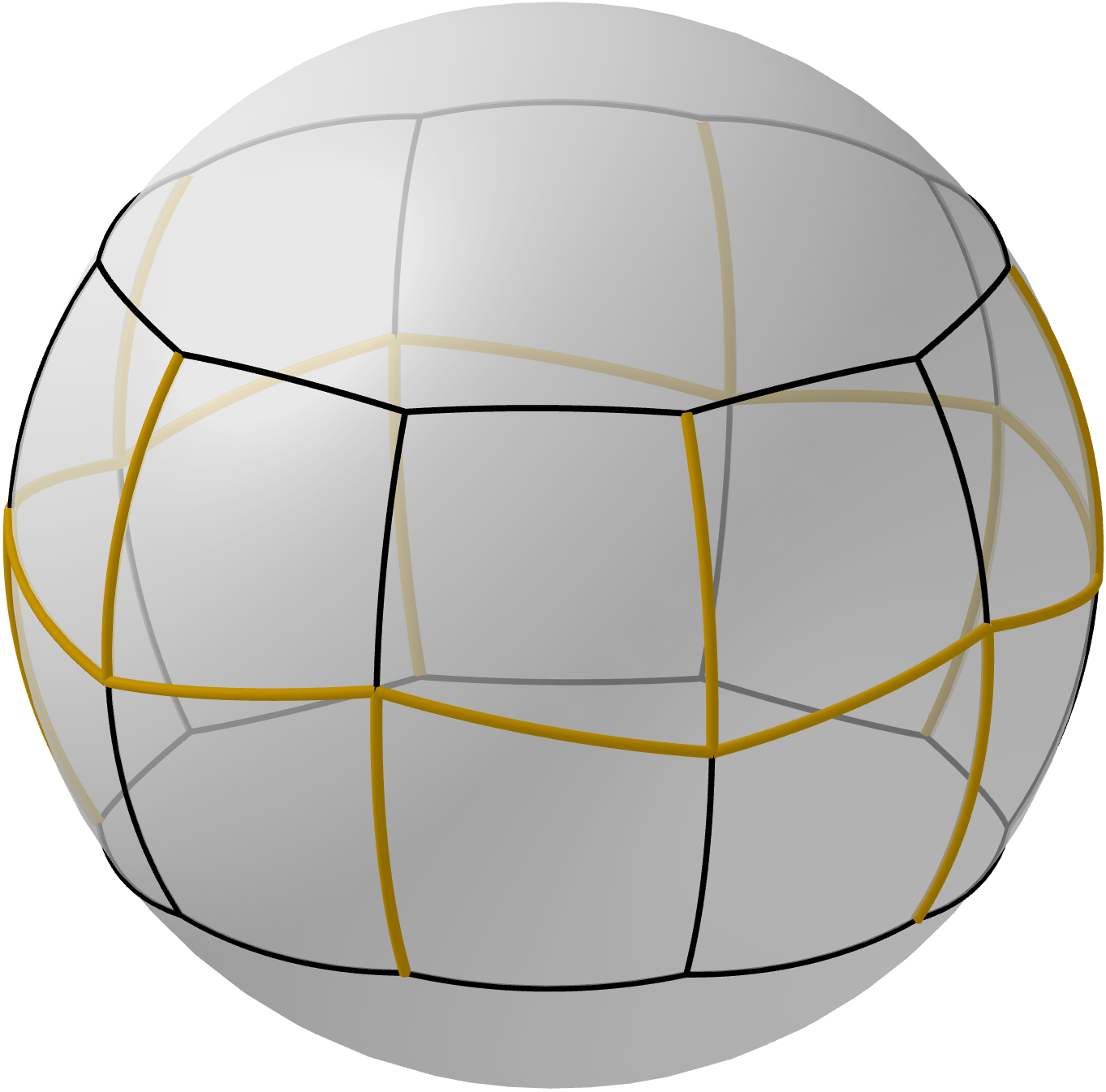}};
\end{scope}

\end{tikzpicture}
\caption{$m=10$}
\end{subfigure}
\caption{The first three members of the \ref{Label:EMT-Subdiv-Equator-Prisms} family with subdivided equatorial tiles}
\label{Fig-EMT-Prisms-Subdiv-Equatorial}
\end{figure}
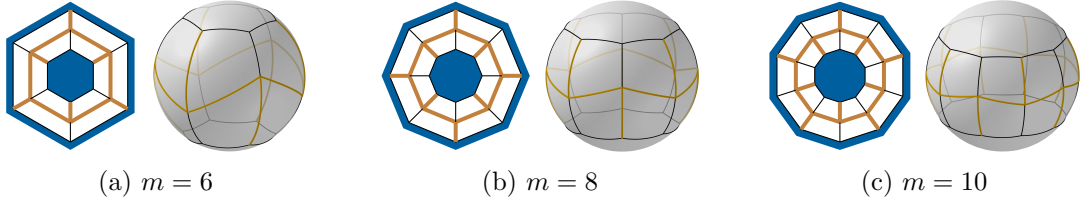

\subsection{Platonic type and Johnson-Zalgaller type}

The tilings of Platonic type or Johnson-Zalgaller type in this paper appear in finite families or in isolation. 

In Figure \ref{Fig-polar-subdiv-cube}, the first picture shows a drawing of the tiling \ref{Label:polar-subdiv-cube}. The drawing, from the perspective of a pole, shows that each of the two antipodal faces of a cube are subdivided into four kites. The kite subdivision is marked by the thick edges. The equatorial faces of the cube are deformed to hexagons, as shown in the $3$D drawing next to it.

In Figure \ref{Fig-Finite-Families-tOcta-tIcosa}, the drawings of a tiling from \ref{Label:trunc-Octa} and a tiling from \ref{Label:trunc-Icosa} without thick edges are respectively the graphs of the truncated octahedron $t\mathcal{O}$ and the truncated icosahedron $t\mathcal{I}$, and the thick edges subdivide each hexagon into three kites. Each tiling generates its corresponding family via independently ``flipping'' the orientation of the subdivision in each tile. Hence we call them {\em seeds}. They are the {\em canonical} choices because the number of vertex types (defined in Section \ref{Subsec-terminology}) is minimum. Details can be seen in Proposition \ref{Prop-be}. Any other member of the family can also be used as a seed.


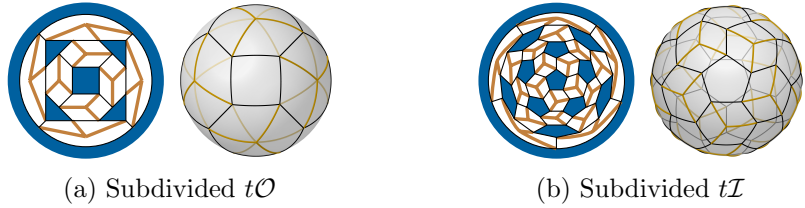
\begin{figure}[h!] 
\centering
\begin{subfigure}{0.4\linewidth}
\centering
\begin{tikzpicture}
\tikzmath{
\s=2.25;
}

\begin{scope}[]
\tikzmath{
\r=0.5;
\dr=0.05*\r;
\rr=0.5*\r;
\R=1.5*\r;
\RR=1.75*\r;
\th=360/4;
\x=\r*cos(\th/2);
\xx=0.5*\x;
\hex=360/6;
}

\fill[teal!75!blue, scale=1.18] (0,0) circle (\RR);

\fill[white] (0,0) circle (\RR);

\foreach \a in {0,1,2,3} {
\tikzset{rotate=\a*\th}

\draw[HGold, line width=1.25]
	(0,\x) -- (2*\xx,\x)
	(0,\x) -- (-\xx,\x+\xx) 
	(0,\x) -- (-\xx,\xx)
	(0,2.1*\x) -- (-\x-\xx,\x+\xx)
	(0,2.1*\x) -- (\xx,\x+\xx)
	(0,2.1*\x) -- (0.5*\th:\RR)
;
}

\foreach \a in {0,1,2,3} {

\tikzset{rotate=\a*\th}

\fill[teal!75!blue]
	(0.5*\th:\rr) -- (1.5*\th:\rr) -- (2.5*\th:\rr) -- (3.5*\th:\rr)
	(2*\xx,\x) -- (3*\xx,\xx) -- (0.5*\th:\R) -- (\xx,\x+\xx)
;
}

\foreach \a in {0,1,2,3} {
\tikzset{rotate=\a*\th}

\draw[]
	(0.5*\th:\rr) -- (1.5*\th:\rr)
	(0.5*\th:\rr) -- (0.5*\th:\r)
	(\xx,\x+\xx) -- (0.5*\th:\r)
	(-\xx,\x+\xx) -- (1.5*\th:\r)
	(0.5*\th:\R) -- (1.5*\th:\R)
	(0.5*\th:\R) -- (0.5*\th:\RR)
;
}

\draw[] (0,0) circle (\RR);

\end{scope}

\begin{scope}[xshift=\s cm]
\node [inner sep=0] (image) at (0,0) 
            {\includegraphics[height=2cm]{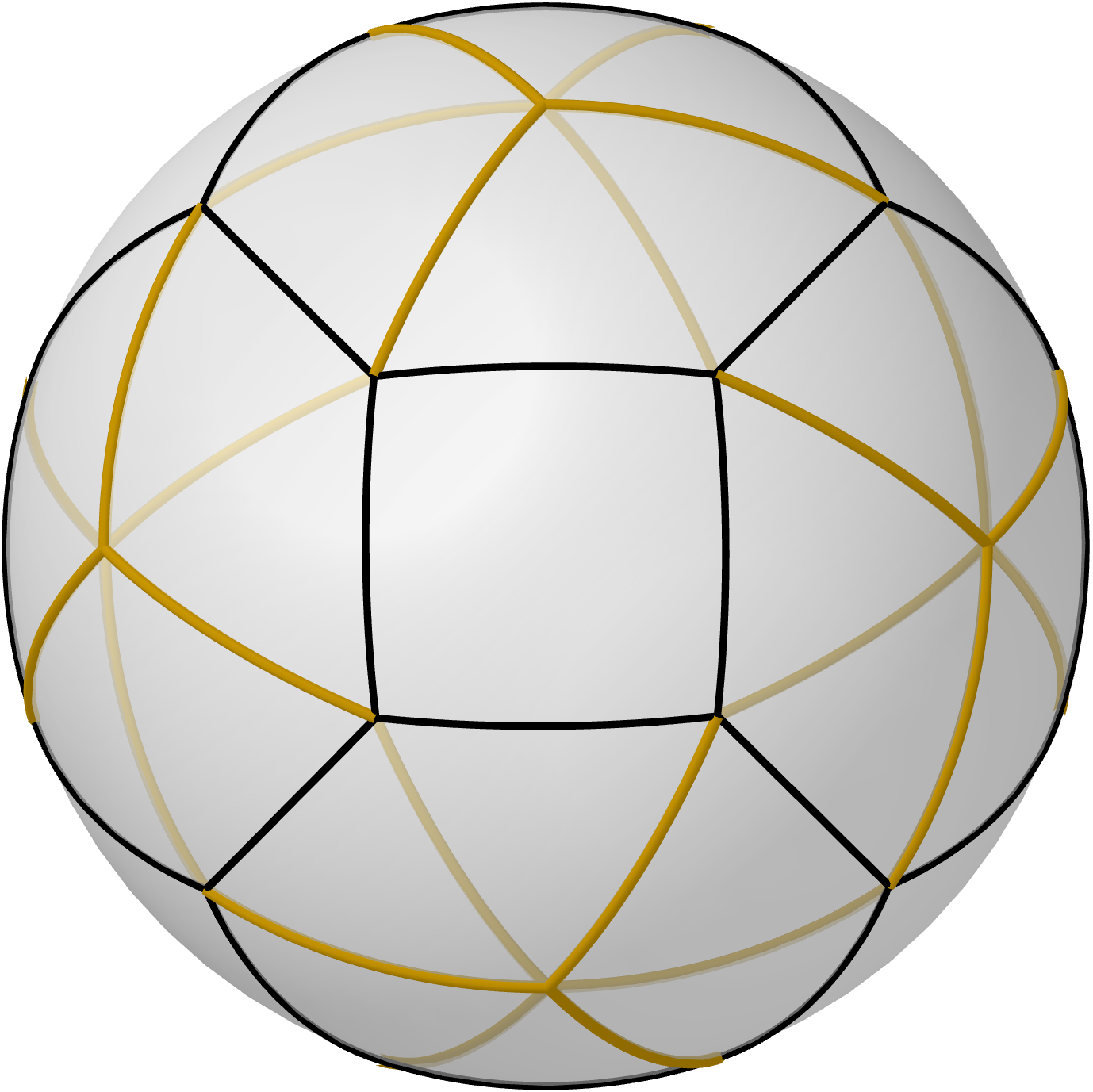}};
\end{scope}
\end{tikzpicture}
\caption{Subdivided $t\mathcal{O}$}
\label{Subfig-subdiv-tO}
\end{subfigure}
\begin{subfigure}{0.4\linewidth}
\centering
\begin{tikzpicture}
\tikzmath{
\s=2.25;
}
\begin{scope}[xshift=2*\s cm]

\tikzmath{
\r=0.375;
\th=360/5;
}

\fill[teal!75!blue] (0,0) circle (2.75*\r);
	
\fill[white] (0,0) circle (2.4*\r);

\fill[teal!75!blue]
	(90:0.4*\r) -- (90+\th:0.4*\r) -- (90+2*\th:0.4*\r) -- (90-2*\th:0.4*\r) -- (90-\th:0.4*\r) -- cycle
;

\foreach \a in {0,...,4}{
\tikzset{rotate=\a*\th}

\draw[HGold, line width=1.25]
	(18:0.4*\r) -- (54:0.75*\r) -- (90:0.8*\r)
	(54:0.75*\r) -- (42:1.1*\r) 
	(66:1.1*\r) -- (54:1.32*\r) -- (30:1.4*\r)
	(54:1.32*\r) -- (66:1.6*\r)
	(114:1.6*\r) -- (90:1.7*\r) -- (78:2*\r)
	(90:1.7*\r) -- (78:1.4*\r)
	(126:2.4*\r) to[out=20,in=180, distance=0.1 cm] 
	(90:2.2*\r) to[out=0,in=135, distance=0.1 cm] (54:2*\r)
	(90:2.2*\r) -- (102:2*\r)
;
}

\foreach \a in {0,...,4}{

\fill[teal!75!blue, rotate=\a*\th]
	(90:0.8*\r) -- (66:1.1*\r) -- (78:1.4*\r) -- (102:1.4*\r) -- (114:1.1*\r) -- cycle
;

\fill[teal!75!blue, rotate=\a*\th]
	(66:1.6*\r) -- (78:2*\r) -- (54:2*\r) -- (30:2*\r) -- (42:1.6*\r)		
;
}

\foreach \a in {0,...,4}
{
\begin{scope}[rotate=72*\a]

\draw
	(18:0.4*\r) -- (90:0.4*\r) -- (90:0.8*\r) -- (66:1.1*\r) -- (42:1.1*\r) -- (18:0.8*\r)
	(66:1.1*\r) -- (78:1.4*\r) -- (102:1.4*\r) -- (114:1.1*\r)
	(78:1.4*\r) -- (66:1.6*\r) -- (42:1.6*\r) -- (30:1.4*\r)
	(66:1.6*\r) -- (78:2*\r) -- (102:2*\r) -- (114:1.6*\r)
	(78:2*\r) -- (54:2*\r) -- (30:2*\r)
	(54:2*\r) -- (54:2.4*\r)
;

\end{scope}
}

\draw[] (0,0) circle (2.4*\r);

\end{scope}

\begin{scope}[xshift=3*\s cm]
\node [inner sep=0] (image) at (0,0) 
            {\includegraphics[height=2cm]{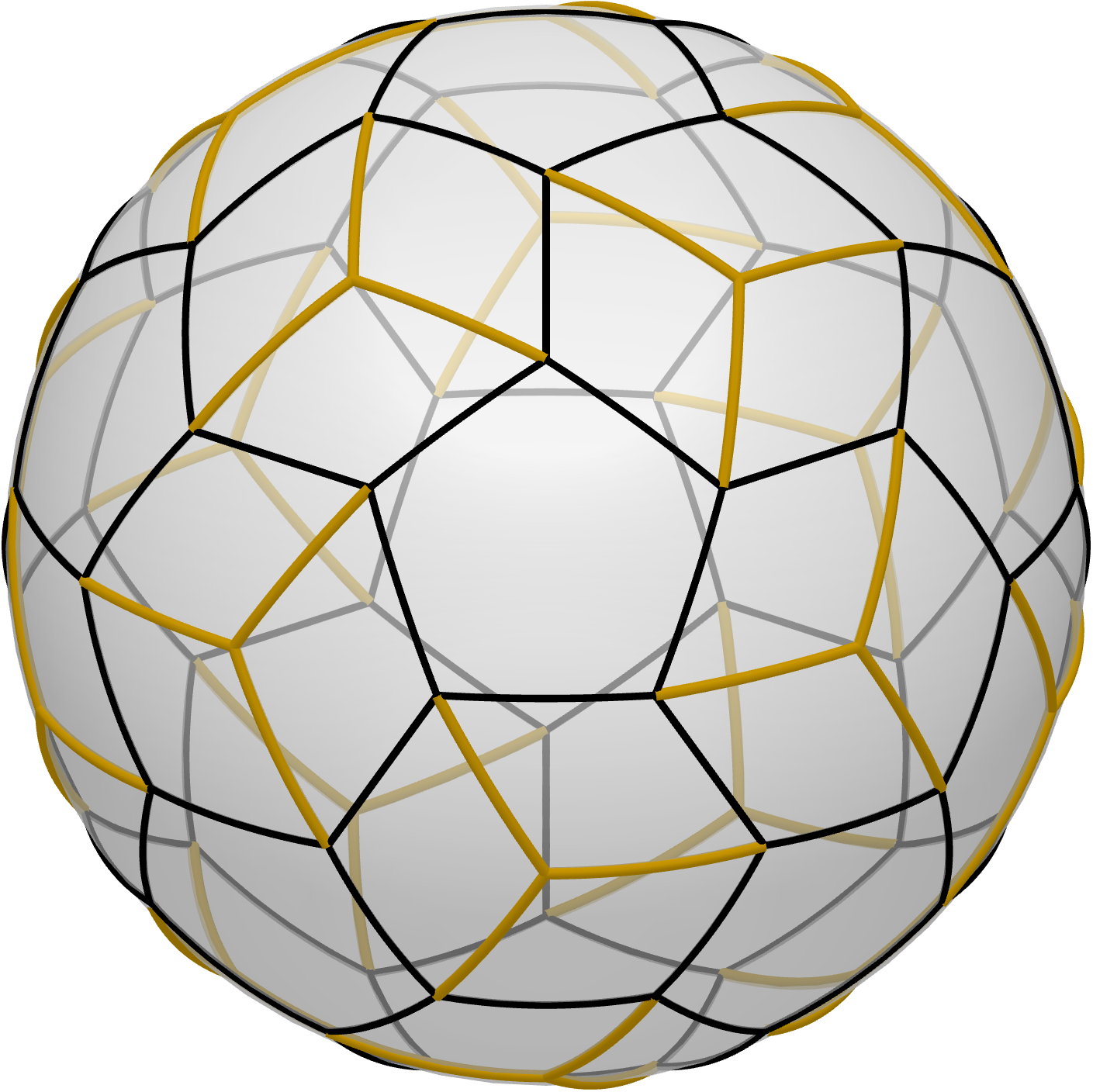}};
\end{scope}
\end{tikzpicture}
\caption{Subdivided $t\mathcal{I}$}
\label{Subfig-subdiv-tI}
\end{subfigure}
\caption{The canonical seeds for generating \ref{Label:trunc-Octa} and \ref{Label:trunc-Icosa}}
\label{Fig-Finite-Families-tOcta-tIcosa}
\end{figure}


\begin{figure}[h!] 
\centering
\begin{tikzpicture}[>=latex]
\tikzmath{
\s=2.25;
\r=0.35;
\n=4;
\nn=\n-1;
\th=360/\n;
\y=\r*sin(45);
}

\foreach \a in {0,...,\nn} {
\tikzset{rotate=\a*\th}
\fill[teal!70!blue!70]
	(90-0.5*\th:1*\r) -- (90-0.5*\th:2.5*\r) -- (90+0.5*\th:2.5*\r) 
	-- (90+0.5*\th:1*\r) -- cycle
;
}

\foreach \a in {0,...,\nn} {
\tikzset{rotate=\a*\th}
\draw[HGold, line width=1.5]
	(0,0) -- (90:\y)
;
\draw[arrows = {-Latex[scale=0.45]}, HGold, line width=1.5]
	(90:2.5*\y) -- (90:4*\y)
;
}

\foreach \a in {0,...,\nn} {
\tikzset{rotate=\a*\th}
\draw[]
	(90-0.5*\th:1*\r) -- (90+0.5*\th:1*\r)
	(90-0.5*\th:1*\r) -- (90-0.5*\th:2.5*\r) 
	(90-0.5*\th:2.5*\r) -- (90+0.5*\th:2.5*\r)
;
}

\begin{scope}[xshift=\s cm]
\node [inner sep=0] (image) at (0,0) 
            {\includegraphics[height=2cm]{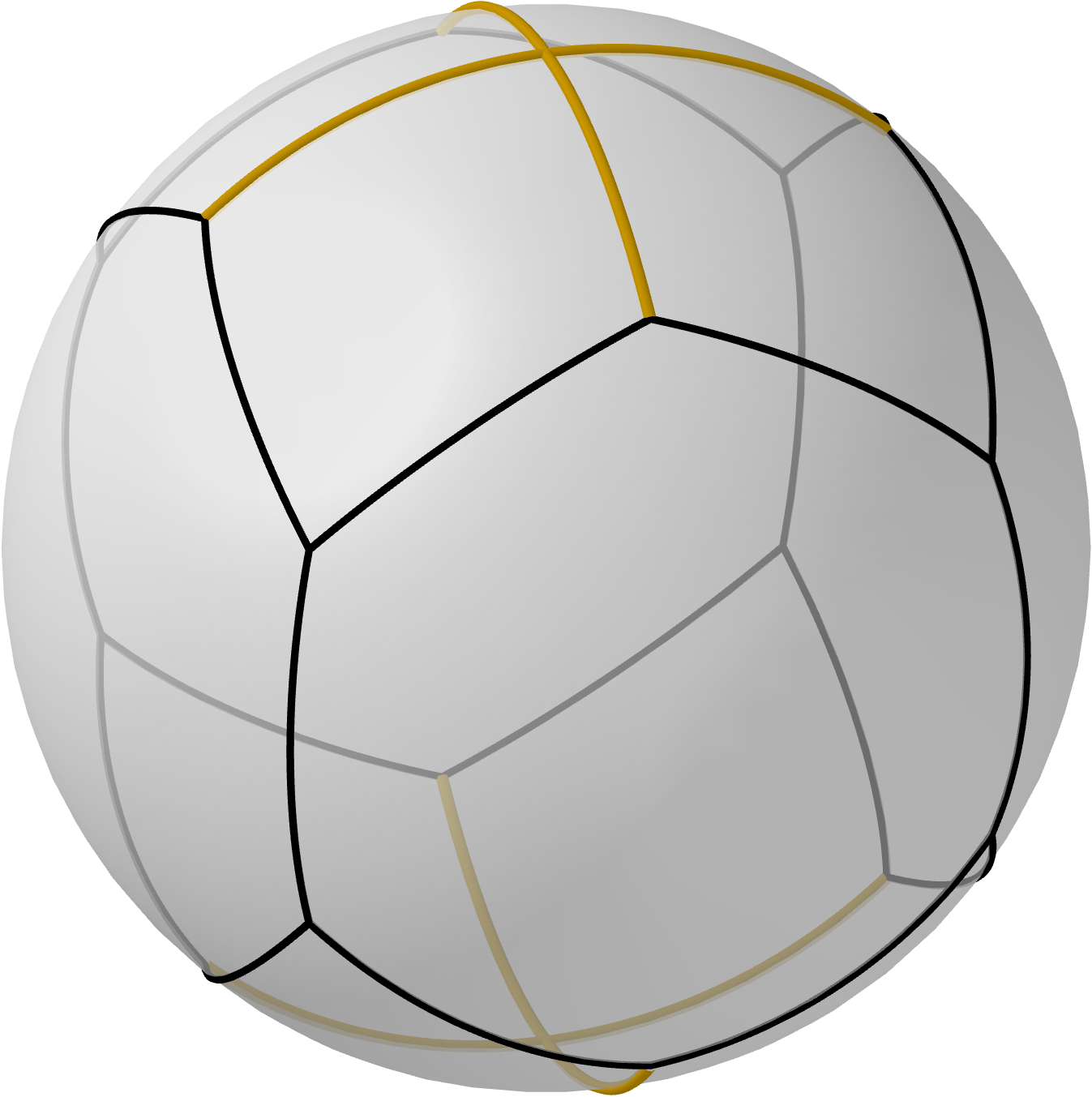}};
\end{scope}
\end{tikzpicture}
\caption{\ref{Label:polar-subdiv-cube} tiling via a kite subdivision of a pair of polar faces of the cube}
\label{Fig-polar-subdiv-cube}
\end{figure}
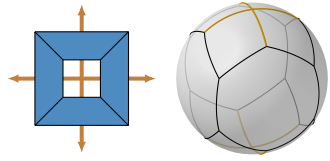

In Figure \ref{Fig-kite-subdiv-trunc-Platonic}, the drawings of \ref{Label:trunc-Platonic} without thick edges are respectively the graphs of the truncated Platonic solids $t\mathcal{T}, t\mathcal{C}, t\mathcal{O}, t\mathcal{D}, t\mathcal{I}$; the thick edges subdivide the derived faces from the truncation. 

\begin{figure}[h!]
\centering
\begin{subfigure}[t]{0.32\linewidth}
\centering
\begin{tikzpicture}
\tikzmath{
\s=2;
\r=0.3;
\th=360/3;
\x=\r*cos(30);
\y=\r*sin(30);
}

\raisebox{1ex}{

\foreach \a in {0,1,2} {
\tikzset{rotate=\a*\th}
\fill[teal!75!blue]
	(90-0.5*\th:\r) -- (90-0.5*\th:2*\r) -- (\x, 3*\y) -- (-\x, 3*\y) 
	-- (90+0.5*\th:2*\r) -- (90+0.5*\th:\r) -- cycle
;

\fill[teal!75!blue]
	(90-0.5*\th:2*\r) circle (1.5*\r)
;

\fill[teal!75!blue]
	(-\x,4*\y) -- (\x,4*\y) -- (0,0) -- cycle
;

}

\fill[white]
	(90-0.5*\th:\r) -- (90+0.5*\th:\r) -- (90+1.5*\th:\r)
;

\foreach \a in {0,1,2} {
\tikzset{rotate=\a*\th}
\fill[white]
	(2*\x, 0) arc (270:270+2*\th:\r) -- (90-0.5*\th:2*\r) -- cycle
;
}

\foreach \a in {0,1,2} {
\tikzset{rotate=\a*\th}

\draw[HGold, line width=1.15]
	(0,0) -- (0,\y)
	(90-0.5*\th:2.5*\r) -- (90-0.5*\th:3*\r) 
	(90-0.5*\th:2.5*\r) to[bend right] ($(\x, 3*\y) !1/2!  (90-0.5*\th:2*\r)$)
	(90+0.5*\th:2.5*\r) to[bend left] ($(-\x, 3*\y) !1/2!  (90+0.5*\th:2*\r)$)
;
}

\foreach \a in {0,1,2} {
\tikzset{rotate=\a*\th}
\draw[]
	(90-0.5*\th:\r) -- (90+0.5*\th:\r)
	(90-0.5*\th:\r) -- (90-0.5*\th:2*\r) 
	(90-0.5*\th:2*\r) -- (\x, 3*\y)
	(90+0.5*\th:2*\r) -- (-\x, 3*\y)
	(-\x, 3*\y) -- (\x, 3*\y)
	(2*\x, 0) arc (270:270+2*\th:\r)
;
}
}

\begin{scope}[xshift=\s cm]
\node [inner sep=0] (image) at (0,0) 
            {\includegraphics[height=2cm]{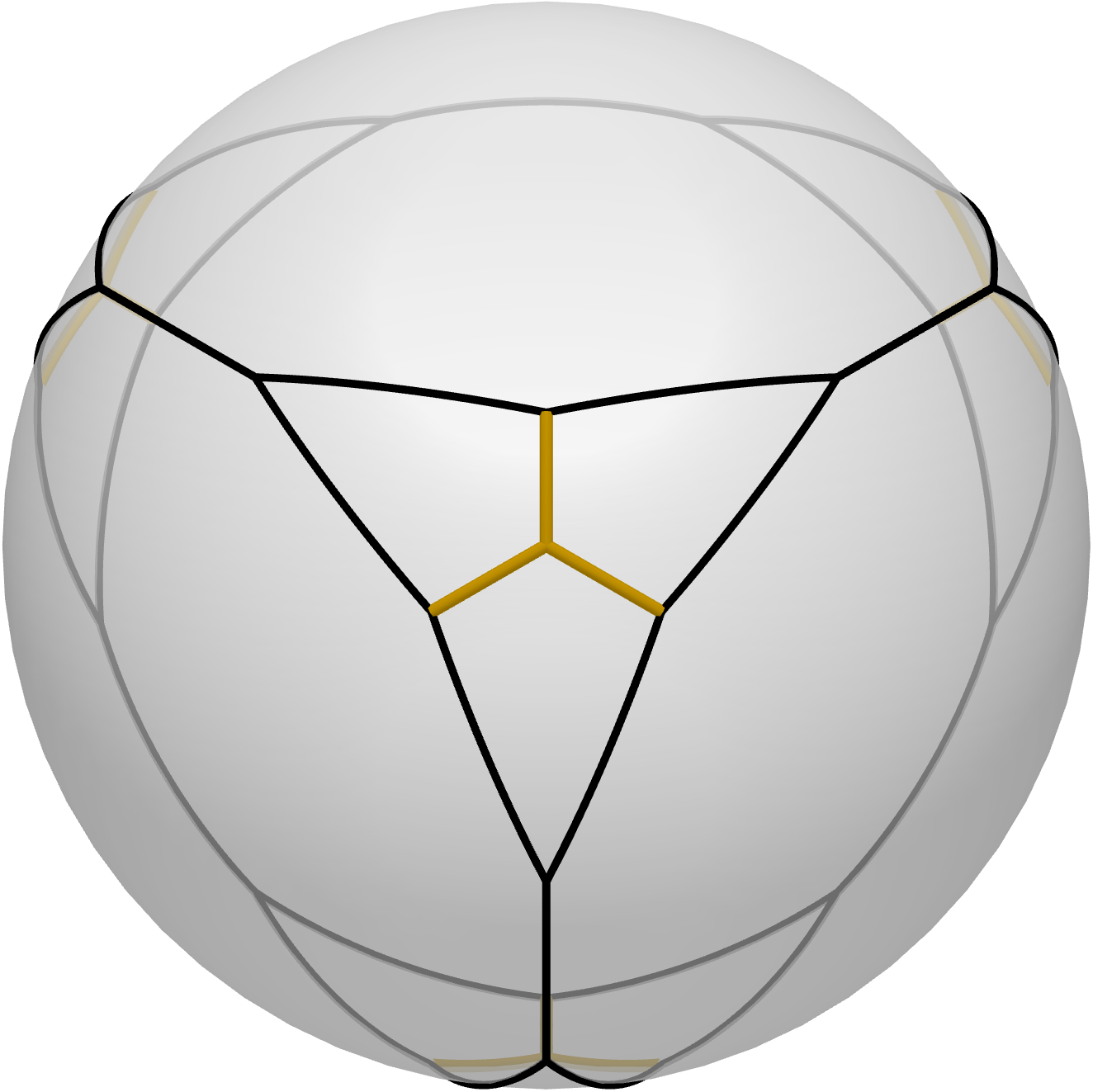}};
\end{scope}

\end{tikzpicture}
\caption{Subdivided $t\mathcal{T}$}
\label{Subfig-kite-subdiv-tT}
\end{subfigure}
\begin{subfigure}[t]{0.33\linewidth}
\centering
\begin{tikzpicture}[>=latex]
\tikzmath{
\s=2;
\r=0.3;
\n=8;
\m=3;
\nn=\n-1;
\th=360/\n;
\ph=360/\m;
\y=\r*cos(0.5*\th);
\x=\r*sin(0.5*\th);
\xx=\r*cos(0.5*\ph);
\yy=\x/tan(0.5*\ph);
\rr=\x*sin(0.5*\ph);
}

\raisebox{1ex}{

\fill[teal!75!blue] circle (2.75*\r);

\fill[white] circle (2.5*\r);

\foreach \a in {0,2,...,\n} {
\tikzset{rotate=\a*\th}
\draw[HGold, line width=1.15]
	(0,\y+\yy) -- (0,\y)
	(0,\y+\yy) -- ($(90-0.5*\th:\r) !1/2! (90:\r+\xx)$)
	(0,\y+\yy) -- ($(90+0.5*\th:\r) !1/2! (90:\r+\xx)$)
	(0,\y+\yy+\rr+0.75*\r) -- ($(0,\y+\yy+\rr+0.5*\r) !1/2! (90-0.5*\th:2.5*\r)$)
	(0,\y+\yy+\rr+0.75*\r) -- ($(0,\y+\yy+\rr+0.5*\r) !1/2! (90+0.5*\th:2.5*\r)$)
	(0,\y+\yy+\rr+0.75*\r) -- (90:2.5*\r)
;
}

\fill[teal!75!blue]
	(90-0.5*\th:\r) -- (90+0.5*\th:\r) -- (90+1.5*\th:\r) -- (90+2.5*\th:\r) -- (90+3.5*\th:\r) 
	-- (90+4.5*\th:\r) -- (90+5.5*\th:\r) -- (90+6.5*\th:\r) -- cycle
;

\foreach \a in {0,2,...,\n} {
\tikzset{rotate=\a*\th}
\fill[teal!75!blue]
	(90-0.5*\th:\r) -- (0,\y+\yy+\rr) -- (0,\y+\yy+\rr+0.5*\r) 
	-- (90-0.5*\th:2.5*\r) arc (90-0.5*\th:90-1.5*\th:2.5*\r) 
	-- (\y+\yy+\rr+0.5*\r,0)  -- (\y+\yy+\rr,0) -- (90-1.5*\th:\r) 
	-- cycle
;
}

\draw[] circle (2.5*\r);

\foreach \a in {0,...,\nn} {
\tikzset{rotate=\a*\th}
\draw[]
	(90-0.5*\th:\r) -- (90+0.5*\th:\r)
;
}

\foreach \a in {0,2,...,\n} {
\tikzset{rotate=\a*\th}
\draw[]
	(90-0.5*\th:\r) -- (90:\r+\xx)
	(90+0.5*\th:\r) -- (90:\r+\xx)
	(90:\r+\xx) -- (0,\y+\yy+\rr+0.5*\r)
	(90:\y+\yy+\rr+0.5*\r) -- (90-0.5*\th:2.5*\r)
	(90:\y+\yy+\rr+0.5*\r) -- (90+0.5*\th:2.5*\r)
;
}

\begin{scope}[xshift=\s cm]
\node [inner sep=0] (image) at (0,0) 
            {\includegraphics[height=2cm]{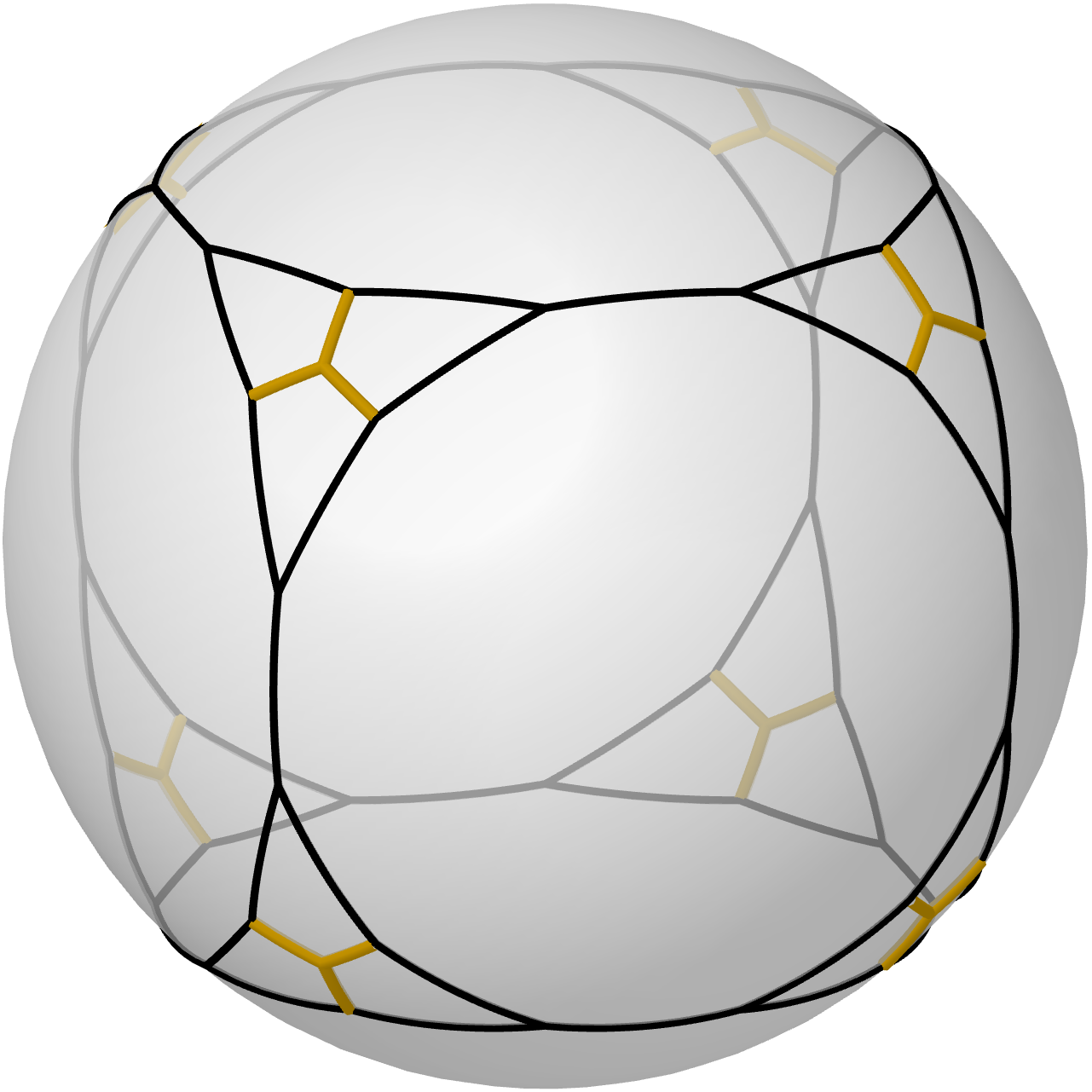}};
\end{scope}

}

\end{tikzpicture}
\caption{Subdivided $t\mathcal{C}$}
\label{Subfig-kite-subdiv-tC}
\end{subfigure}
\begin{subfigure}[t]{0.32\linewidth}
\centering
\begin{tikzpicture}[>=latex]
\tikzmath{
\s=2.15;
\r=0.45;
\dr=0.05*\r;
\rr=0.5*\r;
\R=1.5*\r;
\RR=1.75*\r;
\th=360/4;
\x=\r*cos(\th/2);
\xx=0.5*\x;
\hex=360/6;
}

\fill[white] (0,0) circle (\RR);

\foreach \a in {0,1,2,3} {
\tikzset{rotate=\a*\th}

\draw[HGold, line width=1.15]
	(0,0) -- (0,0.5*\x)
	($(0.5*\th:\r) !1/2! (\xx+\x,\x+\xx)$) -- ($(\xx,\x+\xx) !1/2! (\xx+\x,\x+\xx)$)
	($(0.5*\th:\r) !1/2! (\xx+\x,\x+\xx)$) -- ($(\xx+\x,\x-\xx) !1/2! (\xx+\x,\x+\xx)$)
	($(0.5*\th:\r) !1/2! (\xx+\x,\x+\xx)$) -- ($(0.5*\th:\r) !1/2! (\xx,\x+\xx)$)
	($(0.5*\th:\r) !1/2! (\xx+\x,\x+\xx)$) -- ($(0.5*\th:\r) !1/2! (\xx+\x,\x-\xx)$)
;

\draw[arrows = {-Latex[scale=0.45]}, HGold, line width=1.15]
	(90:\RR) -- (90:1.35*\RR)
;
}

\foreach \a in {0,1,2,3} {
\tikzset{rotate=\a*\th}
\fill[teal!75!blue]
	(0.5*\th:\rr) -- (0.5*\th:\r) -- (\xx,\x+\xx) -- (-\xx,\x+\xx) 
	-- (1.5*\th:\r) -- (1.5*\th:\rr) -- cycle
;
\fill[teal!75!blue]
	(-\xx-\x,\x+\xx) -- (\xx+\x,\x+\xx) -- (0.5*\th:\RR) arc (0.5*\th:1.5*\th:\RR)
;
}

\foreach \a in {0,1,2,3} {
\tikzset{rotate=\a*\th}

\draw[]
	(0.5*\th:\rr) -- (1.5*\th:\rr)
	(0.5*\th:\rr) -- (0.5*\th:\r)
	(\xx,\x+\xx) -- (0.5*\th:\r)
	(-\xx,\x+\xx) -- (1.5*\th:\r)
	(0.5*\th:\R) -- (1.5*\th:\R)
	(0.5*\th:\R) -- (0.5*\th:\RR)
;
}

\draw[] (0,0) circle (\RR);

\begin{scope}[xshift=\s cm]
\node [inner sep=0] (image) at (0,0) 
            {\includegraphics[height=2cm]{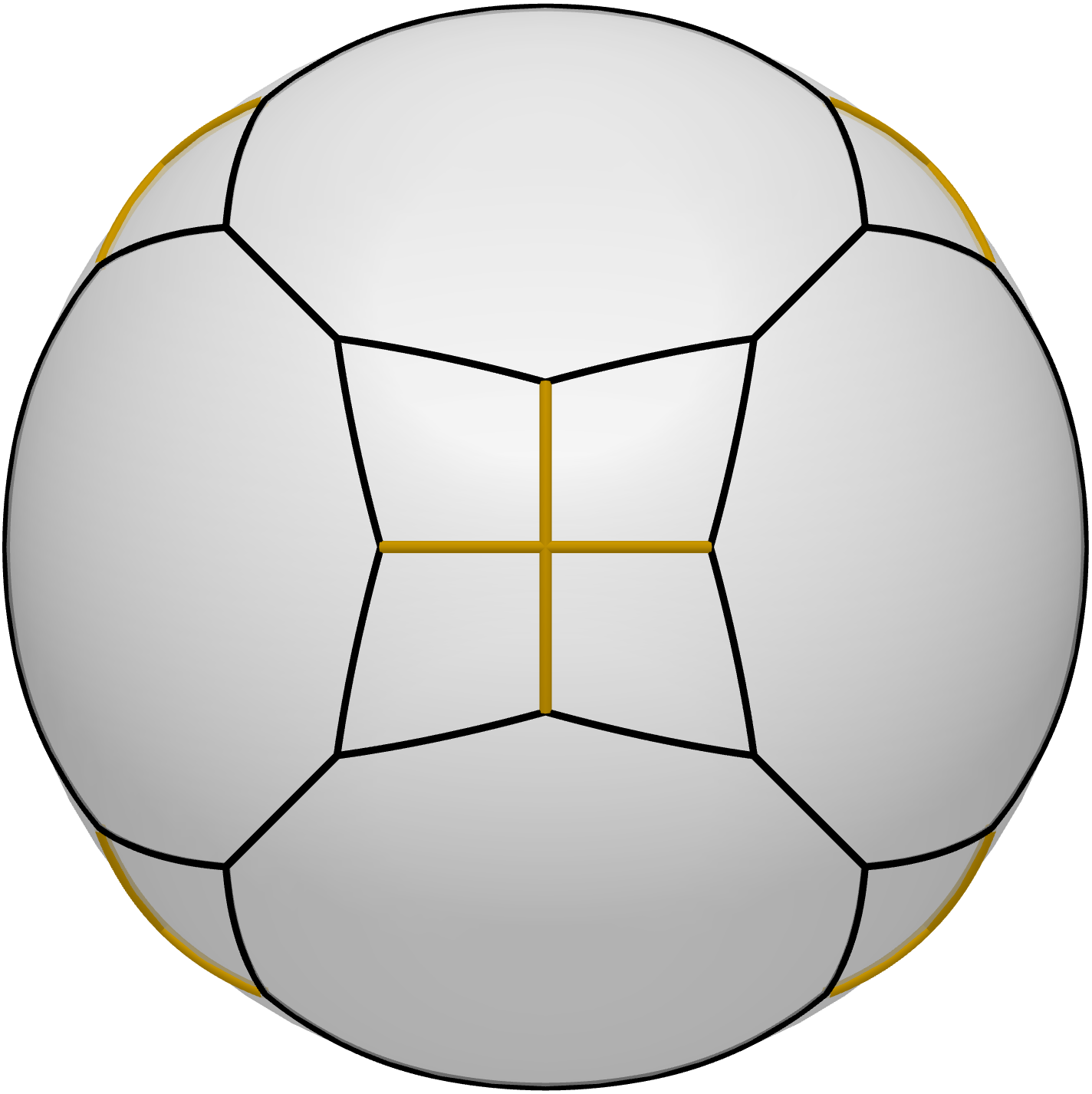}};
\end{scope}

\end{tikzpicture}
\caption{Subdivided $t\mathcal{O}$}
\label{Subfig-kite-subdiv-tO}
\end{subfigure}

\begin{subfigure}[t]{0.4\linewidth}
\centering
\begin{tikzpicture}[>=latex]
\tikzmath{
\s=2.25;
\r=0.18;
\n=10;
\m=3;
\nn=\n-1;
\th=360/\n;
\ph=360/\m;
\y=\r*cos(0.5*\th);
\x=\r*sin(0.5*\th);
\xx=\r*cos(0.5*\ph);
\yy=\x/tan(0.5*\ph);
\rr=\x*sin(0.5*\ph);
}

\raisebox{1ex}{

\fill[teal!75!blue] circle (5.5*\r);
\fill[white] circle (5*\r);

\foreach \a in {0,2,...,\n} {
\tikzset{rotate=\a*\th}
\draw[HGold, line width=1.15]
	(90-1*\th:4.65*\r) -- ($(90-1*\th:4.25*\r) !1/2! (90-0.5*\th:5*\r)$)
	(90-1*\th:4.65*\r) -- ($(90-1*\th:4.25*\r) !1/2! (90-1.5*\th:5*\r)$)
	(90-1*\th:4.65*\r) -- (90-1*\th:5*\r)
;
}

\fill[teal!75!blue]
	(90-0.5*\th:\r) -- (90+0.5*\th:\r) -- (90+1.5*\th:\r) -- (90+2.5*\th:\r)
	-- (90+3.5*\th:\r) -- (90+4.5*\th:\r) -- (90+5.5*\th:\r) -- (90+6.5*\th:\r)
	-- (90+7.5*\th:\r) -- (90+8.5*\th:\r) -- (90+9.5*\th:\r) -- cycle
;

\foreach \a in {0,2,...,\n} {
\tikzset{rotate=\a*\th}
\fill[teal!75!blue]
	(0,\y+\yy+\rr+0.5*\r) -- (90-0.25*\th:2.5*\r) -- (90-0.7*\th:2.85*\r) 
	-- (90-1.3*\th:2.85*\r) -- (90-1.75*\th:2.5*\r) -- (90-2*\th:\y+\yy+\rr+0.5*\r) 
	-- (90-2*\th:\r+\xx) -- (90-1.5*\th:\r) -- (90-0.5*\th:\r)
	-- (0,\r+\xx) -- cycle
;
\fill[teal!75!blue]
	(90-0.25*\th:2.5*\r) -- (90-0.7*\th:2.85*\r) -- (90-1*\th:3.6*\r) 
	-- (90-1*\th:4.25*\r) -- (90-0.5*\th:5*\r) arc (90-0.5*\th:90+0.5*\th:5*\r)
	-- (90+0.5*\th:5*\r) -- (90+1*\th:4.25*\r) -- (90+1*\th:3.6*\r) 
	-- (90+0.7*\th:2.85*\r) -- (90+0.25*\th:2.5*\r) -- cycle
;
}

\foreach \a in {0,2,...,\n} {
\tikzset{rotate=\a*\th}
\draw[HGold, line width=1.15]
	(0,\y+\yy) -- (0,\y)
	(0,\y+\yy) -- ($(90-0.5*\th:\r) !1/2! (90:\r+\xx)$)
	(0,\y+\yy) -- ($(90+0.5*\th:\r) !1/2! (90:\r+\xx)$)
	(0,\y+\yy+\rr+0.85*\r) -- ($(0,\y+\yy+\rr+0.5*\r) !1/2! (90-0.25*\th:2.5*\r)$)
	(0,\y+\yy+\rr+0.85*\r) -- ($(0,\y+\yy+\rr+0.5*\r) !1/2! (90+0.25*\th:2.5*\r)$)
	(0,\y+\yy+\rr+0.85*\r) -- ($(90+0.25*\th:2.5*\r) !1/2! (90-0.25*\th:2.5*\r)$)
	(90-\th:2*\x+2.5*\r) -- ($(90-0.7*\th:2.85*\r) !1/2! (90-1*\th:3.6*\r)$)
	(90-\th:2*\x+2.5*\r) -- ($(90-1.3*\th:2.85*\r) !1/2! (90-1*\th:3.6*\r)$)
	(90-\th:2*\x+2.5*\r) -- ($(90-0.7*\th:2.85*\r) !1/2! (90-1.3*\th:2.85*\r)$)
;
}

\draw[] (0,0) circle (5*\r);

\foreach \a in {0,...,\nn} {
\tikzset{rotate=\a*\th}
\draw[]
	(90-0.5*\th:\r) -- (90+0.5*\th:\r)
;
}

\foreach \a in {0,2,...,\n} {
\tikzset{rotate=\a*\th}
\draw[]
	(90-0.5*\th:\r) -- (90:\r+\xx)
	(90+0.5*\th:\r) -- (90:\r+\xx)
	(90:\r+\xx) -- (90:\y+\yy+\rr+0.5*\r)
	(90:\y+\yy+\rr+0.5*\r) -- (90-0.25*\th:2.5*\r)
	(90:\y+\yy+\rr+0.5*\r) -- (90+0.25*\th:2.5*\r)
	(90-0.25*\th:2.5*\r) -- (90+0.25*\th:2.5*\r)
	(90-0.25*\th:2.5*\r) -- (90-0.7*\th:2.85*\r)
	(90+0.25*\th:2.5*\r) -- (90+0.7*\th:2.85*\r)
	(90-0.7*\th:2.85*\r) -- (90-1.3*\th:2.85*\r)
	(90-0.7*\th:2.85*\r) -- (90-1*\th:3.6*\r)
	(90-1.3*\th:2.85*\r) -- (90-1*\th:3.6*\r)
	(90-1*\th:3.6*\r) -- (90-1*\th:4.25*\r)
	(90-1*\th:4.25*\r) -- (90-0.5*\th:5*\r)
	(90-1*\th:4.25*\r) -- (90-1.5*\th:5*\r)
;
}

\begin{scope}[xshift=\s cm]
\node [inner sep=0] (image) at (0,0) 
            {\includegraphics[height=2cm]{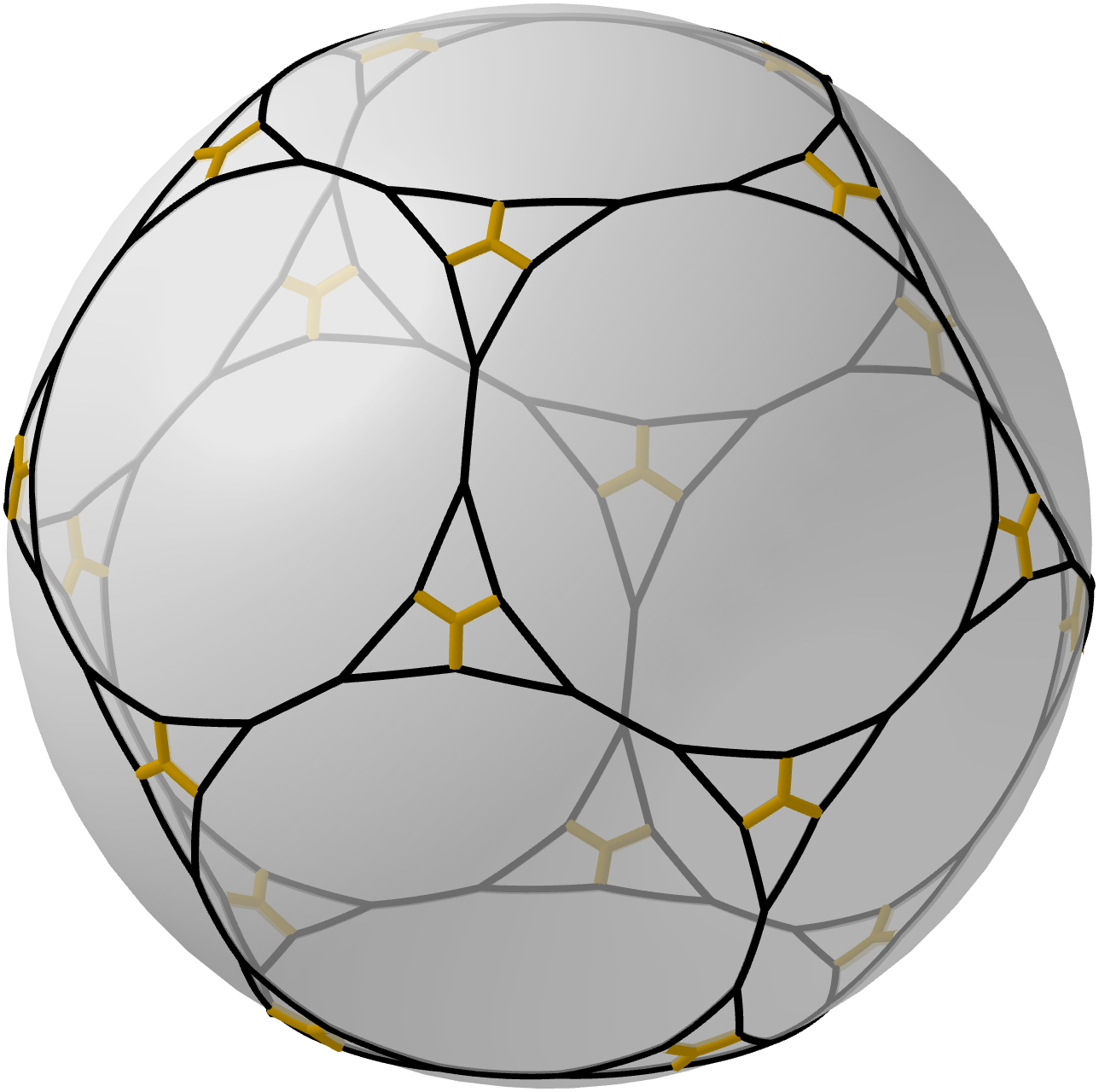}};
\end{scope}

}

\end{tikzpicture}
\caption{Subdivided $t\mathcal{D}$}
\label{Subfig-kite-subdiv-tD}
\end{subfigure}
\begin{subfigure}[t]{0.4\linewidth}
\centering
\begin{tikzpicture}[>=latex]
\tikzmath{
\s=2.15;
\r=0.375;
\th=360/5;
\x=0.4*\r*cos(0.5*\th);
}

\raisebox{1ex}{

\fill[white] (0,0) circle (2.4*\r);

\foreach \a in {0,...,4}{
\tikzset{rotate=\a*\th}
\fill[teal!75!blue]
	(90:0.4*\r) -- (90:0.8*\r) -- (66:1.1*\r) -- (42:1.1*\r)  -- (18:0.8*\r) -- (18:0.4*\r) -- cycle
;
\fill[teal!75!blue]
	(66:1.1*\r) -- (42:1.1*\r) -- (90-72+12:1.4*\r) -- (90-72+24:1.6*\r) -- (66:1.6*\r) -- (78:1.4*\r) -- cycle
;
\fill[teal!75!blue]
	(102:1.4*\r) -- (78:1.4*\r) -- (66:1.6*\r) -- (78:2*\r) -- (102:2*\r) -- (114:1.6*\r) -- cycle
;
\fill[teal!75!blue]
	(102:2*\r) -- (78:2*\r) -- (54:2*\r) -- (54:2.4*\r) arc (54:126:2.4*\r) -- (126:2*\r) -- cycle
;
}

\foreach \a in {0,...,4}{
\tikzset{rotate=\a*\th}

\draw[HGold, line width=1.15]
	(0,0) -- (0,-\x)
	(0,1.1*\r) -- ($(90:0.8*\r) !1/2! (66:1.1*\r)$)
	(0,1.1*\r) -- ($(66:1.1*\r) !1/2! (78:1.4*\r)$)
	(0,1.1*\r) -- ($(102:1.4*\r) !1/2! (78:1.4*\r)$)
	(0,1.1*\r) -- ($(114:1.1*\r) !1/2! (102:1.4*\r)$)
	(0,1.1*\r) -- ($(90:0.8*\r) !1/2! (114:1.1*\r)$)
	(90-36:1.75*\r) -- ($(66:1.6*\r) !1/2! (90-72+24:1.6*\r)$)
	(90-36:1.75*\r) -- ($(66:1.6*\r) !1/2! (78:2*\r)$)
	(90-36:1.75*\r) -- ($(78:2*\r) !1/2! (90-36:2*\r)$)
	(90-36:1.75*\r) -- ($(30:2*\r) !1/2! (90-36:2*\r)$)
	(90-36:1.75*\r) -- ($(30:2*\r) !1/2! (90-72+24:1.6*\r)$)
;	

\draw[arrows = {-Latex[scale=0.45]}, HGold, line width=1.15]
	(0,2.4*\r) -- (0,3.2*\r)
;

}

\foreach \a in {0,...,4}
{
\begin{scope}[rotate=72*\a]

\draw
	(18:0.4*\r) -- (90:0.4*\r) -- (90:0.8*\r) -- (66:1.1*\r) -- (42:1.1*\r) -- (18:0.8*\r)
	(66:1.1*\r) -- (78:1.4*\r) -- (102:1.4*\r) -- (114:1.1*\r)
	(78:1.4*\r) -- (66:1.6*\r) -- (42:1.6*\r) -- (30:1.4*\r)
	(66:1.6*\r) -- (78:2*\r) -- (102:2*\r) -- (114:1.6*\r)
	(78:2*\r) -- (54:2*\r) -- (30:2*\r)
	(54:2*\r) -- (54:2.4*\r)
;

\end{scope}
}

\draw[] (0,0) circle (2.4*\r);

\begin{scope}[xshift=\s cm]
\node [inner sep=0] (image) at (0,0) 
            {\includegraphics[height=2cm]{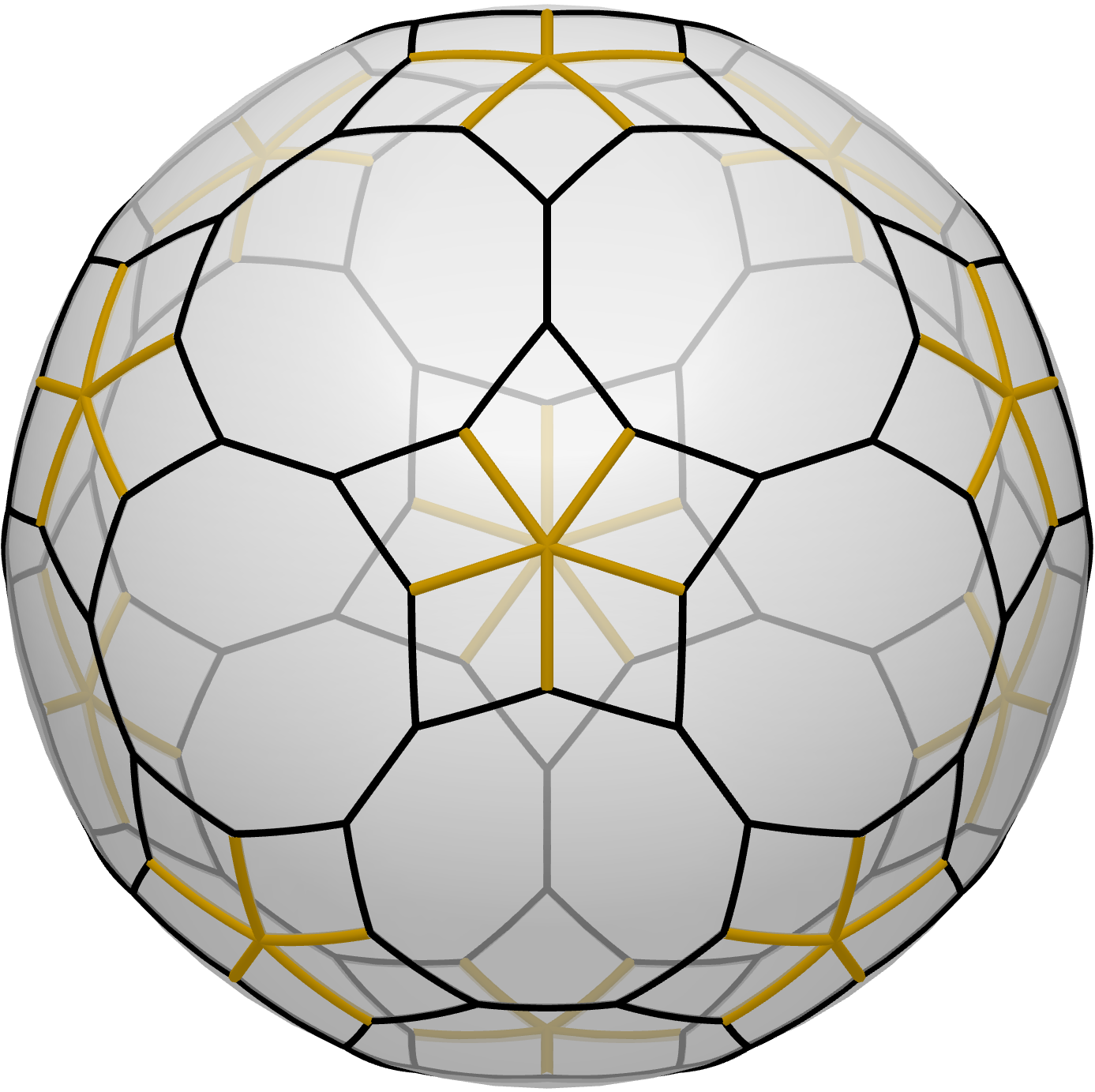}};
\end{scope}

}

\end{tikzpicture}
\caption{Subdivided $t\mathcal{I}$}
\label{Subfig-kite-subdiv-tI}
\end{subfigure}
\caption{\ref{Label:trunc-Platonic} tilings via kite subdivision of the derived faces of the truncated Platonic solids}
\label{Fig-kite-subdiv-trunc-Platonic}
\end{figure}
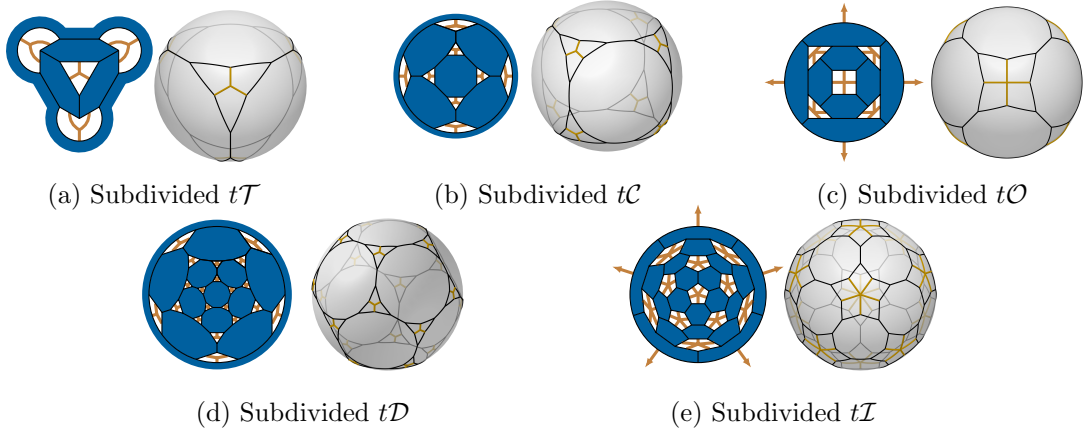


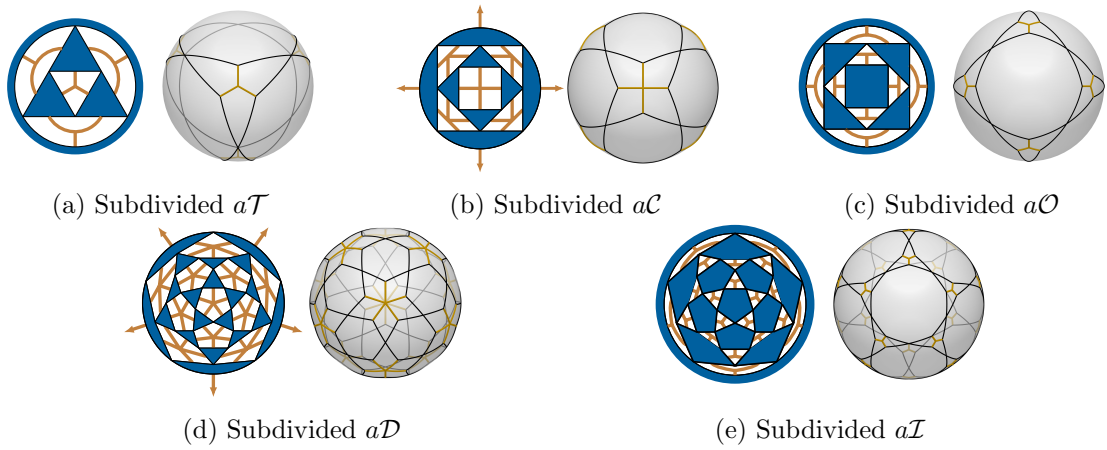
\begin{figure}[h!] 
\centering
\begin{subfigure}[t]{0.3\linewidth}
\centering
\begin{tikzpicture}[>=latex]

\tikzmath{
\xs=2.15;
}

\raisebox{1ex}{

\begin{scope}[]
\tikzmath{
\th=360/3;
\r=0.4;
\x=\r*cos(\th/2);
}

\fill[teal!75!blue] (0,0) circle (2.25*\r);
\fill[white] (0,0) circle (2*\r);
\fill[teal!75!blue] (90:2*\r) -- (90+\th:2*\r) -- (90+2*\th:2*\r) -- cycle;
\fill[white] (270:\r) -- (270+\th:\r) -- (270+2*\th:\r) -- cycle;

\foreach \a in {0,...,2} {

\tikzset{rotate=\a*\th}

\draw[]
	(270:\r) -- (270+\th:\r)
	(90:2*\r) -- (90+\th:2*\r)
;

\draw[HGold, line width=1.25]
	(0,0) -- (90:\x)
	(90-0.5*\th:1.5*\r) -- (90-0.5*\th:2*\r)
;

\arcThroughThreePoints[HGold, line width=1.25]{$(90-0.5*\th:1*\r) !1/2! (90-1*\th:2*\r)$}{90-0.5*\th:1.5*\r}{$(90:2*\r) !1/2! (90-0.5*\th:1*\r)$};
} 

\draw[] (0,0) circle (2*\r);
\end{scope}

\begin{scope}[xshift=\xs cm]
\node [inner sep=0] (image) at (0,0) 
            {\includegraphics[height=2cm]{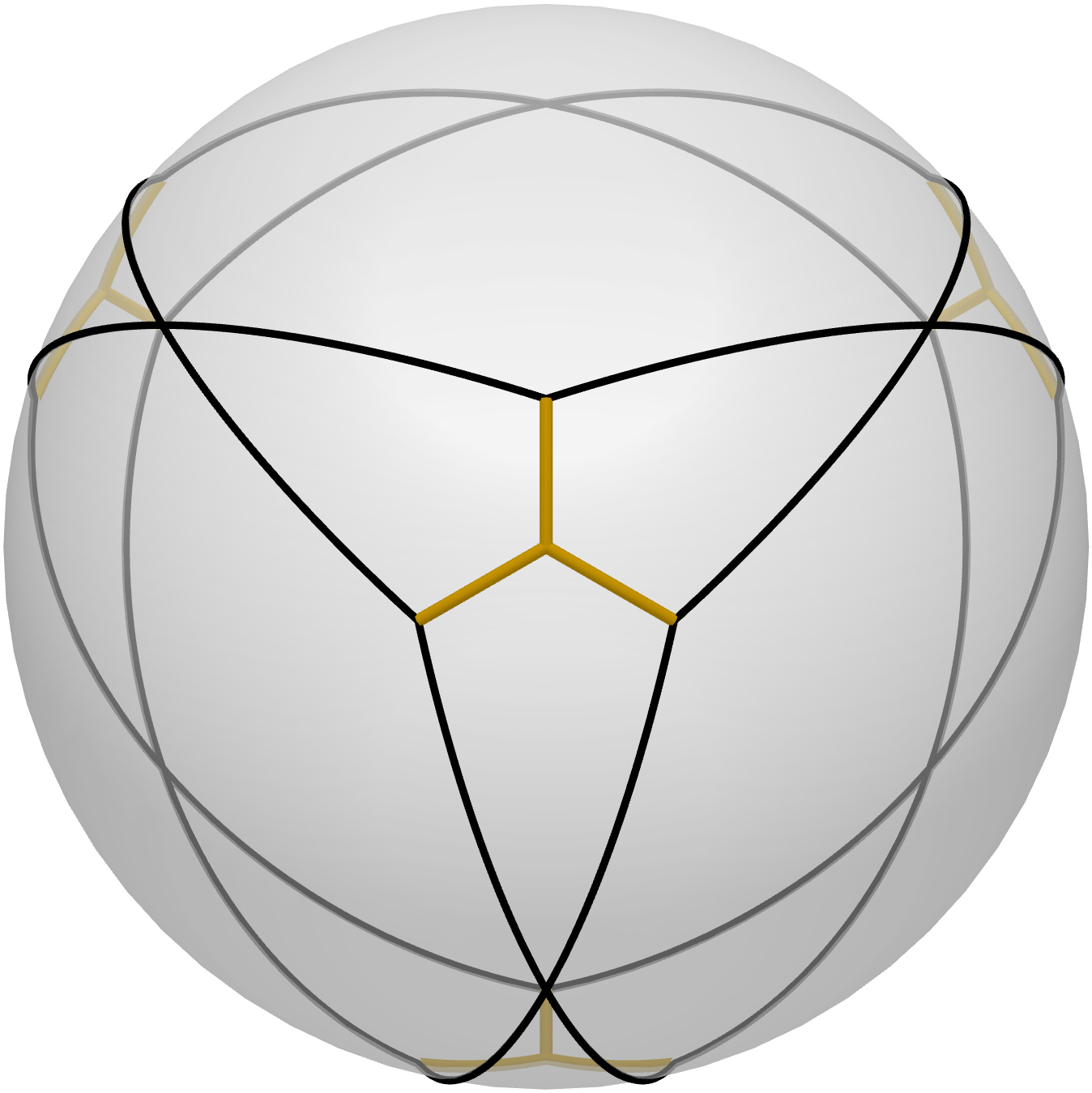}};
\end{scope}
}
\end{tikzpicture}
\caption{Subdivided $a\mathcal{T}$}
\end{subfigure}
\begin{subfigure}[t]{0.37\linewidth}
\centering
\begin{tikzpicture}[>=latex]
\tikzmath{
\xs=2.15;
\ys=4;
}
\begin{scope}[xshift=2*\xs cm]
\tikzmath{
\ph=360/4;
\r=0.4;
\x=\r*cos(0.5*\ph);
}

\fill[teal!75!blue] (0,0) circle (2*\r);
\fill[white] (\x,\x) -- (-\x,\x) -- (-\x,-\x) -- (\x,-\x) -- cycle;

\foreach \a in {0,...,3} {
\tikzset{rotate=\a*\ph}

\fill[white] (2*\x,2*\x) -- (0,2*\x) --  (2*\x,0) -- cycle;

\draw[HGold, line width=1.25]
	($(0,2*\x) !1/2! (2*\x,2*\x)$) -- (1.5*\x, 1.5*\x)
	($(2*\x,0) !1/2! (2*\x,2*\x)$) -- (1.5*\x, 1.5*\x)
	($(0,2*\x) !1/2! (\x,\x)$) -- (1.5*\x, 1.5*\x)
	($(2*\x,0) !1/2! (\x,\x)$) -- (1.5*\x, 1.5*\x)
;
}

\foreach \a in {0,...,3} {
\tikzset{rotate=\a*\ph}

\draw[]
	(\x,\x) -- (-\x,\x)
	(2*\x,0) -- (0,2*\x)
	(2*\x,2*\x) -- (-2*\x,2*\x)
;

\draw[HGold, line width=1.25]
	(0,0) -- (0,\x)
	%
;

\draw[arrows = {-Latex[scale=0.45]}, HGold, line width=1.1]
	(0,2*\r) -- (0,2.8*\r)
;
}

\foreach \a in {0,...,3} {
\tikzset{rotate=\a*\ph}

\draw[fill=teal!75!blue] (\x, \x) -- (0,2*\x) -- (-\x, \x) -- cycle;
\draw[fill=teal!75!blue] (2*\x,2*\x) arc (45:135:2*\r) -- cycle;
}

\draw[] (0,0) circle (2*\r);
\end{scope}

\begin{scope}[xshift=3*\xs cm] 
\node [inner sep=0] (image) at (0,0) 
            {\includegraphics[height=2cm]{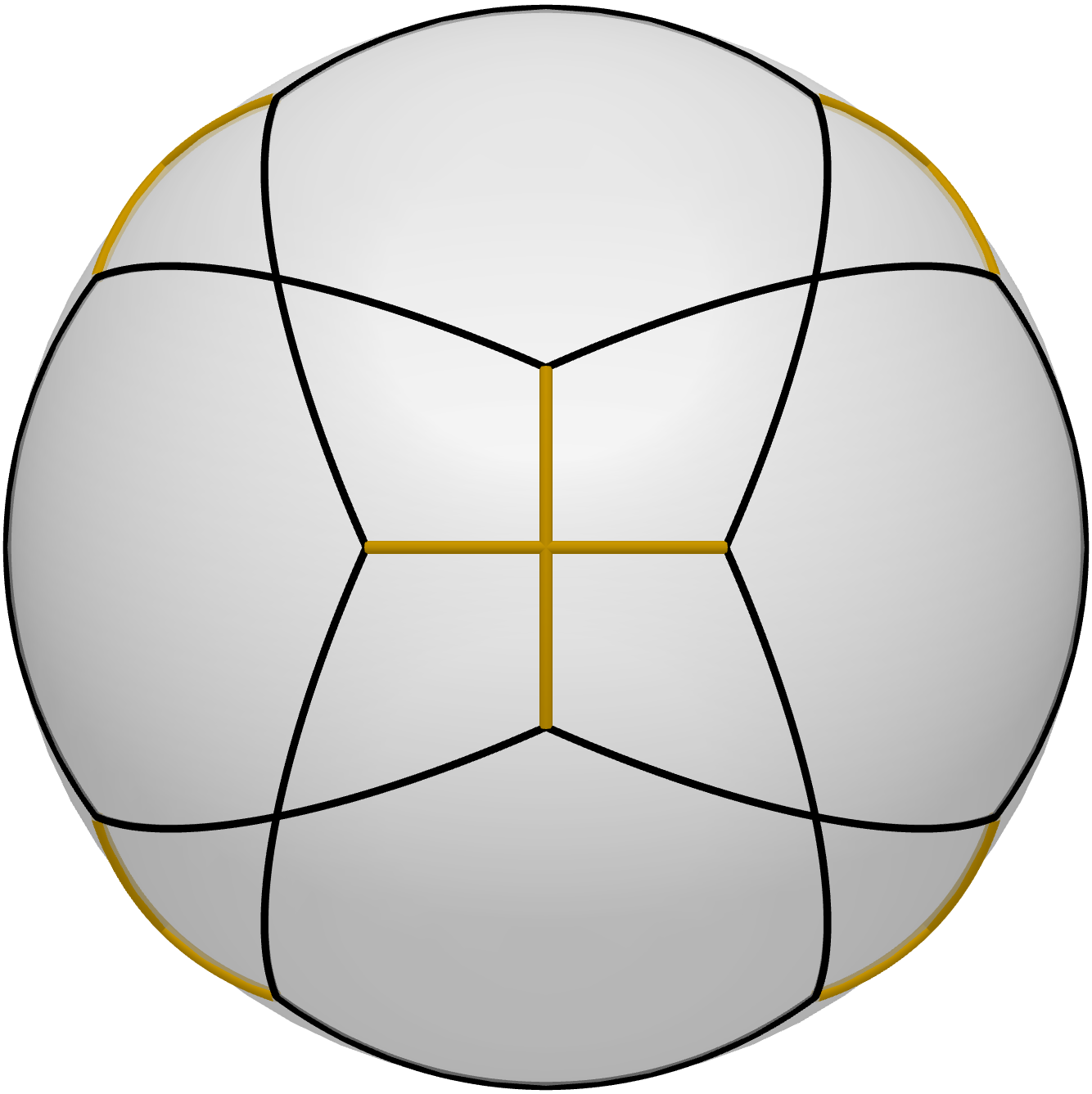}};
\end{scope}
\end{tikzpicture}
\caption{Subdivided $a\mathcal{C}$}
\end{subfigure}
\begin{subfigure}[t]{0.3\linewidth}
\centering
\begin{tikzpicture}[>=latex]

\tikzmath{
\xs=2.15;
}

\raisebox{1ex}{

\begin{scope}[]
\tikzmath{
\ph=360/4;
\r=0.4;
\x=\r*cos(0.5*\ph);
}

\fill[teal!75!blue] (0,0) circle (2.25*\r);

\draw[fill=white] (0,0) circle (2*\r);

\draw[fill=teal!75!blue] (\x,\x) -- (-\x,\x) -- (-\x,-\x) -- (\x,-\x) -- cycle;

\draw[] (0,0) circle (2*\r);

\foreach \a in {0,...,3} {
\tikzset{rotate=\a*\ph}

\draw[fill=teal!75!blue] (2*\x,2*\x) -- (0,2*\x) --  (2*\x,0) -- cycle;

\draw[HGold, line width=1.25]
	(0,2.4*\x) -- (0,2*\r)
	(0,2.4*\x) to[out=0, in=135, distance=0.1cm] (1*\x,2*\x) 
	(0,2.4*\x) to[out=180, in=45, distance=0.1cm] (-1*\x,2*\x) 
;

\draw[HGold, line width=1.25]
	(0,1.5*\x) -- (0,\x)
	(0,1.5*\x) -- ($(0,2*\x) !1/2! (\x,\x)$)
	(0,1.5*\x) -- ($(0,2*\x) !1/2! (-\x,\x)$)
;
}

\foreach \a in {0,...,3} {
\tikzset{rotate=\a*\ph}

\draw[]
	(\x,\x) -- (-\x,\x)
	(2*\x,0) -- (0,2*\x)
	(2*\x,2*\x) -- (-2*\x,2*\x)
;
}

\end{scope}

\begin{scope}[xshift=1*\xs cm]
\node [inner sep=0] (image) at (0,0) 
            {\includegraphics[height=2cm]{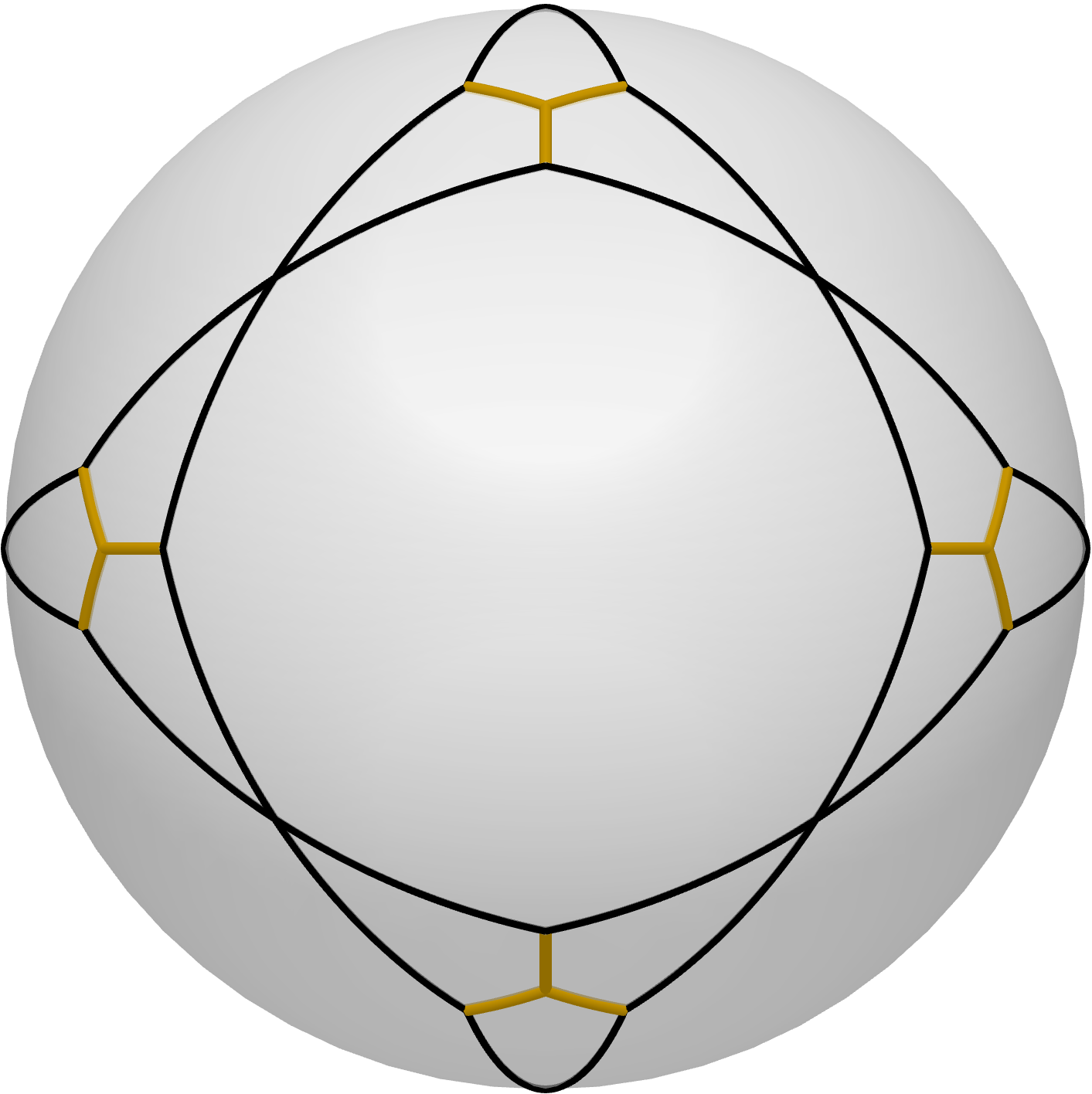}};
\end{scope}


}
\end{tikzpicture}
\caption{Subdivided $a\mathcal{O}$}
\end{subfigure}

\begin{subfigure}[t]{0.45\linewidth}
\centering
\begin{tikzpicture}[>=latex]

\tikzmath{
\xs=2;
\ps=360/5;
\r=0.25;
\x = \r*cos(0.5*\ps);
}

\begin{scope}[]

\foreach \a in {0,...,4} {
\tikzset{rotate=\a*\ps}

\draw[]
	(270:\r) -- (270+\ps:\r)
	(270:\r) -- (270-0.5*\ps:\x+\r)
	(270:\r) -- (270+0.5*\ps:\x+\r)
	(270-0.25*\ps:2.5*\r) -- (270-0.5*\ps:\x+\r)
	(270+0.25*\ps:2.5*\r) -- (270+0.5*\ps:\x+\r)
 	(270-0.25*\ps:2.5*\r) -- (270+0.25*\ps:2.5*\r)
	(90-0.25*\ps:2.5*\r) -- (90+0.25*\ps:2.5*\r) 
	(90-0.25*\ps:2.5*\r) -- (90-0.5*\ps:3.25*\r) 
	(90+0.25*\ps:2.5*\r) -- (90+0.5*\ps:3.25*\r)
	(90:3.75*\r) -- (90-0.5*\ps:3.25*\r) 
	(90:3.75*\r) -- (90+0.5*\ps:3.25*\r) 
;

\draw[HGold, line width=1.25]
	(0,0) -- (0,\x)
	(90-0.5*\ps:\x+\r) -- ($(90-0.5*\ps:1*\r) !1/2! (90-\ps:\x+\r)$)
	(90-0.5*\ps:\x+\r) -- ($(90-0.5*\ps:1*\r) !1/2! (90:\x+\r)$)
	(90-0.5*\ps:\x+\r) -- ($(90:\x+\r) !1/2! (90-0.25*\ps:2.5*\r)$)
	(90-0.5*\ps:\x+\r) -- ($(90-\ps:\x+\r) !1/2! (90-0.75*\ps:2.5*\r)$)
	(90-0.5*\ps:\x+\r) -- ($(90-0.25*\ps:2.5*\r) !1/2! (90-0.75*\ps:2.5*\r)$)
	(90:3*\r) -- ($(90-0.25*\ps:2.5*\r) !1/2! (90+0.25*\ps:2.5*\r)$)
	(90:3*\r) -- ($(90-0.25*\ps:2.5*\r) !1/2! (90-0.5*\ps:3.25*\r)$)
	(90:3*\r) -- ($(90+0.25*\ps:2.5*\r) !1/2! (90+0.5*\ps:3.25*\r)$)
	(90:3*\r) -- ($(90:3.75*\r) !1/2! (90-0.5*\ps:3.25*\r)$)
	(90:3*\r) -- ($(90:3.75*\r) !1/2! (90+0.5*\ps:3.25*\r)$)
;

\draw[arrows = {-Latex[scale=0.45]}, HGold, line width=1.25]
	(90-0.5*\ps:3.75*\r) -- (90-0.5*\ps:5*\r)
;
}

\foreach \a in {0,...,4} {
\tikzset{rotate=\a*\ps}

\draw[fill=teal!75!blue] 
	(90:\x+\r) -- (90-0.5*\ps:\r) -- (90+0.5*\ps:\r) -- cycle
;

\draw[fill=teal!75!blue] 
	(90:\x+\r) -- (90-0.25*\ps:2.5*\r) -- (90+0.25*\ps:2.5*\r) -- cycle
;

\draw[fill=teal!75!blue] 
	(90-0.25*\ps:2.5*\r) -- (90-0.5*\ps:3.25*\r) -- (90-0.75*\ps:2.5*\r) -- cycle
;

\draw[fill=teal!75!blue]  
	(90-0.5*\ps:3.25*\r) -- (90:3.75*\r) arc (90:90-\ps:3.75*\r) -- cycle
;
}

\draw[] (0,0) circle (3.75*\r);

\end{scope}

\begin{scope}[xshift=1.14*\xs cm]
\node [inner sep=0] (image) at (0,0) 
            {\includegraphics[height=2cm]{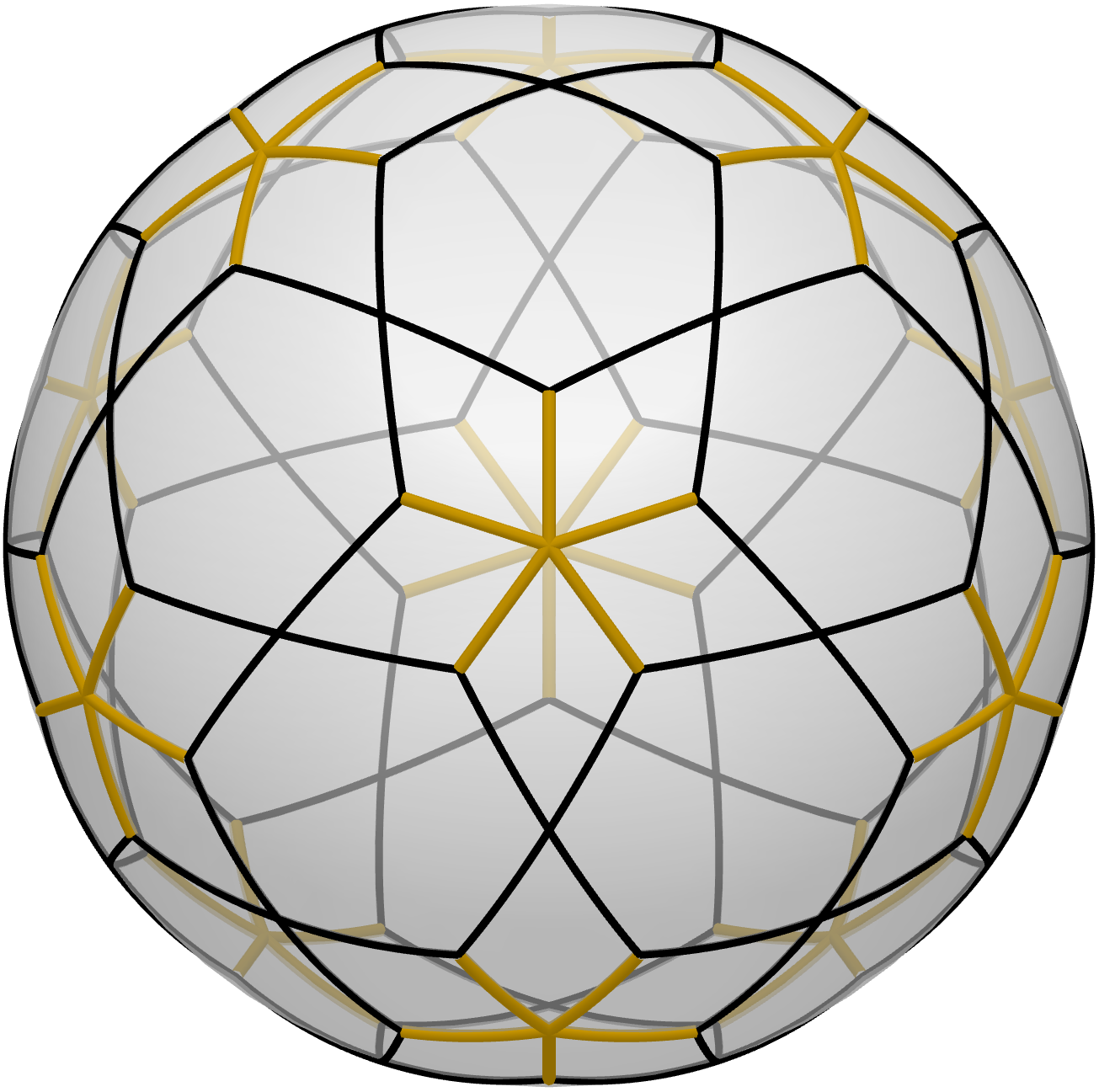}};
\end{scope}

\end{tikzpicture}
\caption{Subdivided $a\mathcal{D}$}
\end{subfigure}
\begin{subfigure}[t]{0.45\linewidth}
\centering
\begin{tikzpicture}[>=latex]

\tikzmath{
\xs=2;
}

\raisebox{1.25ex}{

\begin{scope}[]

\tikzmath{
\ps=360/5;
\r=0.25;
\x = \r*cos(0.5*\ps);
}

\fill[teal!75!blue] (0,0) circle (4.2*\r);

\draw[fill=white] (0,0) circle (3.75*\r);

\draw[] (0,0) circle (3.75*\r);

\foreach \a in {0,...,4} {
\tikzset{rotate=\a*\ps}

\draw[HGold, line width=1.25]
	(90:1.2*\r) -- (90:\x)
	(90:1.2*\r) -- ($(90-0.5*\ps:\r) !1/2! (90:\x+\r)$)
	(90:1.2*\r) -- ($(90+0.5*\ps:\r) !1/2! (90:\x+\r)$)
	(90:2.59*\x) -- ($(90-0.25*\ps:2.5*\r) !1/2! (90:\x+\r)$)
	(90:2.59*\x) -- ($(90+0.25*\ps:2.5*\r) !1/2! (90:\x+\r)$)
	(90:2.59*\x) -- ($(90-0.25*\ps:2.5*\r) !1/2! (90+0.25*\ps:2.5*\r)$)
	(90-0.5*\ps:2.75*\r) -- ($(90-0.25*\ps:2.5*\r) !1/2! (90-0.75*\ps:2.5*\r)$)
	(90-0.5*\ps:2.75*\r) -- ($(90-0.25*\ps:2.5*\r) !1/2! (90-0.5*\ps:3.25*\r)$)
	(90-0.5*\ps:2.75*\r) -- ($(90-0.75*\ps:2.5*\r) !1/2! (90-0.5*\ps:3.25*\r)$)
	(270:3.5*\r) -- (270:3.75*\r) 
	(270:3.5*\r) to[out=180, in=280, distance=0.05 cm] ($(270:3.5*\r) !3/5! (270-0.5*\ps:3.5*\r)$)
	(270:3.5*\r) to[out=0, in=260, distance=0.05 cm] ($(270:3.5*\r) !3/5! (270+0.5*\ps:3.5*\r)$)
;
}

\foreach \a in {0,...,4} {
\tikzset{rotate=\a*\ps}

\draw[]
	(270:\r) -- (270+\ps:\r)
	(270:\r) -- (270-0.5*\ps:\x+\r)
	(270:\r) -- (270+0.5*\ps:\x+\r)
	(270-0.25*\ps:2.5*\r) -- (270-0.5*\ps:\x+\r)
	(270+0.25*\ps:2.5*\r) -- (270+0.5*\ps:\x+\r)
 	(270-0.25*\ps:2.5*\r) -- (270+0.25*\ps:2.5*\r)
	(90-0.25*\ps:2.5*\r) -- (90+0.25*\ps:2.5*\r) 
	(90-0.25*\ps:2.5*\r) -- (90-0.5*\ps:3.25*\r) 
	(90+0.25*\ps:2.5*\r) -- (90+0.5*\ps:3.25*\r)
	(90:3.75*\r) -- (90-0.5*\ps:3.25*\r) 
	(90:3.75*\r) -- (90+0.5*\ps:3.25*\r) 
;
}

\draw[fill=teal!75!blue, line width=0.6] 
	(270:\r) -- (270+\ps:\r) -- (270+2*\ps:\r) -- (270+3*\ps:\r) -- (270+4*\ps:\r) -- cycle
;

\foreach \a in {0,...,4} {
\tikzset{rotate=\a*\ps}

\draw[fill=teal!75!blue, line width=0.6] 
	(270:\r) -- (270-0.5*\ps:\x+\r) -- (270-0.25*\ps:2.5*\r) -- (270+0.25*\ps:2.5*\r)  -- (270+0.5*\ps:\x+\r) -- cycle
;

\draw[fill=teal!75!blue, line width=0.6] 
	(90-0.25*\ps:2.5*\r) -- (90-0.5*\ps:3.25*\r) -- (90:3.75*\r) -- (90+0.5*\ps:3.25*\r) -- (90+0.25*\ps:2.5*\r)
;
}

\end{scope}

\begin{scope}[xshift=1.15*\xs cm] 
\node [inner sep=0] (image) at (0,0) 
            {\includegraphics[height=2cm]{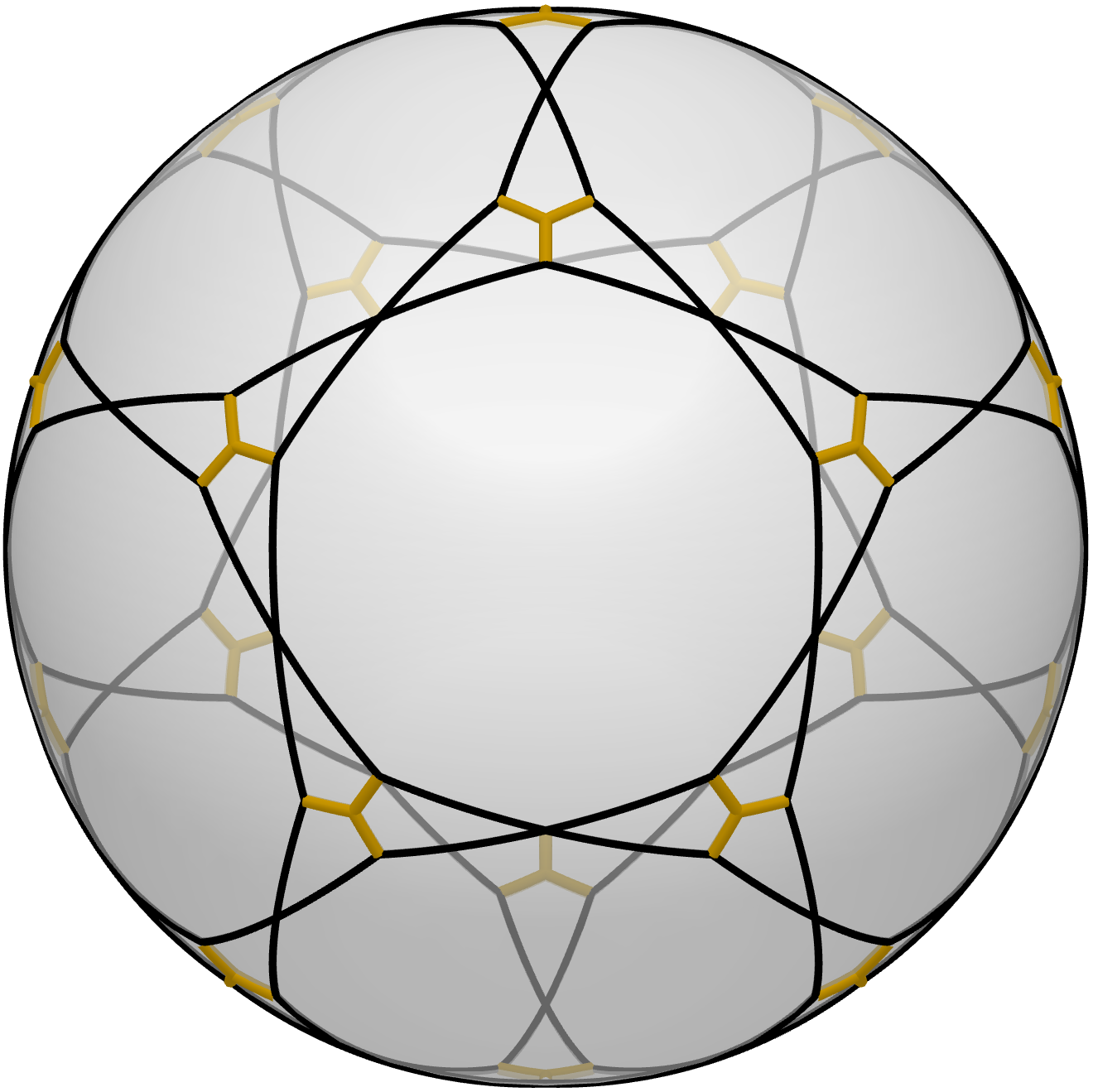}};
\end{scope} 

}
\end{tikzpicture}
\caption{Subdivided $a\mathcal{I}$}
\end{subfigure}
\caption{\ref{Label:rect-Platonic} via a kite subdivision of the (uncoloured) legacy faces in the rectified Platonic solids}
\label{Fig-Platonic-Archimedean}
\end{figure}


In Figure \ref{Fig-Platonic-Archimedean}, the drawings of \ref{Label:rect-Platonic} without thick edges are respectively the graphs of the octahedron $\mathcal{O}$, the cuboctahedron $a\mathcal{C}$ (equivalently $a\mathcal{O}$) and the icosidodecahedron $a\mathcal{D}$ (equivalently $a\mathcal{I}$) as rectified Platonic solids. The thick edges divide the legacy faces from the original solids into kites. Because of duality, one may also view them as kite subdivision of the derived faces of the rectification. 

In Figure \ref{Fig-Platonic-multigraph-Cube-Dodeca}, the drawings of \ref{Label:multigr-Cube-Dodeca} without thick edges are the $2$-edged multigraphs of the cube and the dodecahedron. Here a multigraph is {\em $2$-edged}, which means that each pair of adjacent vertices are the end vertices of exactly two incident edges between them. The thick edges divide the faces of the cube and the dodecahedron into kites.

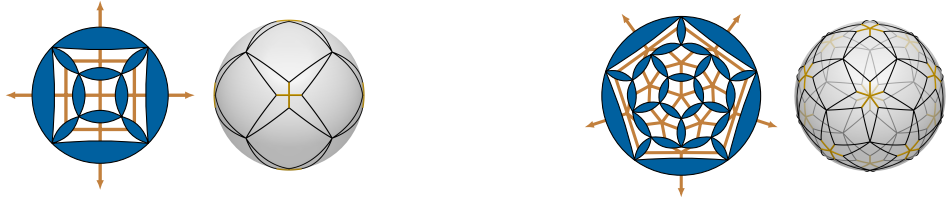
\begin{figure}[h!] 
\centering
\begin{subfigure}[t]{0.45\linewidth}
\centering
\begin{tikzpicture}[>=latex]
\tikzmath{
\xs=2.5;
}

\begin{scope}[]
\tikzmath{
\r=0.4;
\th=360/4;
\x=\r*cos(0.5*\th);
}

\foreach \a in {0,...,3} {
\tikzset{rotate=\a*\th}

\draw[HGold, line width=1.25]
	(0,0) -- (0,\x)
	(0,1.6*\x) -- (0,\x)
	(0,1.6*\x) -- (0,2.5*\x)
	(0,1.6*\x) -- (0.5*\th: 1.6*\r)
	(0,1.6*\x) -- (1.5*\th: 1.6*\r)
;
\draw[arrows = {-Latex[scale=0.45]}, HGold, line width=1.25]
	(90:2.25*\r) -- (90:3.15*\r)
;
}

\foreach \a in {0,...,3} {
\tikzset{rotate=\a*\th}

\fill[teal!75!blue]
	(0.5*\th:\r) to[out=150, in=30] (1.5*\th:\r) to[out=330, in=210] (0.5*\th:\r) 
	(0.5*\th:\r) to[out=75, in=195] (0.5*\th:2.25*\r) to[out=255, in=15] (0.5*\th:\r) 
	(0.5*\th:2.25*\r) to[out=185, in=355] (1.5*\th:2.25*\r) to[out=45, in=135] (0.5*\th:2.25*\r)
;
}

\foreach \a in {0,...,3} {
\tikzset{rotate=\a*\th}

\draw[]
	(0.5*\th:\r) to[out=150, in=30] (1.5*\th:\r)
	(0.5*\th:\r) to[out=210, in=-30] (1.5*\th:\r)
	(0.5*\th:\r) to[out=75, in=195] (0.5*\th:2.25*\r)
	(0.5*\th:\r) to[out=15, in=255] (0.5*\th:2.25*\r)
	(0.5*\th:2.25*\r) to[out=185, in=355] (1.5*\th:2.25*\r)
;
}

\draw[] (0,0) circle (2.25*\r);

\end{scope}

\begin{scope}[xshift=\xs cm] 
\node [inner sep=0] (image) at (0,0) 
            {\includegraphics[height=2cm]{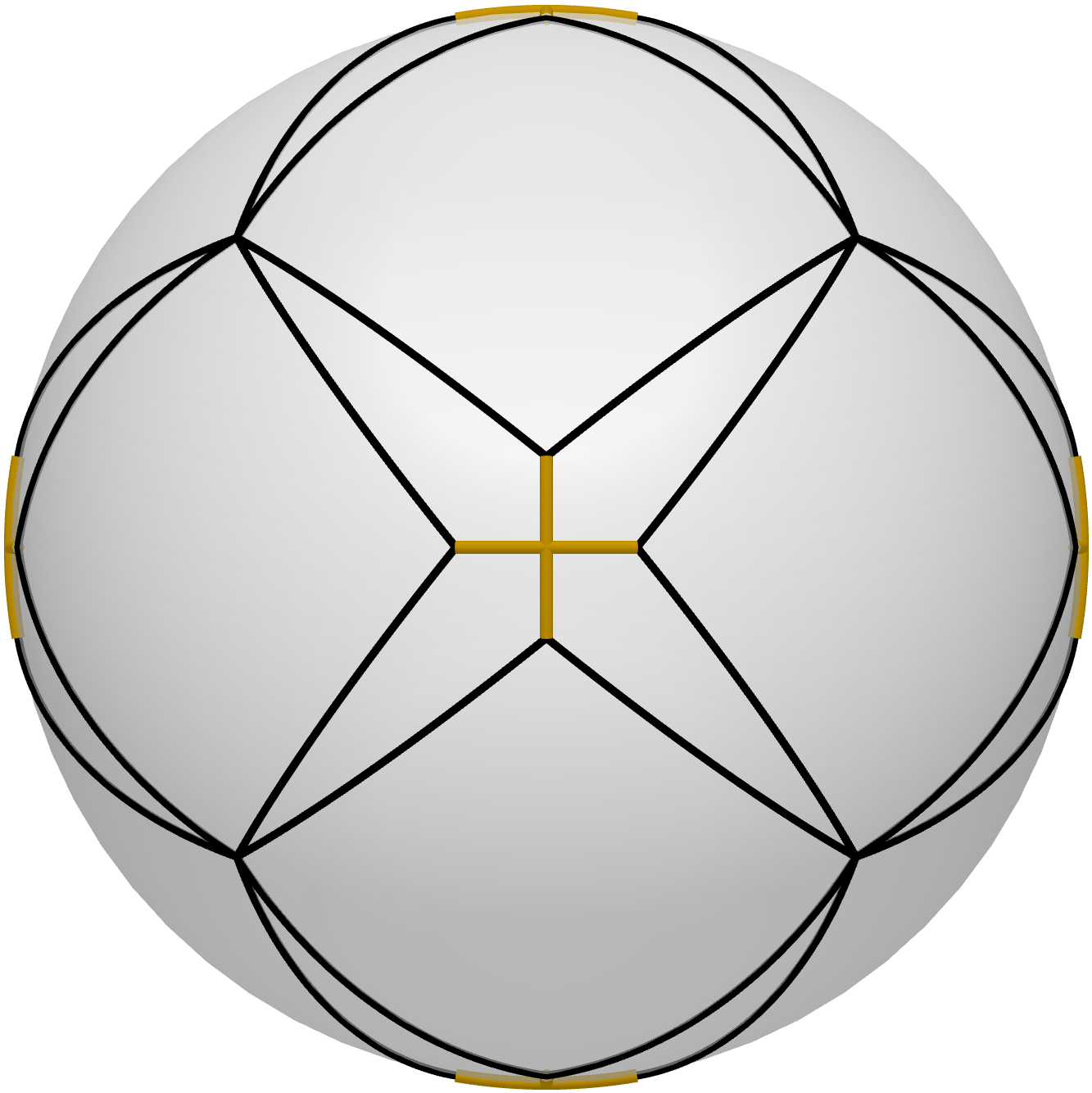}};
\end{scope} 
\node at (0,-0.5*\xs) {};
\end{tikzpicture}
\caption{Subdivided $2$-edged multigraph of $\mathcal{C}$}
\end{subfigure}
\begin{subfigure}[t]{0.54\linewidth}
\centering
\begin{tikzpicture}[>=latex]

\tikzmath{
\xs=2.5;
}

\raisebox{0ex}{

\begin{scope}[]
\tikzmath{
\r=0.3;
\ph=360/5;
\X=\r*cos(0.5*\ph);
}

\foreach \a in {0,...,4} {
\tikzset{rotate=\a*\ph}

\draw[HGold, line width=1.15]
	(0,0) -- (0,\X)
	(0,1.8*\X) -- (0,\X)
	(0,1.8*\X) -- (90-0.5*\ph:1.5*\r)
	(0,1.8*\X) -- (90+0.5*\ph:1.5*\r)
	(0,1.8*\X) -- (90-0.25*\ph:2.2*\r)
	(0,1.8*\X) -- (90+0.25*\ph:2.2*\r)
	(270:2.5*\r) -- (270+0.25*\ph:2.2*\r) 
	(270:2.5*\r) -- (270-0.25*\ph:2.2*\r) 
	(270:2.5*\r) -- (270+0.5*\ph:3.1*\r) 
	(270:2.5*\r) -- (270-0.5*\ph:3.1*\r) 
	(270:2.5*\r) -- (270:3*\r)
;

\draw[arrows = {-Latex[scale=0.45]}, HGold, line width=1.15]
	(90-0.5*\ph:3.5*\r) -- (90-0.5*\ph:4.5*\r)
;
}

\foreach \a in {0,...,4} {
\tikzset{rotate=\a*\ph}

\fill[teal!75!blue]
	(90-0.5*\ph:\r) to[out=150, in=30] (90+0.5*\ph:\r) to[out=-30, in=210] (90-0.5*\ph:\r)
	(270:\r) to[out=300, in=60] (270:2.25*\r) to[out=120, in=240] (270:\r) 
	(90:2.25*\r) to[out=60, in=300] (90:3.5*\r) to[out=240, in=120] (90:2.25*\r) 
	(90:3.5*\r) to[out=-40, in=180-\ph+40]  (90-\ph:3.5*\r) arc (90-\ph:90:3.5*\r)
;

\foreach \aa in {-1,1} {
\tikzset{xscale=\aa}

\fill[teal!75!blue]
	(270+2*\ph:2.25*\r) to[out=185, in=320] (90:2.25*\r) to[out=20, in=130] (270+2*\ph:2.25*\r) 
;

}
}

\foreach \a in {0,...,4} {
\tikzset{rotate=\a*\ph}

\draw[]
	(90-0.5*\ph:\r) to[out=150, in=30] (90+0.5*\ph:\r)
	(90-0.5*\ph:\r) to[out=210, in=-30] (90+0.5*\ph:\r)
	(270:\r) to[out=300, in=60] (270:2.25*\r)	
	(270:\r) to[out=240, in=120] (270:2.25*\r)
	(90:2.25*\r) to[out=60, in=300] (90:3.5*\r)
	(90:2.25*\r) to[out=120, in=240] (90:3.5*\r)
	(90:3.5*\r) to[out=-40, in=180-\ph+40]  (90-\ph:3.5*\r)
;

\foreach \aa in {-1,1} {
\tikzset{xscale=\aa}

\draw[]
	(270+2*\ph:2.25*\r) to[out=185, in=320] (90:2.25*\r)
	(270+2*\ph:2.25*\r) to[out=130, in=20] (90:2.25*\r)
;

}
}

\draw[] (0,0) circle (3.5*\r);

\end{scope}

\begin{scope}[xshift=\xs cm] 

\node [inner sep=0] (image) at (0,0) 
            {\includegraphics[height=2cm]{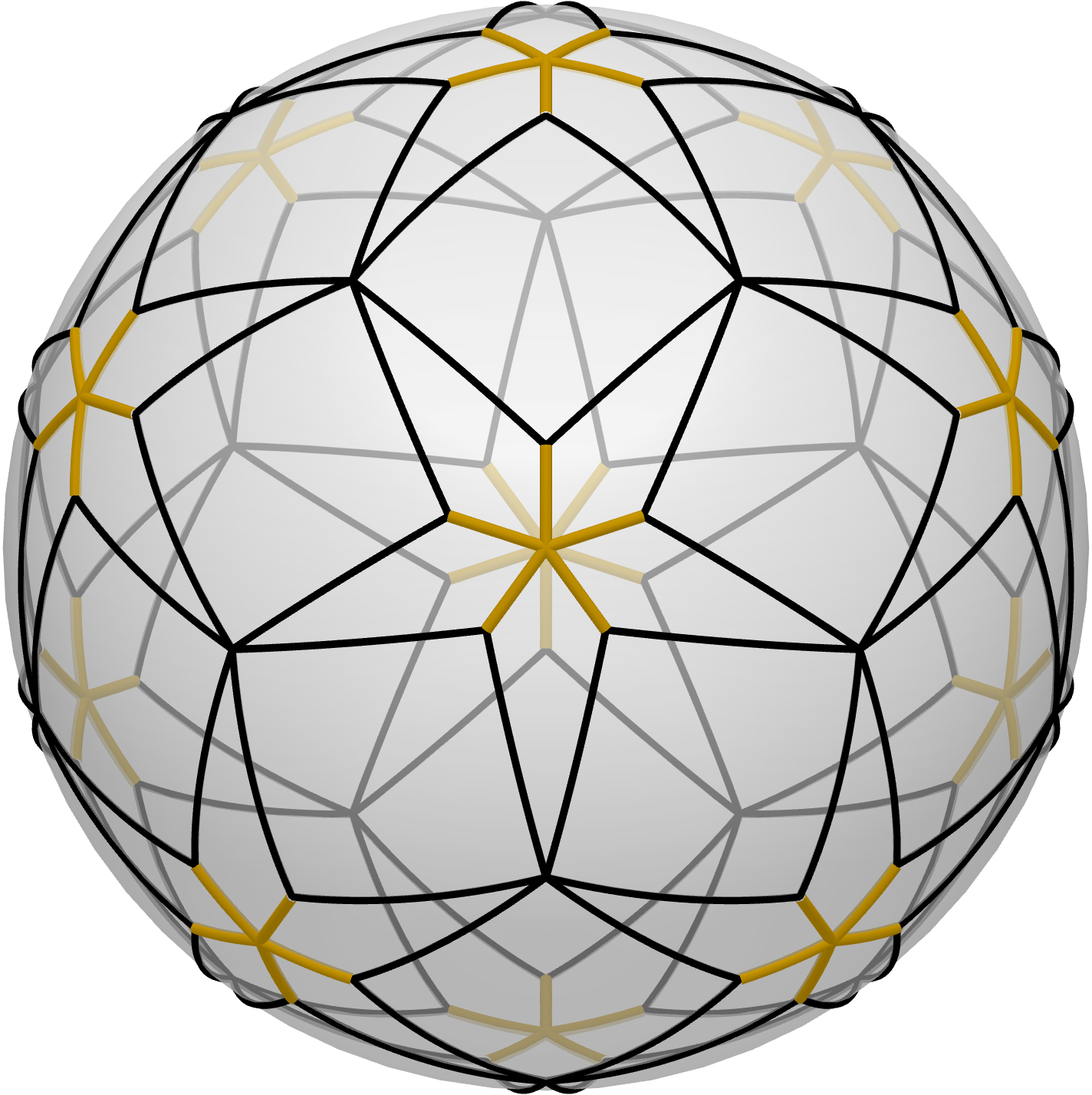}};

\end{scope}

}

\end{tikzpicture}
\caption{Subdivided $2$-edged multigraph of $\mathcal{D}$}
\end{subfigure}
\caption{\ref{Label:multigr-Cube-Dodeca} via a kite subdivision of the $2$-edged multigraphs of the cube and the dodecahedron}
\label{Fig-Platonic-multigraph-Cube-Dodeca}
\end{figure}


In Figure \ref{Fig-Platonic-thicken-Octa}, the first row shows \ref{Label:tri-subdiv-thick-Octa} and the second row shows \ref{Label:quad-subdiv-thick-Octa}. The first column shows the tilings via two operations on the octahedron $\mathcal{O}$, and the second shows those via the same operations on the flipped octahedron $F\mathcal{O}$. The first operation thickens the $1$-skeletons of $\mathcal{O}$ and $F\mathcal{O}$ respectively and subdivides them into squares, see Figure \ref{Fig-skeleton-op}. The other operation is a triangular or a kite subdivision of the legacy faces of the original solids. In fact, \ref{Label:tri-subdiv-thick-Octa} and \ref{Label:quad-subdiv-thick-Octa} are related via ``flipping'' the subdivision in each face all at once. For further details, see the discussions of Figures \ref{Fig-be-Tilings-sq-hex-al2be-al3ga2} and \ref{Fig-be,de-Tilings-sq-hex-de3-al3be-al2ga2}.

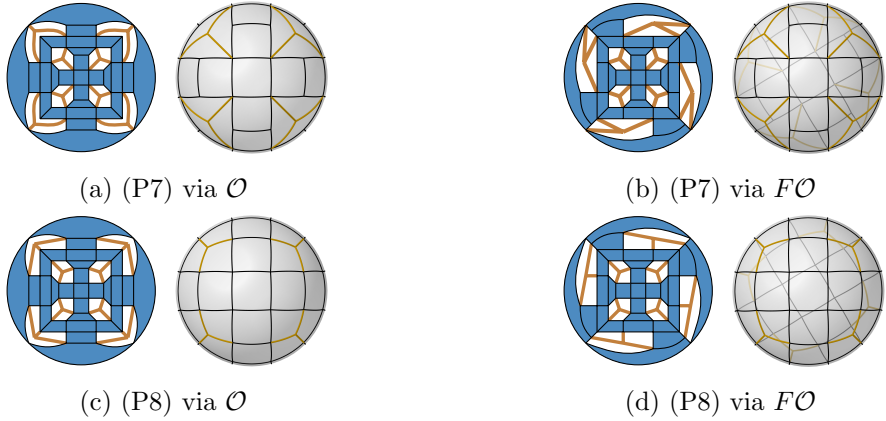
\begin{figure}[h!] 
\centering
\begin{subfigure}[t]{0.475\linewidth}
\centering
\begin{tikzpicture}
\tikzmath{
\xs=2.25;
\s=0.65;
\ys=3.5;
\r=0.14;
\rr=0.15*\r;
\ps=360/3;
\X=\r*cos(0.5*\ps);
\ph=360/6;
\l=\r*cos(0.5*\ph);
\th=360/4;
\x=\r*cos(\th/2);
}


\begin{scope}[]

\fill[teal!70!blue!70, scale=1] (0,0) circle (7*\r);

\fill[white] (0,0) circle (7*\r);

\fill[teal!70!blue!70]
	(\x, \x) -- (-\x, \x) -- (-\x, -\x) -- (\x, -\x) -- cycle
;

\foreach \a in {0,1,2,3} {
\tikzset{rotate=\a*\th}

\fill[teal!70!blue!70]
	(\x, \x) -- (\x, 3*\x) -- (2*\x, 4*\x) -- (2*\x, 7*\x) to[out=20,in=160] (7*\x, 7*\x) arc (45:135:7*\r) 
	to[out=20,in=160] (-2*\x, 7*\x) -- (-2*\x, 4*\x) -- (-\x, 3*\x) -- (-\x, \x) -- cycle
;

\fill[teal!70!blue!70]
	(2*\x, 4*\x) -- (2*\x, 5.5*\x) -- (5.5*\x, 5.5*\x) 
	-- (5.5*\x, 2*\x) -- (4*\x, 2*\x) -- (4*\x, 4*\x) -- cycle
;

\draw[HGold, line width=1.25]
	(2.5*\x, 2.5*\x) -- (2*\x, 4*\x)
	(2.5*\x, 2.5*\x) -- (4*\x, 2*\x)
	(2.5*\x, 2.5*\x) -- (\x, \x)
	(6.25*\x, 6.25*\x) -- (7*\x, 7*\x)
	(6.25*\x, 6.25*\x) to[out=180, in=45] (2*\x, 5.5*\x)
	(6.25*\x, 6.25*\x) to[out=270, in=45] (5.5*\x, 2*\x)
;
}

\foreach \aa in {0,1,2,3} {
\tikzset{shift={(\aa*\th:2*\x)}}

\foreach \a in {0,1,2,3} {
\tikzset{rotate=\a*\th}

\draw[]
	(0.5*\th:\r) -- (1.5*\th:\r)
;
}
}

\foreach \a in {0,1,2,3} {
\tikzset{rotate=\a*\th}

\draw[]
	(\x,3*\x) -- (2*\x, 4*\x)
	(3*\x,\x) -- (4*\x, 2*\x)
	(2*\x, 4*\x) -- (4*\x, 4*\x)
	(4*\x,2*\x) -- (4*\x, 4*\x)
	(2*\x, 4*\x) -- (-2*\x, 4*\x)
	(2*\x, 4*\x) -- (2*\x, 5.5*\x)
	(-2*\x, 4*\x) -- (-2*\x, 5.5*\x)
	(-2*\x, 5.5*\x) -- (2*\x, 5.5*\x)
	(4*\x, 4*\x) -- (5.5*\x, 5.5*\x)
	(2*\x, 5.5*\x) -- (5.5*\x, 5.5*\x)
	(-2*\x, 5.5*\x) -- (-5.5*\x, 5.5*\x)
	(2*\x, 5.5*\x) -- (2*\x, 7*\x)
	(2*\x, 7*\x) to[out=20,in=160] (7*\x, 7*\x)
	(5.5*\x, 2*\x) -- (7*\x,2*\x)
	(7*\x,2*\x) to[out=70,in=290] (7*\x, 7*\x)
	(-2*\x, 7*\x) -- (2*\x, 7*\x)
;
}

\draw (0,0) circle (7*\r);


\end{scope}

\begin{scope}[xshift=\xs cm]
\node [inner sep=0] (image) at (0,0) 
            {\includegraphics[height=2cm]{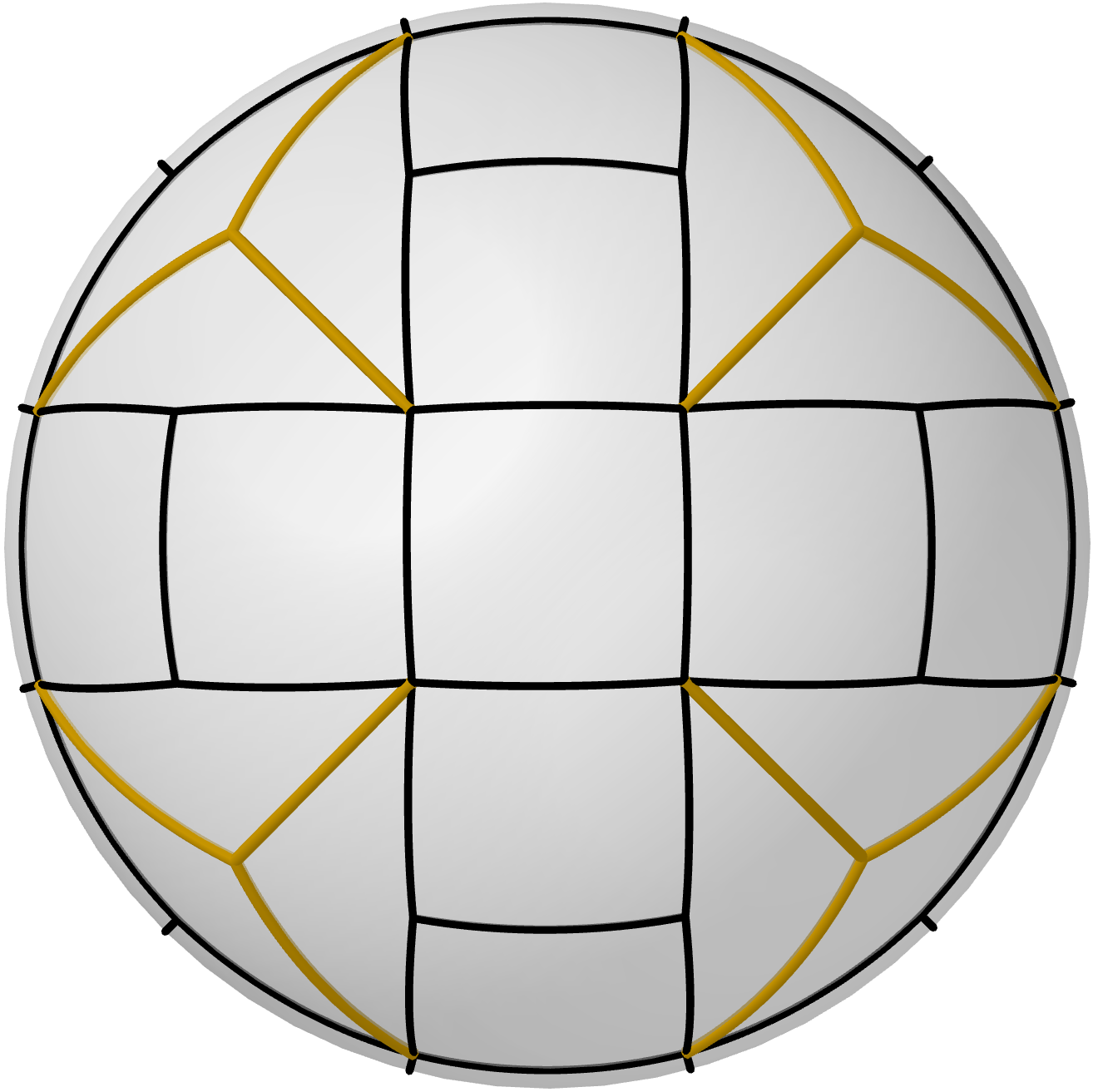}};
\end{scope}
\end{tikzpicture}
\caption{\ref{Label:tri-subdiv-thick-Octa} via $\mathcal{O}$}
\end{subfigure}
\begin{subfigure}[t]{0.475\linewidth}
\centering
\begin{tikzpicture}
\tikzmath{
\xs=2.25;
\s=0.65;
\ys=3.5;
\r=0.14;
\rr=0.15*\r;
\ps=360/3;
\X=\r*cos(0.5*\ps);
\ph=360/6;
\l=\r*cos(0.5*\ph);
\th=360/4;
\x=\r*cos(\th/2);
}

\begin{scope}[]

\fill[teal!70!blue!70, scale=1] (0,0) circle (7*\r);

\fill[white] (0,0) circle (7*\r);

\fill[teal!70!blue!70]
	(\x, \x) -- (-\x, \x) -- (-\x, -\x) -- (\x, -\x) -- cycle
;

\foreach \a in {0,1,2,3} {
\tikzset{rotate=\a*\th}

\fill[teal!70!blue!70]
	(-2*\x, 5.5*\x) -- (-2*\x, 7.75*\x) to[out=20, in=160] (7*\x, 7*\x) arc (45:135:7*\r) -- (-5.5*\x, 5.5*\x) -- cycle
;

\fill[teal!70!blue!70]
	(\x, \x) -- (\x, 3*\x) -- (2*\x, 4*\x) -- (2*\x, 5.5*\x) 
	-- (-2*\x, 5.5*\x) -- (-2*\x, 4*\x) -- (-\x, 3*\x) -- (-\x, \x) -- cycle
;

\fill[teal!70!blue!70]
	(2*\x, 4*\x) -- (2*\x, 5.5*\x) -- (5.5*\x, 5.5*\x) 
	-- (5.5*\x, 2*\x) -- (4*\x, 2*\x) -- (4*\x, 4*\x)   -- cycle
;

\draw[
HGold, 
line width=1.5]
	(2.5*\x, 2.5*\x) -- (2*\x, 4*\x)
	(2.5*\x, 2.5*\x) -- (4*\x, 2*\x)
	(2.5*\x, 2.5*\x) -- (\x, \x)
	(2*\x, 7.25*\x) -- (7*\x, 7*\x)
	(2*\x, 7.25*\x) -- (5.5*\x, 5.5*\x)
	(2*\x, 7.25*\x) -- (-2*\x, 5.5*\x)
;

}

\foreach \aa in {0,1,2,3} {
\tikzset{shift={(\aa*\th:2*\x)}}

\foreach \a in {0,1,2,3} {
\tikzset{rotate=\a*\th}

\draw[]
	(0.5*\th:\r) -- (1.5*\th:\r)
;
}
}

\foreach \a in {0,1,2,3} {
\tikzset{rotate=\a*\th}

\draw[]
	(\x,3*\x) -- (2*\x, 4*\x)
	(3*\x,\x) -- (4*\x, 2*\x)
	(2*\x, 4*\x) -- (4*\x, 4*\x)
	(4*\x,2*\x) -- (4*\x, 4*\x)
	(2*\x, 4*\x) -- (-2*\x, 4*\x)
	(2*\x, 4*\x) -- (2*\x, 5.5*\x)
	(-2*\x, 4*\x) -- (-2*\x, 5.5*\x)
	(-2*\x, 5.5*\x) -- (2*\x, 5.5*\x)
	(4*\x, 4*\x) -- (5.5*\x, 5.5*\x)
	(2*\x, 5.5*\x) -- (5.5*\x, 5.5*\x)
	(-2*\x, 5.5*\x) -- (-5.5*\x, 5.5*\x)
	(5.5*\x, 5.5*\x) -- (7*\x, 7*\x)
	(-2*\x, 5.5*\x) -- (-2*\x, 7.75*\x)
	(-2*\x, 7.75*\x) to[out=180, in=45] (-6.25*\x, 6.25*\x)
	(-2*\x, 7.75*\x) to[out=20, in=160] (7*\x, 7*\x)
;
}

\draw (0,0) circle (7*\r);

\end{scope}

\begin{scope}[xshift=\xs cm]
\node [inner sep=0] (image) at (0,0) 
            {\includegraphics[height=2cm]{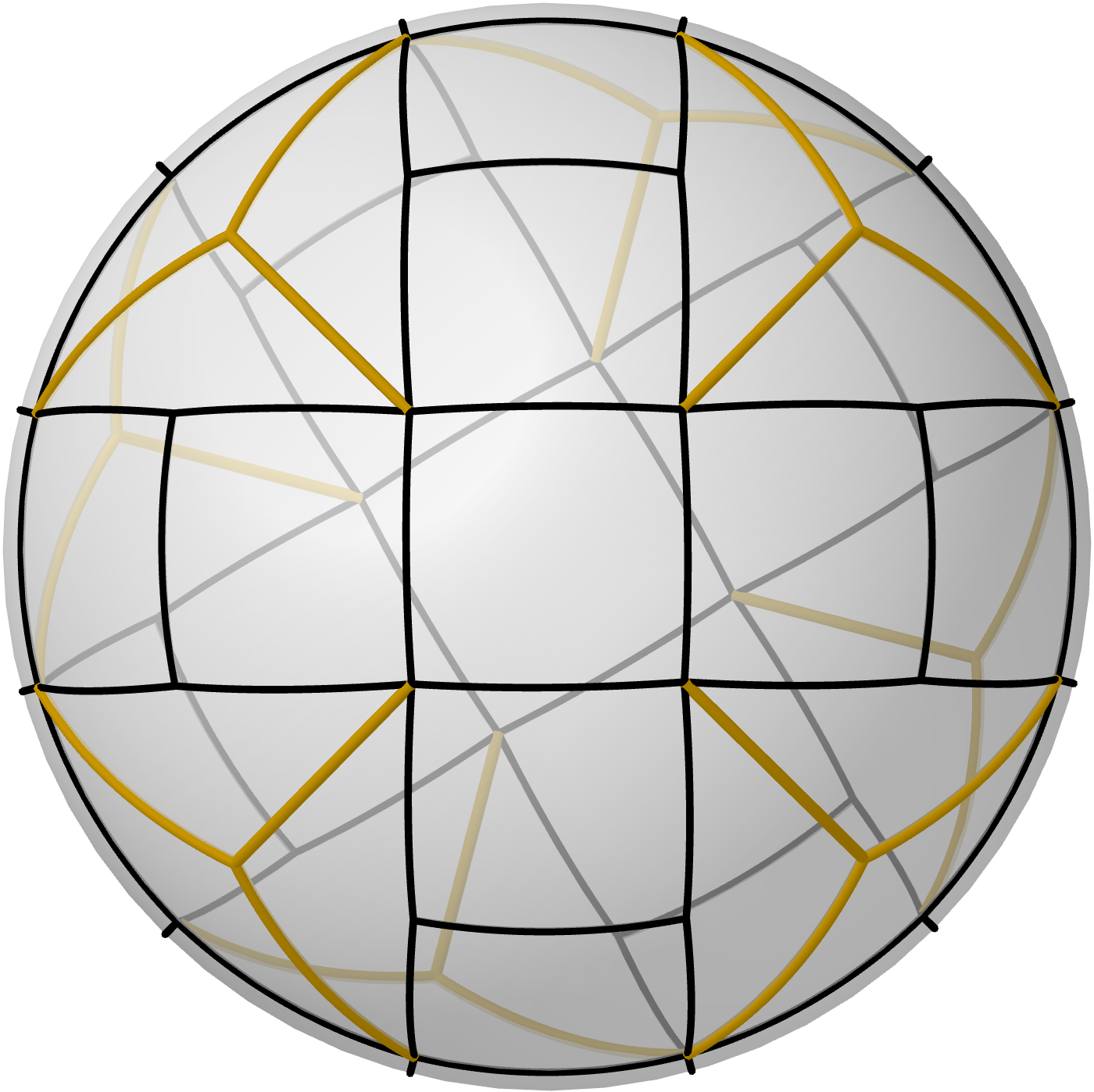}};
\end{scope}
\end{tikzpicture}
\caption{\ref{Label:tri-subdiv-thick-Octa} via $F\mathcal{O}$}
\end{subfigure}
\begin{subfigure}[t]{0.475\linewidth}
\centering
\begin{tikzpicture}
\tikzmath{
\xs=2.25;
\s=0.65;
\ys=3.5;
\r=0.14;
\rr=0.15*\r;
\ps=360/3;
\X=\r*cos(0.5*\ps);
\ph=360/6;
\l=\r*cos(0.5*\ph);
\th=360/4;
\x=\r*cos(\th/2);
}

\begin{scope}[]

\fill[teal!70!blue!70, scale=1] (0,0) circle (7*\r);

\fill[white] (0,0) circle (7*\r);

\fill[teal!70!blue!70]
	(\x, \x) -- (-\x, \x) -- (-\x, -\x) -- (\x, -\x) -- cycle
;

\foreach \a in {0,1,2,3} {
\tikzset{rotate=\a*\th}

\fill[teal!70!blue!70]
	(\x, \x) -- (\x, 3*\x) -- (2*\x, 4*\x) -- (2*\x, 7*\x) to[out=20,in=160] (7*\x, 7*\x) arc (45:135:7*\r) 
	to[out=20,in=160] (-2*\x, 7*\x) -- (-2*\x, 4*\x) -- (-\x, 3*\x) -- (-\x, \x) -- cycle
;

\fill[teal!70!blue!70]
	(2*\x, 4*\x) -- (2*\x, 5.5*\x) -- (5.5*\x, 5.5*\x) 
	-- (5.5*\x, 2*\x) -- (4*\x, 2*\x) -- (4*\x, 4*\x) -- cycle
;

\draw[HGold, line width=1.25]
	(2.5*\x, 2.5*\x) -- (\x, 3*\x)
	(2.5*\x, 2.5*\x) -- (3*\x, \x)
	(2.5*\x, 2.5*\x) -- (4*\x, 4*\x)
	(6.25*\x, 6.25*\x) -- (5.5*\x, 5.5*\x)
	(6.25*\x, 6.25*\x) to[] (2*\x, 7*\x)
	(6.25*\x, 6.25*\x) to[] (7*\x, 2*\x)
;
}

\foreach \aa in {0,1,2,3} {
\tikzset{shift={(\aa*\th:2*\x)}}

\foreach \a in {0,1,2,3} {
\tikzset{rotate=\a*\th}

\draw[]
	(0.5*\th:\r) -- (1.5*\th:\r)
;
}
}

\foreach \a in {0,1,2,3} {
\tikzset{rotate=\a*\th}

\draw[]
	(\x,3*\x) -- (2*\x, 4*\x)
	(3*\x,\x) -- (4*\x, 2*\x)
	(2*\x, 4*\x) -- (4*\x, 4*\x)
	(4*\x,2*\x) -- (4*\x, 4*\x)
	(2*\x, 4*\x) -- (-2*\x, 4*\x)
	(2*\x, 4*\x) -- (2*\x, 5.5*\x)
	(-2*\x, 4*\x) -- (-2*\x, 5.5*\x)
	(-2*\x, 5.5*\x) -- (2*\x, 5.5*\x)
	(4*\x, 4*\x) -- (5.5*\x, 5.5*\x)
	(2*\x, 5.5*\x) -- (5.5*\x, 5.5*\x)
	(-2*\x, 5.5*\x) -- (-5.5*\x, 5.5*\x)
	(2*\x, 5.5*\x) -- (2*\x, 7*\x)
	(2*\x, 7*\x) to[out=20,in=160] (7*\x, 7*\x)
	(5.5*\x, 2*\x) -- (7*\x,2*\x)
	(7*\x,2*\x) to[out=70,in=290] (7*\x, 7*\x)
	(-2*\x, 7*\x) -- (2*\x, 7*\x)
;
}

\draw (0,0) circle (7*\r);

\end{scope}

\begin{scope}[xshift=\xs cm]
\node [inner sep=0] (image) at (0,0) 
            {\includegraphics[height=2cm]{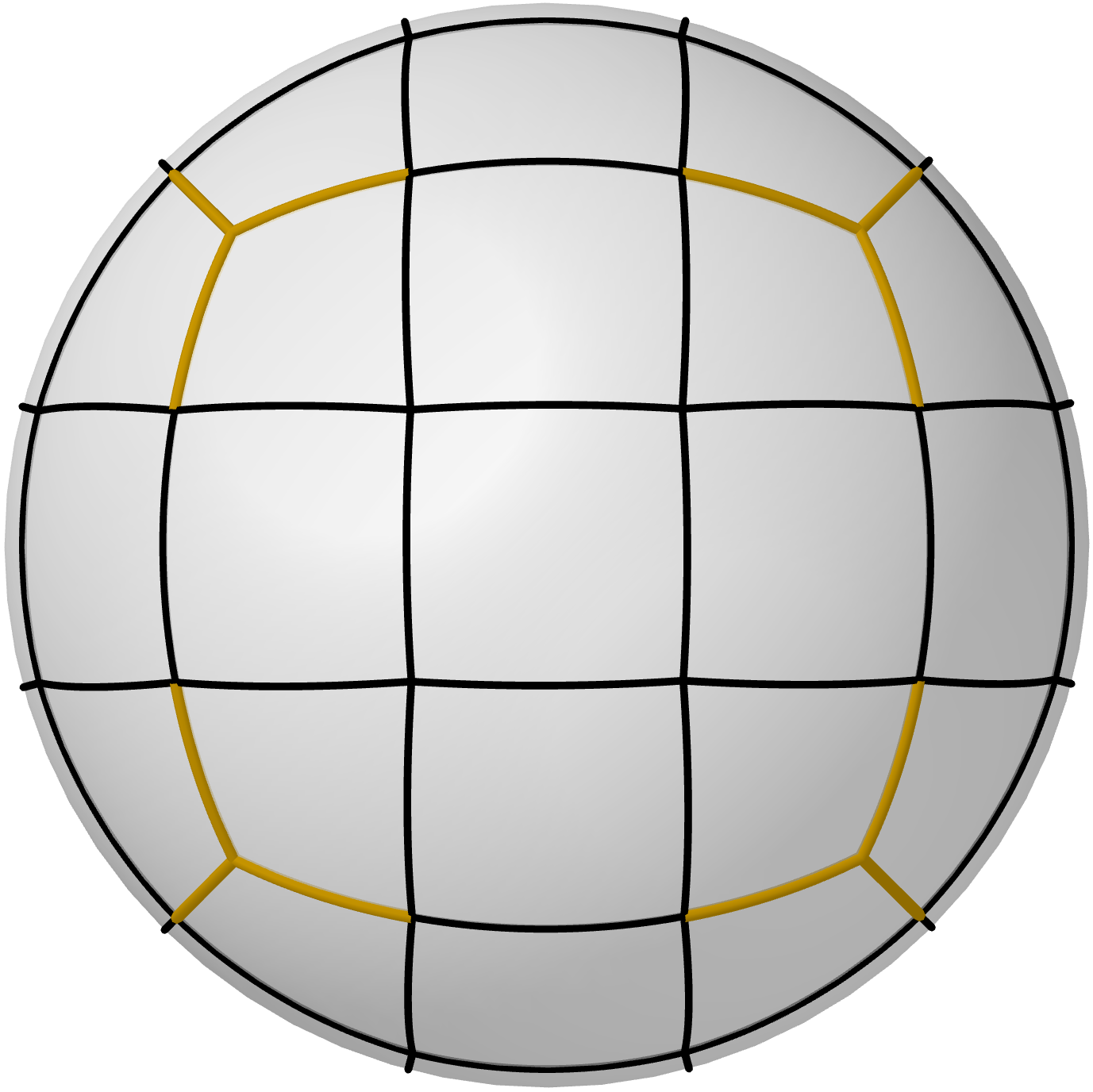}};
\end{scope}
\end{tikzpicture}
\caption{\ref{Label:quad-subdiv-thick-Octa} via $\mathcal{O}$}
\end{subfigure}
\begin{subfigure}[t]{0.475\linewidth}
\centering
\begin{tikzpicture}
\tikzmath{
\xs=2.25;
\s=0.65;
\ys=3.5;
\r=0.14;
\rr=0.15*\r;
\ps=360/3;
\X=\r*cos(0.5*\ps);
\ph=360/6;
\l=\r*cos(0.5*\ph);
\th=360/4;
\x=\r*cos(\th/2);
}

\begin{scope}[]

\fill[teal!70!blue!70, scale=1] (0,0) circle (7*\r);

\fill[white] (0,0) circle (7*\r);

\fill[teal!70!blue!70]
	(\x, \x) -- (-\x, \x) -- (-\x, -\x) -- (\x, -\x) -- cycle
;

\foreach \a in {0,1,2,3} {
\tikzset{rotate=\a*\th}

\fill[teal!70!blue!70]
	(-2*\x, 5.5*\x) -- (-2*\x, 7.75*\x) to[out=20, in=160] (7*\x, 7*\x) arc (45:135:7*\r) -- (-5.5*\x, 5.5*\x) -- cycle
;

\fill[teal!70!blue!70]
	(\x, \x) -- (\x, 3*\x) -- (2*\x, 4*\x) -- (2*\x, 5.5*\x) 
	-- (-2*\x, 5.5*\x) -- (-2*\x, 4*\x) -- (-\x, 3*\x) -- (-\x, \x) -- cycle
;

\fill[teal!70!blue!70]
	(2*\x, 4*\x) -- (2*\x, 5.5*\x) -- (5.5*\x, 5.5*\x) 
	-- (5.5*\x, 2*\x) -- (4*\x, 2*\x) -- (4*\x, 4*\x)   -- cycle
;

\draw[HGold, line width=1.25]
	(2.5*\x, 2.5*\x) -- (\x, 3*\x)
	(2.5*\x, 2.5*\x) -- (3*\x, \x)
	(2.5*\x, 2.5*\x) -- (4*\x, 4*\x)
	(2*\x, 7*\x) -- (6.25*\x, 6.25*\x)
	(2*\x, 7*\x) -- (-2*\x, 7.75*\x) 
	(2*\x, 7*\x) -- (2*\x, 5.5*\x)
;
}

\foreach \aa in {0,1,2,3} {
\tikzset{shift={(\aa*\th:2*\x)}}

\foreach \a in {0,1,2,3} {
\tikzset{rotate=\a*\th}

\draw[]
	(0.5*\th:\r) -- (1.5*\th:\r)
;

}
}

\foreach \a in {0,1,2,3} {
\tikzset{rotate=\a*\th}

\draw[]
	(\x,3*\x) -- (2*\x, 4*\x)
	(3*\x,\x) -- (4*\x, 2*\x)
	(2*\x, 4*\x) -- (4*\x, 4*\x)
	(4*\x,2*\x) -- (4*\x, 4*\x)
	(2*\x, 4*\x) -- (-2*\x, 4*\x)
	(2*\x, 4*\x) -- (2*\x, 5.5*\x)
	(-2*\x, 4*\x) -- (-2*\x, 5.5*\x)
	(-2*\x, 5.5*\x) -- (2*\x, 5.5*\x)
	(4*\x, 4*\x) -- (5.5*\x, 5.5*\x)
	(2*\x, 5.5*\x) -- (5.5*\x, 5.5*\x)
	(-2*\x, 5.5*\x) -- (-5.5*\x, 5.5*\x)
	(5.5*\x, 5.5*\x) -- (7*\x, 7*\x)
	(-2*\x, 5.5*\x) -- (-2*\x, 7.75*\x)
	(-2*\x, 7.75*\x) to[out=180, in=45] (-6.25*\x, 6.25*\x)
	(-2*\x, 7.75*\x) to[out=20, in=160] (7*\x, 7*\x)
;
}

\draw (0,0) circle (7*\r);
\end{scope}

\begin{scope}[xshift=\xs cm] 
\node [inner sep=0] (image) at (0,0) 
            {\includegraphics[height=2cm]{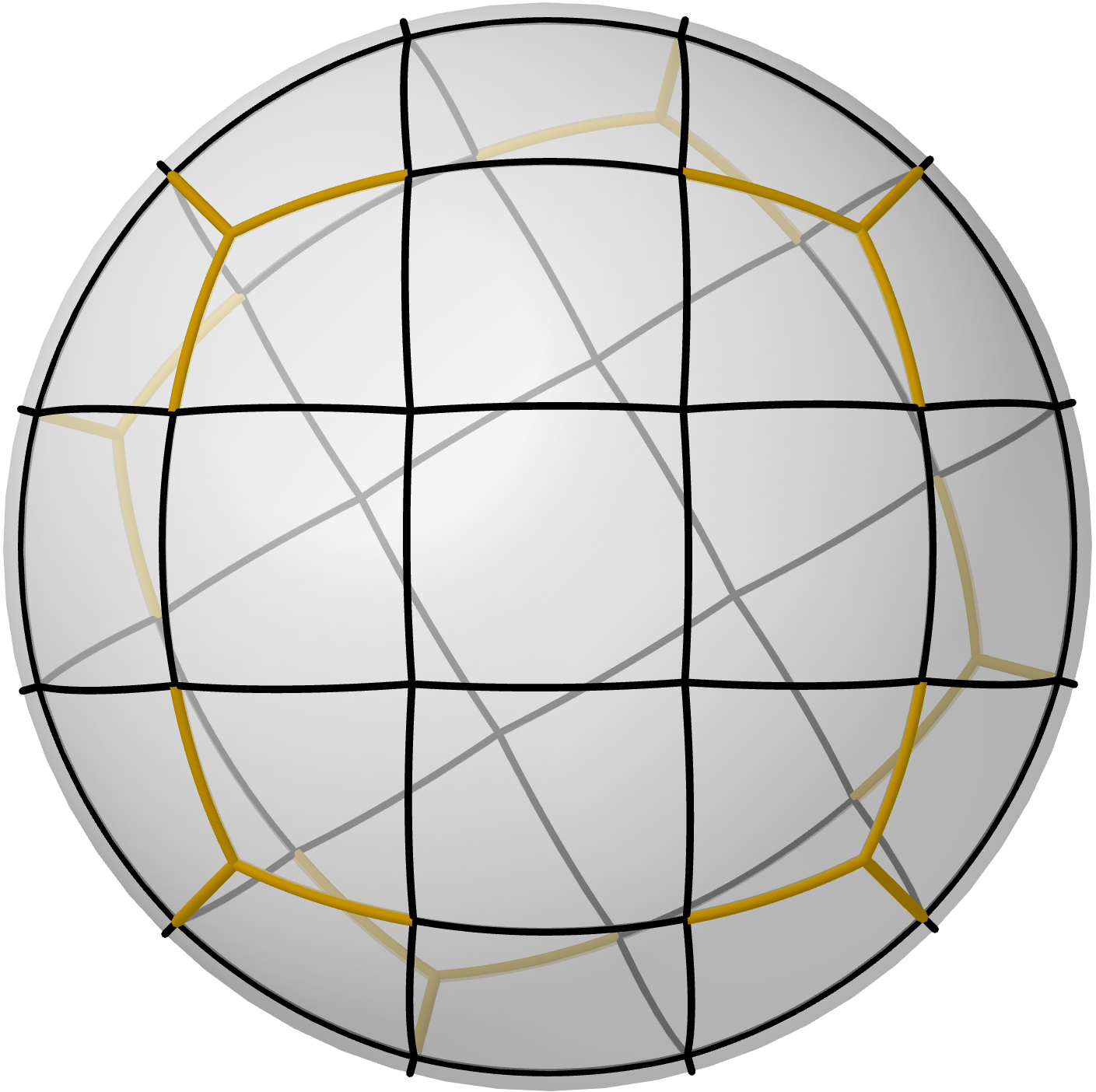}};
\end{scope}
\end{tikzpicture}
\caption{\ref{Label:quad-subdiv-thick-Octa} via $F\mathcal{O}$}
\end{subfigure}
\caption{\ref{Label:tri-subdiv-thick-Octa} in the first row and \ref{Label:quad-subdiv-thick-Octa} in the second;
the left (resp. the right) two tilings via subdividing the thickened $1$-skeletons and the (uncoloured) faces from the octahedron (resp. the flipped octahedron)}
\label{Fig-Platonic-thicken-Octa}
\end{figure}


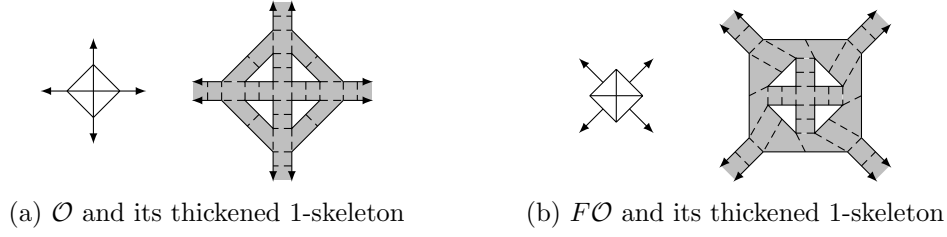
\begin{figure}[h!] 
\centering
\begin{subfigure}[t]{0.45\linewidth}
\centering
\begin{tikzpicture}[>=latex]
\tikzmath{
\r=0.35;
\rr=0.1*\r;
\ys=1;
\xs=2.5;
\th=90;
\x=\r*cos(0.5*\th);
}

\begin{scope}[]

\foreach \a in {0,1,2,3} {
\tikzset{rotate=\a*\th}

\draw[]
	(0,0) -- (0,1*\r) 
	(0,\r) -- (\r,0)
;

\draw[->]
	(0,\r) -- (0,2*\r)
;

}


\end{scope}

\begin{scope}[xshift=\xs cm] 

\foreach \a in {0,1,2,3} {
\tikzset{rotate=\a*\th}

\fill[gray!50]
	(2*\x, 0.5*\x) -- (2*\x, -0.5*\x) -- (3.25*\x, -0.5*\x) -- (3.25*\x, 0.5*\x) -- (0.5*\x,3.25*\x) -- (0.5*\x,2*\x) -- cycle
;

\fill[gray!50]
	(-0.5*\x, 0.5*\x) -- (-0.5*\x, -0.5*\x) -- (4.75*\x, -0.5*\x) -- (4.75*\x, 0.5*\x) -- cycle
;

}

\foreach \a in {0,1,2,3} {
\tikzset{rotate=\a*\th}

\draw[]
	(0.5*\x, 0.5*\x) -- (0.5*\x, 2*\x) -- (2*\x, 0.5*\x) -- cycle
	(0.5*\x, 3.25*\x) -- (3.25*\x, 0.5*\x)
	%
;


\draw[->]
	(0.5*\x, 3.25*\x) -- (0.5*\x, 4.85*\x)
;

\draw[->]
	(-0.5*\x, 3.25*\x) -- (-0.5*\x, 4.85*\x)
;

\draw[densely dashed]
	(0.5*\x, 0.5*\x) -- (-0.5*\x, 0.5*\x)
	(1.25*\x, 0.5*\x) -- (1.25*\x, -0.5*\x)
	(2*\x, 0.5*\x) -- (2*\x, -0.5*\x)
	($(0.5*\x, 2*\x) !1/2! (2*\x, 0.5*\x)$) -- ($(0.5*\x, 3.25*\x) !1/2! (3.25*\x, 0.5*\x)$)
	(2*\x, 0.5*\x) -- (3.25*\x, 0.5*\x)
	(2*\x, -0.5*\x) -- (3.25*\x, -0.5*\x)
	(3.25*\x, 0.5*\x) -- (3.25*\x, -0.5*\x)
	(4*\x, 0.5*\x) -- (4*\x, -0.5*\x) 
;

}


\end{scope}

\end{tikzpicture}
\caption{$\mathcal{O}$ and its thickened $1$-skeleton}
\end{subfigure}
\begin{subfigure}[t]{0.45\linewidth}
\centering
\begin{tikzpicture}[>=latex]
\tikzmath{
\r=0.35;
\rr=0.1*\r;
\ys=1;
\xs=2.5;
\th=90;
\x=\r*cos(0.5*\th);
}

\begin{scope}[]

\foreach \a in {0,1,2,3} {
\tikzset{rotate=\a*\th}

\draw[]
	(0,0) -- (0,1*\r) 
	(0,\r) -- (\r,0)
;

\draw[->]
	(45:\x) -- (45:2*\r)
;
}


\end{scope}

\begin{scope}[xshift=\xs cm]

\foreach \a in {0,1,2,3} {
\tikzset{rotate=\a*\th}

\fill[gray!50]
	(2*\x, 0.5*\x) -- (2*\x, -0.5*\x) -- (3*\x, 0) -- (3*\x,2.25*\x) -- (4.5*\x, 3.75*\x) -- (3.75*\x,4.5*\x)
	-- (2.25*\x, 3*\x) -- (-0.5*\x, 3*\x) -- cycle
	(-0.5*\x, 0.5*\x) -- (2*\x, 0.5*\x) -- (2*\x, -0.5*\x) -- (-0.5*\x, -0.5*\x) -- cycle
;
}

\foreach \a in {0,1,2,3} {
\tikzset{rotate=\a*\th}

\draw[densely dashed]
	%
	%
;

\draw[HGold, line width=1.15, densely dashed]
	%
;
}

\foreach \a in {0,1,2,3} {
\tikzset{rotate=\a*\th}

\draw[]
	(0.5*\x, 0.5*\x) -- (0.5*\x, 2*\x) -- (2*\x, 0.5*\x) -- cycle
	(2.25*\x,3*\x) -- (-2.25*\x,3*\x) 
	%
;


\draw[->]
	(2.25*\x,3*\x) -- (3.85*\x,4.6*\x) 
;

\draw[->]
	(3*\x,2.25*\x) -- (4.6*\x, 3.85*\x)
;

\draw[densely dashed]
	(0.5*\x, 0.5*\x) -- (-0.5*\x, 0.5*\x)
	(1.25*\x, 0.5*\x) -- (1.25*\x, -0.5*\x)
	(2*\x, 0.5*\x) -- (2*\x, -0.5*\x)
	(0.5*\x, 2*\x) -- (0, 3*\x) 
	(-0.5*\x, 2*\x) -- (-2.25*\x,3*\x) 
	($(0.5*\x, 2*\x) !1/2! (2*\x, 0.5*\x)$) -- (2.25*\x,3*\x) 
	(2.25*\x,3*\x) -- (3*\x,2.25*\x) 
	(3.75*\x,3*\x) -- (3*\x,3.75*\x) 
;

}


\end{scope}

\end{tikzpicture}
\caption{$F\mathcal{O}$ and its thickened $1$-skeleton}
\end{subfigure}

\caption{The thickened skeletons of the octahedron and the flipped octahedron}
\label{Fig-skeleton-op}
\end{figure}


In Figure \ref{Fig-Tiling-Subdiv-J1}, the drawing without thick edges is the graph of the square pyramid $\mathcal{J}_1$; the thick edges subdivide each triangle into three kites and turn the square into a regular octagon. 

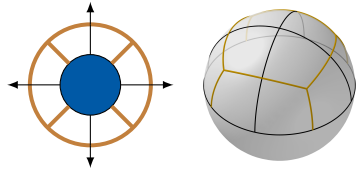
\begin{figure}[h!]
\centering
\begin{tikzpicture}[>=latex]
\tikzmath{
\r=0.4;
\n=8;
\nn=\n-1;
\th=360/\n;
\XS=2.5;
}

\fill[teal!70!blue] (0,0) circle (\r);

\draw[] (0,0) circle (\r);

\draw[orange!75!brown!75!black!80, line width=1.5] (0,0) circle (2*\r);

\foreach \a in {0,2,...,6} {
\tikzset{rotate=\a*\th}
\draw[->]
	(90:2*\r) -- (90:2.75*\r)
;
}

\foreach \a in {0,2,...,6} {
\tikzset{rotate=\a*\th}
\draw[]
	(90:\r) -- (90:2*\r) 
;
\draw[orange!75!brown!75!black!80, line width=1.5]
	(90-\th:\r) -- (90-\th:2*\r)
;
}

\begin{scope}[] 
\node [inner sep=0] (image) at (\XS,0) 
            {\includegraphics[height=2cm]{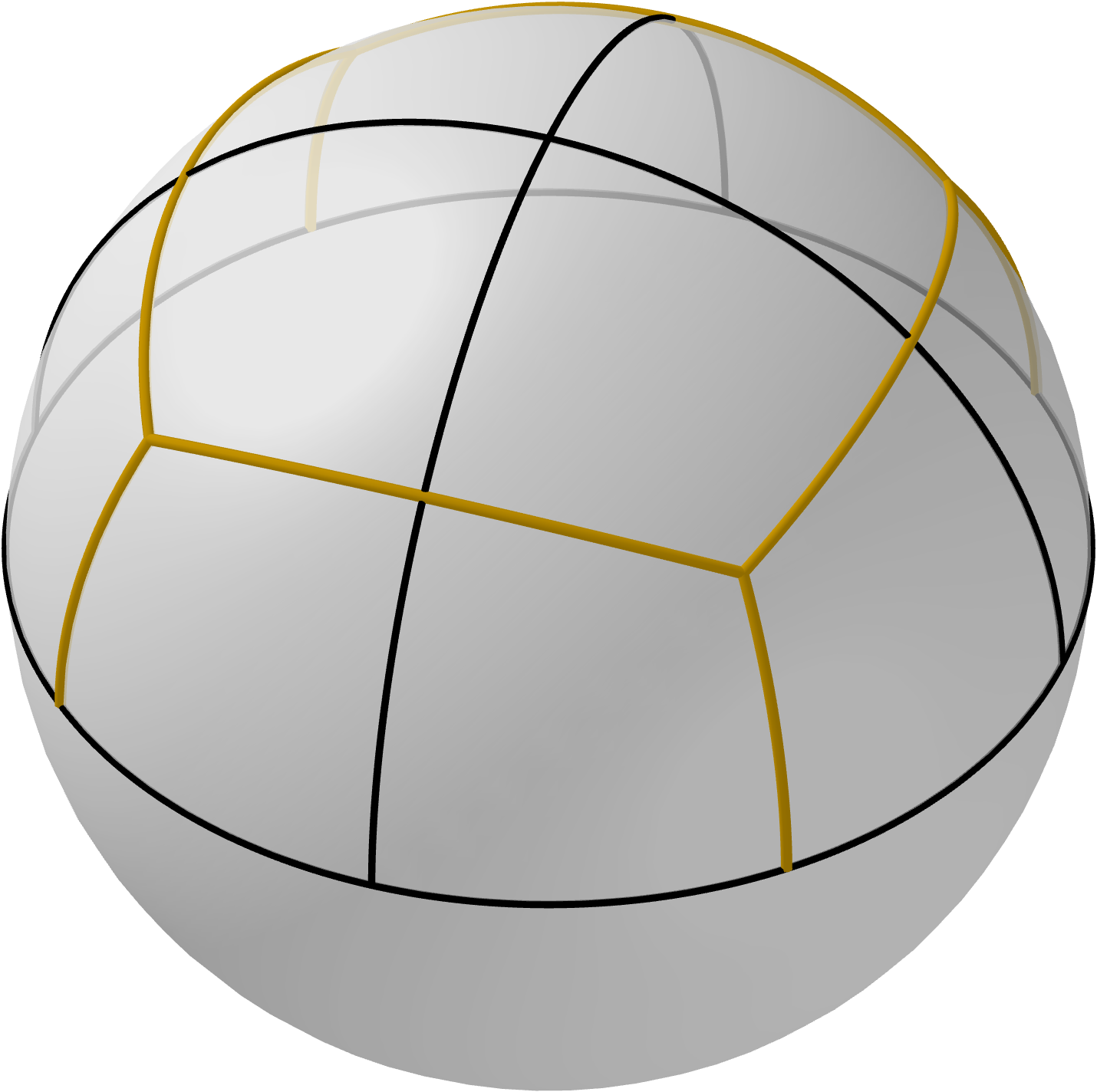}};
\end{scope}

\end{tikzpicture}
\caption{The kite subdivision of the square pyramid $\mathcal{J}_1$}
\label{Fig-Tiling-Subdiv-J1}
\end{figure}

\section{Terminology \& Tools}
\label{Sec-Toolbox-Strategy}

The key to classification is to determine the admissible vertex types in tilings in terms of their incident angle combinations. Our tools are developed to serve this purpose. Terminology is inherited from \cite{cly,cly2} with minimum introduction to be self-contained.

\subsection{Terminology}
\label{Subsec-terminology}

A vertex with full incident angle information is represented by its incident angles counting multiplicities (without specifying the angle arrangement) in a product form. For example, $\alpha^a\beta^b\gamma^c$ denotes a vertex having $a$ copies of $\alpha$ and $b$ copies of $\beta$ and $c$ copies of $\gamma$ as incident angles. Its {\em vertex angle sum} is
\begin{align*}
a \alpha + b \beta + c \gamma = 2\pi.
\end{align*} 
In practice, we implicitly assume $a,b,c >0$ unless otherwise specified. For example, $\alpha\beta^2$ represents a vertex with one $\alpha$, two $\beta$'s and no $\gamma$, i.e., $a=1, b=2$ and $c=0$. 

A vertex with partial incident angle information can also be similarly represented. For example, $\alpha\beta^2\cdots$ means a vertex with at least one $\alpha$ and two $\beta$'s. 

In fact, the same notations above may represent the equivalence class of vertices having the same full/partial incident angle combination. We call such an equivalence class of vertices a {\em vertex type} and the vertex angle sum remains the same. For this reason, in practice we may not need to distinguish between vertex and vertex type. 

Various algebraic, geometric and combinatorial constraints lead to a list of {\em admissible} vertex types in a tiling, which is called an {\em anglewise vertex combination} ($\AVC$). One constraint is the consistency of the linear system of vertex angle sums. The list can be refined recursively by further geometric or combinatorial deductions, until one obtains either a tiling or a contradiction. An example (\eqref{Eq-AVC-al2be-al2ga2-ded} of \ref{Label:EMT-Prisms}) is given below,
\begin{align*} 
\AVC = \{ \alpha^2\beta, \alpha^2\gamma^2, \delta^d \}.
\end{align*}
One can check that the linear system given by the above $\AVC$ is underdetermined. The generic parameter $d$ takes different values in different tilings within the family. 

To represent the angle arrangement(s) at a vertex, we use the following system of notations. The $x$-edge (resp. the $y$-edge) between angles are stylised as \quotes{ $\vert$ } (resp. \quotes{ $\bvert$ }).  For example, $\alpha_1\gamma_2\cdots$ in Figure \ref{Subfig-albe-alga} denotes a vertex where an $\alpha$ from tile $T_1$ and a $\gamma$ from tile $T_2$ are incident angles; and $\alpha_1\vert\gamma_2\cdots$ serves the same purpose and specifies the adjacency of angles at an $x$-edge. At times, for simplicity we may use $\alpha_1\vert\gamma_2$ just to emphasise the angle arrangement. In addition, the same picture shows that $\alpha\vert\gamma\cdots$ is a vertex if and only if $\alpha\vert\beta\cdots$ is a vertex. Similarly, Figure \ref{Subfig-bebe-gaga} shows that $\gamma\vert\gamma\cdots$ is a vertex if and only if $\beta\vert\beta\cdots$ is also a vertex; and in Figure \ref{Subfig-gaga-dede}, $\gamma \, \bvert \, \gamma \cdots$ is a vertex if and only if $\delta \, \bvert \, \delta \cdots$ is a vertex. Figure \ref{Subfig-dedede} shows an example of the angle arrangement of a full vertex, $\delta^3$, denoted by $\bvert \, \delta \, \bvert \, \delta \, \bvert \, \delta \, \bvert$. The above system of notations serves as substitutes for pictures to streamline the later discussion. 

\begin{figure}[h!] 
\centering
\begin{subfigure}[t]{0.28\linewidth}
\centering
\begin{tikzpicture}
\tikzmath{
\r=0.625;
\th=360/4;
\x=\r*cos(0.5*\th);
\R = sqrt(\x^2+(3*\x)^2);
\aR = acos(3*\x/\R);
}

\begin{scope}

\draw[]
	(0,\x) -- (-1.5*\x, 1.25*\x)
	(0,-\x) -- (-1.5*\x, -1.25*\x)
;

\foreach \a in {0,...,3} {
\draw[xshift=\x cm, rotate=\th*\a]
	(0.5*\th:\r) -- (1.5*\th:\r)
;
}

\foreach \a in {0,3} {
\tikzset{shift={(-\x,0)}}
\node at (0.5*\th+\a*\th: 0.6*\r) {\small $\alpha$};
}

\node at (-1.75*\x, 0.2*\x) {$\vdots$};

{
\tikzset{shift={(\x,0)}}

\draw[line width=2]
	(0.5*\th:\r) -- (3.5*\th:\r)
	(2.5*\th:\r) -- (3.5*\th:\r)
;

\node at (0.5*\th: 0.575*\r) {\small $\gamma$};
\node at (1.55*\th: 0.55*\r) {\small $\beta$};
\node at (2.5*\th: 0.575*\r) {\small $\gamma$};
\node at (3.5*\th: 0.575*\r) {\small $\delta$};
}

\node[inner sep=1,draw=black,line width=0.5,shape=circle] at (-\x,0) {\footnotesize $1$};
\node[inner sep=1,draw=black,line width=0.5,shape=circle] at (\x,0) {\footnotesize $2$};

\end{scope}
\end{tikzpicture}
\caption{$\alpha_1\vert\beta_2\cdots$, $\alpha_1\vert\gamma_2\cdots$}
\label{Subfig-albe-alga}
\end{subfigure}
\begin{subfigure}[t]{0.28\linewidth}
\centering
\begin{tikzpicture}

\tikzmath{
\r=0.625;
\th=360/4;
\x=\r*cos(0.5*\th);
\R = sqrt(\x^2+(3*\x)^2);
\aR = acos(3*\x/\R);
}

\begin{scope}[]

\foreach \aa in {-1, 1} {

\tikzset{shift={(\aa*\x,0)}}

\foreach \a in {0,...,3} {

\draw[rotate=\th*\a]
	(0.5*\th:\r) -- (1.5*\th:\r)
;

}
}

\foreach \aa in {-1, 1} {

\tikzset{shift={(\aa*\x,0)}, xscale=\aa}

\draw[line width=2]
	(0.5*\th:\r) -- (1.5*\th:\r)
	(0.5*\th:\r) -- (3.5*\th:\r)
;

\node at (0.5*\th: 0.575*\r) {\small $\delta$};
\node at (1.55*\th: 0.575*\r) {\small $\gamma$};
\node at (2.5*\th: 0.575*\r) {\small $\beta$};
\node at (3.5*\th: 0.575*\r) {\small $\gamma$};

}

\node[inner sep=1,draw=black,line width=0.5,shape=circle] at (-\x,0) {\footnotesize $1$};
\node[inner sep=1,draw=black,line width=0.5,shape=circle] at (\x,0) {\footnotesize $2$};

\end{scope}
\end{tikzpicture}
\caption{$\beta_1\vert\beta_2\cdots$, $\gamma_1\vert\gamma_2\cdots$}
\label{Subfig-bebe-gaga}
\end{subfigure}
\begin{subfigure}[t]{0.28\linewidth}
\centering
\begin{tikzpicture}

\tikzmath{
\r=0.625;
\th=360/4;
\x=\r*cos(0.5*\th);
\R = sqrt(\x^2+(3*\x)^2);
\aR = acos(3*\x/\R);
}

\begin{scope}[]

\foreach \aa in {-1, 1} {
\tikzset{shift={(\aa*\x,0)}}

\foreach \a in {0,...,3} {
\draw[rotate=\th*\a]
	(0.5*\th:\r) -- (1.5*\th:\r)
;
}
}

\draw[shift={(\x,0)}, line width=2]
	(1.5*\th:\r) -- (2.5*\th:\r)
;

\foreach \aa in {-1, 1} {
\tikzset{shift={(\aa*\x,0)}, xscale=\aa}

\draw[line width=2]
	(0.5*\th:\r) -- (1.5*\th:\r)
;

\node at (0.5*\th: 0.575*\r) {\small $\gamma$};
\node at (1.55*\th: 0.575*\r) {\small $\delta$};
\node at (2.5*\th: 0.575*\r) {\small $\gamma$};
\node at (3.5*\th: 0.575*\r) {\small $\beta$};
}

\node[inner sep=1,draw=black,line width=0.5,shape=circle] at (-\x,0) {\footnotesize $1$};
\node[inner sep=1,draw=black,line width=0.5,shape=circle] at (\x,0) {\footnotesize $2$};

\end{scope}
\end{tikzpicture}
\caption{$\gamma_1\,\bvert\,\gamma_2\cdots$, $\delta_1\,\bvert\,\delta_2\cdots$}
\label{Subfig-gaga-dede}
\end{subfigure}
\begin{subfigure}[t]{0.125\linewidth}
\centering
\begin{tikzpicture}

\tikzmath{
\r=0.625;
\th=360/4;
\x=\r*cos(0.5*\th);
\R = sqrt(\x^2+(3*\x)^2);
\aR = acos(3*\x/\R);
}

\begin{scope}[] 
\tikzmath{
\ph=360/3;
}

\foreach \a in {0,1,2} {
\tikzset{rotate=\a*\ph}
\draw[line width=2]
	(0,0) -- (90:\r)
;

\node at (90-0.5*\ph:0.4*\r) {\small $\delta$};
}

\end{scope}

\end{tikzpicture}
\caption{$\bvert\,\delta\,\bvert\,\delta\,\bvert\,\delta\,\bvert$}
\label{Subfig-dedede}
\end{subfigure}
\caption{Adjacent pairs of tiles; and $\delta^3$}
\label{Fig-a4-a2b2-adj}
\end{figure}
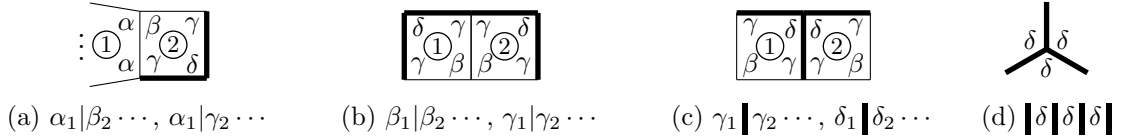

As per Figure \ref{Fig-kite-mgon}, the kite angle sum and the $m$-gon angle sum are
\begin{align}
\label{Eq-mgon-sum}
m\alpha > (m-2)\pi, \\
\label{Eq-kite-sum}
\beta+2\gamma+\delta > 2\pi.
\end{align}
Given $m$-gons with $m\ge4$, equation \eqref{Eq-mgon-sum} further implies that
\begin{align}
\label{Eq-alpha-lbound}
\alpha > (1-\tfrac{2}{m})\pi\ge\tfrac{1}{2}\pi.
\end{align}

\subsection{Combinatorics}

Local and global combinatorics shape vertex types in a tiling. A local fact follows immediately from Figures \ref{Subfig-bebe-gaga} and \ref{Subfig-gaga-dede}. 

\begin{lem}\label{Lem-bebe-gaga-dede} In a tiling with kites, $\beta\vert\beta\cdots$ is a vertex if and only if $\gamma\vert\gamma\cdots$ is a vertex; and $\gamma\,\bvert\,\gamma\cdots$ is a vertex if and only if $\delta\,\bvert\,\delta\cdots$ is a vertex. 
\end{lem}

\begin{lem}[Parity Lemma]\label{Lem-parity} The number of $\gamma$ at a vertex is even. 
\end{lem}

\begin{proof} A vertex with $\gamma$ is of the form $\vert \gamma \, \bvert \cdots$. The edges at both ends must be matched and $\gamma$ is the only angle with both edges $\vert$ and $\bvert$. Hence the claim follows.
\end{proof}

\begin{remark}\label{Rmk-benede}
Since $\gamma$ must appear at some vertex, the above Parity Lemma~\ref{Lem-parity} implies that there is a vertex of type $\gamma^2\cdots$. Its vertex angle sum then implies $\gamma<\pi$. Now if $\beta=\delta$, then by $\gamma<\pi$ the kite can be bisected into two congruent triangles as shown in Figure \ref{Fig-a4-a2b2-kite}, in which $\beta_1=\delta_1$ and $\beta_2=\delta_2$; this would imply that $a=b$ and so the kite would in fact be a square. Hence, in every dihedral tiling under our assumption, we must have either $\beta>\delta$ or $\beta<\delta$.
\end{remark}

The following lemma excludes certain vertices.  

\begin{lem}\label{Lem-bega2-ga2de} If $\beta\gamma^2$ or $\gamma^2\delta$ is a vertex, then the tiling is monohedral. 
\end{lem}


\begin{proof} Suppose $\beta\gamma^2$ is a vertex and it has tiles $T_1, T_2, T_3$ as in Figure \ref{Subfig-kt-emt-bega2}. Parity Lemma (Lemma \ref{Lem-parity}) then implies $\beta_3\gamma_1\cdots=\beta\gamma^2$, which determines $T_4$. Similarly, $\beta_4\gamma_3\cdots=\beta\gamma^2$ determines $T_5$. Since $T_4, T_3, T_5$ form an identical neighbourhood as $T_1, T_2, T_3$ do, the pattern repeats until it returns to the left-hand boundary of $T_1, T_2$, forming a monohedral tiling by kites.

\begin{figure}[h!]
\centering
\begin{subfigure}{0.45\linewidth}
\centering
\begin{tikzpicture}[>=latex]
\tikzmath{ 
\s=2;
\r=1.5;
\rr=0.075*\r;
\tz=2;
\tzz=\tz-1;
}

\begin{scope}[]

\foreach \a in {0,...,\tz}
{
\tikzset{xshift=0.8*\r*\a cm}

\draw[]
	(0.8*\r,0.15*\r) -- (0.4*\r,-0.15*\r)
;

\draw[arrows = {-Latex[scale=0.45]},line width=1.5]	
	(0.8*\r,0.15*\r) -- (0.8*\r,0.6*\r)
;
\draw[arrows = {-Latex[scale=0.45]},line width=1.5]	
	(0.4*\r,-0.15*\r) -- (0.4*\r,-0.6*\r)
;

}

\foreach \a in {0,...,\tzz}{

\tikzset{xshift=0.8*\r*\a cm}

\draw[]
	(0.8*\r,0.15*\r) -- (1.2*\r,-0.15*\r)
;	

	(0.8*\r,0.15*\r) -- (1.2*\r,-0.15*\r)
;

\node at (0.4*\r+0.8*\r,0.5*\r) {\small $\delta$};
\node at (0.15*\r+0.75*\r,0.25*\r) {\small $\gamma$};
\node at (0.4*\r+1.075*\r,0.25*\r) {\small $\gamma$};
\node at (0.4*\r+0.8*\r,0.025*\r) {\small $\beta$};

\node at (0.8*\r,-0.025*\r) {\small $\beta$};
\node at (0.15*\r+0.35*\r,-0.25*\r) {\small $\gamma$};
\node at (0.0*\r+1.075*\r,-0.25*\r) {\small $\gamma$};
\node at (0.8*\r,-0.5*\r) {\small $\delta$};

}

\foreach \gs in {3} {

\tikzmath{
\op=1-\gs*0.16; 
\c=100-\gs*16;
}

\begin{scope}[xshift=\gs*0.8*\r cm]

\draw[line width=0.5, opacity=\op]
	(0.8*\r,0.6*\r) -- (0.8*\r,0.15*\r)
	(0.4*\r,-0.6*\r) -- (0.4*\r,-0.15*\r)
	(0.8*\r,0.15*\r) -- (0.4*\r,-0.15*\r)
;

\draw[arrows = {-Latex[scale=0.45]}, line width=1.5, opacity=\op]	
	(0.8*\r,0.15*\r) -- (0.8*\r,0.6*\r) 
;
\draw[arrows = {-Latex[scale=0.45]}, line width=1.5, opacity=\op]	
	(0.4*\r,-0.15*\r) -- (0.4*\r,-0.6*\r) 
;

\draw[line width=0.5, opacity=\op]
	(0*\r,0.15*\r) -- (0.4*\r,-0.15*\r)
;	

\node at (0.4*\r,0.5*\r) {\small \textcolor{gray!\c}{$\delta$}};
\node at (0.1*\r,0.25*\r) {\small \textcolor{gray!\c}{$\gamma$}};
\node at (0.7*\r,0.25*\r) {\small \textcolor{gray!\c}{$\gamma$}};
\node at (0.4*\r,0.025*\r) {\small \textcolor{gray!\c}{$\beta$}};

\node at (0.0*\r,-0.025*\r) {\small \textcolor{gray!\c}{$\beta$}};
\node at (-0.3*\r,-0.25*\r) {\small $\gamma$};
\node at (0.3*\r,-0.25*\r) {\small \textcolor{gray!\c}{$\gamma$}};
\node at (0.0*\r,-0.5*\r) {\small \textcolor{gray!\c}{$\delta$}};

\end{scope}
}

\node[inner sep=1,draw=black,line width=0.5,shape=circle] at (1.2*\r, 0.26*\r) {\footnotesize $1$};
\node[inner sep=1,draw=black,line width=0.5,shape=circle] at (0.8*\r, -0.26*\r) {\footnotesize $2$};
\node[inner sep=1,draw=black,line width=0.5,shape=circle] at (1.6*\r, -0.26*\r) {\footnotesize $3$};

\node[inner sep=1,draw=black,line width=0.5,shape=circle] at (1.2*\r+0.8*\r, 0.26*\r) {\footnotesize $4$};
\node[inner sep=1,draw=black,line width=0.5,shape=circle] at (1.6*\r+0.8*\r, -0.26*\r) {\footnotesize $5$};

\node at (3.6*\r,0) {\Large $\cdots$}; 

\end{scope}
\end{tikzpicture}
\caption{$\AVC = \{ \beta\gamma^2, \delta^d \}$, $d\ge3$}
\label{Subfig-kt-emt-bega2}
\end{subfigure}
\begin{subfigure}{0.45\linewidth}
\centering
\begin{tikzpicture}[>=latex]

\tikzmath{ 
\s=2;
\r=1.5;
\rr=0.075*\r;
\tz=2;
\tzz=\tz-1;
}
\begin{scope}[xshift=3.5*\s cm]

\foreach \a in {0,...,\tz}
{
\tikzset{xshift=0.8*\r*\a cm}

\draw[->]
	(0.8*\r,0.15*\r) -- (0.8*\r,0.6*\r)
;
\draw[->]
	(0.4*\r,-0.15*\r) -- (0.4*\r,-0.6*\r) 
;

\draw[line width=1.5]	
	(0.8*\r,0.15*\r) -- (0.4*\r,-0.15*\r)
;

}

\foreach \a in {0,...,\tzz}{

\tikzset{xshift=0.8*\r*\a cm}

\draw[]
	(0.8*\r,0.15*\r) -- (1.2*\r,-0.15*\r)
;	

\draw[line width=1.5]
	(0.8*\r,0.15*\r) -- (1.2*\r,-0.15*\r)
;

\node at (0.4*\r+0.8*\r,0.5*\r) {\small $\beta$};
\node at (0.15*\r+0.75*\r,0.25*\r) {\small $\gamma$};
\node at (0.4*\r+1.075*\r,0.25*\r) {\small $\gamma$};
\node at (0.4*\r+0.8*\r,0.025*\r) {\small $\delta$};

\node at (0.8*\r,-0.025*\r) {\small $\delta$};
\node at (0.15*\r+0.35*\r,-0.25*\r) {\small $\gamma$};
\node at (0.0*\r+1.075*\r,-0.25*\r) {\small $\gamma$};
\node at (0.8*\r,-0.5*\r) {\small $\beta$};

}

\foreach \gs in {3} {

\tikzmath{
\op=1-\gs*0.16; 
\c=100-\gs*16;
}

\begin{scope}[xshift=\gs*0.8*\r cm]

\draw[->, line width=0.5, opacity=\op]
	(0.8*\r,0.15*\r) -- (0.8*\r,0.6*\r) 
;
\draw[->, line width=0.5, opacity=\op]
	 (0.4*\r,-0.15*\r) -- (0.4*\r,-0.6*\r)
;

\draw[line width=1.5, opacity=\op]	
	(0*\r,0.15*\r) -- (0.4*\r,-0.15*\r)
	(0.8*\r,0.15*\r) -- (0.4*\r,-0.15*\r)
;

\draw[line width=0.5, opacity=\op]
	(0*\r,0.15*\r) -- (0.4*\r,-0.15*\r)
;	

\node at (0.4*\r,0.5*\r) {\small \textcolor{gray!\c}{$\beta$}};
\node at (0.1*\r,0.25*\r) {\small \textcolor{gray!\c}{$\gamma$}};
\node at (0.7*\r,0.25*\r) {\small \textcolor{gray!\c}{$\gamma$}};
\node at (0.4*\r,0.025*\r) {\small \textcolor{gray!\c}{$\delta$}};

\node at (0.0*\r,-0.025*\r) {\small \textcolor{gray!\c}{$\delta$}};
\node at (-0.3*\r,-0.25*\r) {\small $\gamma$};
\node at (0.3*\r,-0.25*\r) {\small \textcolor{gray!\c}{$\gamma$}};
\node at (0.0*\r,-0.5*\r) {\small \textcolor{gray!\c}{$\beta$}};

\end{scope}
}

\node[inner sep=1,draw=black,line width=0.5,shape=circle] at (1.2*\r, 0.26*\r) {\footnotesize $1$};
\node[inner sep=1,draw=black,line width=0.5,shape=circle] at (0.8*\r, -0.26*\r) {\footnotesize $2$};
\node[inner sep=1,draw=black,line width=0.5,shape=circle] at (1.6*\r, -0.26*\r) {\footnotesize $3$};

\node[inner sep=1,draw=black,line width=0.5,shape=circle] at (1.2*\r+0.8*\r, 0.26*\r) {\footnotesize $4$};
\node[inner sep=1,draw=black,line width=0.5,shape=circle] at (1.6*\r+0.8*\r, -0.26*\r) {\footnotesize $5$};

\node at (3.6*\r,0) {\Large $\cdots$}; 

\end{scope}

\end{tikzpicture}
\caption{$\AVC = \{ \gamma^2\delta, \beta^b \}$, $b\ge3$}
\label{Subfig-kt-emt-ga2de}
\end{subfigure}
\caption{The monohedral earth map tiling with $\beta\gamma^2$ and that with $\gamma^2\delta$}
\label{Fig-kt-emt-bega2-ga2de}
\end{figure}

The argument for $\gamma^2\delta$ is analogous and the tiling is given in Figure \ref{Subfig-kt-emt-ga2de}. 
\end{proof}

\begin{lem}\label{Lem-3-div-m} Along the boundary of an $m$-gon with its angles specified by $\hat{\alpha}$, if the angle arrangement of each $\hat{\alpha}$ at its vertex is either $\alpha \vert \hat{\alpha} \vert \beta$ or $\gamma \vert \hat{\alpha} \vert \gamma$, then $m$ is divisible by $3$.
\end{lem}

\begin{proof} Without loss of generality, we may start at a vertex with angle arrangement $\gamma \vert \hat{\alpha} \vert \gamma$ on the left of Figure \ref{Fig-mgon-3-multiple}. Its neighbour to the right must have an angle arrangement $\beta \vert \hat{\alpha} \vert \alpha$, which then determines the next vertex to its right. Then the fourth vertex has $\gamma \vert \hat{\alpha} \vert \gamma$. It is either the same vertex as the first one or the same pattern continues until it comes back to the starting vertex, and hence the claim.

\begin{figure}[h!]
\centering
\begin{tikzpicture}
\tikzmath{
\XS=1;
\Y=0.8;
}

\foreach \xs in {0,1,2} {
\tikzset{xshift=\xs*\XS cm}
\draw[]
	(0,0) -- (\XS,0)
;
}

\foreach \aa in {-1,1} {
\tikzset{xshift=1.5*\XS cm, xscale=\aa}
\draw[]
	(0.5*\XS,0) -- (0.3*\XS,\Y)
	(0.5*\XS,0) -- (0.7*\XS,\Y)
;
\draw[line width=1.5]
	(1.5*\XS,0) -- (1.3*\XS,\Y)
	(1.5*\XS,0) -- (1.7*\XS,\Y)
;

\node at (0.25*\XS,0.2*\Y) {\small $\alpha$};
\node at (0.75*\XS,0.2*\Y) {\small $\beta$};
\node at (1.25*\XS,0.2*\Y) {\small $\gamma$};
\node at (1.75*\XS,0.2*\Y) {\small $\gamma$};

\node at (0.5*\XS,-0.2*\Y) {\small $\hat{\alpha}$};
\node at (1.5*\XS,-0.2*\Y) {\small $\hat{\alpha}$};
}

\end{tikzpicture}
\caption{The boundary of an $m$-gon having angle arrangements $\alpha \vert \hat{\alpha} \vert \beta$ or $\gamma \vert \hat{\alpha} \vert \gamma$} \qedhere
\label{Fig-mgon-3-multiple}
\end{figure}
\end{proof}

\begin{lem}\label{Lem-even-m} Along the boundary of an $m$-gon with angles $\alpha$, if the vertices have alternating angle arrangements between $\theta \vert \alpha \vert \theta$ and $\varphi \vert \alpha \vert \varphi$ for distinct $\theta,\varphi \neq \alpha$, then $m$ is even. 
\end{lem}

Figure \ref{Fig-odd-m-bealbe-gaalga} illustrates examples of the alternating vertex condition along an $m$-gon. In the context of $\theta, \varphi$ being the angles $\beta,\gamma$ of the kite, the alternating vertex condition automatically holds (for example, see Figure \ref{Subfig-mgon-m=5/7}).

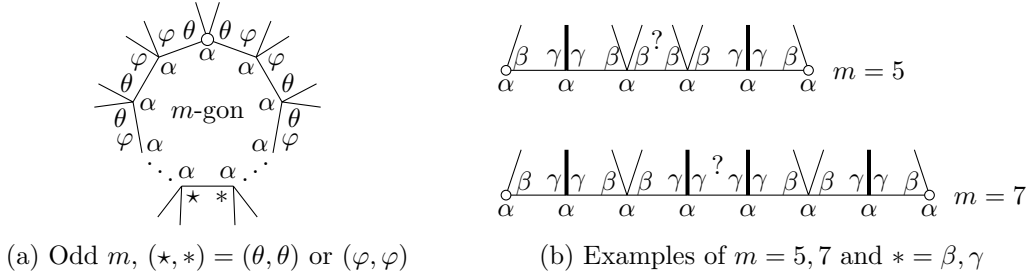
\begin{figure}[h!]
\centering
\begin{subfigure}[t]{0.425\linewidth}
\centering
\begin{tikzpicture}
\tikzmath{
\r=1;
\rr=0.08*\r;
\n=9;
\nn=\n-1;
\th=360/\n;
}

\foreach \a in {0,...,2,4,6,7,8} {
\tikzset{rotate=\a*\th}
\draw[]
	(90:\r) -- (90+\th:\r)
;
}

\foreach \a in {-2,...,2,4,5} {
\tikzset{rotate=\a*\th}
\draw[]
	(90:\r) -- (90-0.15*\th:1.5*\r)
	(90:\r) -- (90+0.15*\th:1.5*\r)
;
}

\draw[fill=white] (90:\r) circle (\rr);

\foreach \a in {0,...,\nn} {
\tikzset{rotate=\a*\th}
\node at (90:0.8*\r) {\small $\alpha$};
}

\foreach \a in {-2,0,2} {
\tikzset{rotate=\a*\th}
\node at (90-0.3*\th:1.15*\r) {\small $\theta$};
\node at (90+0.3*\th:1.15*\r) {\small $\theta$};

\node at (90-0.7*\th:1.15*\r) {\small $\varphi$};
\node at (90+0.7*\th:1.15*\r) {\small $\varphi$};
}

\node at (270-1.1*\th:0.9*\r) {\small $\ddots$};
\node at (270+1.1*\th:0.9*\r) {\small $\iddots$};

\node at (270-0.25*\th: 1.1*\r) {\small $\star$};
\node at (270+0.25*\th: 1.1*\r) {\small $\ast$};

\node at (0,0) {\small $m$-gon};
\end{tikzpicture}
\caption{Odd $m$, $(\star,\ast)=(\theta,\theta)$ or $(\varphi,\varphi)$}
\label{Subfig-mgon-odd-m}
\end{subfigure}
\begin{subfigure}[t]{0.52\linewidth}
\centering
\begin{tikzpicture}

\tikzmath{
}

\begin{scope}[]

\tikzmath{
\s=1.5;
\l=0.8;
\rr=0.08*\l;
}

\begin{scope}[yshift=0.25*\s cm]

\foreach \xs in {0,...,4} {

\draw[xshift=\xs*\l cm]
	(0,0) -- (\l,0)
;

}

\foreach \xs in {1,4} {

\draw[xshift=\xs*\l cm, line width=1.5]
	(0,0) -- (0,0.75*\l)
;

}

\draw[]
	(0,0) -- (0.25*\l, 0.75*\l)
	(2*\l,0) -- (1.75*\l,0.75*\l)
	(2*\l,0) -- (2.25*\l,0.75*\l)
	(3*\l,0) -- (2.75*\l,0.75*\l)
	(3*\l,0) -- (3.25*\l,0.75*\l)
	(5*\l,0) -- (4.75*\l, 0.75*\l)
;

\draw[fill=white] (0,0) circle (\rr);
\draw[fill=white] (5*\l,0) circle (\rr);

\foreach \xs in {0,...,5} {

\tikzset{xshift=\xs*\l cm}

\node at (0,-0.25*\l) {\small $\alpha$};

}


\node at (0.25*\l,0.2*\l) {\small $\beta$};

\foreach \xs in {0,1} {

\tikzset{xshift=\xs*\l cm}

\node at (1.75*\l,0.2*\l) {\small $\beta$};
\node at (2.25*\l,0.2*\l) {\small $\beta$};

}

\node at (2.5*\l,0.5*\l) {\small $?$};

\node at (4.7*\l,0.2*\l) {\small $\beta$};

\node at (0.8*\l,0.2*\l) {\small $\gamma$};
\node at (1.2*\l,0.2*\l) {\small $\gamma$};

\node at (3.8*\l,0.2*\l) {\small $\gamma$};
\node at (4.2*\l,0.2*\l) {\small $\gamma$};

\node at (6*\l, 0) {\small $m=5$};

\end{scope}

\begin{scope}[yshift=-0.85*\s cm]

\foreach \xs in {0,...,6} {

\draw[xshift=\xs*\l cm]
	(0,0) -- (\l,0)
;

}

\foreach \xs in {1,3,4,6} {

\draw[xshift=\xs*\l cm, line width=1.5]
	(0,0) -- (0,0.75*\l)
;

}

\draw[]
	(0,0) -- (0.25*\l, 0.75*\l)
	(2*\l,0) -- (1.75*\l,0.75*\l)
	(2*\l,0) -- (2.25*\l,0.75*\l)
	(5*\l,0) -- (4.75*\l,0.75*\l)
	(5*\l,0) -- (5.25*\l,0.75*\l)
	(7*\l,0) -- (6.75*\l, 0.75*\l)
;

\draw[fill=white] (0,0) circle (\rr);
\draw[fill=white] (7*\l,0) circle (\rr);

\foreach \xs in {0,...,7} {

\tikzset{xshift=\xs*\l cm}

\node at (0,-0.25*\l) {\small $\alpha$};

}

\node at (0.3*\l,0.2*\l) {\small $\beta$};

\foreach \xs in {0,3} {

\tikzset{xshift=\xs*\l cm}

\node at (1.7*\l,0.2*\l) {\small $\beta$};
\node at (2.3*\l,0.2*\l) {\small $\beta$};

}

\node at (6.7*\l,0.2*\l) {\small $\beta$};

\foreach \xs in {0,2,3,5} {

\tikzset{xshift=\xs*\l cm}

\node at (0.8*\l,0.2*\l) {\small $\gamma$};
\node at (1.2*\l,0.2*\l) {\small $\gamma$};

}

\node at (3.5*\l,0.5*\l) {\small $?$};

\node at (8*\l, 0) {\small $m=7$};

\end{scope}

\end{scope}

\end{tikzpicture}
\caption{Examples of $m=5,7$ and $\ast = \beta, \gamma$} 
\label{Subfig-mgon-m=5/7}
\end{subfigure}
\caption{Alternating angle arrangements $\theta\vert\alpha\vert\theta$ and $\varphi\vert\alpha\vert
\varphi$ around an $m$-gon} \qedhere
\label{Fig-odd-m-bealbe-gaalga}
\end{figure}

\begin{proof} The alternating angle arrangements along the boundary of an $m$-gon mean that the edges have alternating labels between $\theta\varphi$ and $\varphi\theta$ in the exterior. This induces a $2$-edge colouring of the boundary and hence $m$ is even.   
\end{proof}

Globally, combinatorics of the tilings deduce the existence of vertices of a small degree. Let $f_{m}$ denote the number of $m$-gons, $v_k$ denote the number of vertices of degree $k$, and $v,e,f$ denote the total numbers of vertices, edges, faces respectively. Since we consider $m\ge 4$ only and the vertices have degrees $\ge3$, we have $f=\sum_{m\ge4} f_m$ and $v=\sum_{k\ge3}v_k$. Following from Euler polyhedral formula and Dehn-Sommerville formulae, we have
\begin{align}
\label{Eq-basic-Euler}
&v-e+f=2, \\
\label{Eq-basic-DS1}
&2e = \sum_{k\ge3} kv_k, \\
\label{Eq-basic-DS2}
&2e=\sum_{m\ge4} m f_{m}.
\end{align}
Equation \eqref{Eq-basic-DS2} implies $2e\ge4f$. Combining this inequality with \eqref{Eq-basic-Euler} and \eqref{Eq-basic-DS1} gives
\begin{align*}
2 = v-e+f \le 
v - \tfrac{e}{2} = \sum_{k\ge 3} (1-\tfrac{k}{4})v_k 
= (1 - \tfrac{3}{4})v_3 + ( 1 - \tfrac{4}{4} ) v_4 + ( 1 - \tfrac{5}{4} ) v_5 + \cdots.
\end{align*}
This implies $v_3 >0$, meaning that there must exist at least one vertex of degree exactly $3$. This fact will be used frequently without reference. 

We remark that, for $m=3$, the same argument above leads to $v_3, v_4$ or $v_5 >0$, meaning that there must be a vertex of degree $3, 4$ or $5$.

\begin{lem}[Counting Lemma] \label{Lem-counting} If two angles $\theta,\varphi$ appear in the same prototile such that the number of $\varphi$ is $t$ times of that of $\theta$, and at every vertex of a tiling the number of $\varphi$ is at least $t$ times of that of $\theta$, then at every vertex the number of $\varphi$ equals $t$ times the number of $\theta$. 
\end{lem}

\begin{proof} Let $\# \theta, \# \varphi$ denote the total numbers of $\theta, \varphi$ respectively. The hypothesis implies $\# \varphi = t \# \theta$. Suppose every vertex is $\theta^{p_i}\varphi^{q_i}\cdots$ for some $p_i, q_i \ge 0$ and $tp_i \le q_i$. Counting the two angles at the vertices, we get
\begin{align*}
&\# \theta = \sum_i  p_i \#\theta^{p_i}\varphi^{q_i}\cdots,&
&\# \varphi = \sum_i  q_i \#\theta^{p_i}\varphi^{q_i}\cdots.&
\end{align*}
By $\# \varphi = t \# \theta$ and $t p_i \le q_i$ and $\#\theta^{p_i}\varphi^{q_i}\cdots > 0$, we get $(tp_i - q_i) = 0$ for every $i$, which implies the desired result.
\end{proof}

\begin{expl}\label{Eg-Counting} As an example of applying the Counting Lemma \ref{Lem-counting}, which will be used later in the proofs of Lemma \ref{Lem-vertex} and Proposition \ref{Prop-de}, suppose $\delta\cdots=\gamma^{c_1}\delta, \alpha\gamma^{c_2}\delta$, where $c_1, c_2\ge2$. Then $\gamma\cdots=\gamma^{c_1}\delta, \alpha\gamma^{c_2}\delta$, etc. Counting $\gamma$ among the vertices gives
\begin{align*}
\# \gamma \ge c_1 \# \gamma^{c_1}\delta + c_2 \# \alpha\gamma^{c_2}\delta \ge 2( \# \gamma^{c_1}\delta + \# \alpha\gamma^{c_2}\delta  ) = 2\# \delta.
\end{align*}
In the kite, the number of $\gamma$ is twice that of $\delta$. Counting Lemma implies $c_1=c_2=2$ and hence $\gamma\cdots=\delta\cdots=\gamma^2\delta, \alpha\gamma^2\delta$. The dihedral hypothesis and Lemma \ref{Lem-bega2-ga2de} then dismiss $\gamma^2\delta$ and therefore $\gamma\cdots=\delta\cdots= \alpha\gamma^2\delta$. It means that there are vertices of the type $\alpha\gamma^2\delta$ and they are the only places where $\gamma,\delta$ appear.
\end{expl}

Given the above example, we may simply say \quotes{by Counting Lemma on $\gamma,\delta$ (or $\beta,\gamma$)} for similar arguments in the later discussion.

\begin{lem}\label{Lem-vertex} In a dihedral tiling by regular polygons and kites, each vertex belongs to one of the following types,
\begin{align}
\label{Eq-vertex-al}
\alpha\cdots &= \alpha^3, \alpha^a\beta^b, \alpha^a\gamma^c, \alpha\beta^b\gamma^c, \alpha\gamma^c\delta^d; \\
\label{Eq-vertex-be}
\beta\cdots &= \beta^b, \alpha^a\beta^b, \beta^b\gamma^c((b,c)\neq(1,2)), \alpha\beta^b\gamma^c;  \\
\label{Eq-vertex-ga}
\gamma\cdots &= \gamma^c, \alpha^a\gamma^c, \beta^b\gamma^c((b,c)\neq(1,2)), \gamma^c\delta^d((c,d)\neq(2,1)), \alpha\beta^b\gamma^c, \alpha\gamma^c\delta^d;  \\
\label{Eq-vertex-de}
\delta\cdots &=\delta^d, \gamma^c\delta^d((c,d)\neq(2,1)), \alpha\gamma^c\delta^d.
\end{align}
where $a\le3$, and $c\ge2$ is even. Notably, no vertex is of the type $\alpha^a\delta^d, \beta\delta\cdots, \beta\gamma^2$ or $\gamma^2\delta$.
\end{lem}

\begin{proof} Given $\alpha>\frac{1}{2}\pi$ and $\gamma$ appearing an even number of times at a vertex (Parity Lemma) and the kite angle sum, no vertex has all of $\alpha,\beta,\gamma,\delta$. We also know $\alpha^a=\alpha^3$.

The adjacent edge combinations of $\alpha, \beta, \gamma, \delta$ are $\vert \alpha \vert$, $\vert \beta \vert$, $\vert \gamma \, \bvert$, and $\bvert \, \delta \, \bvert$ respectively, which rule out $\alpha^a\delta^d, \beta^b\delta^d, \alpha^a\beta^b\delta^d$. Parity Lemma implies $\beta\gamma\delta\cdots=\beta\gamma^2\delta\cdots$ which has angle sum $\beta+2\gamma + \delta > 2\pi$. Then $\beta\gamma\delta\cdots$ is also not a vertex, and hence it rules out $\beta\delta\cdots$. 

The above yields the lists $\beta\cdots = \beta^b, \alpha^a\beta^b, \beta^b\gamma^c, \alpha^a\beta^b\gamma^c$ and $\delta\cdots =\delta^d, \gamma^c\delta^d, \alpha^a\gamma^c\delta^d$. It suffices to show that $\alpha^2\beta\gamma\cdots, \alpha^2\gamma\delta\cdots$ are not vertices. Parity Lemma implies $\alpha^2\beta\gamma\cdots=\alpha^2\beta\gamma^2\cdots$ and $\alpha^2\gamma\delta\cdots=\alpha^2\gamma^2\delta\cdots$.

Assume that $\alpha^2\beta\gamma^2\cdots$ is a vertex. Then its vertex angle sum and the kite angle sum imply $\beta+2\gamma+\delta > 2\pi \ge 2\alpha+\beta+2\gamma$. Combining it with $\alpha>\frac{1}{2}\pi$ deduces $\delta>2\alpha>\pi$. Then $\delta\cdots=\delta^d, \gamma^c\delta^d, \alpha^a\gamma^c\delta^d$ becomes $\delta\cdots=\gamma^c\delta, \alpha\gamma^c\delta$. As per Example \ref{Eg-Counting}, Counting Lemma on $\gamma,\delta$ implies $\gamma^c\delta, \alpha\gamma^c\delta=\gamma^2\delta, \alpha\gamma^2\delta$ and then Lemma \ref{Lem-bega2-ga2de} implies $\delta\cdots=\alpha\gamma^2\delta$. Now Counting Lemma on $\gamma,\delta$ implies $\gamma\cdots=\alpha\gamma^2\delta$, which contradicts $\alpha^2\beta\gamma^2\cdots$ being a vertex.

An analogous argument rules out $\alpha^2\gamma^2\delta\cdots$. The four lists follow from the above. 
\end{proof}

\begin{lem} \label{Lem-albe-alga} In a dihedral tiling by regular polygons and kites, there must exist vertices of each of the types $\alpha\beta\cdots, \alpha\gamma\cdots$. Moreover,
\begin{align}
\label{Eq-vertex-albe}
\alpha\beta\cdots &= \alpha^a\beta^b, \alpha\beta^b\gamma^c; \\
\label{Eq-vertex-alga}
\alpha\gamma\cdots &= \alpha\gamma^2\cdots = \alpha^a\gamma^c, \alpha\beta^b\gamma^c, \alpha\gamma^c\delta^d
\end{align}
where $a,b,c,d$ are positive integers such that $1 \le a \le 3$ and $c\ge2$ is even.
\end{lem}

\begin{proof} The vertices $\alpha\beta\cdots, \alpha\gamma\cdots$ appear along the common edge between a regular polygon and a kite (Figure \ref{Subfig-albe-alga}). Then \eqref{Eq-vertex-albe} and \eqref{Eq-vertex-alga} follow from Parity Lemma and Lemma \ref{Lem-vertex}.
\end{proof}

\subsection{Geometry}

Subdividing a regular $m$-gon from its centre into isosceles triangles and applying the spherical cosine law \cite{lnp}, we derive  
\begin{align}\label{Eq-a4-cosx}
\cos x = \cot^2 \tfrac{1}{2}\alpha + \frac{\cos \tfrac{2}{m}\pi}{\sin^2 \tfrac{1}{2}\alpha}.
\end{align}

For $\gamma<\pi$, the kite can be divided into two congruent triangles along the bisector of $\beta,\delta$ as shown in Figure \ref{Fig-a4-a2b2-kite}. 

\begin{figure}[h!] 
\centering
\begin{tikzpicture}

\tikzmath{
\s=1;
\r=0.85;
\th=360/4;
\x=\r*cos(0.5*\th);
\R = sqrt(\x^2+(3*\x)^2);
\aR = acos(3*\x/\R);
}

\begin{scope}[] 

\foreach \a in {0,...,3} {

\draw[rotate=\a*\th]
	(\th:\r) -- (2*\th:\r)
;

}

\draw[line width=1.5]
	(270:\r) -- (270-\th:\r)
	(270:\r) -- (270+\th:\r)
;

\draw[dashed]
	(90:\r) -- (270:\r)
;

\node at (0.5*\th:0.9*\r) {\small $x$};
\node at (1.5*\th:0.9*\r) {\small $x$};
\node at (-0.5*\th:0.9*\r) {\small $y$};
\node at (2.5*\th:0.9*\r) {\small $y$};

\node at (1.25*\th:0.5*\r) {\small $\beta_1$};
\node at (0.725*\th:0.5*\r) {\small $\beta_2$};
\node at (180:0.7*\r) {\small $\gamma$};
\node at (0:0.7*\r) {\small $\gamma$};
\node at (2.75*\th:0.5*\r) {\small $\delta_1$};
\node at (3.25*\th:0.5*\r) {\small $\delta_2$};

\end{scope}

\end{tikzpicture}
\caption{Triangulation of the kite}
\label{Fig-a4-a2b2-kite}
\end{figure}
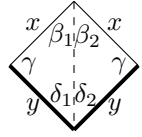

The spherical cosine law on the triangle with angles $\frac{1}{2}\beta, \gamma, \frac{1}{2}\delta$ implies
\begin{align}
\label{Eq-x2y2-cosx}
\cos x &= \frac{\cos \frac{1}{2}\delta + \cos \frac{1}{2}\beta \cos \gamma}{ \sin \frac{1}{2}\beta \sin \gamma}, \\
\label{Eq-x2y2-cosy}
\cos y &=  \frac{ \cos \frac{1}{2}\beta + \cos \gamma\cos \frac{1}{2}\delta }{\sin \gamma \sin \frac{1}{2}\delta}.
\end{align}

Combining \eqref{Eq-a4-cosx} and \eqref{Eq-x2y2-cosx}, we deduce
\begin{align*}
\sin \tfrac{1}{2}\beta \sin \gamma (1 - \sin^2\tfrac{1}{2}\alpha + \cos \tfrac{2}{m}\pi) = \sin^2\tfrac{1}{2}\alpha ( \cos \tfrac{1}{2}\delta + \cos \tfrac{1}{2}\beta \cos \gamma),
\end{align*}
which is simplified as %
\begin{align} \label{Eq-angle-id}
\sin\tfrac{1}{2}\beta\sin\gamma (1+ \cos \tfrac{2}{m}\pi) = \sin^2\tfrac{1}{2}\alpha (\cos\tfrac{1}{2}\delta+\cos(\gamma-\tfrac{1}{2}\beta) ).
\end{align}

\subsection{Strategy and Construction}

Admissible vertex types allow us to construct the associated tiling(s). They can be determined by the necessary conditions imposed by facts established in this section. Typically, we use the following facts:
\begin{itemize}
\item there exists a degree $3$ vertex;
\item $\alpha\beta\cdots$ and $\alpha\gamma\cdots$ are vertices, while $\beta\gamma^2, \gamma^2\delta, \alpha^a\delta^d$ and $\beta\delta\cdots$ are not;
\item the angles at vertices satisfy Lemma \ref{Lem-vertex} and Counting Lemma;
\item $\beta\vert\beta\cdots$ is a vertex if and only if $\gamma\vert\gamma\cdots$ is a vertex; likewise for $\gamma \, \bvert \, \gamma \cdots$ and $\delta \, \bvert \, \delta$.
\end{itemize} 
Each admissible vertex comes with a vertex angle sum equation. In principle, with three independent equations and \eqref{Eq-angle-id}, one can determine the angle values and hence all the vertex types.

\section{Classification}
\label{Sec-Tilings}

We structure the classification using Remark~\ref{Rmk-benede}: either $\beta>\delta$ or $\beta<\delta$ holds. The three possible scenarios, (\ref{Prop-de}) $\delta > \beta, \frac{2}{3}\pi$, 
(\ref{Prop-be}) $\beta > \delta, \frac{2}{3}\pi$, and (\ref{Prop-be-de}) both $\beta, \delta\le \frac{2}{3}\pi$, tie the classification together in the three propositions in this section.

The tilings come from Propositions \ref{Prop-be} and \ref{Prop-be-de}, which filter through some of the near cases. Indeed, in one case we actually obtain a ``pseudo-tiling'' which is combinatorially feasible. The nature of classification inevitably demands substantial effort in these propositions and we provide the reader with diagrams for the workflows in the proofs in Figures \ref{Fig-flow-4.2}, \ref{Fig-flow-4.3} respectively.

\begin{prop}\label{Prop-de} If $\delta > \beta, \frac{2}{3}\pi$, then there is no dihedral tiling. 
\end{prop}

\begin{proof} We will determine the vertex types using the information about angle sizes. 

Given in the hypothesis, $\delta$ is a bigger angle. The kite angle sum $\beta+2\gamma + \delta > 2\pi$ and $\delta > \beta$ imply $\gamma+\delta > \beta + \gamma, \pi$. Combined with $\alpha>\frac{1}{2}\pi$ and $\delta > \frac{2}{3}\pi$, list \eqref{Eq-vertex-de} is then reduced to $\delta\cdots=\gamma^c\delta, \alpha\gamma^c\delta$. In Example \ref{Eg-Counting}, we deduce $\gamma\cdots=\delta\cdots=\alpha\gamma^2\delta$, meaning that $\gamma,\delta$ appear only at $\alpha\gamma^2\delta$. 

The remaining vertices are given by combinations of $\alpha,\beta$. From \eqref{Eq-vertex-al} and \eqref{Eq-vertex-be}, they are $\alpha^3, \alpha^a\beta^b, \beta^b$. Notably, $\alpha\beta\cdots=\alpha^a\beta^b$. Lemma \ref{Lem-albe-alga} then implies that $\alpha^a\beta^b$ is a vertex. Combining the kite angle sum and the vertex angle sum of $\alpha\gamma^2\delta$ gives $\beta+2\gamma+\delta>2\pi = \alpha+2\gamma+\delta$, which implies $\beta>\alpha>\frac{1}{2}\pi$. This imply that $\alpha^a\beta^b$ is of degree $3$. Since $\gamma$ appears only at $\alpha\gamma^2\delta$ which has angle arrangement $\vert \alpha \vert \gamma \, \bvert \, \delta \, \bvert \, \gamma \vert$, it rules out $\gamma \vert \gamma\cdots$. Then Lemma \ref{Lem-bebe-gaga-dede} rules out $\beta\vert\beta\cdots$, meaning that $\beta^b, \alpha\beta^b$ are not vertices. Hence a degree $3$ vertex $\alpha^a\beta^b$ can only be $\alpha^2\beta$. Given $\beta>\alpha$, it then rules out $\alpha^3$. The above concludes $\alpha\cdots=\alpha^2\beta, \alpha\gamma^2\delta$ and $\beta\cdots=\alpha^2\beta$. Therefore the vertex types are
\begin{align*}
\AVC = \{ \alpha^2\beta, \alpha\gamma^2\delta \}.
\end{align*}

The vertex angle sum of $\alpha^2\beta$ and $\beta>\alpha$ and \eqref{Eq-mgon-sum} imply $\beta>\frac{2}{3}\pi>\alpha > (1-\frac{2}{m})\pi$, and hence $m\le5$, i.e., $m=4,5$. 

Along the boundary of an $m$-gon, its vertices are of the type either $\alpha^2\beta$ or $\alpha\gamma^2\delta$. Given the unique angle arrangements of $\alpha^2\beta$ and $\alpha\gamma^2\delta$ along the boundary of an $m$-gon, Lemma \ref{Lem-3-div-m} implies that $m$ is a multiple of $3$, contradicting $m=4,5$.  \qedhere
\end{proof}

\begin{prop}\label{Prop-be} If $\beta>\delta, \frac{2}{3}\pi$, then the dihedral tilings are
\begin{itemize}
\item
the \ref{Label:EMT-Prisms} family with flip modifications, 
\item
the \ref{Label:trunc-Octa} family, 
\item 
the \ref{Label:trunc-Icosa} family, and 
\item
the two \ref{Label:tri-subdiv-thick-Octa} tilings.
\end{itemize}
\end{prop}

\begin{figure}[h!]
\centering
\begin{tikzpicture}[>=]
\tikzmath{
\x=1.5;
\y=0.25;
}

\draw[]
	(0,0) -- (0,\y)
	(0,0) -- (0,-\y)
	(-\x,0) -- (-\x,-\y)
	(\x,0) -- (\x,-\y)
	(-5*\x,2.75*\y) -- (-5*\x,3.75*\y) 
	(-3.5*\x,2.75*\y) -- (-3.5*\x,3.75*\y) 
	(-2*\x,2.75*\y) -- (-2*\x,3.75*\y)
	(0,2.75*\y) -- (0,3.75*\y)
	(-1*\x,3.75*\y) -- (-1*\x,4.75*\y)
	(-4*\x,3.75*\y) -- (-4*\x,4.75*\y)
	(-1*\x,6.75*\y) -- (-1*\x,7.75*\y)
	(-6*\x,6.75*\y) -- (-6*\x,7.75*\y)
	(-4*\x,6.75*\y) -- (-4*\x,7.75*\y)
	(-3*\x,7.75*\y) --  (-3*\x,8.75*\y)
	(-\x,0) -- (0,0) -- (\x,0)
	(-2*\x,3.75*\y) -- (-1*\x,3.75*\y) -- (0,3.75*\y)
	(-5*\x,3.75*\y) -- (-3.5*\x,3.75*\y) 
	(-6*\x,7.75*\y) -- (-1*\x,7.75*\y) 
;

\node at (-3*\x,9.5*\y) {\tiny $\alpha\beta\cdots$};

\node at (-6*\x,5.75*\y) {\tiny $\nexists$ $\alpha^2\beta, \alpha\beta^2$};
\node at (-6*\x,4.5*\y) {\tiny $\Downarrow$};
\node at (-6*\x,3.5*\y) {\tiny $Kt\mathcal{O}, Kt\mathcal{I}$};

\node at (-5*\x,1.75*\y) {\tiny $\gamma \vert \alpha \vert \gamma$};
\node at (-5*\x,0.5*\y) {\tiny $\Downarrow$};
\node at (-5*\x,-0.5*\y) {\tiny $\varnothing$};

\node at (-3.5*\x,1.75*\y) {\tiny $\nexists$ $\gamma \vert \alpha \vert \gamma$};
\node at (-3.5*\x,0.5*\y) {\tiny $\Downarrow$};
\node at (-3.5*\x,-0.5*\y) {\tiny (EM3) \& Flip, (P7)};

\node at (-4*\x,5.75*\y) {\tiny $\alpha^2\beta$, $\nexists$ $\alpha\beta^2$};

\node at (-1*\x,5.75*\y) {\tiny $\alpha\beta^2$, $\nexists$ $\alpha^2\beta$};

\node at (-2*\x,1.75*\y) {\tiny $\alpha\beta\gamma^2$};
\node at (-2*\x,0.5*\y) {\tiny $\Downarrow$};
\node at (-2*\x,-0.5*\y) {\tiny (P1), (P2)};

\node at (0,1.75*\y) {\tiny $\nexists$ $\alpha\beta\gamma^2$};


\node at (-\x,-1.75*\y) {\tiny $\alpha\gamma^2$};
\node at (-\x,-3*\y) {\tiny $\Downarrow$};
\node at (-\x,-4*\y) {\tiny $\varnothing^{\ast}$};

\node at (0,-1.75*\y) {\tiny $\alpha\gamma^4$};
\node at (0,-3*\y) {\tiny $\Downarrow$};
\node at (0,-4*\y) {\tiny $\subset$ (P1), (P2)};

\node at (\x,-1.75*\y) {\tiny $\alpha\gamma^2\delta^d$};
\node at (\x,-3*\y) {\tiny $\Downarrow$};
\node at (\x,-4*\y) {\tiny $\varnothing$};

\end{tikzpicture}
\caption{The workflow of the proof for Proposition \ref{Prop-be}; *a combinatorial tiling appears under $\alpha\gamma^2$}
\label{Fig-flow-4.2}
\end{figure}
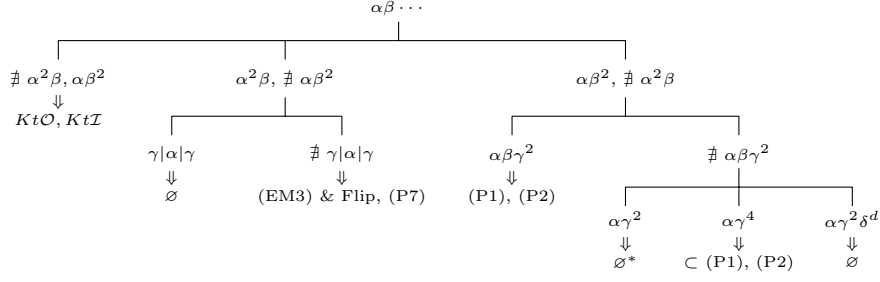

\begin{proof} 
From the hypothesis, $\beta$ is a bigger angle. The kite angle sum and $\beta>\delta$ imply $\beta+\gamma>\gamma+\delta, \pi$. Parity Lemma and $\alpha>\frac{1}{2}\pi$ and $\beta>\frac{2}{3}\pi$ and $\beta+\gamma>\pi$ reduce \eqref{Eq-vertex-be} to 
\begin{align}\label{Eq-be>de,2/3-vertex-be}
\beta\cdots=\alpha^2\beta, \alpha\beta^2, \beta\gamma^{c\ge4}, \alpha\beta\gamma^c.
\end{align}
In particular, we have $\alpha\beta\cdots=\alpha^2\beta, \alpha\beta^2, \alpha\beta\gamma^c$.
Lemma \ref{Lem-albe-alga} asserts that one of them is a vertex. The angle sums imply that $\alpha^2\beta,\alpha\beta^2$ are mutually exclusive, which means that they cannot both be vertices. We will study the cases: both of $\alpha^2\beta,\alpha\beta^2$ are not vertices; and one of them is a vertex.

\begin{case*}[$\nexists \, \alpha^2\beta,\alpha\beta^2$] List \eqref{Eq-be>de,2/3-vertex-be} becomes $\beta\cdots=\beta\gamma^{c\ge4}, \alpha\beta\gamma^c$. Counting Lemma on $\beta, \gamma$ then rules out $\beta\gamma^{c\ge4}$ which implies $\beta\cdots=\alpha\beta\gamma^2$. Counting Lemma on $\beta,\gamma$ again implies $\gamma\cdots=\beta\cdots=\alpha\beta\gamma^2$. Now $\alpha\beta\gamma^2$ has to be a vertex.

With $\beta\cdots=\gamma\cdots=\alpha\beta\gamma^2$, list \eqref{Eq-vertex-al} becomes $\alpha\cdots=\alpha^3, \alpha\beta\gamma^2$ and list \eqref{Eq-vertex-de} becomes $\delta\cdots=\delta^d$. Since $\delta$ appears only at $\delta^d$, now $\delta^d$ has to be a vertex. Furthermore, the vertex angle sum of $\alpha\beta\gamma^2$ and the kite angle sum imply $\delta>\alpha > \frac{1}{2}\pi$, and hence $\delta^d=\delta^3$. Then $\delta>\alpha$ and the vertex angle sum of $\delta^3$ rule out $\alpha^3$. Therefore the admissible vertices are
\begin{align}\label{AVC-de3-albega2}
\AVC = \{ \delta^3, \alpha\beta\gamma^2 \}. 
\end{align}
Next, $\delta^3$ and $\delta>\alpha$ imply $\tfrac{2}{3}\pi>\alpha$. Then \eqref{Eq-mgon-sum} implies $\tfrac{2}{3}\pi>\alpha > (1-\tfrac{2}{m})\pi$, and hence $m\le5$. Thus, $m=4,5$, namely the regular prototile is either a square or a pentagon.

The neighbourhood of $\delta^3$ consists of three kites as shown in Figure \ref{Fig-be-hex-de3} and has a hexagonal boundary. Consider each such neighbourhood as a hexagon with angles $\beta^3\bar{\gamma}^3$ where $\bar{\gamma}$ denotes $\gamma^2$ combined. Then AVC \eqref{AVC-de3-albega2} becomes $\{ \alpha\beta\bar{\gamma} \}$ for the prototiles given by the hexagon and the square (resp. the regular pentagon). From this viewpoint, a straightforward construction gives the truncated octahedron (resp. the truncated icosahedron) as shown in Figure \ref{Fig-Finite-Families-tOcta-tIcosa}. The tilings are therefore given by subdividing each hexagon into three kites. 

\begin{figure}[h!] 
\centering
\begin{tikzpicture}

\tikzmath{
\s=1;
\r=0.75;
\dr=0.1*\r;
\rr=0.5*\r;
\R=1.5*\r;
\RR=2.5*\r;
\th=360/4;
\x=\r*cos(\th/2);
\xx=0.5*\x;
\hex=360/6;
}

\begin{scope}[]

\foreach \a in {0,1,...,5} {

\tikzset{rotate=\a*\hex}

\draw[]
	(90:\r) -- (90+\hex:\r)
;

}

\foreach \a in {0,2,4} {

\tikzset{rotate=\a*\hex}

\draw[line width=1.5]
	(0,0) -- (90:\r)
;

\fill (90:\r) circle (\dr);

\node at (90-1*\hex:0.75*\r) {\small $\beta$};
\node at (90-1*\hex:0.25*\r) {\small $\delta$};
\node at (0.2*\r, 0.65*\r) {\small $\gamma$};
\node at (-0.2*\r, 0.65*\r) {\small $\gamma$};
}

\end{scope}

\end{tikzpicture}
\caption{A hexagon centred at $\delta^3$}
\label{Fig-be-hex-de3}
\end{figure}
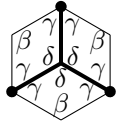
\end{case*}

\begin{case*}[$\alpha^2\beta$] The vertex angle sum of $\alpha^2\beta$ and $\alpha>\frac{1}{2}\pi$ and $\beta>\frac{2}{3}\pi$ imply $\pi>\beta>\frac{2}{3}\pi>\alpha > \frac{1}{2}\pi$. Then \eqref{Eq-mgon-sum} implies $\frac{2}{3}\pi> (1-\tfrac{2}{m})\pi$, and hence $m=4,5$.

The vertex $\alpha^2\beta$ and $\beta>\alpha$ rule out $\alpha\beta^2$. Then \eqref{Eq-be>de,2/3-vertex-be} becomes $\beta\cdots=\alpha^2\beta, \beta\gamma^{c\ge4}, \alpha\beta\gamma^c$. Notably, $\beta^2\cdots$ is not a vertex, and hence the same for $\beta\vert\beta\cdots$. Lemma \ref{Lem-bebe-gaga-dede} then rules out $\gamma\vert\gamma\cdots$. Hence $\gamma^c, \gamma^c\delta^d(=$ $\bvert \, \delta \, \bvert \, \gamma \vert \gamma \cdots)$ and $\beta\gamma^{c\ge4}$ are not vertices; and the angle arrangements for $\alpha\beta\gamma^c$ are $\bvert \, \gamma \vert \alpha \vert \beta \vert \gamma \, \bvert$ and $\bvert \, \gamma \vert \alpha \vert \gamma \, \bvert \, \gamma \vert \beta \vert \gamma \, \bvert$, which correspond to $\alpha\beta\gamma^2, \alpha\beta\gamma^4$. List \eqref{Eq-be>de,2/3-vertex-be} becomes
\begin{align}
\label{Eq-be>de,2/3-al2be-be}
\beta\cdots&=\alpha^2\beta, \alpha\beta\gamma^2, \alpha\beta\gamma^4.
\end{align}

Similarly, the absence of $\gamma\vert\gamma\cdots$ implies $\alpha\gamma^c\delta^d=$ $\bvert \, \gamma \vert \alpha \vert \gamma \, \bvert \, \delta \cdots \delta \, \bvert$ $=\alpha\gamma^2\delta^d$. Combined with \eqref{Eq-be>de,2/3-al2be-be}, lists \eqref{Eq-vertex-ga} and \eqref{Eq-vertex-de} become
\begin{align}
\label{Eq-be>de,2/3-al2be-ga}
\gamma\cdots&= \alpha^{a}\gamma^c,  \alpha\beta\gamma^2, \alpha\beta\gamma^4, \alpha\gamma^2\delta^d; \\
\label{Eq-be>de,2/3-al2be-de}
\delta\cdots&=\delta^d, \alpha\gamma^2\delta^d.
\end{align}

Given \eqref{Eq-be>de,2/3-al2be-be} and the case assumption, we will show that $\alpha\beta\gamma^2$ and $\alpha\beta\gamma^4$, which are mutually exclusive, do not appear in any tiling. 

First, assume that $\alpha\beta\gamma^2$ appears in a tiling. The vertices $\alpha^2\beta, \alpha\beta\gamma^2$ imply $2\gamma=\alpha>\frac{1}{2}\pi$. Combined with $2\gamma+\delta>\pi$, it implies $\alpha+2\gamma+2\delta=4\gamma+2\delta>2\pi$. Using \eqref{Eq-be>de,2/3-al2be-ga} we then determine $\alpha\gamma\delta\cdots=\alpha\gamma^{2}\delta$. It is ruled out by $\alpha\beta\gamma^2$ and $\beta\neq\delta$. Hence \eqref{Eq-be>de,2/3-al2be-ga} becomes $\gamma\cdots=\alpha\beta\gamma^2, \alpha^a\gamma^c$. One of them appears as a vertex. Substituting $2\gamma=\alpha$ into the vertex angle sum of $\alpha^a\gamma^c$ gives $2\pi=a\alpha+c\gamma=(a+\bar{c})\alpha$ where $\bar{c}:=\frac{1}{2}c$. Then $\alpha>\frac{1}{2}\pi$ implies $a+\bar{c}=3$. Since $\frac{2}{3}\pi > \alpha$, the vertex angle sum is $<2\pi$, a contradiction. This rules out $\alpha^a\gamma^c$ and \eqref{Eq-be>de,2/3-al2be-ga} becomes $\gamma\cdots=\alpha\beta\gamma^2$. Counting Lemma on $\beta, \gamma$ implies $\beta\cdots=\alpha\beta\gamma^2$, contradicting $\alpha^2\beta$ being a vertex. 

Second, assume that $\alpha\beta\gamma^4$ appears in a tiling. From \eqref{Eq-be>de,2/3-al2be-de}, we know $\delta\cdots =\delta^d, \alpha\gamma^2\delta^d$. One of them is a vertex. The vertices $\alpha^2\beta, \alpha\beta\gamma^4$ imply $4\gamma=\alpha$. Combining it with the kite angle sum and $\alpha\beta\gamma^4$ and $\alpha>\frac{1}{2}\pi$, we get $\delta>\alpha+2\gamma = \frac{3}{2}\alpha > \frac{3}{4}\pi$. This rules out $\delta^d$. Moreover, $4\gamma=\alpha$ and $\delta>\alpha+2\gamma$ imply $2\delta=3\alpha$ and $\alpha+6\gamma+\delta>4\alpha$. Then $\alpha\gamma^c\delta^d=\alpha\gamma^2\delta, \alpha\gamma^4\delta$. However, $\beta>\delta$ and $\alpha\beta\gamma^4$ imply that $\alpha\gamma^2\delta, \alpha\gamma^4\delta$ have the angle sums $<2\pi$, meaning no vertex for $\delta\cdots$, a contradiction. 

The above therefore reduces \eqref{Eq-be>de,2/3-al2be-be} to 
\[
\beta\cdots=\alpha^2\beta.
\]


Lemma \ref{Lem-albe-alga} asserts that $\alpha\gamma\cdots$ is a vertex. Using lists \eqref{Eq-vertex-al}, \eqref{Eq-be>de,2/3-al2be-ga}, \eqref{Eq-be>de,2/3-al2be-de} and $\alpha>\tfrac{1}{2}\pi$ and the absence of $\gamma\vert\gamma\cdots$ and Parity Lemma, we deduce 
\begin{align}
\label{Eq-be>de,2/3-al2be-alga}
\alpha\gamma\cdots =  \alpha\gamma^2, \alpha^2\gamma^2, \alpha^2\gamma^4, \alpha^3\gamma^2, \alpha^3\gamma^4, \alpha^3\gamma^6, \alpha\gamma^2\delta^{d}. 
\end{align}

Note that, the presence of $\alpha^2\beta$ and $\beta>\alpha$ rule out $\alpha^3$ and \eqref{Eq-vertex-al} implies $\alpha^3\cdots=\alpha^3\gamma^c$, which will be useful for determining $\alpha\gamma\cdots$.

We claim that $\alpha^3\cdots$ is possible only for $m=4$. For $m=5$, inequality \eqref{Eq-mgon-sum} implies $\alpha>\tfrac{3}{5}\pi$. For $k\ge1$, the vertex angle sums of $\alpha^2\beta, \alpha^3\gamma^{2k}$ imply $\beta=2\pi - 2\alpha$ and $2\gamma=\tfrac{1}{k}(2\pi-3\alpha)$. Substituting them into the kite angle sum gives $\delta>\tfrac{1}{k}((2k+3)\alpha-2\pi)$. Then $\alpha>\tfrac{3}{5}\pi$ and $k\ge1$ further implies $\delta > \tfrac{1}{5k}(6k-1)\pi>\pi$. Now $\beta>\delta > \pi$ and $\alpha>\tfrac{3}{5}\pi$ mean that $\alpha^2\beta$ has an angle sum $>2\pi$, a contradiction. 


As $\alpha\gamma\cdots$ appears as one of the vertices in \eqref{Eq-be>de,2/3-al2be-alga}, we have the following observation. The absence of $\gamma\vert\gamma\cdots$ means that each of $\alpha\gamma^2, \alpha^2\gamma^4, \alpha^3\gamma^4, \alpha^3\gamma^6, \alpha\gamma^2\delta^d$ has a unique angle arrangement: 
\begin{align*}
&\alpha\gamma^2= \, \bvert \gamma \vert \alpha \vert \gamma \, \bvert, \quad
\alpha^2\gamma^4= \, \bvert \gamma \vert \alpha \vert \gamma \, \bvert \, \gamma \vert \alpha \vert \gamma \, \bvert, \quad
\alpha^3\gamma^4= \, \bvert \gamma \vert \alpha \vert \gamma \, \bvert \, \gamma \vert \alpha \vert \alpha \vert \gamma \, \bvert, \\
&\alpha^3\gamma^6=\, \bvert \gamma \vert \alpha \vert \gamma \, \bvert \, \gamma \vert \alpha \vert \gamma \, \bvert \, \gamma \vert \alpha \vert \gamma \, \bvert, \quad
\alpha\gamma^2\delta^d=\, \bvert \gamma \vert \alpha \vert \gamma \, \bvert \, \delta \cdots \delta \, \bvert.
\end{align*}
Similarly, $\alpha^2\gamma^2$ has a unique angle arrangement $\bvert \, \gamma \vert \alpha \vert \alpha \vert \gamma \, \bvert$; and $\alpha^3\gamma^2=$ $\bvert \, \gamma \vert \alpha \vert \alpha \vert \alpha \vert \gamma \, \bvert$. Note that, $\alpha^3\cdots=\alpha^3\gamma^c$ implies $\alpha \vert \alpha \vert \alpha \cdots = \vert \alpha \vert \alpha \vert \alpha \vert \gamma \, \bvert \, \gamma \vert = \alpha^3\gamma^2$. 

What $\alpha\gamma^2, \alpha^2\gamma^4, \alpha^3\gamma^4, \alpha^3\gamma^6, \alpha\gamma^2\delta^d$ have in common is the angle arrangement $\gamma \vert \alpha \vert \gamma$, while the remaining vertices $\alpha^2\gamma^2, \alpha^3\gamma^2$ do not have that. We proceed with two subcases depending on the presence of $\gamma \vert \alpha \vert \gamma\cdots$ (one of $\alpha\gamma^2, \alpha^2\gamma^4, \alpha^3\gamma^4, \alpha^3\gamma^6, \alpha\gamma^2\delta^d$) and otherwise (either $\alpha^2\gamma^2$ or $\alpha^3\gamma^2$). 

\begin{subcase*}[$\alpha^2\beta$ and $\exists$ $\gamma \vert \alpha \vert \gamma\cdots$]

For $m=4$, the arrangement $\gamma \vert \alpha \vert \gamma$ determines tiles $T_1, T_2, T_3$ in Figure \ref{Subfig-a4-a2b2-AAD-al2be-gaalga-m=4}. Then $\alpha_2\beta_1\cdots, \alpha_2\beta_3\cdots=\alpha^2\beta$ determine $T_4, T_5$ resulting in $\alpha \vert \alpha \vert \alpha \cdots$, which is $\alpha^3\gamma^2$. Now $\alpha^3\gamma^2$ is a vertex, which rules out $\alpha\gamma^2, \alpha^3\gamma^4, \alpha^3\gamma^6$. It remains to consider one of $\alpha^2\gamma^4, \alpha\gamma^2\delta^d$ in the subsubcase assumption. The vertices $\alpha^2\beta, \alpha^3\gamma^2$ imply $\beta=2\pi - 2\alpha$ and $2\gamma = 2\pi - 3\alpha$ respectively. On the one hand, the vertex angle sum of $\alpha^2\gamma^4$ and second equation imply $\alpha=\frac{1}{2}\pi$, a contradiction; on the other hand, the two equations and the kite angle sum imply $2\gamma+\delta>2\alpha>\pi$ and $\delta > 5\alpha - 2\pi > \tfrac{1}{2}\pi$, and hence $\alpha\gamma^2\delta^d=\alpha\gamma^2\delta$. However, $\alpha\gamma^2\delta, \alpha^3\gamma^2$ imply $\delta=2\alpha>\pi > \beta$, a contradiction again. 

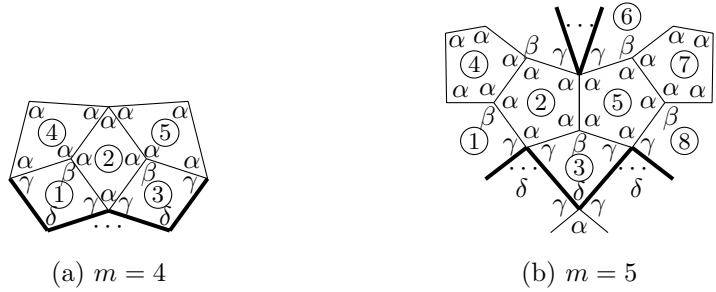
\begin{figure}[h!] 
\centering
\begin{subfigure}[t]{0.4\linewidth}
\centering
\begin{tikzpicture}

\tikzmath{
\xs=6;
\ys=1;
}

\begin{scope}[]

\tikzmath{
\r=0.85;
\gon=5;
\th=360/\gon;
\x=\r*cos(\th/2);
}

\foreach \a in {1,...,4}{

\draw[rotate=\th*\a]
	(0:0) -- (270:\r)
;
}

\foreach \a in {1,...,3}{
\draw[rotate=\th*\a]
	(270:\r) -- (270+0.5*\th:2*\x)
;
}

\foreach \a in {2,...,4}{
\draw[rotate=\th*\a]
	(270:\r) -- (270-0.5*\th:2*\x)
;
}

\draw[]
	(90:2*\x) -- (90-0.5*\th:2.6*\x) 
	(90:2*\x) -- (90+0.5*\th:2.6*\x) 
	(90-\th:2*\x) -- (90-0.5*\th:2.6*\x) 
	(90+\th:2*\x) -- (90+0.5*\th:2.6*\x) 
;

\draw[line width=1.5]
	(0,0) -- (270-\th:\r)
	(270-\th:\r) -- (270-1.5*\th:2*\x)
	(0,0) -- (270+\th:\r)
	(270+\th:\r) -- (270+1.5*\th:2*\x)
;

\node at (137.5:1.05*\x) {\small $\beta$}; 
\node at (185:1.1*\x) {\small $\delta$};
\node at (165:0.35*\x) {\small $\gamma$};
\node at (162.5:1.65*\x) {\small $\gamma$};

\node at (90:0.3*\x) {\small $\alpha$}; 
\node at (90:1.7*\x) {\small $\alpha$};
\node at (65:1.1*\x) {\small $\alpha$};
\node at (115:1.1*\x) {\small $\alpha$};

\node at (15:0.35*\x) {\small $\gamma$}; 
\node at (17.5:1.65*\x) {\small $\gamma$};
\node at (42.5:1.05*\x) {\small $\beta$}; 
\node at (355:1.1*\x) {\small $\delta$};

\node at (150:1.85*\x) {\small $\alpha$}; 
\node at (90+0.5*\th:1.15*\r) {\small $\alpha$};
\node at (90+0.5*\th:2.35*\x) {\small $\alpha$};
\node at (100:1.9*\x) {\small $\alpha$};

\node at (30:1.85*\x) {\small $\alpha$}; 
\node at (90-0.5*\th:1.15*\r) {\small $\alpha$};
\node at (90-0.5*\th:2.35*\x) {\small $\alpha$};
\node at (80:1.9*\x) {\small $\alpha$};

\node at (275:0.25*\r) {\small $\cdots$};

\node[inner sep=1,draw,shape=circle] at (90+\th:1*\x) {\footnotesize $1$};
\node[inner sep=1,draw,shape=circle] at (90:1*\x) {\footnotesize $2$};
\node[inner sep=1,draw,shape=circle] at (90-\th:1*\x) {\footnotesize $3$};
\node[inner sep=1,draw,shape=circle] at (90+0.5*\th:1.5*\r) {\footnotesize $4$};
\node[inner sep=1,draw,shape=circle] at (90-0.5*\th:1.5*\r) {\footnotesize $5$};

\end{scope}

\end{tikzpicture}
\caption{$m=4$}
\label{Subfig-a4-a2b2-AAD-al2be-gaalga-m=4}
\end{subfigure}
\begin{subfigure}[t]{0.4\linewidth}
\centering
\begin{tikzpicture}

\tikzmath{
\xs=6;
\ys=1;
}

\begin{scope}[xshift=\xs cm, yshift=\ys cm]

\tikzmath{
\r=0.625;
\g=5;
\ph=360/\g;
\x=\r*cos(\ph/2);
}

\foreach \aa in {-1,1} {

\tikzset{xshift=\aa*\x cm, xscale=-\aa}

\foreach \a in {0,...,4} {

\draw[rotate=\a*\ph]
	(0.5*\ph:\r) -- (-0.5*\ph:\r)
;

\node at (0.5*\ph+\a*\ph:0.65*\r) {\small $\alpha$};

}
}

\foreach \aa in {-1,1} {

\tikzset{xscale=\aa}

\draw[xshift=-\x cm]
	(2.5*\ph:\r) -- (2.5*\ph:2*\r)
	(1.5*\ph:\r) -- (1.75*\ph:2*\r)
	(1.75*\ph:2*\r) -- (2*\ph:2.5*\r)
	(2.5*\ph:2*\r) -- (2*\ph:2.5*\r)
	(0+\x,-2.25*\r) -- (0+1.75*\x,-2.75*\r)
;

\draw[line width=1.5]
	([xshift=-\x cm]0.5*\ph:\r) -- (0.6*\x,2.5*\x)
;

\draw[xshift=-\x cm, line width=1.5]
	(-1.5*\ph:\r) -- (0+\x,-2.25*\r)
;

\draw[xshift=-\x cm, line width=1.5]
	(3.5*\ph:\r) -- (3.25*\ph:2*\r)
;

\node at (-1.8*\x,-1.15*\x) {\small $\gamma$};
\node at (-2.4*\x,-0.4*\x) {\small $\beta$};

\node at (2.4*\x,0.35*\x) {\small $\alpha$};
\node at (1.8*\x,1.1*\x) {\small $\alpha$};
\node at (2.55*\x,1.7*\x) {\small $\alpha$};
\node at (3.2*\x,1.6*\x) {\small $\alpha$};
\node at (3.1*\x,0.3*\x) {\small $\alpha$};

}

\node at (0,-1.125*\x) {\small $\beta$};
\node at (0.8*\x,-1.3*\x) {\small $\gamma$};
\node at (-0.8*\x,-1.3*\x) {\small $\gamma$};
\node at (0,-2.3*\x) {\small $\delta$};

\node at (-1.25*\x,1.45*\x) {\small $\beta$};
\node at (-0.5*\x,1.2*\x) {\small $\gamma$};

\node at (1.25*\x,1.45*\x) {\small $\beta$};
\node at (0.5*\x,1.2*\x) {\small $\gamma$};

\node at (0.08*\x,2*\x) {\small $\cdots$};

\node at (-1.4*\x,-1.75*\x) {\small $\cdots$};
\node at (1.45*\x,-1.75*\x) {\small $\cdots$};

\node at (-1.5*\x,-2.225*\x) {\small $\delta$};
\node at (-0.5*\x,-2.75*\x) {\small $\gamma$};

\node at (1.5*\x,-2.225*\x) {\small $\delta$};
\node at (0.5*\x,-2.75*\x) {\small $\gamma$};

\node at (0,-3.25*\x) {\small $\alpha$};

\node[inner sep=1,draw,shape=circle] at (-2.75*\x,-\x) {\footnotesize $1$};
\node[inner sep=1,draw,shape=circle] at (-\x,0) {\footnotesize $2$};
\node[inner sep=1,draw,shape=circle] at (0,-1.7*\x) {\footnotesize $3$};
\node[inner sep=1,draw,shape=circle] at (-2.75*\x,\x) {\footnotesize $4$};
\node[inner sep=1,draw,shape=circle] at (\x,0) {\footnotesize $5$};
\node[inner sep=1,draw,shape=circle] at (1.25*\x,2.25*\x) {\footnotesize $6$};
\node[inner sep=1,draw,shape=circle] at (2.75*\x,\x) {\footnotesize $7$};
\node[inner sep=1,draw,shape=circle] at (2.75*\x,-\x) {\footnotesize $8$};

\end{scope}

\end{tikzpicture}
\caption{$m=5$}
\label{Subfig-a4-a2b2-AAD-al2be-gaalga-m=5}
\end{subfigure}
\caption{The deductions of $\gamma \vert \alpha \vert \gamma$ for $m=4,5$
}
\label{Fig-a4-a2b2-AAD-al2be-gaalga}
\end{figure}

For $m=5$, the arrangement $\gamma \vert \alpha \vert \gamma$ determines the angles in $T_1, T_2, ..., T_5$ in Figure \ref{Subfig-a4-a2b2-AAD-al2be-gaalga-m=5}. Knowing that $\alpha^3\cdots$ is not a vertex, the symmetry across $T_4, T_2, T_5$ allows us to assume the top $\alpha_2\alpha_4\cdots$ to be $\alpha^2\beta$. Parity Lemma and \eqref{Eq-be>de,2/3-al2be-alga} then imply $\vert \alpha \vert \gamma \, \bvert \, \cdots \, \bvert  \gamma \vert \alpha \vert = \alpha^2\gamma^2$, which now has to be a vertex, and it rules out $\alpha\gamma^2, \alpha^2\gamma^4$. Hence \eqref{Eq-be>de,2/3-al2be-alga} becomes $\alpha\gamma\cdots=\alpha^2\gamma^2, \alpha\gamma^2\delta^d$, and therefore $\gamma \vert \alpha \vert \gamma \cdots=\alpha\gamma^2\delta^d$. The top $\alpha_2\alpha_4\cdots=\alpha^2\gamma^2$ in the figure determines the angles in $T_6$. As $\alpha_5\beta_6\cdots=\alpha^2\beta$, we also determine $T_7$, and hence $\alpha_5\gamma_3\cdots= \gamma_3  \vert  \alpha_5  \vert \gamma_8  \cdots = \alpha\gamma^2\delta^d$, and the same for $ \gamma_1  \vert  \alpha_2  \vert \gamma_3  \cdots$. With \eqref{Eq-be>de,2/3-al2be-de}, this further implies $\delta_3\cdots =\alpha\gamma^2\delta$. The vertex angle sums of $\alpha^2\beta, \alpha^2\gamma^2, \alpha\gamma^2\delta$ and $m=5$ and \eqref{Eq-angle-id} determine
\begin{align*}
&\alpha = 2\cos^{-1}\lambda = (0.63471...)\pi,&
&\beta = 2\pi-2\alpha = 4\sin^{-1}\lambda = (0.73057...)\pi,& \\
&\gamma = \pi-\alpha = 2\sin^{-1}\lambda = (0.36528...)\pi,&
&\delta = 2\cos^{-1}\lambda = (0.63471...)\pi.&
\end{align*}
where $\lambda=\tfrac{1}{8}(3-\sqrt5+\sqrt{62-22\sqrt5})$. The angle values and no $\gamma\vert\gamma\cdots$ determine the vertices 
\begin{align*}
\AVC = \{ \alpha^2\beta, \alpha^2\gamma^2, \alpha\gamma^2\delta \},
\end{align*}
which contradicts Counting Lemma on $\gamma,\delta$. 
\end{subcase*}

\begin{subcase*}[$\alpha^2\beta, \alpha^3\gamma^2$] We have $\alpha\gamma\cdots$ $=\alpha^3\gamma^2$ and \eqref{Eq-be>de,2/3-al2be-ga} becomes $\gamma\cdots=\alpha^3\gamma^2$. Then \eqref{Eq-be>de,2/3-al2be-de} becomes $\delta\cdots=\delta^d$. We already know $\beta\cdots=\alpha^2\beta$, and $\alpha^3\cdots$ is a vertex only for $m=4$ from the subcase. The knowledge of $\beta\cdots, \gamma\cdots, \delta\cdots$ determines $\alpha\cdots=\alpha^2\beta, \alpha^3\gamma^2$ and $m=4$.

By $\pi>\beta>\alpha>\frac{1}{2}\pi$ and $\alpha^2\beta, \alpha^3\gamma^2$, we get $\gamma<\frac{1}{4}\pi$. Then the kite angle sum implies $\delta>\frac{1}{2}\pi$ and we get $\delta^d=\delta^3$. Therefore the vertices are 
\begin{align}\label{Eq-AVC-al2be-de3-al3ga2}
\AVC = \{ \alpha^2\beta, \delta^3, \alpha^3\gamma^2 \}.
\end{align}
Together with $m=4$ and \eqref{Eq-angle-id}, we obtain a solution for $\alpha \in (\frac{1}{2}\pi,\frac{2}{3}\pi)$ below,
\begin{align}
&\alpha = \cos^{-1}\tfrac{1}{4}(\sqrt{3} - 2) = (0.52133...)\pi,& \\
&\beta = 2\pi - 2\alpha = \cos^{-1}(-\tfrac{1}{8}(1+4\sqrt{3}))=(0.95732...)\pi,& \notag \\
&\gamma = \pi-\tfrac{3}{2}\alpha = \cos^{-1}\tfrac{1}{8}(1+3\sqrt{3})=(0.21799...)\pi,&  \notag \\
&\delta=\tfrac{2}{3}\pi.& \notag
\end{align}
Furthermore, the angle values and \eqref{Eq-a4-cosx} and \eqref{Eq-x2y2-cosy} imply
\begin{align}
&x = \cos^{-1} \tfrac{1}{33} (15 + 8\sqrt{3}),&
&y=\tan^{-1} \tfrac{1}{13}(\sqrt 2 + 3\sqrt 6).&
\end{align}

As seen in Figure \ref{Fig-be-hex-de3}, three kites at $\delta^3$ form a hexagon with alternating angles $\beta$'s and $\gamma^2$'s along the boundary. In the hexagon, we denote the marked $\gamma^2$-angles by ``$\bullet$" and the unmarked $\beta$-angles by ``$\_$" (the same notation will also be used for $\AVC$ \eqref{Eq-AVC-de3-al3be-al2ga2}). Then $\AVC$ \eqref{Eq-AVC-al2be-de3-al3ga2} becomes $\{ \alpha^2\_, \alpha^3\bullet \}$, whereby we determine the nine squares around a hexagon in  Figure \ref{Fig-be-Tilings-sq-hex-al2be-al3ga2} outlined by thick dashed lines.

Up to symmetry, it suffices to consider vertex $\circ$ in Figure \ref{Fig-be-Tilings-sq-hex-al2be-al3ga2} having degree $3$ or $4$ with respect to the squares and hexagons. If $\circ$ has degree $3$, then $\{ \alpha^2\beta, \alpha^3\bullet \}$ uniquely determine the tiling in the second picture; if $\circ$ has degree $4$, then it determines the tiling in the third. Therefore, the two \ref{Label:tri-subdiv-thick-Octa} tilings with $\AVC = \{ \alpha^2\beta, \alpha^3\gamma^2, \delta^3 \}$ are given by subdividing each hexagon into three kites: join \quotes{the centre} with the $\bullet$'s as illustrated by thin dashed lines.

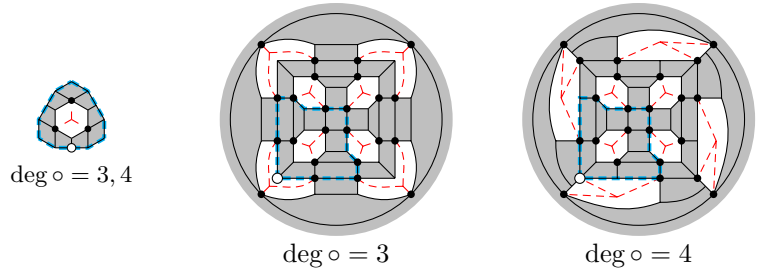
\begin{figure}[h!] 
\centering
\begin{tikzpicture}

\tikzmath{
\s=1;
\r=0.2;
\ps=360/3;
\X=\r*cos(0.5*\ps);
\ph=360/6;
\l=\r*cos(0.5*\ph);
\th=360/4;
\x=\r*cos(\th/2);
}

\begin{scope}[scale=1.25]

\tikzmath{
\rr=0.15*\r;
}

\foreach \a in {0,1,2} {

\tikzset{rotate=\a*\ps}

\fill[gray!50] 
	(90+0.5*\ph:2*\l) -- (90+\ph:1.5*\r) -- (90 + \ph:\r) -- (90:\r)
;

\fill[gray!50]
	(90:\r) -- (90-0.5*\ph:2*\l) -- (90:\r+2*\X) -- (90+0.5*\ph:2*\l) -- cycle
;

\fill[gray!50] 
	(90-0.5*\ph:2*\l) -- (90-\ph:1.5*\r) -- (90 - \ph:\r) -- (90:\r)
;

}

\foreach \a in {0,1,2} {

\tikzset{rotate=\a*\ps}

\draw[cyan, densely dashed, line width=1.5] 
	(90:\r+2*\X) -- (90-0.5*\ph:2*\l) -- (90-0.5*\ph-\ph:2*\l) -- (90-2*\ph:\r+2*\X) 
;

\draw[red, dashed]
	(0,0) -- (90:\r)
;
}

\foreach \a in {0,1,2} {

\tikzset{rotate=\a*\ps}

\draw[]
	(90:\r) -- (90-0.5*\ph:2*\l) 
	(90-0.5*\ph:2*\l) -- (90:\r+2*\X)
	(90:\r) -- (90+0.5*\ph:2*\l)
	(90+0.5*\ph:2*\l) -- (90:\r+2*\X)
	(90-\ph:\r) -- (90-\ph:1.5*\r)
	(90-0.5*\ph:2*\l) -- (90-1.5*\ph:2*\l)
;

\filldraw 
	(90:\r) circle (\rr)
;

}

\foreach \b in {0,...,5} {

\tikzset{rotate=\b*\ph}

\draw[]
	(90:\r) -- (90+\ph:\r)
;
}

\draw[fill=white] (0,-1.5*\r) circle (1.5*\rr);

\node at (0,-3*\r) {\footnotesize $\deg \circ = 3, 4$};

\end{scope}

\begin{scope}[xshift=3.5*\s cm, scale=1]

\tikzmath{
\rr=0.25*\r;
}

\fill[gray!50, scale=1.1] (0,0) circle (7*\r);

\fill[white] (0,0) circle (7*\r);

\fill[gray!50]
	(\x, \x) -- (-\x, \x) -- (-\x, -\x) -- (\x, -\x) -- cycle
;

\foreach \a in {0,1,2,3} {

\tikzset{rotate=\a*\th}

\fill[gray!50]
	(\x, \x) -- (\x, 3*\x) -- (2*\x, 4*\x) -- (2*\x, 7*\x) to[out=20,in=160] (7*\x, 7*\x) arc (45:135:7*\r) 
	to[out=20,in=160] (-2*\x, 7*\x) -- (-2*\x, 4*\x) -- (-\x, 3*\x) -- (-\x, \x) -- cycle
;

\fill[gray!50]
	(2*\x, 4*\x) -- (2*\x, 5.5*\x) -- (5.5*\x, 5.5*\x) 
	-- (5.5*\x, 2*\x) -- (4*\x, 2*\x) -- (4*\x, 4*\x) -- cycle
;

\draw[red, densely dashed]
	(2.5*\x, 2.5*\x) -- (2*\x, 4*\x)
	(2.5*\x, 2.5*\x) -- (4*\x, 2*\x)
	(2.5*\x, 2.5*\x) -- (\x, \x)
	(6.25*\x, 6.25*\x) -- (7*\x, 7*\x)
	(6.25*\x, 6.25*\x) to[out=180, in=45] (2*\x, 5.5*\x)
	(6.25*\x, 6.25*\x) to[out=270, in=45] (5.5*\x, 2*\x)
;

}

\draw[cyan, densely dashed, line width=1.5]
	(\x, \x) -- (-3*\x, \x) -- (-4*\x, 2*\x) -- (-5.5*\x, 2*\x) 
	-- (-5.5*\x, -5.5*\x) -- (2*\x, -5.5*\x) 
	-- (2*\x, -4*\x) -- (1*\x, -3*\x) -- cycle
;

\foreach \aa in {0,1,2,3} {

\tikzset{shift={(\aa*\th:2*\x)}}

\foreach \a in {0,1,2,3} {

\tikzset{rotate=\a*\th}

\draw[]
	(0.5*\th:\r) -- (1.5*\th:\r)
;

}
}

\foreach \a in {0,1,2,3} {

\tikzset{rotate=\a*\th}

\draw[]
	(\x,3*\x) -- (2*\x, 4*\x)
	(3*\x,\x) -- (4*\x, 2*\x)
	(2*\x, 4*\x) -- (4*\x, 4*\x)
	(4*\x,2*\x) -- (4*\x, 4*\x)
	(2*\x, 4*\x) -- (-2*\x, 4*\x)
	(2*\x, 4*\x) -- (2*\x, 5.5*\x)
	(-2*\x, 4*\x) -- (-2*\x, 5.5*\x)
	(-2*\x, 5.5*\x) -- (2*\x, 5.5*\x)
	(4*\x, 4*\x) -- (5.5*\x, 5.5*\x)
	(2*\x, 5.5*\x) -- (5.5*\x, 5.5*\x)
	(-2*\x, 5.5*\x) -- (-5.5*\x, 5.5*\x)
	(2*\x, 5.5*\x) -- (2*\x, 7*\x)
	(2*\x, 7*\x) to[out=20,in=160] (7*\x, 7*\x)
	(5.5*\x, 2*\x) -- (7*\x,2*\x)
	(7*\x,2*\x) to[out=70,in=290] (7*\x, 7*\x)
	(-2*\x, 7*\x) -- (2*\x, 7*\x)
;

\fill (\x,\x) circle (\rr); 
\fill (2*\x, 4*\x) circle (\rr);
\fill (4*\x, 2*\x) circle (\rr);
\fill (2*\x, 5.5*\x) circle (\rr); 
\fill (-2*\x, 5.5*\x) circle (\rr); 
\fill (7*\x, 7*\x) circle (\rr); 

}

\draw (0,0) circle (7*\r);

\draw[fill=white] (-5.5*\x, -5.5*\x) circle (1.35*\rr);

\node at (0,-9*\r) {\small $\deg \circ = 3$};

\end{scope}

\begin{scope}[xshift=7.5*\s cm, scale=1]

\tikzmath{
\rr=0.25*\r;
}

\fill[gray!50, scale=1.1] (0,0) circle (7*\r);

\fill[white] (0,0) circle (7*\r);

\fill[gray!50]
	(\x, \x) -- (-\x, \x) -- (-\x, -\x) -- (\x, -\x) -- cycle
;

\foreach \a in {0,1,2,3} {

\tikzset{rotate=\a*\th}

\fill[gray!50]
	(-2*\x, 5.5*\x) -- (-2*\x, 7.75*\x) to[out=20, in=160] (7*\x, 7*\x) arc (45:135:7*\r) -- (-5.5*\x, 5.5*\x) -- cycle
;

\fill[gray!50]
	(\x, \x) -- (\x, 3*\x) -- (2*\x, 4*\x) -- (2*\x, 5.5*\x) 
	-- (-2*\x, 5.5*\x) -- (-2*\x, 4*\x) -- (-\x, 3*\x) -- (-\x, \x) -- cycle
;

\fill[gray!50]
	(2*\x, 4*\x) -- (2*\x, 5.5*\x) -- (5.5*\x, 5.5*\x) 
	-- (5.5*\x, 2*\x) -- (4*\x, 2*\x) -- (4*\x, 4*\x)   -- cycle
;

\draw[red, densely dashed]
	(2.5*\x, 2.5*\x) -- (2*\x, 4*\x)
	(2.5*\x, 2.5*\x) -- (4*\x, 2*\x)
	(2.5*\x, 2.5*\x) -- (\x, \x)
	(2*\x, 7.25*\x) -- (7*\x, 7*\x)
	(2*\x, 7.25*\x) -- (5.5*\x, 5.5*\x)
	(2*\x, 7.25*\x) -- (-2*\x, 5.5*\x)
;

}

\draw[cyan, densely dashed, line width=1.5]
	(\x, \x) -- (-3*\x, \x) -- (-4*\x, 2*\x) -- (-5.5*\x, 2*\x) 
	-- (-5.5*\x, -5.5*\x) -- (2*\x, -5.5*\x) 
	-- (2*\x, -4*\x) -- (1*\x, -3*\x) -- cycle
;

\foreach \aa in {0,1,2,3} {

\tikzset{shift={(\aa*\th:2*\x)}}

\foreach \a in {0,1,2,3} {

\tikzset{rotate=\a*\th}

\draw[]
	(0.5*\th:\r) -- (1.5*\th:\r)
;

}
}

\foreach \a in {0,1,2,3} {

\tikzset{rotate=\a*\th}

\draw[]
	(\x,3*\x) -- (2*\x, 4*\x)
	(3*\x,\x) -- (4*\x, 2*\x)
	(2*\x, 4*\x) -- (4*\x, 4*\x)
	(4*\x,2*\x) -- (4*\x, 4*\x)
	(2*\x, 4*\x) -- (-2*\x, 4*\x)
	(2*\x, 4*\x) -- (2*\x, 5.5*\x)
	(-2*\x, 4*\x) -- (-2*\x, 5.5*\x)
	(-2*\x, 5.5*\x) -- (2*\x, 5.5*\x)
	(4*\x, 4*\x) -- (5.5*\x, 5.5*\x)
	(2*\x, 5.5*\x) -- (5.5*\x, 5.5*\x)
	(-2*\x, 5.5*\x) -- (-5.5*\x, 5.5*\x)
	(5.5*\x, 5.5*\x) -- (7*\x, 7*\x)
	(-2*\x, 5.5*\x) -- (-2*\x, 7.75*\x)
	(-2*\x, 7.75*\x) to[out=180, in=45] (-6.25*\x, 6.25*\x)
	(-2*\x, 7.75*\x) to[out=20, in=160] (7*\x, 7*\x)
;

\fill (\x,\x) circle (\rr); 
\fill (2*\x, 4*\x) circle (\rr);
\fill (4*\x, 2*\x) circle (\rr);

\fill (-2*\x, 5.5*\x) circle (\rr); 
\fill (5.5*\x, 5.5*\x) circle (\rr); 
\fill (7*\x,7*\x) circle (\rr); 

}

\draw (0,0) circle (7*\r);

\draw[fill=white] (-5.5*\x, -5.5*\x) circle (1.35*\rr);

\node at (0,-9*\r) {\small $\deg \circ = 4$};

\end{scope}

\end{tikzpicture}
\caption{The tilings by squares and hexagons with $\AVC = \{ \alpha^2\_, \alpha^3\bullet \}$, via a kite subdivision of the hexagons, determine the two \ref{Label:tri-subdiv-thick-Octa}  tilings with $\AVC = \{ \alpha^2\beta, \alpha^3\gamma^2, \delta^3 \}$}
\label{Fig-be-Tilings-sq-hex-al2be-al3ga2}
\end{figure}
\end{subcase*}

\begin{subcase*}[$\alpha^2\beta, \alpha^2\gamma^2$] We have $\alpha\gamma\cdots=\alpha^2\gamma^2$. and \eqref{Eq-be>de,2/3-al2be-ga} becomes $\gamma\cdots=\alpha^2\gamma^2$. It then rules out $\alpha^3\cdots(=\alpha^3\gamma^c)$. Meanwhile, \eqref{Eq-be>de,2/3-al2be-de} becomes $\delta\cdots=\delta^d$. We already know $\beta\cdots=\alpha^2\beta$. Given the knowledge of $\beta\cdots, \gamma\cdots, \delta\cdots$ and the absence of $\alpha^3\cdots$, we determine $\alpha\cdots=\alpha^2\beta, \alpha^2\gamma^2$, and therefore
\begin{align} \label{Eq-AVC-al2be-al2ga2-ded}
\AVC = \{ \alpha^2\beta, \alpha^2\gamma^2, \delta^d \}.
\end{align}
It implies that a vertex with $m$-gon and two kites as incident faces must be $\alpha^2\gamma^2=$ $\bvert \, \gamma \vert \alpha \vert \alpha \vert \gamma \, \bvert$.

We already know $m=4,5$ from the case ($\alpha^2\beta$).

For $m=5$, consider pentagon $T_1$ adjacent to kite $T_2$ in Figure \ref{Fig-be-m=5-al2be-al2ga2-ded}. Then $\alpha_1\beta_2\cdots=\alpha^2\beta$ and $\alpha_1\gamma_2\cdots=\alpha^2\gamma^2$ determine $T_3, T_4$ respectively. As $\alpha_1\alpha_3\cdots, \alpha_1\alpha_4\cdots$ is one of $\alpha^2\beta, \alpha^2\gamma^2$, tiles $T_5, T_6$ incident to the bottom right $\alpha_1\cdots$ must be kites, and the vertex can only be $\alpha^2\gamma^2$. However, the tiles are not arranged as $\bvert \, \gamma \vert \alpha \vert \alpha \vert \gamma \, \bvert$, a contradiction. 

\begin{figure}[h!] 
\centering
\begin{tikzpicture}

\begin{scope}[] 

\tikzmath{
\r=0.625;
\ph=360/5;
}

\foreach \a in {0,...,4} {

\tikzset{rotate=\a*\ph}

\draw[]
	(90:\r) -- (90+\ph:\r)
;

\node at (90:0.65*\r)  {\small $\alpha$};

}

\node at (90-2*\ph:1.35*\r)  {\small $?$};

\draw[]
	(90:\r) -- ([shift={(90:\r)}]45:\r)
	(90-\ph:\r) -- (90-\ph:1.8*\r)
	(90+\ph:\r) -- (90+\ph:1.8*\r)
	(90+2*\ph:\r) -- (90+2*\ph:1.8*\r)
;

\draw[line width=1.5]
	(90:\r) -- ([shift={(90:\r)}]135:\r)
;

\draw[dashed]
	(90-\ph:\r) -- (90-1.25*\ph:1.75*\r)
	(90+2*\ph:\r) -- (90+2.25*\ph:1.75*\r)
;

\node at (-0.4*\r,1*\r) {\small $\gamma$};
\node at (-1*\r,0.55*\r) {\small $\beta$};

\node at (-1.15*\r,0.1*\r) {\small $\alpha$};
\node at (-0.85*\r,-0.75*\r) {\small $\alpha$};

\node at (0.4*\r,1*\r) {\small $\alpha$};
\node at (1*\r,0.55*\r) {\small $\alpha$};

\node at (90:1.35*\r) {\small $\gamma$};

\node[inner sep=1,draw,shape=circle] at (0,0) {\footnotesize $1$};
\node[inner sep=1,draw,shape=circle] at (90+0.6*\ph:1.5*\r) {\footnotesize $2$};
\node[inner sep=1,draw,shape=circle] at (90+1.5*\ph:1.5*\r) {\footnotesize $3$};
\node[inner sep=1,draw,shape=circle] at (90-0.6*\ph:1.5*\r) {\footnotesize $4$};
\node[inner sep=1,draw,shape=circle] at (90-1.6*\ph:1.5*\r) {\footnotesize $5$};
\node[inner sep=1,draw,shape=circle] at (270:1.45*\r) {\footnotesize $6$};

\end{scope} 

\end{tikzpicture}
\caption{The deduction around a regular pentagon for $\AVC=\{ \alpha^2\beta, \alpha^2\gamma^2, \delta^d \}$}
\label{Fig-be-m=5-al2be-al2ga2-ded}
\end{figure}
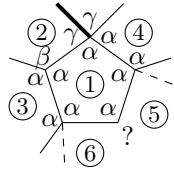

For $m=4$, if both $\alpha_1\vert\alpha_2\cdots$'s in squares $T_1, T_2$ in Figure \ref{Subfig-be-timezone-al2be-al2ga2-ded} are $\alpha^2\beta$, then they determine $T_3, T_4$. Next, $\alpha_2\gamma_3\cdots, \alpha_2\gamma_4\cdots=\alpha^2\gamma^2$ determine $T_5, T_6, T_7$. Then $\alpha_5\beta_6\cdots$ determines $T_8$. The pattern in $T_5, T_6, T_7, T_8$ is the same as that in $T_1, T_2, T_3, T_4$. If both $\alpha_1\vert\alpha_2\cdots$'s are $\alpha^2\gamma^2$'s in squares $T_1, T_2$, a similar discussion will also result in the same pattern.

\begin{figure}[h!] 
\centering
\begin{subfigure}[t]{0.45\linewidth}
\centering
\begin{tikzpicture}

\tikzmath{
\r=0.625;
\th=360/4;
\x=\r*cos(0.5*\th);
}

\begin{scope}

\foreach \xs in {0,1} {

\tikzset{xshift=4*\xs*\x cm}

\foreach \aa in {-1,1} {

\tikzset{shift={(\aa*\x,0)}}

\foreach \a in {0,1,2,3} {

\tikzset{rotate=\a*\th}

\draw[]
	(0.5*\th:\r) -- (1.5*\th:\r)
;

\node at (0.5*\th:0.6*\r) {\small $\alpha$};

}
}

\foreach \aa in {-1,1} {

\tikzset{xscale=\aa}

\draw[line width=1.5]
	(2*\x,\x) -- (2*\x,3*\x)
	(2*\x,-\x) -- (2*\x,-3*\x)
;

}

\node at (0, 2.5*\x) {\small $\delta$};
\node at (0, 1.5*\x) {\small $\beta$};
\node at (-1.6*\x, 1.5*\x) {\small $\gamma$};
\node at (1.6*\x, 1.5*\x) {\small $\gamma$};

\node at (0, -1.5*\x) {\small $\beta$};
\node at (-1.6*\x, -1.5*\x) {\small $\gamma$};
\node at (1.6*\x, -1.5*\x) {\small $\gamma$};
\node at (0, -2.5*\x) {\small $\delta$};

}

\node[inner sep=1,draw,shape=circle] at (-\x, 0) {\footnotesize $1$};
\node[inner sep=1,draw,shape=circle] at (\x, 0) {\footnotesize $2$};
\node[inner sep=1,draw,shape=circle] at (\x, 2*\x) {\footnotesize $3$};
\node[inner sep=1,draw,shape=circle] at (\x, -2*\x) {\footnotesize $4$};

\node[inner sep=1,draw,shape=circle] at (-\x+4*\x, 0) {\footnotesize $5$};
\node[inner sep=1,draw,shape=circle] at (\x+4*\x, 2*\x) {\footnotesize $6$};
\node[inner sep=1,draw,shape=circle] at (\x+4*\x, -2*\x) {\footnotesize $7$};
\node[inner sep=1,draw,shape=circle] at (\x+4*\x, 0) {\footnotesize $8$};


\end{scope}

\end{tikzpicture}
\caption{\ref{Label:EMT-Prisms}}
\label{Subfig-be-timezone-al2be-al2ga2-ded}
\end{subfigure}
\begin{subfigure}[t]{0.5\linewidth}
\centering
\begin{tikzpicture}

\tikzmath{
\r=0.625;
\th=360/4;
\x=\r*cos(0.5*\th);
}

\begin{scope}[]

\foreach \xs in {0,1} {

\tikzset{xshift=4*\xs*\x cm}

\foreach \aa in {-1,1} {

\tikzset{shift={(\aa*\x,0)}}

\foreach \a in {0,1,2,3} {

\tikzset{rotate=\a*\th}

\draw[]
	(0.5*\th:\r) -- (1.5*\th:\r)
;

\node at (0.5*\th:0.6*\r) {\small $\alpha$};

}
}

\foreach \aa in {-1,1} {

\tikzset{xscale=\aa}

\draw[line width=1.5]
	(2*\x,\x) -- (2*\x,3*\x)
;

\draw[line width=1.5, opacity=0.25]
	(4*\x,-\x) -- (4*\x,-3*\x)
;

\draw[opacity=0.25]
	(2*\x,-\x) -- (4*\x,-\x)
;

}

\draw[line width=1.5]
	(0, -\x) -- (0, -3*\x)
;

\node at (0, 2.5*\x) {\small $\delta$};
\node at (0, 1.5*\x) {\small $\beta$};
\node at (-1.6*\x, 1.5*\x) {\small $\gamma$};
\node at (1.6*\x, 1.5*\x) {\small $\gamma$};

}

\node at (2*\x, -1.5*\x) {\small $\beta$};
\node at (0.4*\x, -1.5*\x) {\small $\gamma$};
\node at (3.6*\x, -1.5*\x) {\small $\gamma$};
\node at (2*\x, -2.5*\x) {\small $\delta$};

\node at (-0.4*\x, -1.5*\x) {\small $\gamma$};
\node at (-2*\x, -1.5*\x) {\small $\beta$};
\node at (-2*\x, -2.5*\x) {\small $\delta$};

\node at (6*\x, -1.5*\x) {\small $\beta$};
\node at (4.4*\x, -1.5*\x) {\small $\gamma$};
\node at (6*\x, -2.5*\x) {\small $\delta$};

\node[inner sep=1,draw,shape=circle] at (-\x, 0) {\footnotesize $1$};
\node[inner sep=1,draw,shape=circle] at (\x, 0) {\footnotesize $2$};
\node[inner sep=1,draw,shape=circle] at (\x, 2*\x) {\footnotesize $3$};
\node[inner sep=1,draw,shape=circle] at (-\x, -2*\x) {\footnotesize $4$};
\node[inner sep=1,draw,shape=circle] at (\x, -2*\x) {\footnotesize $5$};

\node[inner sep=1,draw,shape=circle] at (\x+2*\x, 0) {\footnotesize $6$};
\node[inner sep=1,draw,shape=circle] at (\x+4*\x, 2*\x) {\footnotesize $7$};
\node[inner sep=1,draw,shape=circle] at (\x+4*\x, 0) {\footnotesize $8$};
\node[inner sep=1,draw,shape=circle] at (\x+4*\x, -2*\x) {\footnotesize $9$};

\end{scope}

\end{tikzpicture}
\caption{Flipped \ref{Label:EMT-Prisms} (shifted Antarctic Circle)}
\label{Subfig-be-timezone-flipped-al2be-al2ga2-ded}
\end{subfigure}
\caption{Time zones of \ref{Label:EMT-Prisms} with $\AVC = \{ \alpha^2\beta, \alpha^2\gamma^2, \delta^d \}$}
\label{Fig-be-Tilings-al2be-al2ga2-ded}
\end{figure}
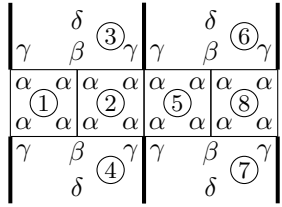
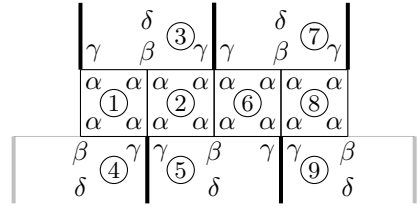

If the $\alpha_1\vert\alpha_2\cdots$'s in $T_1, T_2$ in Figure \ref{Subfig-be-timezone-flipped-al2be-al2ga2-ded} are $\alpha^2\beta$ and $\alpha^2\gamma^2$ respectively, then they determine $T_3, T_4, T_5$. Next, $\alpha_2\gamma_3\cdots=\alpha^2\gamma^2$ determines $T_6,T_7$. Then $\alpha_6\gamma_5\cdots=\alpha^2\gamma^2$ determines $T_8, T_9$. The pattern in $T_5, T_6, T_7, T_8$ is the same as that in $T_1, T_2, T_3, T_4$. 

Therefore, we have two types of time zones given by $\AVC = \{ \alpha^2\beta, \alpha^2\gamma^2, \delta^d \}$.

Combinatorially, the \ref{Label:EMT-Prisms} family are given by copies of one of the two time zones in Figure \ref{Fig-be-Tilings-al2be-al2ga2-ded} for $\delta^d$ with $d\ge3$. 

Geometrically, we verify the realisation of the family by starting with the vertex angle sums from $\AVC$ \eqref{Eq-AVC-al2be-al2ga2-ded}:
\begin{align}\label{Eq-al2be-al2ga2-ded}
\beta = 2\pi - 2\alpha, \quad
\gamma=\pi-\alpha, \quad
\delta = \tfrac{2}{d}\pi. 
\end{align}
Substituting the above into \eqref{Eq-angle-id} and specifying $m = 4$, we get 
\begin{align*}
4\cos^{2} \tfrac{1}{2}\alpha\sin^2 \tfrac{1}{2}\alpha = \sin^2 \tfrac{1}{2}\alpha( \cos \tfrac{1}{d}\pi + 1).
\end{align*}
For $\alpha \in (\tfrac{1}{2}\pi, \pi)$, we know $\sin \tfrac{1}{2}\alpha \neq 0$. Then $\cos^2 \tfrac{1}{2}\alpha = \tfrac{1}{2}(1+\cos \alpha)$ implies
\begin{align*}
\cos \tfrac{1}{d}\pi =- 1 + 4 \cos^2  \tfrac{1}{2}\alpha = 1 + 2 \cos \alpha.
\end{align*}
The above gives
\begin{align}\label{Eq-al2be-al2ga2-ded2}
\alpha = \cos^{-1} \tfrac{1}{2}(\cos \tfrac{1}{d}\pi - 1),
\end{align}
So a solution $\alpha\in(\frac12\pi,\pi)$ exists if and only if $\cos\frac1d\pi-1\in[-1,0)$, i.e. $d\ge 2$. In other words, for every positive integer $d\ge2$, a tiling exists: the kite has angle values for $\beta,\gamma, \delta$ given by \eqref{Eq-al2be-al2ga2-ded} and the square has angle value $\alpha$ given by \eqref{Eq-al2be-al2ga2-ded2}. By \eqref{Eq-x2y2-cosx} and \eqref{Eq-x2y2-cosy}, the edge lengths for each $d\ge2$ are given by
\begin{align}
&x = \cos^{-1} \frac{1+\cos \frac{1}{d}\pi}{3 - \cos \frac{1}{d}\pi},&
&y=\tfrac{1}{2}\pi-\tfrac{1}{2}\cos^{-1} \frac{1+\cos \frac{1}{d}\pi}{3 - \cos \frac{1}{d}\pi} = \tfrac{1}{2}(\pi - x).&
\end{align}
\end{subcase*}
\end{case*}

\begin{case*}[$\alpha\beta^2$] The vertex angle sum of $\alpha\beta^2$ and $\beta+\gamma>\pi$ imply $2\gamma>\alpha$. The same vertex angle sum and $\alpha>\frac{1}{2}\pi$ and $\beta>\frac{2}{3}\pi$ imply $\frac{3}{4}\pi>\beta>\frac{2}{3}\pi>\alpha>\frac{1}{2}\pi$. Then $\alpha^3$ is not a vertex and $\alpha^a\beta^b=\alpha\beta^2$. The kite angle sum and $\frac{3}{4}\pi>\beta$ imply $2\gamma+\delta>\frac{5}{4}\pi$. By \eqref{Eq-mgon-sum} and $\frac{2}{3}\pi>\alpha > (1-\tfrac{2}{m})\pi$, we also get $m=4,5$.

For $\delta\cdots$, by $\alpha>\frac{1}{2}\pi$ and $2\gamma+\delta>\pi$ we get $4\gamma+2\delta, 2\alpha+2\gamma+\delta>2\pi$. Then Parity Lemma and \eqref{Eq-vertex-de} imply $\gamma^c\delta^d=\gamma^4\delta, \gamma^2\delta^{d\ge2}$ and $\alpha\gamma^c\delta^d=\alpha\gamma^4\delta, \alpha\gamma^2\delta^d$. Moreover, $2\gamma>\alpha>\frac{1}{2}\pi$ and $2\gamma+\delta>\pi$ rule out $\alpha\gamma^4\delta$. Therefore \eqref{Eq-vertex-de} becomes
\begin{align}\label{Eq-be>de-albe2-vertex-de}
\delta \cdots= \delta^d, \gamma^4\delta, \gamma^2\delta^{d\ge2}, \alpha\gamma^2\delta^d. 
\end{align}

For $\beta\cdots$, by $\alpha\beta^2$ and $\beta+\gamma>\pi$ and $\alpha>\frac{1}{2}\pi$ we get $2\gamma> \alpha >\frac{1}{2}\pi$. Then Parity Lemma and \eqref{Eq-be>de,2/3-vertex-be} imply $\beta\gamma\cdots = \beta\gamma^4, \alpha\beta\gamma^2$. Hence \eqref{Eq-be>de,2/3-vertex-be} becomes
\begin{align}\label{Eq-be>de-albe2-vertex-be}
\beta\cdots= \alpha\beta^2,  \beta\gamma^4, \alpha\beta\gamma^2.
\end{align}
Notably, $\alpha\beta\cdots =\alpha\beta^2, \alpha\beta\gamma^2$. We proceed with the subcases in which $\alpha\beta\gamma^2$ is a vertex and otherwise.

\begin{subcase*}[$\alpha\beta^2, \alpha\beta\gamma^2$] We use the subcase assumption to determine which one in \eqref{Eq-be>de-albe2-vertex-de} is a vertex. The kite angle sum and $\alpha\beta\gamma^2$ and $\alpha>\frac{1}{2}\pi$ imply $\delta>\alpha>\frac{1}{2}\pi$, meaning $\delta^d=\delta^3$. The vertices $\alpha\beta^2, \alpha\beta\gamma^2$ imply $2\gamma=\beta$. Combining it with the kite angle sum, we get $4\gamma+\delta=\beta+2\gamma+\delta>2\pi$, which rules out $\gamma^4\delta\cdots$. Meanwhile, $2\gamma=\beta>\frac{2}{3}\pi$ and $\delta>\frac{1}{2}\pi$ imply $\gamma^2\delta^{d\ge2}=\gamma^2\delta^2$. The vertex $\alpha\beta\gamma^2$ and $\beta>\delta$ imply $d\ge2$ in $\alpha\gamma^2\delta^d$. However, $2\gamma>\frac{2}{3}\pi$ and $\delta>\alpha>\frac{1}{2}\pi$ imply $\alpha+2\gamma+2\delta>2\pi$ which rules out $\alpha\gamma^2\delta^d$. Therefore \eqref{Eq-be>de-albe2-vertex-de} becomes $\delta\cdots=\delta^3, \gamma^2\delta^2$. 

It follows that $\gamma\delta\cdots=\gamma^2\delta^2$, which has a unique angle arrangement $\bvert \, \delta \, \bvert \, \gamma \vert \gamma \, \bvert \, \delta \, \bvert$. It has incident tiles $T_1, T_2, T_3, T_4$ as in Figure \ref{Subfig-be-AAD-albe2-albega2-ga2de2-alga4}. Then $\gamma_3\delta_2\cdots=\gamma^2\delta^2$ determine $T_5, T_6$, and $\gamma_5\delta_6\cdots=\gamma^2\delta^2$ further determine $T_7, T_8$. By \eqref{Eq-be>de-albe2-vertex-be}, we have $\beta^2\cdots=\alpha\beta^2$. Then $\beta_1\beta_2\cdots, \beta_5\beta_7\cdots=\alpha\beta^2$ determine $T_9, T_{10}$. The $x$-edges between $\alpha_9,\alpha_{10}$ and Parity Lemma imply $\bvert \, \gamma_2 \vert \alpha_9 \vert \cdots  \vert \alpha_{10} \vert \gamma_5 \, \bvert$ $=\alpha^3\gamma^2\cdots$ or $\alpha^2\gamma^4\cdots$. However, $\alpha>\frac{1}{2}\pi$ and $2\gamma>\frac{2}{3}\pi$ imply that they have angle sums $>2\pi$, a contradiction. Therefore $\gamma^2\delta^2$ is not a vertex.

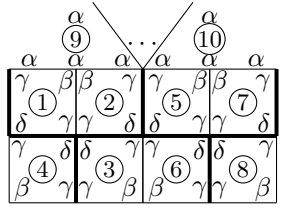
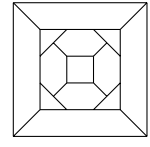
\begin{figure}[h!] 
\centering
\begin{subfigure}[t]{0.45\linewidth}
\centering
\begin{tikzpicture}

\tikzmath{
\r=0.625;
\th=360/4;
\x=\r*cos(0.5*\th);
}

\foreach \ay in {0,1} {

\tikzset{yshift=2*\ay*\x cm}

\foreach \ax in {-1,1,3,5} {

\tikzset{xshift=\ax*\x cm}

\foreach \a in {0,1,2,3} {

\tikzset{rotate=\a*\th}

\draw[]
	(0.5*\th:\r) -- (1.5*\th:\r)
;

}
}
}

\draw[]
	(2*\x,3*\x) -- (3.5*\x,5*\x)
	(2*\x,3*\x) -- (0.5*\x,5*\x)
;

\foreach \ax in {0,4} {

\tikzset{xshift=\ax*\x cm}


\node at (0,3.25*\x) {\small $\alpha$};
\node at (0,4.55*\x) {\small $\alpha$};
\node at (-1.4*\x,3.25*\x) {\small $\alpha$};
\node at (1.4*\x,3.25*\x) {\small $\alpha$};

\node at (-0.3*\x,2.6*\x) {\small $\beta$};
\node at (-0.3*\x,1.35*\x) {\small $\gamma$};
\node at (-1.6*\x,2.6*\x) {\small $\gamma$};
\node at (-1.625*\x,1.4*\x) {\small $\delta$};

\node at (0.3*\x,2.6*\x) {\small $\beta$};
\node at (0.3*\x,1.35*\x) {\small $\gamma$};
\node at (1.6*\x,2.6*\x) {\small $\gamma$};
\node at (1.625*\x,1.4*\x) {\small $\delta$};

\node at (0.3*\x,0.6*\x) {\small $\delta$};
\node at (0.3*\x,-0.6*\x) {\small $\gamma$};
\node at (1.7*\x,0.6*\x) {\small $\gamma$};
\node at (1.6*\x,-0.6*\x) {\small $\beta$};

\node at (-0.3*\x,0.6*\x) {\small $\delta$};
\node at (-1.7*\x,0.6*\x) {\small $\gamma$};
\node at (-0.3*\x,-0.6*\x) {\small $\gamma$};
\node at (-1.6*\x,-0.6*\x) {\small $\beta$};
}

\foreach \ax in {0,2,4} {

\tikzset{xshift=\ax*2*\x cm}

\draw[line width=1.5]
	(-2*\x,\x) -- (-2*\x,3*\x)
;
}

\draw[line width=1.5]
	(-2*\x,\x) -- (6*\x,\x)
	(0,\x) -- (0,-1*\x)
	(4*\x,\x) -- (4*\x,-1*\x)
;

\node at (2.05*\x, 3.75*\x) {\small $\cdots$};

\node[inner sep=1,draw,shape=circle] at (-\x,2*\x) {\footnotesize $1$};
\node[inner sep=1,draw,shape=circle] at (\x,2*\x) {\footnotesize $2$};
\node[inner sep=1,draw,shape=circle] at (\x,0) {\footnotesize $3$};
\node[inner sep=1,draw,shape=circle] at (-\x,0) {\footnotesize $4$};

\node[inner sep=1,draw,shape=circle] at (-\x+4*\x,2*\x) {\footnotesize $5$};
\node[inner sep=1,draw,shape=circle] at (-\x+4*\x,0) {\footnotesize $6$};
\node[inner sep=1,draw,shape=circle] at (\x+4*\x,2*\x) {\footnotesize $7$};
\node[inner sep=1,draw,shape=circle] at (\x+4*\x,0) {\footnotesize $8$};

\node[inner sep=1,draw,shape=circle] at (0,3.9*\x) {\footnotesize $9$};
\node[inner sep=0.25,draw,shape=circle] at (0+4*\x,3.9*\x) {\footnotesize $10$};

\end{tikzpicture}
\caption{The deduction of $\gamma^2\delta^2$}
\label{Subfig-be-AAD-albe2-albega2-ga2de2-alga4}
\end{subfigure}
\begin{subfigure}[t]{0.45\linewidth}
\centering
\begin{tikzpicture}

\tikzmath{
\r=0.5;
\dr=0.05*\r;
\rr=0.5*\r;
\R=1.5*\r;
\RR=2.5*\r;
\th=360/4;
\x=\r*cos(\th/2);
\xx=0.5*\x;
\hex=360/6;
}

\begin{scope}[]

\foreach \a in {0,1,2,3} {

\tikzset{rotate=\a*\th}

\draw[]
	(0.5*\th:\rr) -- (1.5*\th:\rr)
	(0.5*\th:\rr) -- (0.5*\th:\r)
	(\xx,\x+\xx) -- (0.5*\th:\r)
	(-\xx,\x+\xx) -- (1.5*\th:\r)
	(0.5*\th:\R) -- (1.5*\th:\R)
	(0.5*\th:\R) -- (0.5*\th:\RR)
	(0.5*\th:\RR) -- (1.5*\th:\RR)
;

%
%
}


\end{scope}

\end{tikzpicture}
\caption{A drawing of truncated octahedron}
\label{Subfigure-be-albe2-de3-albega2-alga4-tO}
\end{subfigure}
\caption{The deduction of $\gamma^2\delta^2$ and the underlying structure of tilings with $\AVC = \{ \alpha\beta^2, \delta^3, \alpha\beta\gamma^2, \alpha\gamma^4 \}$}
\label{Fig-be-AAD-albe2-albega2-ga2de2-alga4}
\end{figure}

Now it suffices to discuss $\delta\cdots=\delta^3$ in the context of $m=4,5$. 

For $m=5$, we must have $\frac{2}{3}\pi>\alpha>\frac{3}{5}\pi$. The vertex angle sums of $\alpha\beta^2, \alpha\beta\gamma^2$ and $\delta^3$ and \eqref{Eq-angle-id} determine 
\begin{align} \label{Eq-ang-vals-m=5-albe2-albega2-de3}
&\alpha= 2\cos^{-1} \tfrac{1}{12}(9-\sqrt5) = (0.61878...)\pi,&
&\beta = \cos^{-1} \tfrac{1}{12}(\sqrt5-9)=(0.69060...)\pi,& \\ \notag
&\gamma = \tfrac{1}{2} \cos^{-1} \tfrac{1}{12}(\sqrt5-9)=(0.34530...)\pi ,&
&\delta =\tfrac{2}{3}\pi,& \\ \notag
&x = 2\tan^{-1} \tfrac{1}{6}(\sqrt{5}-1),&
&y= \tan^{-1} \tfrac{1}{\sqrt{3}}(3-\sqrt{5}),&
\end{align}
which gives the vertices below
\begin{align} \label{AVC-albe2-de3-albega2-alga4}
\AVC = \{ \alpha\beta^2, \delta^3, \alpha\beta\gamma^2, \alpha\gamma^4 \}.
\end{align}

For $m=4$, we must have $\frac{2}{3}\pi>\alpha>\frac{1}{2}\pi$. Similarly we determine \begin{align} \label{Eq-ang-vals-albe2-albega2-de3}
&\alpha = \cos^{-1}(-\tfrac{1}{9})=(0.53544...)\pi,&
&\beta = \cos^{-1}(-\tfrac{2}{3})=(0.73227...)\pi,& \\ \notag
&\gamma = \cos^{-1}\tfrac{1}{\sqrt6}=(0.36613...)\pi,&
&\delta = \tfrac{2}{3}\pi,& \\ \notag
&x=\cos^{-1}\tfrac{4}{5},&
&y = \tfrac{1}{2}\cos^{-1}\tfrac{1}{5}.&
\end{align}
The angle values determine the same $\AVC$ in \eqref{AVC-albe2-de3-albega2-alga4}. 

We now construct the tilings with $\AVC$ \eqref{AVC-albe2-de3-albega2-alga4}. Since $\AVC$ \eqref{AVC-de3-albega2} is a subset of $\AVC$ \eqref{AVC-albe2-de3-albega2-alga4}, the same construction there applies as follows. As the three kites at $\delta^3$ form a regular hexagon (see Figure \ref{Fig-be-hex-de3}), for $m=4,5$ the other three vertices, $\alpha\beta^2, \alpha\beta\gamma^2, \alpha\gamma^4$, can be viewed as degree $3$ vertex $\angle_{\square} \angle_{\varhexagon} \angle_{\varhexagon}$ (resp. $\angle_{\pentagon} \angle_{\varhexagon} \angle_{\varhexagon}$) where the angles $\angle_{\varhexagon} = \beta, \gamma^2$ and $\angle_{\square}=\alpha$ for $m=4$ (resp. $\angle_{\pentagon}$ for $m=5$). Therefore, the underlying structure of the tilings after merging the kites at each $\delta^3$ is the truncated octahedron (resp. the truncated icosahedron). Each hexagon is then divided by $\delta^3$ independently. 

For $m=4$, up to isomorphism the fourteen tilings are listed in Figure \ref{Fig-be-Tilings-albe2-de3-albega2-alga4}. They can be obtained manually by analysing their graphs with $x$-, $y$-edges labelled and their subgraphs given by only the $y$-edges. The first one is the same tiling in Figure \ref{Subfig-subdiv-tO}, which has $\AVC = \{ \delta^3, \alpha\beta\gamma^2 \}$. It is the only tiling with such $\AVC$, which is minimal in terms of the number of vertex types in $\AVC$ \eqref{AVC-albe2-de3-albega2-alga4}, while all other tilings in Figure \ref{Fig-be-Tilings-albe2-de3-albega2-alga4} have all the vertices from $\AVC$ \eqref{AVC-albe2-de3-albega2-alga4}. All these tilings can be generated by flipping subdivision(s) in one tiling. Given its minimality, we call the first tiling the canonical seed. The flipping is reflected in the appearance of both $\alpha\beta^2$ and $\alpha\gamma^4$.

\begin{figure}[h!] 
\centering
\begin{tikzpicture}

\tikzmath{
\s=1;
\xs=0.45*\s;
\ys=1.8*\s;
\r=0.4;
\dr=0.05*\r;
\rr=0.5*\r;
\R=1.5*\r;
\RR=2.5*\r;
\th=360/4;
\x=\r*cos(\th/2);
\xx=0.5*\x;
\hex=360/6;
\lw=1.25;
}

\def\CaseArrayOne{
(1,1,1,1,1,1,1,1),
(1,1,1,1,-1,1,1,1),
(1,1,1,1,1,-1,1,-1),
(1,1,1,1,-1,-1,1,1),
(1,1,1,1,-1,-1,-1,1),
(1,1,1,1,-1,-1,-1,-1),
(1,1,1,-1,1,-1,1,1)
}

\def\CaseArrayTwo{
(1,1,1,-1,-1,1,-1,1),
(1,1,1,-1,1,-1,-1,1),
(1,1,1,-1,-1,-1,-1,1),
(1,1,1,-1,-1,1,-1,-1),
(-1,1,1,1,-1,-1,-1,1),
(1,-1,1,-1,-1,1,-1,1),
(1,-1,1,-1,1,-1,1,-1)
}

\def\AllCasesArray{
\CaseArrayOne,
\CaseArrayTwo
}

\foreach \allcases [count=\allcnt] in \AllCasesArray {

\tikzmath{
\cY=\allcnt;
}

\begin{scope}[yshift=-\cY*\ys cm]

\edef\caseArray{\allcases}

\foreach \Case [count=\cnt] in \caseArray {

\edef\CaseTemp{
\noexpand\GetOne\Case\noexpand\CaseOne
\noexpand\GetTwo\Case\noexpand\CaseTwo
\noexpand\GetThree\Case\noexpand\CaseThree
\noexpand\GetFour\Case\noexpand\CaseFour
\noexpand\GetFive\Case\noexpand\CaseFive
\noexpand\GetSix\Case\noexpand\CaseSix
\noexpand\GetSeven\Case\noexpand\CaseSeven
\noexpand\GetEight\Case\noexpand\CaseEight
}

\CaseTemp

\begin{scope}[xshift=\cnt*4*\xs cm]

\fill[gray!50, scale=1.15]
	(0.5*\th:\RR) -- (1.5*\th:\RR) -- (2.5*\th:\RR) -- (3.5*\th:\RR) 
;

\fill[white]
	(0.5*\th:\RR) -- (1.5*\th:\RR) -- (2.5*\th:\RR) -- (3.5*\th:\RR) 
;

\foreach \a in {0,1,2,3} {

\tikzset{rotate=\a*\th}

\fill[gray!50]
	(0.5*\th:\rr) -- (1.5*\th:\rr) -- (2.5*\th:\rr) -- (3.5*\th:\rr)
	(2*\xx,\x) -- (3*\xx,\xx) -- (0.5*\th:\R) -- (\xx,\x+\xx)
;

}

\foreach \a in {0,1,2,3} {

\tikzset{rotate=\a*\th}

\draw[]
	(0.5*\th:\rr) -- (1.5*\th:\rr)
	(0.5*\th:\rr) -- (0.5*\th:\r)
	(\xx,\x+\xx) -- (0.5*\th:\r)
	(-\xx,\x+\xx) -- (1.5*\th:\r)
	(0.5*\th:\R) -- (1.5*\th:\R)
	(0.5*\th:\R) -- (0.5*\th:\RR)
	(0.5*\th:\RR) -- (1.5*\th:\RR)
;
}


\draw[line width=\lw, rotate=0*90]
	(0,\x) -- (\CaseOne*2*\xx,\x)
	(0,\x) -- (-\CaseOne*\xx,\x+\xx) 
	(0,\x) -- (-\CaseOne*\xx,\xx)
;

\draw[line width=\lw, rotate=1*90]
	(0,\x) -- (\CaseTwo*2*\xx,\x)
	(0,\x) -- (-\CaseTwo*\xx,\x+\xx) 
	(0,\x) -- (-\CaseTwo*\xx,\xx)
;

\draw[line width=\lw, rotate=2*90]
	(0,\x) -- (\CaseThree*2*\xx,\x)
	(0,\x) -- (-\CaseThree*\xx,\x+\xx) 
	(0,\x) -- (-\CaseThree*\xx,\xx)
;

\draw[line width=\lw, rotate=3*90]
	(0,\x) -- (\CaseFour*2*\xx,\x)
	(0,\x) -- (-\CaseFour*\xx,\x+\xx) 
	(0,\x) -- (-\CaseFour*\xx,\xx)
;

\draw[line width=\lw, rotate=0*90]
	(0,2*\x) -- (-\CaseFive*\x-\CaseFive*\xx,\x+\xx)
	(0,2*\x) -- (\CaseFive*\xx,\x+\xx)
	(0,2*\x) -- (\CaseFive*\x+\CaseFive*3*\xx,\x+3*\xx)
;

\draw[line width=\lw, rotate=1*90]
	(0,2*\x) -- (-\CaseSix*\x-\CaseSix*\xx,\x+\xx)
	(0,2*\x) -- (\CaseSix*\xx,\x+\xx)
	(0,2*\x) -- (\CaseSix*\x+\CaseSix*3*\xx,\x+3*\xx)
;

\draw[line width=\lw, rotate=2*90]
	(0,2*\x) -- (-\CaseSeven*\x-\CaseSeven*\xx,\x+\xx)
	(0,2*\x) -- (\CaseSeven*\xx,\x+\xx)
	(0,2*\x) -- (\CaseSeven*\x+\CaseSeven*3*\xx,\x+3*\xx)
;

\draw[line width=\lw, rotate=3*90]
	(0,2*\x) -- (-\CaseEight*\x-\CaseEight*\xx,\x+\xx)
	(0,2*\x) -- (\CaseEight*\xx,\x+\xx)
	(0,2*\x) -- (\CaseEight*\x+\CaseEight*3*\xx,\x+3*\xx)
;


\end{scope}

}

\end{scope}
}

\end{tikzpicture}
\caption{The $14$ tilings given by $\AVC$ \eqref{AVC-albe2-de3-albega2-alga4} $=\{ \alpha\beta^2, \delta^3, \alpha\beta\gamma^2, \alpha\gamma^4 \}$}
\label{Fig-be-Tilings-albe2-de3-albega2-alga4}
\end{figure}
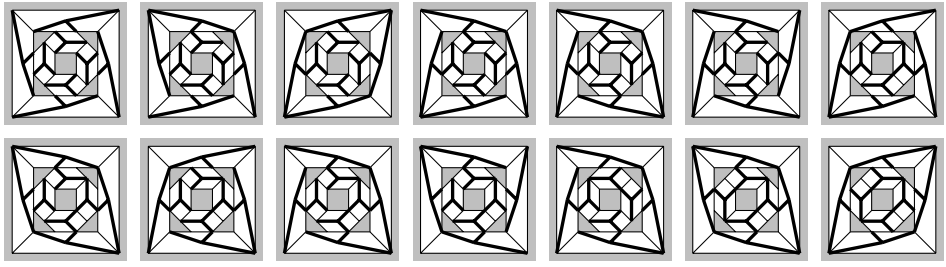

For $m=5$, up to isomorphism there are $8924$ tilings via computer enumeration\footnote{A computer-aided result by the first author.}. The reasoning for the canonical seed is the same as that for $m=4$. 
\end{subcase*}

\begin{subcase*}[$\alpha\beta\cdots=\alpha\beta^2$] The vertex $\alpha\beta^2$ has a unique angle arrangement  $\vert \alpha \vert \beta \vert \beta \vert$. Then Lemma \ref{Lem-bebe-gaga-dede} implies that $\gamma\vert\gamma\cdots$ is a vertex.

By \eqref{Eq-vertex-alga} and $\alpha\beta\cdots=\alpha\beta^2$, we know $\alpha\gamma\cdots=\alpha^a\gamma^c, \alpha\gamma^c\delta^d$ with even $c$. Then using $2\gamma>\alpha>\frac{1}{2}\pi$ and $2\gamma+\delta>\frac{5}{4}\pi$, we get
\begin{align}\label{Eq-albe2-vertex-alga}
\alpha\gamma\cdots=\alpha\gamma^2, \alpha^2\gamma^2, \alpha\gamma^4, \alpha\gamma^2\delta^d.
\end{align}
Lemma \ref{Lem-albe-alga} asserts that one of them is a vertex. 

In the absence of $\alpha^3$ (established under the case), lists \eqref{Eq-be>de-albe2-vertex-de}, \eqref{Eq-be>de-albe2-vertex-be}, \eqref{Eq-albe2-vertex-alga} imply $\alpha^2\cdots=\alpha^2\gamma^2$, meaning that the shared vertices between two adjacent $m$-gons have to be both $\alpha^2\gamma^2$. In what follows, we will rule out $\alpha^2\gamma^2$ for $m=4$ and show that there is no tiling for $m=5$.

For $m=4$, two shared vertices $\alpha^2\gamma^2$ between two squares $T_1, T_2$ are shown in Figure \ref{Subfig-be-AAD-albe2-al2ga2-m=4}. They determine $T_3, T_4$. Then $\alpha_2\beta_3\cdots, \alpha_2\beta_4\cdots=\alpha\beta^2$ lead to two $\beta$'s in $T_5$, a contradiction. 

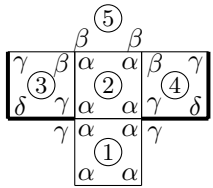
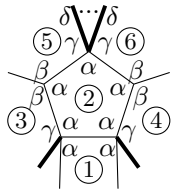
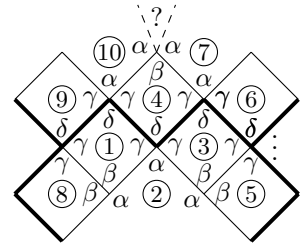
\begin{figure}[h!] 
\centering
\begin{subfigure}[t]{0.32\linewidth}
\centering
\begin{tikzpicture}
\tikzmath{
\r=0.625;
\th=360/4;
\x=\r*cos(0.5*\th);
}

\begin{scope}[]

\foreach \aa in {-1,1} {

\tikzset{shift={(0,\aa*\x)}}

\foreach \a in {0,...,3} {

\draw[rotate=\th*\a]
	(0.5*\th:\r) -- (1.5*\th:\r)
;

\node at (0.5*\th+\a*\th: 0.65*\r) {\small $\alpha$};

}
}

\foreach \aa in {-1,1} {

\tikzset{shift={(\aa*2*\x,\x)}, xscale=\aa}

\foreach \a in {0,...,3} {

\draw[rotate=\th*\a]
	(0.5*\th:\r) -- (1.5*\th:\r)
;

}

\draw[line width=1.5]
	(-\x,-\x) -- (\x,-\x)
	(\x,-\x) -- (\x,\x)
;

\node at (3.5*\th: 0.62*\r) {\small $\delta$};
\node at (2.5*\th: 0.62*\r) {\small $\gamma$};
\node at (0.5*\th: 0.62*\r) {\small $\gamma$};
\node at (1.5*\th: 0.6*\r) {\small $\beta$};

\node at (-1.2*\x,1.35*\x) {\small $\beta$};

}

\foreach \aa in {-1,1} {
\tikzset{xscale=\aa}

\node at (1.4*\x,-0.5*\x) {\small $\gamma$};

}

\node[inner sep=1,draw,shape=circle] at (0,-\x) {\footnotesize $1$};
\node[inner sep=1,draw,shape=circle] at (0,\x) {\footnotesize $2$};
\node[inner sep=1,draw,shape=circle] at (-2*\x,\x) {\footnotesize $3$};
\node[inner sep=1,draw,shape=circle] at (2*\x,\x) {\footnotesize $4$};
\node[inner sep=1,draw,shape=circle] at (0,3*\x) {\footnotesize $5$};

\end{scope}

\end{tikzpicture}
\caption{Deduction of $\alpha^2\gamma^2$, $m=4$}
\label{Subfig-be-AAD-albe2-al2ga2-m=4}
\end{subfigure}
\centering
\begin{subfigure}[t]{0.395\linewidth}
\centering
\begin{tikzpicture}

\begin{scope} 

\tikzmath{
\r=0.625;
\ph=360/5;
}

\foreach \a in {0,...,4} {
\tikzset{rotate=\a*\ph}

\draw[]
	(90:\r) -- (90+\ph:\r)
;

\node at (90:0.65*\r)  {\small $\alpha$};
}

\foreach \a in {1,4} {
\tikzset{rotate=\a*\ph}

\draw[]
	(90:\r) -- (90:1.8*\r)
;
}

\foreach \b in {2,3} {

\tikzset{rotate=\b*\ph}

\draw[line width=1.5]
	(90:\r) -- (90:1.8*\r)
;
}

\foreach \aa in {-1,1} {
\tikzset{xscale=\aa}

\draw[line width=1.5]
	(90:\r) -- (80:2*\r)
;

\draw[]
	(90-2*\ph:\r) -- (90-2.25*\ph: 2*\r) 
;

\node at (0.5*\r,1.75*\r) {\small $\delta$};
\node at (0.35*\r,1.125*\r) {\small $\gamma$};
\node at (1*\r,0.55*\r) {\small $\beta$};

\node at (1.1*\r,0*\r) {\small $\beta$};
\node at (0.85*\r,-0.75*\r) {\small $\gamma$};

\node at (0.4*\r,-1.1*\r) {\small $\alpha$};
}

\node at (0,1.85*\r) {\small $...$};

\node[inner sep=1,draw,shape=circle] at (90-0.5*\ph:1.5*\r) {\footnotesize $6$};

\node[inner sep=1,draw,shape=circle] at (90-1.5*\ph:1.5*\r) {\footnotesize $4$};

\node[inner sep=1,draw,shape=circle] at (270:1.5*\r) {\footnotesize $1$};
\node[inner sep=1,draw,shape=circle] at (0,0) {\footnotesize $2$};
\node[inner sep=1,draw,shape=circle] at (90+1.5*\ph:1.5*\r) {\footnotesize $3$};
\node[inner sep=1,draw,shape=circle] at (90+0.5*\ph:1.5*\r) {\footnotesize $5$};

\end{scope} 

\end{tikzpicture}
\caption{Deduction of $\alpha^2\gamma^2 / \gamma \vert \alpha \vert \gamma$, $m=5$}
\label{Subfig-a4-a2b2-AAD-albe2-al2ga2-m=5}
\end{subfigure}
\begin{subfigure}[t]{0.27\linewidth}
\centering
\begin{tikzpicture}
\tikzmath{
\r=0.625;
\th=360/4;
\x=\r*cos(0.5*\th);
}

\foreach \aa in {0,2} {

\tikzset{shift={(\aa*\th:\r)}}

\foreach \a in {0,1,2,3} {

\tikzset{rotate=\a*\th}

\draw[]
	(90:\r) -- (90+\th:\r)
;

}
}

\foreach \aa in {-1,1} {

\tikzset{xscale=\aa}

\draw[line width=1.5]
	(0,0) -- (\r,\r)
	(\r,\r) -- (2*\r,0)
;

\node at (-0.4*\r,0) {\small $\gamma$};
\node at (-1.6*\r,0) {\small $\gamma$};
\node at (-\r,0.6*\r) {\small $\delta$};
\node at (-\r,-0.6*\r) {\small $\beta$};

}

\foreach \aa in {-1,1} {
\tikzset{xscale=\aa}
\draw[]
	(\r,\r) -- (0,2*\r)
	(2*\r,2*\r)  -- (3*\r,\r)
	(\r,\r) -- (2*\r,2*\r)
	(\r, -\r) -- (2*\r, -2*\r)
	(2*\r,-2*\r) -- (3*\r,-\r)
;

\draw[dashed]
	(0,2*\r) -- (0.4*\r,3*\r)
;

\draw[line width=1.5]
	(2*\r,0) -- (3*\r,\r)
	(2*\r, 0) -- (3*\r, -\r)
	(3*\r,-\r) -- (2*\r,-2*\r)
;

}

\node at (0,1.55*\r) {\small $\beta$};
\node at (0,0.45*\r) {\small $\delta$};
\node at (0.6*\r,1*\r) {\small $\gamma$};
\node at (-0.6*\r,1*\r) {\small $\gamma$};

\node at (0,-0.4*\r) {\small $\alpha$};
\node at (0.75*\r,-1.15*\r) {\small $\alpha$};
\node at (-0.75*\r,-1.15*\r) {\small $\alpha$};

\foreach \aa in {-1,1} {
\tikzset{xscale=\aa}
\node at (1.4*\r,-1*\r) {\small $\beta$};
\node at (2*\r,-0.4*\r) {\small $\gamma$};

\node at (2*\r,0.4*\r) {\small $\delta$};
\node at (1.4*\r,\r) {\small $\gamma$};

\node at (0.35*\r,2.1*\r) {\small $\alpha$};
\node at (1*\r,1.4*\r) {\small $\alpha$};
}

\node at (0,2.75*\r) {\small $?$};

\node at (2.5*\r,0.15*\r) {\small $\vdots$};

\node at (2*\r,0.4*\r) {\small $\delta$};
\node at (1.4*\r,\r) {\small $\gamma$};

\node[inner sep=1,draw,shape=circle] at (-\r,0) {\footnotesize $1$};
\node[inner sep=1,draw,shape=circle] at (0,-\r) {\footnotesize $2$};
\node[inner sep=1,draw,shape=circle] at (\r,0) {\footnotesize $3$};
\node[inner sep=1,draw,shape=circle] at (0,\r) {\footnotesize $4$};

\node[inner sep=1,draw,shape=circle] at (2*\r,-1*\r) {\footnotesize $5$};
\node[inner sep=1,draw,shape=circle] at (2*\r,\r) {\footnotesize $6$};

\node[inner sep=1,draw,shape=circle] at (\r,2*\r) {\footnotesize $7$};

\node[inner sep=1,draw,shape=circle] at (-2*\r,-1*\r) {\footnotesize $8$};
\node[inner sep=1,draw,shape=circle] at (-2*\r,\r) {\footnotesize $9$};
\node[inner sep=0.25,draw,shape=circle] at (-\r,2*\r) {\footnotesize $10$};
\end{tikzpicture}
\caption{Deduction of $\alpha\gamma^2\delta$}
\label{Subfig-be-AAD-albe2-alga2de}
\end{subfigure}
\caption{Deductions in subsubcases with $\alpha^2\gamma^2$ and with $\alpha\gamma^2\delta^d$}
\label{Fig-be-AAD-albe2-al2ga2}
\end{figure}

For $m=5$, we claim that $\alpha^2\gamma^2$ is a vertex if and only if one of $\alpha\gamma^2, \alpha\gamma^4, \alpha\gamma^2\delta^d$ is also a vertex; and this will lead to a contradiction. Given $\alpha\beta\cdots=\alpha\beta^2$, if $\alpha^2\gamma^2$ is not a vertex, then $\alpha\gamma\cdots=\alpha\gamma^2, \alpha\gamma^4, \alpha\gamma^2\delta^d$ as per \eqref{Eq-albe2-vertex-alga}, and hence Lemma \ref{Lem-even-m} implies $m$ is even, a contradiction. Now, $\alpha^2\gamma^2$ is a vertex. The same deduction on the two shared $\alpha^2\gamma^2$'s between two adjacent pentagons $T_1,T_2$ in Figure \ref{Subfig-a4-a2b2-AAD-albe2-al2ga2-m=5} results in $\gamma \vert \alpha \vert \gamma \cdots$. This vertex comes from one of $\alpha\gamma^2, \alpha\gamma^4, \alpha\gamma^2\delta^d$ in \eqref{Eq-albe2-vertex-alga}. Given $2\gamma>\alpha$ and $\alpha^2\gamma^2$, it means that $\gamma \vert \alpha \vert \gamma \cdots$ can only be $\alpha\gamma^2\delta^d$. Now both $\alpha^2\gamma^2, \alpha\gamma^2\delta^d$ are vertices. Recall that $\gamma\vert\gamma\cdots$ is a vertex but is not any of \eqref{Eq-albe2-vertex-alga}. Then $\gamma\vert\gamma\cdots$ has no $\alpha$. The vertices $\alpha\beta^2, \alpha^2\gamma^2$ and $2\gamma>\alpha$ imply $4\gamma> \alpha+2\gamma = 2\beta$, which gives $2\gamma>\beta>\tfrac{2}{3}\pi$. With \eqref{Eq-be>de-albe2-vertex-be} and $\alpha\beta\cdots=\alpha\beta^2$, we know $\gamma^4\cdots=\gamma^4$, which means that $\gamma\vert\gamma\cdots$ has no $\beta$. Moreover, $\alpha^2\gamma^2, \gamma^4$ imply $\alpha=\frac{1}{2}\pi$, a contradiction. Then $\gamma\vert\gamma\cdots$ has exactly two $\gamma$'s and hence must have a $\delta$. List \eqref{Eq-be>de-albe2-vertex-de} therefore implies $\gamma\vert\gamma\cdots =\gamma^2\delta^{d\ge2}$. For $m=5$, inequality \eqref{Eq-mgon-sum} implies $\alpha>\frac{3}{5}\pi$. The vertex angle sums of $\alpha\beta^2, \alpha^2\gamma^2$ imply $\beta=\pi - \tfrac{1}{2}\alpha$ and $2\gamma=2\pi - 2\alpha$, substituting which into the kite angle sum gives $\delta > \tfrac{5}{2}\alpha - \pi$. Then $\alpha>\frac{3}{5}\pi$ further implies $\delta>\tfrac{1}{2}\pi$. Now $2\gamma>\alpha>\frac{3}{5}\pi$ and $\delta>\tfrac{1}{2}\pi$ determine $\gamma\vert\gamma\cdots =\gamma^2\delta^{d\ge2}=\gamma^2\delta^2$. Substituting its vertex angle sum and $\beta=\pi - \tfrac{1}{2}\alpha$ and $2\gamma=2\pi - 2\alpha$ into \eqref{Eq-angle-id} determine
\begin{align*}
&\alpha=(0.62334...)\pi,&
&\beta=(0.68832...)\pi,&
&\gamma=(0.37665...)\pi,&
&\delta=(0.62334...)\pi,&
\end{align*}
for $\alpha \in (\frac{3}{5}\pi, \frac{2}{3}\pi)$. The angle values determine 
\begin{align*}
\AVC = \{ \alpha\beta^2, \alpha^2\gamma^2, \gamma^2\delta^2, \alpha\gamma^2\delta \}.
\end{align*}
Then at an $\alpha\gamma^2\delta$ in Figure \ref{Subfig-be-AAD-albe2-alga2de}, the angles of the four tiles $T_1, T_2, T_3, T_4$ are determined. Next, $\alpha_2\beta_3\cdots=\alpha\beta^2$ determines $T_5$ and then $\gamma_3 \vert \gamma_5\cdots=\gamma^2\delta^2$ determines the angles in $T_6$. Now $\gamma_4\gamma_6\delta_3\cdots=\alpha\gamma^2\delta$ determines the angles in $T_7$. By mirror symmetry, $T_8, T_9, T_{10}$ are also determined. This results in $\alpha_7\alpha_{10}\beta_4\cdots$, which is not admissible, a contradiction. 

It remains to discuss one of $\alpha\gamma^2, \alpha\gamma^4, \alpha\gamma^2\delta^d$ for $m=4$.

\begin{subsubcase*}[$\alpha\beta\cdots=\alpha\beta^2$, and $\exists \, \alpha\gamma^2$ for $m=4$] The assumption reduces \eqref{Eq-albe2-vertex-alga} to $\alpha\gamma\cdots=\alpha\gamma^2$. With $\beta>\alpha$ from the case, $\alpha\gamma^2$ rules out $\beta\gamma^4$, and reduces \eqref{Eq-be>de-albe2-vertex-be} to $\beta\cdots=\alpha\beta^2$. 

Meanwhile, $\alpha\beta^2, \alpha\gamma^2$ imply $\gamma=\beta>\frac{2}{3}\pi$. Then $\alpha\gamma\cdots=\alpha\gamma^2$ and \eqref{Eq-be>de-albe2-vertex-de} imply $\delta\cdots=\delta^d, \gamma^2\delta^d$. In the absence of $\alpha^3$ (established under the case), the vertex types are
\begin{align*}
\AVC = \{ \alpha\beta^2, \alpha\gamma^2, \gamma^2\delta^d, \delta^d \}.
\end{align*}

Since $\alpha\gamma^2$ has a unique angle arrangement $\vert\alpha\vert\gamma \, \bvert \, \gamma\vert$, Lemma \ref{Lem-bebe-gaga-dede} imply that $\delta \, \bvert \, \delta\cdots$ is a vertex, which determines tiles $T_1, T_2$ in Figure \ref{Subfig-a4-a2b2-AAD-albe2-dede-m=4}. Then $\gamma_1 \, \bvert \, \gamma_2\cdots=\alpha\gamma^2$ determines $T_3$. Next, $\alpha_3\beta_1\cdots, \alpha_3\beta_2 \cdots=\alpha\beta^2$ determine $T_4, T_5$, meaning $\delta_1 \, \bvert \, \delta_2 \cdots= \gamma \, \bvert \, \delta_1 \, \bvert \, \delta_2 \, \bvert \, \gamma \cdots$ and $\delta_4 \, \bvert \, \delta_5 \cdots= \vert \gamma \, \bvert \, \delta_4 \, \bvert \, \delta_5 \, \bvert \, \gamma \vert \cdots$. They can only be $\gamma^2\delta^2$. Therefore $\delta \, \bvert \, \delta \cdots = \gamma^2\delta^2$.

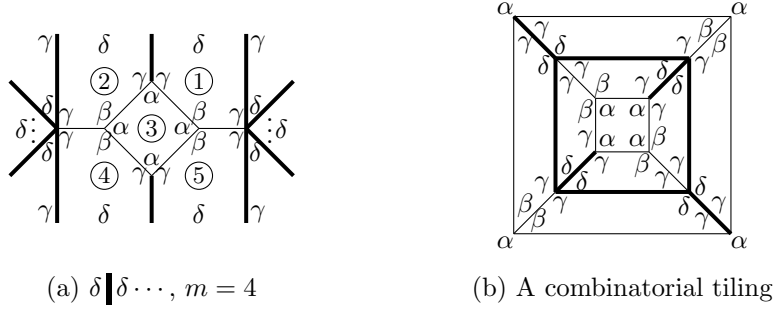
\begin{figure}[h!] 
\centering
\begin{subfigure}[t]{0.4\linewidth}
\centering
\begin{tikzpicture}

\raisebox{2ex}{
\begin{scope}[] 

\tikzmath{
\r=0.625;
\th=360/4;
\x=\r*cos(0.5*\th);
\R = sqrt(\x^2+(3*\x)^2);
\aR = acos(3*\x/\R);
}

\foreach \a in {0,...,3} {
\draw[rotate=\th*\a]
	(0*\th:\r) -- (1*\th:\r)
;

\node at (0*\th+\a*\th: 0.65*\r) {\small $\alpha$};
}

\foreach \a in {-1,1} {
\tikzset{xscale=\a}

\draw[]
	(0*\th:\r) -- (0*\th:2*\r)
;

\draw[line width=1.5]
	(2*\r,0) -- (3*\r, \r)
	(2*\r,0) -- (3*\r, -\r)
;
}

\foreach \b in {0,1} {

\draw[line width=1.5, rotate=2*\th*\b]
	(\th:\r) -- (\th:2*\r) 
	(2*\r,0) -- (2*\r,2*\r)
	(-2*\r,0) -- (-2*\r,2*\r)
;

\tikzset{rotate=\b*2*\th}

\node at (\r, 0.32*\r) {\small $\beta$};
\node at (\r, 1.8*\r) {\small $\delta$};
\node at (0.25*\r, \r) {\small $\gamma$};
\node at (1.8*\r, 0.25*\r) {\small $\gamma$};

\node at (-\r, 0.32*\r) {\small $\beta$};
\node at (-\r, 1.8*\r) {\small $\delta$};
\node at (-0.25*\r, \r) {\small $\gamma$};
\node at (-1.8*\r, 0.25*\r) {\small $\gamma$};

\node at (2.75*\r, 0) {\small $\delta$};
\node at (2.25*\r, 1.8*\r) {\small $\gamma$};
\node at (-2.25*\r, 1.8*\r) {\small $\gamma$};

\node at (2.2*\r, 0.5*\r) {\small $\delta$};
\node at (2.2*\r, -0.5*\r) {\small $\delta$};
}

\node at (2.5*\r, 0.125*\r) {\small $\vdots$};
\node at (-2.5*\r, 0.125*\r) {\small $\vdots$};

\node[inner sep=1,draw,shape=circle] at (\r,\r) {\footnotesize $1$};
\node[inner sep=1,draw,shape=circle] at (-\r,\r) {\footnotesize $2$};
\node[inner sep=1,draw,shape=circle] at (0,0) {\footnotesize $3$};
\node[inner sep=1,draw,shape=circle] at (-\r,-\r) {\footnotesize $4$};
\node[inner sep=1,draw,shape=circle] at (\r,-\r) {\footnotesize $5$};

\end{scope}
}

\end{tikzpicture}
\caption{$\delta \, \bvert \, \delta \cdots$, $m=4$}
\label{Subfig-a4-a2b2-AAD-albe2-dede-m=4}
\end{subfigure}
\begin{subfigure}[t]{0.4\linewidth}
\centering
\begin{tikzpicture}

\begin{scope}[] 

\tikzmath{
\r=0.625;
\th=90;
\x=\r*cos(\th/2);
}

\foreach \a in {0,1,2,3} {

\tikzset{rotate=\a*\th}

\draw[]
	(0.8*\x, 0.8*\x) -- (-0.8*\x, 0.8*\x)
	(3.25*\x,3.25*\x) -- (-3.25*\x,3.25*\x)
;

\draw[line width=1.5]
	(2*\x, 2*\x) -- (-2*\x, 2*\x) 
;

\node at (0.45*\x, 0.45*\x) {\small $\alpha$};
\node at (3.5*\x, 3.5*\x) {\small $\alpha$};

}

\foreach \a in {0,2} {

\tikzset{rotate=\a*\th}

\draw[]
	(-0.8*\x, 0.8*\x) -- (-2*\x,2*\x)
	(2*\x,2*\x) -- (3.25*\x,3.25*\x)
;

\draw[line width=1.5]
	(0.8*\x, 0.8*\x) -- (2*\x,2*\x)
	(-2*\x,2*\x) -- (-3.25*\x,3.25*\x)
;

\node at (-0.6*\x,1.2*\x) {\small $\beta$};
\node at (0.6*\x,1.2*\x) {\small $\gamma$};
\node at (1.2*\x,1.65*\x) {\small $\delta$};
\node at (-1.2*\x,1.65*\x) {\small $\gamma$};

\node at (-1.1*\x,0.4*\x) {\small $\beta$};
\node at (1.1*\x,0.4*\x) {\small $\gamma$};
\node at (1.65*\x, 1.2*\x) {\small $\delta$};
\node at (1.65*\x, -1.2*\x) {\small $\gamma$};

\node at (2.5*\x, 2.85*\x) {\small $\beta$};
\node at (1.85*\x, 2.4*\x) {\small $\gamma$};
\node at (-1.85*\x, 2.4*\x) {\small $\delta$};
\node at (-2.4*\x, 2.85*\x) {\small $\gamma$};

\node at (2.4*\x,1.8*\x) {\small $\gamma$};
\node at (-2.35*\x, 1.8*\x) {\small $\delta$};
\node at (2.9*\x, 2.4*\x) {\small $\beta$};
\node at (2.9*\x, -2.4*\x) {\small $\gamma$};
}

\end{scope}

\end{tikzpicture}
\caption{A combinatorial tiling}
\label{Subfig-a4-a2b2-AAD-albe2-dede-comb}
\end{subfigure}
\caption{The deduction of $\delta \, \bvert \, \delta$ and a combinatorial tiling with $\AVC = \{ \alpha\beta^2, \alpha\gamma^2, \gamma^2\delta^2 \}$}
\label{Fig-a4-a2b2-AAD-albe2-dede}
\end{figure}

We update the $\AVC$ below,
\begin{align}\label{AVC-albe2-alga2-ga2de2}
\AVC = \{ \alpha\beta^2, \alpha\gamma^2, \gamma^2\delta^2 \}.
\end{align}
For $m=4$, the vertex angle sums and \eqref{Eq-angle-id} yield no solution for $\alpha \in (\frac{1}{2}\pi, 2\pi)$. Note that a combinatorial tiling, as shown in Figure \ref{Subfig-a4-a2b2-AAD-albe2-dede-comb}, can be constructed using $\AVC$ \eqref{AVC-albe2-alga2-ga2de2}.
\end{subsubcase*}

\begin{subsubcase*}[$\alpha\beta\cdots=\alpha\beta^2$, and $\exists \, \alpha\gamma^4$ for $m=4$] The kite angle sum and $\alpha\beta^2, \alpha\gamma^4$ imply $2\gamma=\beta$ and $\delta>\alpha$. Then $\frac{3}{4}\pi>\beta>\frac{2}{3}\pi>\alpha>\frac{1}{2}\pi$ and $\beta>\delta$ imply $\frac{3}{4}\pi>2\gamma=\beta>\delta, \frac{2}{3}\pi>\alpha>\frac{1}{2}\pi$. The inequalities refine \eqref{Eq-be>de-albe2-vertex-de} as follows. By $\delta>\frac{1}{2}\pi$, we have $\delta^d=\delta^3$. By $\delta>\alpha$, we have $4\gamma+\delta>\alpha+4\gamma=2\pi$ ruling out $\gamma^4\delta\cdots$. Then by $\frac{3}{4}\pi>2\gamma>\frac{2}{3}\pi$ and $\frac{3}{4}\pi>\delta>\frac{1}{2}\pi$, we have $\gamma^2\delta^d= \gamma^2\delta^2$. By $\alpha, \delta>\frac{1}{2}\pi$ and $2\gamma > \frac{2}{3}\pi$, we have $\alpha\gamma^2\delta^d = \alpha\gamma^2\delta$. By $2\gamma=\beta>\delta$, we have $2\pi=\alpha+2\beta>\alpha+2\gamma+\delta$ ruling out $\alpha\gamma^2\delta$. Hence $\delta\cdots=\delta^3, \gamma^2\delta^2$. The latter combined with $\alpha\beta^2, \alpha\gamma^4$ and \eqref{Eq-angle-id} imply the same angle values for $\alpha\beta^2$, $\alpha\beta\gamma^2, \gamma^2\delta^2$. Given $\alpha\beta\cdots=\alpha\beta^2$, the angle values determine $\AVC = \{ \alpha\beta^2, \gamma^2\delta^2, \alpha\gamma^4 \}$. The same deduction in Figure \ref{Subfig-be-AAD-albe2-albega2-ga2de2-alga4} leads to a contradiction. The vertices $\alpha\beta^2, \alpha\gamma^4, \delta^3$ and \eqref{Eq-angle-id} give \eqref{Eq-ang-vals-albe2-albega2-de3}. The angle values determine $\AVC = \{ \alpha\beta^2,  \delta^3, \alpha\gamma^4 \}$, which is a subset of $\AVC$ \eqref{AVC-albe2-de3-albega2-alga4}. Hence any associated tilings belong to the families given by $\AVC$ \eqref{AVC-albe2-de3-albega2-alga4}. Notably, for $m=4$, the tiling is the second last in Figure \ref{Fig-be-Tilings-albe2-de3-albega2-alga4}. 
\end{subsubcase*}

\begin{subsubcase*}[$\alpha\beta\cdots=\alpha\beta^2$, and $\exists \, \alpha\gamma^2\delta^d$ for $m=4$] Based on the previous discussion, we may assume $\alpha\gamma\cdots=\alpha\gamma^2\delta^d$. The vertex angle sum of $\alpha\gamma^2\delta^d$ implies $\gamma^2\delta^d=\gamma^2\delta^{d\ge2}$. 

Suppose $d\ge2$ in $\alpha\gamma^2\delta^{d}$. With $\alpha>\frac{1}{2}\pi$ and $2\gamma+\delta > \frac{5}{4}\pi$, it implies $\delta <\frac{1}{4}\pi$, which further deduces $2\gamma> \pi$ and rules out $\gamma^4\cdots$. Next, lists \eqref{Eq-be>de-albe2-vertex-de}, \eqref{Eq-be>de-albe2-vertex-be} and $\alpha\gamma\cdots=\alpha\gamma^2\delta^d$ imply $\gamma\cdots= \gamma^2\delta^d, \alpha\gamma^2\delta^d$, and hence $\gamma\vert\gamma\cdots=\gamma^2\delta^d$. The unique angle arrangement of $\alpha\gamma^2\delta^{d\ge2}$ determines tiles $T_1, T_2, T_3$ in Figure \ref{Subfig-be-AAD-albe2-alga2de} and the angles $\delta$'s on top of them: in place of $T_4$ there is a cluster of kites. Now $\alpha_2\beta_3\cdots=\alpha\beta^2$ determines the angles in $T_5$, and $\gamma_3\vert\gamma_4\cdots=\gamma^2\delta^d$ determines $T_6$. We have $\delta_3\cdots=\vert \gamma \, \bvert \, \delta_3 \, \bvert \, \gamma_6 \vert \cdots \vert$, which is neither $\gamma^2\delta^{d\ge2}$ nor $\alpha\gamma^2\delta^{d\ge2}$, contradicting $\gamma\cdots$ above.

Suppose $d=1$ in $\alpha\gamma^2\delta^{d}$. Then $\alpha\beta^2, \alpha\gamma^2\delta$ imply $\delta=2\beta-2\gamma$, which combined with $\beta>\alpha, \delta, \tfrac{2}{3}\pi$ gives $2\gamma>\beta>\alpha, \delta, \tfrac{2}{3}\pi$. The inequalities further imply $6\gamma > \beta+4\gamma > \alpha + 4\gamma > 4\gamma + \delta > \alpha+2\gamma+\delta = 2\pi$, meaning that $\gamma^4\cdots=\gamma^4$. Notably, it follows from \eqref{Eq-be>de-albe2-vertex-be} that $\beta\gamma\cdots$ is not a vertex. Now $\gamma\vert\gamma\cdots$ is either $\gamma^4$ or $\gamma^2\cdots$ with no more $\gamma$'s. With $\alpha\beta\cdots=\alpha\beta^2$ and \eqref{Eq-be>de-albe2-vertex-de}, \eqref{Eq-be>de-albe2-vertex-be}, the latter can only be $\gamma^2\delta^{d\ge2}$. Hence $\gamma\vert\gamma\cdots=\gamma^4, \gamma^2\delta^{d\ge2}$.

For $\gamma^4$, together with $\alpha\beta^2, \alpha\gamma^2\delta$ and $m=4$, the vertex angle sums and \eqref{Eq-angle-id} determine
\begin{align*}
&\alpha=(0.53291...)\pi,& 
&\beta=(0.73354...)\pi,&
&\gamma=\tfrac{1}{2}\pi,&
&\delta=(0.46708...)\pi.& 
\end{align*}
The angle values determine $\AVC = \{ \alpha\beta^2, \alpha\gamma^2\delta, \gamma^4 \}$, which contradicts Counting Lemma on $\gamma,\delta$. Then $\gamma^4$ is not a vertex, and hence $\gamma\vert\gamma\cdots=\gamma^2\delta^{d\ge2}$, which has to be a vertex. 

For $\gamma^2\delta^{d\ge2}$, the absence of $\beta\gamma\cdots, \gamma^4\cdots$ and the subsubcase assumption $\alpha\gamma\cdots=\alpha\gamma^2\delta^d$ and \eqref{Eq-be>de-albe2-vertex-de} imply $\gamma\cdots = \alpha\gamma^2\delta, \gamma^2\delta^{d\ge2}$, contradicting Counting Lemma on $\gamma,\delta$. \qedhere
\end{subsubcase*}
\end{subcase*}
\end{case*}
\end{proof}

\begin{prop}\label{Prop-be-de} If $\beta, \delta \le \frac{2}{3}\pi$, then the dihedral tilings are
\begin{itemize}
\item
the \ref{Label:EMT-Herschel} family, 
\item 
the \ref{Label:EMT-Antiprisms} family, 
\item
the even \ref{Label:EMT-Subdiv-Equator-Prisms} family, 
\item 
the kite subdivided square pyramid $\mathcal{J}_1$, 
\item
the \ref{Label:polar-subdiv-cube} tiling,
\item
the three \ref{Label:trunc-Platonic} tilings, 
\item
the five \ref{Label:rect-Platonic} tilings, 
\item
the two \ref{Label:multigr-Cube-Dodeca} tilings, 
\item
the two \ref{Label:quad-subdiv-thick-Octa} tilings, 
\item
and the canonical seeds of the \ref{Label:trunc-Octa}, \ref{Label:trunc-Icosa} families.
\end{itemize}
\end{prop}

\begin{figure}[h!]
\centering
\begin{tikzpicture}[>=]
\tikzmath{
\x=1.5;
\y=0.25;
}

\draw[]
	(0,0) -- (0,\y)
	(0,2.75*\y) -- (0,3.75*\y)
	(-2*\x,0) -- (-2*\x,\y)
	(-2*\x,2.75*\y) -- (-2*\x,3.75*\y) -- (-2*\x,4.75*\y)
	(-2*\x,6.75*\y) -- (-2*\x,7.75*\y)
	(-4*\x,2.75*\y) -- (-4*\x,3.75*\y)
	(-1*\x,0) -- (-1*\x,-\y) 
	(-3*\x,0) -- (-3*\x,-\y) 
	(-1*\x-2*\x/3,0) -- (-1*\x-2*\x/3,-\y)
	(-3*\x+2*\x/3,0) -- (-3*\x+2*\x/3,-\y)
	(4*\x,6.75*\y) -- (4*\x,7.75*\y)
	(3*\x,6.75*\y) -- (3*\x,7.75*\y)
	(2*\x,6.75*\y) -- (2*\x,7.75*\y)
	(1*\x,6.75*\y) -- (1*\x,7.75*\y)
	(0,7.75*\y) -- (0,8.75*\y)
	(-4*\x,3.75*\y) -- (-2*\x,3.75*\y) -- (0,3.75*\y)
	(-3*\x,0) -- (-2*\x,0) -- (-1*\x,0) 
	(-2*\x,7.75*\y) -- (4*\x,7.75*\y)
;

\foreach \xs in {0,...,3} {
\tikzset{xshift=\xs*\x cm}
\draw[]
	(0,0) -- (\x,0)
;
}

\foreach \xs in {0,...,4} {
\tikzset{xshift=\xs*\x cm}
\draw[]
	(0,0) -- (0,-\y)
;
}

\node at (0,9.5*\y) {\tiny $\alpha\gamma\cdots$};

\node at (-2*\x,5.75*\y) {\tiny $\alpha\gamma^2$};

\node at (1*\x,5.75*\y) {\tiny $\alpha^2\gamma^2$};
\node at (1*\x,4.5*\y) {\tiny $\Downarrow$};
\node at (1*\x,3.5*\y) {\tiny (P8)};

\node at (2*\x,5.75*\y) {\tiny $\alpha\gamma^4$};
\node at (2*\x,4.5*\y) {\tiny $\Downarrow$};
\node at (2*\x,3.5*\y) {\tiny $\varnothing$};

\node at (3*\x,5.75*\y) {\tiny $\alpha\beta^b\gamma^2$};
\node at (3*\x,4.5*\y) {\tiny $\Downarrow$};
\node at (3*\x,3.5*\y) {\tiny $Kt\mathcal{O}, Kt\mathcal{I}$};

\node at (4*\x,5.75*\y) {\tiny $\alpha\gamma^2\delta^d$};
\node at (4*\x,4.5*\y) {\tiny $\Downarrow$};
\node at (4*\x,3.5*\y) {\tiny $\varnothing$};

\node at (-4*\x,0.5*\y) {\tiny $\Downarrow$};
\node at (-4*\x,-0.75*\y) {\tiny $K\mathcal{J}_1$};

\node at (0,1.75*\y) {\tiny $\nexists$ $\gamma^4, \gamma^2\delta^d$};
\node at (-2*\x,1.75*\y) {\tiny $\gamma^2\delta^d$, $\nexists$ $\gamma^4$};
\node at (-4*\x,1.75*\y) {\tiny $\gamma^4$};

\node at (-3*\x,-1.75*\y) {\tiny $\alpha\beta^2$};
\node at (-3*\x,-3*\y) {\tiny $\Downarrow$};
\node at (-3*\x,-4*\y) {\tiny (EM4)};

\node at (-3*\x+2*\x/3,-1.75*\y) {\tiny $\alpha^2\beta^2$};
\node at (-3*\x+2*\x/3,-3*\y) {\tiny $\Downarrow$};
\node at (-3*\x+2*\x/3,-4*\y) {\tiny $\varnothing$};

\node at (-1*\x-2*\x/3,-1.75*\y) {\tiny $\alpha^2\beta^3$};
\node at (-1*\x-2*\x/3,-3*\y) {\tiny $\Downarrow$};
\node at (-1*\x-2*\x/3,-4*\y) {\tiny $\varnothing$};

\node at (-1*\x,-1.75*\y) {\tiny $\beta^2\gamma^2$};
\node at (-1*\x,-3*\y) {\tiny $\Downarrow$};
\node at (-1*\x,-4*\y) {\tiny $\varnothing$};


\node at (0,-1.75*\y) {\tiny $\alpha^2\beta$};
\node at (0,-3*\y) {\tiny $\Downarrow$};
\node at (0,-4*\y) {\tiny (P3), (P4)};

\node at (\x,-1.75*\y) {\tiny $\alpha^3\beta^2$};
\node at (\x,-3*\y) {\tiny $\Downarrow$};
\node at (\x,-4*\y) {\tiny $\varnothing$};

\node at (2*\x,-1.75*\y) {\tiny $\alpha^2\beta^2$};
\node at (2*\x,-3*\y) {\tiny $\Downarrow$};
\node at (2*\x,-4*\y) {\tiny (EM3), (P5)};

\node at (3*\x,-1.75*\y) {\tiny $\alpha^3\beta$};
\node at (3*\x,-3*\y) {\tiny $\Downarrow$};
\node at (3*\x,-4*\y) {\tiny (EM2)};

\node at (4*\x,-1.75*\y) {\tiny $\alpha^3\beta^3$};
\node at (4*\x,-3*\y) {\tiny $\Downarrow$};
\node at (4*\x,-4*\y) {\tiny (P6)};
\end{tikzpicture}
\caption{The workflow of the proof of Proposition \ref{Prop-be-de}}
\label{Fig-flow-4.3}
\end{figure}
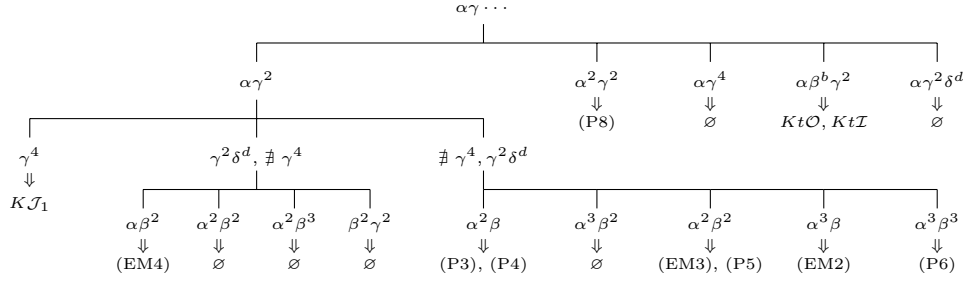

\begin{proof} The hypothesis and $\alpha>\frac{1}{2}\pi$ imply $2\alpha>\beta,\delta$. The kite angle sum and $\beta, \delta \le \frac{2}{3}\pi$ imply $2\gamma > \frac{2}{3}\pi \ge \beta, \delta$. A pair of $\gamma$'s contribute a bigger angle value at a vertex. The above inequalities deduce
\begin{align*}
2\alpha+\beta+2\gamma, 2\alpha+2\gamma+\delta, \beta+4\gamma, 4\gamma+\delta > \beta+ 2\gamma + \delta > 2\pi. 
\end{align*}

By $\alpha>\frac{1}{2}\pi$ and $2\gamma>\frac{2}{3}\pi$ and $\beta+4\gamma, 4\gamma+\delta >2\pi$,  Parity Lemma and \eqref{Eq-vertex-ga} imply $\gamma^4\cdots=\gamma^4, \alpha\gamma^4$ and $\gamma^c\delta^d=\gamma^2\delta^{d\ge2}$. The above facts will be used without reference. 

At $ \gamma \vert \gamma  \cdots$, Parity Lemma implies that it has either two more $\gamma$'s or none. The former deduces $\gamma^4\cdots=\gamma^4, \alpha\gamma^4$, whereas the latter deduces $\bvert \, \gamma \vert \gamma \, \bvert \, \delta \cdots \delta \, \bvert =$ $\gamma^2\delta^{d\ge2}$. Hence 
\begin{align}\label{Eq-be,de-vertex-gaxga}
\gamma \vert \gamma \cdots &= \gamma^4, \alpha\gamma^4, \gamma^2\delta^{d\ge2}.
\end{align}

List \eqref{Eq-vertex-alga} implies $\alpha\gamma\cdots=\alpha^a\gamma^c, \alpha^a\beta^b\gamma^c, \alpha^a\gamma^c\delta^d$. We determine $\alpha^a\gamma^c=\alpha\gamma^2, \alpha^2\gamma^2, \alpha\gamma^4$ from $\gamma^4\cdots=\gamma^4, \alpha\gamma^4$ and $\alpha>\frac{1}{2}\pi$ and $2\gamma>\frac{2}{3}\pi$ and Parity Lemma; and $\alpha^a\beta^b\gamma^c=\alpha\beta^b\gamma^2$ from $2\alpha+\beta+2\gamma > 2\pi$ and the same facts; and $\alpha^a\gamma^c\delta^d=\alpha\gamma^2\delta^d$ similarly. Hence 
\begin{align}\label{Eq-be,de-vertex-alga}
\alpha\gamma\cdots=\alpha\gamma^2, \alpha^2\gamma^2, \alpha\gamma^4, \alpha\beta^b\gamma^2, \alpha\gamma^2\delta^d. 
\end{align}
In addition, \eqref{Eq-vertex-ga}, \eqref{Eq-vertex-de} and \eqref{Eq-be,de-vertex-alga} imply
\begin{align}\label{Eq-be,de-vertex-gade}
\gamma\delta\cdots = \gamma^2\delta^{d\ge2}, \alpha\gamma^2\delta^d.
\end{align}

By Lemma \ref{Lem-albe-alga}, it suffices to discuss one of \eqref{Eq-be,de-vertex-alga} being a vertex in the cases below.

\begin{case*}[$\alpha\gamma^2$] The kite angle sum and the vertex angle sum of $\alpha\gamma^2$ imply $\beta+\delta > \alpha > \frac{1}{2}\pi$. 

The case assumption $\alpha\gamma^2$ rules out $\alpha^2\gamma^2, \alpha\gamma^4, \alpha\beta^b\gamma^2, \alpha\gamma^2\delta^d$, reducing \eqref{Eq-be,de-vertex-gaxga} to $\gamma\vert\gamma\cdots=\gamma^4, \gamma^2\delta^{d\ge2}$, and \eqref{Eq-be,de-vertex-alga} to $\alpha\gamma\cdots=\alpha\gamma^2$, and \eqref{Eq-be,de-vertex-gade} to $\gamma\delta\cdots=\gamma^2\delta^{d\ge2}$.

The knowledge of $\beta^{b}\gamma^{c}$ will be useful. For $b\ge1$ and $c\ge4$, the angle sums of $\beta^b\gamma^{c}, \alpha\gamma^2$ imply $\alpha=b\beta+(c-2)\gamma$. Combined with $\beta+\delta > \alpha$, it further implies $\delta > (b-1)\beta + (c-2) \gamma \ge 2\gamma > \frac{2}{3}\pi$, contradicting the hypothesis. As per \eqref{Eq-vertex-be}, we have $\beta^{b}\gamma^{c}=\beta^{b\ge2}\gamma^{2}$.

Lemma \ref{Lem-albe-alga} asserts that $\alpha\beta\cdots$ is a vertex. Then $\alpha\gamma\cdots=\alpha\gamma^2$ and \eqref{Eq-vertex-be} deduces $\alpha\beta\cdots=\alpha^a\beta^b$. Such a vertex is shaped by whether or not $\beta\vert\beta$ is possible. By Lemma \ref{Lem-bebe-gaga-dede}, vertices $\beta\vert\beta\cdots$ and $\gamma\vert\gamma\cdots$ coexist. As $\gamma\vert\gamma\cdots=\gamma^4, \gamma^2\delta^{d\ge2}$, further information about $a,b$ in $\alpha^a\beta^b$ depends on the subcases whether or not one of $\gamma^4, \gamma^2\delta^d$ is a vertex.

\begin{subcase*}[$\alpha\gamma^2, \gamma^4$] The vertex $\gamma^4$ implies $\gamma = \frac{1}{2}\pi$. Substituting it into the kite angle sum and the angle sum of $\alpha\gamma^2$ implies $\beta+\delta > \alpha = 2\gamma= \pi$. Then $\delta\le\frac{2}{3}\pi$ implies $\beta\ge\frac{1}{3}\pi$. Now $\alpha=\pi$ and $\beta\ge\frac{1}{3}\pi$ deduce $\alpha^a\beta^b=\alpha\beta^2$. The vertices $\alpha\beta^2, \alpha\gamma^2$ and $\gamma = \frac{1}{2}\pi$ imply $\beta=\gamma=\frac{1}{2}\pi$. Substituting them into the kite angle sum gives $\delta>\frac{1}{2}\pi$. Then \eqref{Eq-vertex-de} and \eqref{Eq-be,de-vertex-gade} imply $\delta\cdots=\delta^3$, which has to be a vertex, which deduces $\delta=\tfrac{2}{3}\pi$. The angle values are summarised below
\begin{align}
&\alpha=\pi,& &\beta=\gamma=\tfrac{1}{2}\pi,& &\delta=\tfrac{2}{3}\pi.&
\end{align}
Substituting them into \eqref{Eq-angle-id} determines $m=8$. The angle values also determine
\begin{align}\label{Eq-AVC-alga2-ga4}
\AVC = \{ \alpha\beta^2, \alpha\gamma^2, \delta^3, \beta^4, \gamma^4 \}.
\end{align}

Given $\alpha=\pi$, the boundary of an $m$-gon is a great circle, meaning that for $m=8$ the regular octagon in Figure \ref{Fig-kite-subdiv-J1} is a hemisphere and hence there is only one such tile. The remaining tiles are kites that tile the other hemisphere, which has area $2\pi$. As seen in Figure \ref{Fig-be-hex-de3}, the three incident kites at $\delta^3$ form a hexagon. Given $\gamma=\tfrac{1}{2}\pi$, the angle of the $\gamma^2$-vertices of such a hexagon is $\pi$. This means that removing $\delta^3$ and the $y$-edges results in a regular triangle $\triangle$ with edge length $2x$ and three $\beta$'s. The area of $\triangle$ is $3\beta-\pi = \tfrac{1}{2}\pi$, meaning that there are exactly four $\triangle$'s. Given $\alpha\beta\cdots=\alpha\beta^2$, placing two $\triangle$'s at an $\alpha$ immediately determines the placement of the other two $\triangle$'s, which is the square pyramid (Johnson-Zalgaller's solid $\mathcal{J}_1$). Putting back the $\delta^3$ and the $y$-edges into each $\triangle$ gives the tiling in Figure \ref{Fig-kite-subdiv-J1}.

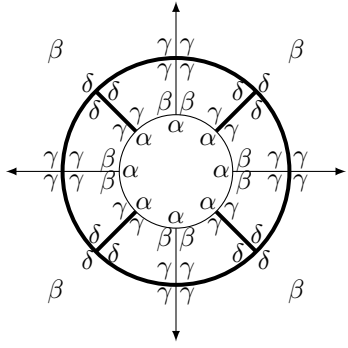
\begin{figure}[h!]
\centering
\begin{tikzpicture}[>=latex]
\tikzmath{
\r=0.75;
\n=8;
\nn=\n-1;
\th=360/\n;
}

\draw[] (0,0) circle (\r);

\draw[line width=1.5] (0,0) circle (2*\r);

\foreach \a in {0,2,...,6} {
\tikzset{rotate=\a*\th}
\draw[->]
	(90:2*\r) -- (90:3*\r)
;
}

\foreach \a in {0,2,...,6} {
\tikzset{rotate=\a*\th}
\draw[]
	(90:\r) -- (90:2*\r) 
;
\draw[line width=1.5]
	(90-\th:\r) -- (90-\th:2*\r)
;

\node at (0.2*\r,1.2*\r) {\small $\beta$};
\node at (-0.2*\r,1.2*\r) {\small $\beta$};

\node at (0.2*\r,1.75*\r) {\small $\gamma$};
\node at (-0.2*\r,1.75*\r) {\small $\gamma$};

\node at (0.2*\r,2.2*\r) {\small $\gamma$};
\node at (-0.2*\r,2.2*\r) {\small $\gamma$};

\node at (55:1.2*\r) {\small $\gamma$};
\node at (35:1.2*\r) {\small $\gamma$};

\node at (52.5:1.8*\r) {\small $\delta$};
\node at (37.5:1.8*\r) {\small $\delta$};

\node at (45:2.2*\r) {\small $\delta$};

\node at (45:3*\r) {\small $\beta$};
}

\foreach \a in {0,...,\nn} {
\tikzset{rotate=\a*\th}
\node at (0,0.8*\r) {\small $\alpha$};
}

\end{tikzpicture}
\caption{Kite subdivision of the equilateral square pyramid $J_1$}
\label{Fig-kite-subdiv-J1}
\end{figure}
\end{subcase*}

\begin{subcase*}[$\alpha\gamma^2, \gamma^2\delta^{d\ge2}$] Recall $\alpha\gamma\cdots = \alpha\gamma^2$ and $\gamma\delta\cdots= \gamma^2\delta^{d\ge2}$ from the case. Based on the previous discussion, we may also exclude $\gamma^4$ and reduce \eqref{Eq-be,de-vertex-gaxga} to $\gamma \vert \gamma \cdots =\gamma^2\delta^{d\ge2}$. The kite angle sum and $\gamma^2\delta^{d\ge2}$ imply $\beta+2\gamma+\delta > 2\pi = 2\gamma + d \delta \ge 2\gamma + 2\delta$, giving $\beta>\delta$. With $\beta+\delta>\alpha>\frac{1}{2}\pi$, it implies $2\beta>\beta+\delta>\alpha>\frac{1}{2}\pi$, which then determine
\begin{align}\label{Eq-be,de-alga2-ga2ded-albe}
\alpha\beta\cdots=\alpha^a\beta^b=\alpha\beta^b(2\le b \le 5), \alpha^2\beta^b(1\le b \le 3), \alpha^3\beta.
\end{align}

Given $\alpha\gamma^2$ and no $\gamma^4\cdots$, Parity Lemma and \eqref{Eq-vertex-ga} imply $\gamma \, \bvert \, \gamma  \cdots=\alpha\gamma^2, \beta^{b\ge2}\gamma^2$ (the same for $\bvert \, \gamma \vert \cdots \vert \gamma \, \bvert$).

The unique angle arrangement $\bvert \, \delta \, \bvert \, \gamma \vert \gamma \, \bvert \, \delta \, \bvert \, \cdots$ of $\gamma^2\delta^{d\ge2}$ determines tiles $T_1, T_2, T_3, T_4$ in Figure \ref{Fig-bede-alga2-AAD-ga2ded}. From \eqref{Eq-be,de-vertex-gade}, $\gamma_4\delta_1\cdots, \gamma_3\delta_2\cdots=\gamma^2\delta^{d\ge2}$ and their angle arrangements determine $T_5, T_6$. Since $\gamma \, \bvert \, \gamma  \cdots=\alpha\gamma^2, \beta^{b\ge2}\gamma^2$, it means $\gamma_1\gamma_5\cdots, \gamma_2\gamma_6\cdots=\alpha\gamma^2, \beta^b\gamma^2$. If one of them is $\alpha\gamma^2$ and the other is $\beta^{b\ge2}\gamma^2$, then $\beta_1\beta_2\cdots$ becomes $\alpha\beta^2\gamma\cdots$ and contradicts $\alpha\gamma\cdots=\alpha\gamma^2$. Hence either both $\gamma_1\gamma_5\cdots, \gamma_2\gamma_6\cdots$ are $\alpha\gamma^2$'s or both are $\beta^{b\ge2}\gamma^2$'s. Each outcome paves the way for the next step. If both are $\alpha\gamma^2$, then \eqref{Eq-be,de-alga2-ga2ded-albe} deduces that $\beta_1\beta_2\cdots$ is one of $\alpha\beta^2, \alpha^2\beta^2, \alpha^2\beta^3$; if both are $\beta^b\gamma^2$, then $\beta_1\beta_2\cdots$ has angle arrangement $\vert \beta_1 \vert \gamma \,\bvert \cdots \, \bvert \gamma \vert \beta_2 \vert$. Since $\bvert \, \gamma \vert \cdots \vert \gamma \, \bvert = \alpha\gamma^2, \beta^b\gamma^2$, it further implies $\beta_1\beta_2\cdots=\beta^2\gamma^2$. To sum up, $\gamma^2\delta^{d\ge2}$ leads to $\beta_1\beta_2\cdots$ being one of $\alpha\beta^2, \alpha^2\beta^2, \alpha^2\beta^3, \beta^2\gamma^2$.

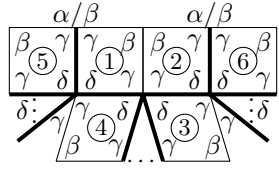
\begin{figure}[h!] 
\centering
\begin{tikzpicture}

\tikzmath{
\s=1;
\r=0.625;
\th=360/4;
\x=\r*cos(0.5*\th);
\R = sqrt(\x^2+(3*\x)^2);
\aR = acos(3*\x/\R);
}

\begin{scope}[]

\foreach \aa in {0,1} {

\tikzset{shift={(0.5*\th+\aa*\th:\r)}}

\foreach \a in {0,...,3} {

\draw[rotate=\th*\a]
	(0.5*\th:\r) -- (1.5*\th:\r)
;
}
}

\foreach \aa in {-1,1}{

\tikzset{xscale=\aa}

\draw[]
	(2*\x,0) -- (2.6*\x,-2*\x)
	(0.6*\x,-2*\x) -- (2.6*\x,-2*\x)
	(2*\x, 2*\x) -- (4*\x, 2*\x)
	(4*\x, 2*\x) -- (4*\x, 0)
;

}

\foreach \b in {-1, 1} {

\draw[line width=1.5, xscale=\b]
	(0,0) -- (2*\x,0)
	(2*\x,0) -- (2*\x, 2*\x)
	(0,0) -- (0.6*\x,-2*\x)
	(2*\x, 0) -- (4*\x, 0)
	(2*\x, 0) -- (3.75*\x,-1.4*\x)
;

}

\foreach \aa in {-1,1}{

\tikzset{shift={(\aa*\x, \x)}, xscale=\aa}

\foreach \a in {0,2} {

\tikzset{rotate=\a*\th}

\node at (0.55*\x, 0.55*\x) {\small $\gamma$};

}

\node at (-0.55*\x, 0.5*\x) {\small $\beta$};
\node at (0.55*\x, -0.55*\x) {\small $\delta$};

\node at (-4.6*\x, 0.5*\x) {\small $\beta$};
\node at (-4.6*\x, -0.6*\x) {\small $\gamma$};
}

\foreach \aa in {-1,1}{

\tikzset{shift={(\aa*\x, -\x)}, xscale=\aa}

\node at (-0.4*\x, 0.5*\x) {\small $\delta$};

\node at (0.8*\x, 0.5*\x) {\small $\gamma$};

\node at (1.1*\x, -0.6*\x) {\small $\beta$};

\node at (-0.2*\x, -0.6*\x) {\small $\gamma$};
}

\foreach \aa in {-1,1}{

\tikzset{xscale=\aa}

\node at (2.55*\x, -0.9*\x) {\small $\gamma$};
\node at (3.25*\x, -0.25*\x) {\small $\vdots$};
\node at (3.6*\x, -0.45*\x) {\small $\delta$};

}

\foreach \aa in {-1,1}{

\tikzset{xscale=\aa}

\node at (2.4*\x, 1.55*\x) {\small $\gamma$};
\node at (2.4*\x, 0.45*\x) {\small $\delta$};

}

\node at (2*\x, 2.35*\x) {\small $\alpha / \beta$}; 
\node at (-2*\x, 2.35*\x) {\small $\alpha / \beta$};

\node at (0,-2*\x) {\footnotesize $\cdots$};

\node[inner sep=1,draw,shape=circle] at (-\x,\x) {\footnotesize $1$};
\node[inner sep=1,draw,shape=circle] at (\x,\x) {\footnotesize $2$};
\node[inner sep=1,draw,shape=circle] at (1.25*\x,-\x) {\footnotesize $3$};
\node[inner sep=1,draw,shape=circle] at (-1.25*\x,-\x) {\footnotesize $4$};
\node[inner sep=1,draw,shape=circle] at (3*\x,\x) {\footnotesize $6$};
\node[inner sep=1,draw,shape=circle] at (-3*\x,\x) {\footnotesize $5$};

\end{scope}

\end{tikzpicture}
\caption{The deduction of $\gamma^2\delta^{d\ge2}$}
\label{Fig-bede-alga2-AAD-ga2ded}
\end{figure}

\begin{subsubcase*}[$\alpha\gamma^2, \gamma^2\delta^{d\ge2}$ and $\beta_1\beta_2\cdots=\alpha\beta^2$] The vertex $\alpha\beta^2$ excludes $\alpha^2\beta^2, \alpha^2\beta^3$, and with $\frac{2}{3}\pi \ge \beta$ it implies $\alpha \ge \frac{2}{3}\pi \ge \beta$, which also excludes $\alpha^3\beta$. Hence \eqref{Eq-be,de-alga2-ga2ded-albe} becomes $\alpha\beta\cdots=\alpha\beta^2$.

Recall $2\beta>\alpha$ and $\gamma \, \bvert \, \gamma  \cdots=\alpha\gamma^2, \beta^{b\ge2}\gamma^2$ from the subcase. Hence $\gamma \, \bvert \, \gamma  \cdots=\alpha\gamma^2$.

Given $\gamma\delta\cdots=\gamma^2\delta^{d\ge2}$ from the case and $\gamma \vert \gamma \cdots =\gamma^2\delta^{d\ge2}$ from the subcase, together with $\alpha\beta\cdots=\alpha\beta^2$ and $\gamma \, \bvert \, \gamma  \cdots=\alpha\gamma^2$, the same deduction through $T_1,T_2,T_3, T_4, T_5$ on $\delta \, \bvert \, \delta \cdots$ in Figure \ref{Subfig-a4-a2b2-AAD-albe2-dede-m=4} determines $\delta \, \bvert \, \delta \cdots=\gamma^2\delta^2$. This is because the argument on $\delta_1 \, \bvert \, \delta_2 \cdots$ does not require $T_4,T_5$ to be adjacent, and hence does not depend on the exact value of $m$. Now $\alpha\gamma^2, \alpha\beta^2, \gamma^2\delta^2$ are vertices, which determine $\beta=\gamma=\pi-\frac{1}{2}\alpha$ and $\delta=\frac{1}{2}\alpha$. For $m=4$, we already know from $\AVC$ \eqref{AVC-albe2-alga2-ga2de2} that these vertices yield no solutions for $\alpha$. Hence $m\ge5$. Given $\frac{2}{3}\pi \ge \beta$ in the hypothesis, $\beta=\pi-\frac{1}{2}\alpha$ implies $\alpha\ge \frac{2}{3}\pi$ and then $\delta=\frac{1}{2}\alpha$ implies $\delta\ge\frac{1}{3}\pi$. 

Next, we claim that $\delta\cdots=\gamma^2\delta^2$ and $\beta\cdots=\alpha\beta^2$. Accordingly, since $\alpha\beta^2$ is a vertex and $\beta^{b}\gamma^{c}=\beta^{b\ge2}\gamma^{2}$, list \eqref{Eq-vertex-be} will imply \eqref{Eq-be,de-alga2-ga2de2-albe2-be} below. List \eqref{Eq-vertex-de} and the subcase assumption will imply \eqref{Eq-be,de-alga2-ga2de2-albe2-de} below. Recall $\alpha\gamma\cdots=\alpha\gamma^2$ from the case; and recall that $\gamma^4$ has been excluded, which as per \eqref{Eq-be,de-vertex-gaxga} rules out $\gamma^c$ in \eqref{Eq-vertex-ga}. Combining the above gives \eqref{Eq-be,de-alga2-ga2de2-albe2-ga} below:
\begin{align}
\label{Eq-be,de-alga2-ga2de2-albe2-be}
&\beta\cdots=\alpha\beta^2, \beta^{b}, \beta^{b\ge2}\gamma^2; \\
\label{Eq-be,de-alga2-ga2de2-albe2-ga}
&\gamma\cdots=\alpha\gamma^2, \beta^{b\ge2}\gamma^2, \gamma^2\delta^2; \\
\label{Eq-be,de-alga2-ga2de2-albe2-de}
&\delta\cdots=\delta^d, \gamma^2\delta^{2}.
\end{align}

Now we prove the two claims. Assume that $\delta^d$ is also a vertex, then $\delta\ge\frac{1}{3}\pi$ implies $\delta^d = \delta^3, \delta^4, \delta^5, \delta^6$. Given $\beta=\pi-\frac{1}{2}\alpha$ and $\delta=\frac{1}{2}\alpha$, they further imply 
\begin{align*}
&\delta^3:& &\alpha=\tfrac{4}{3}\pi,& &\beta=\gamma=\tfrac{1}{3}\pi,& &\delta=\tfrac{2}{3}\pi;& \\
&\delta^4:& &\alpha=\pi,& &\beta=\gamma=\tfrac{1}{2}\pi,& &\delta=\tfrac{1}{2}\pi;& \\
&\delta^5:& &\alpha=\tfrac{4}{5}\pi,& &\beta=\gamma=\tfrac{3}{5}\pi,& &\delta=\tfrac{2}{5}\pi;& \\
&\delta^6:& &\alpha=\tfrac{2}{3}\pi,& &\beta=\gamma=\tfrac{2}{3}\pi,& &\delta=\tfrac{1}{3}\pi.&
\end{align*}
None of them yield integer solutions for $m\ge5$ in \eqref{Eq-angle-id}, ruling out $\delta^d$. Hence $\delta\cdots=\gamma^2\delta^{2}$.  

Similarly, if $\alpha^3$ or $\beta^3$ is a vertex, then the vertex angle sum implies $\alpha=\beta=\gamma=\tfrac{2}{3}\pi$ and $\delta=\tfrac{1}{3}\pi$, which has no solution for $m\ge5$ as shown above. Therefore $\alpha^3, \beta^3$ are not vertices.

Given $\alpha\beta\cdots=\alpha\beta^2$ and $\alpha\gamma\cdots=\alpha\gamma^2$ and the absence of $\alpha^3$, list \eqref{Eq-vertex-al} becomes $\alpha\cdots=\alpha\beta^2, \alpha\gamma^2$. Then, along the boundary of an $m$-gon, the vertices are either $\alpha\beta^2$ or $\alpha\gamma^2$ arranged alternatingly. Lemma \ref{Lem-even-m} implies that $m$ is even and hence $\#\alpha\beta^2=\# \alpha\gamma^2$ in a tiling.

Assume that one of $\beta^{b\ge4}, \beta^{b\ge2}\gamma^2$ is a vertex. A priori, both could appear in a tiling. Since $\beta^2\gamma^2$ and $\gamma^2\delta^{2}$ imply $\beta=\delta$, contradicting $\beta\neq\delta$, it suffices to consider $\beta^{b\ge3}\gamma^2$. For $\beta^{b_1\ge4}$ and $\beta^{b_2\ge3}\gamma^2$ and $\beta=\gamma$, we have $b_1\beta=b_2\beta+2\gamma=(b_2+2)\beta$. Hence $b_1-2=b_2\ge3$, which imply $b_1\ge5$. Moreover, $\gamma^2\delta^{2}$ and one of $\beta^{b\ge5}, \beta^{b\ge3}\gamma^2$ imply $\beta<\delta$. Then $\beta=\pi-\frac{1}{2}\alpha$ and $\delta=\frac{1}{2}\alpha$ imply $\alpha>\pi$. Counting $\beta, \gamma, \delta$ in a tiling via \eqref{Eq-be,de-alga2-ga2de2-albe2-be}, \eqref{Eq-be,de-alga2-ga2de2-albe2-ga} and \eqref{Eq-be,de-alga2-ga2de2-albe2-de} gives
\begin{align*} 
&\# \beta = 2 \# \alpha\beta^2 + b_1 \# \beta^{b_1} + b_2 \# \beta^{b_2} \gamma^2; \\
&\# \gamma = 2 \# \alpha\gamma^2 + 2\#\beta^{b_2}\gamma^2 + 2\# \gamma^2\delta^2; \\
&\# \delta = 2 \# \gamma^2\delta^2.
\end{align*}
Counting $\gamma, \delta$ (resp. $\beta,\delta$) via the kites gives $2\#\delta=\# \gamma$ (resp. $\#\beta= \#\delta$). Substituting $2\#\delta=\# \gamma$ and the third counting equation into the second gives
\begin{align*} 
\# \delta = 2 \#\alpha\gamma^2 + 2 \#\beta^{b_2}\gamma^2.
\end{align*}
Then combining it with $\#\beta= \#\delta$ and the first counting equation gives
\begin{align*} 
2(\# \alpha\beta^2 - \# \alpha\gamma^2) + b_1 \# \beta^{b_1} + (b_2 - 2)\#\beta^{b_2}\gamma^2 = 0.
\end{align*}
Since $ \# \alpha\beta^2=\# \alpha\gamma^2$ and $b_1 - 2 = b_2\ge3$, this implies $\#\beta^{b_1} = \#\beta^{b_2}\gamma^2 = 0$, contradicting one of $\beta^{b_1},\beta^{b_2}\gamma^2$ being a vertex in a tiling.

Now we conclude $\alpha\cdots=\alpha\beta^2, \alpha\gamma^2$, and $\beta\cdots=\alpha\beta^2$, and $\gamma\cdots=\alpha\gamma^2, \gamma^2\delta^2$, and $\delta\cdots=\gamma^2\delta^2$. Therefore the vertex types are
\begin{align}\label{AVC-albe2-alga2-ga2de2-m>=6}
\AVC = \{ \alpha\beta^2, \alpha\gamma^2, \gamma^2\delta^2 \}. 
\end{align}
Since $m\ge5$ is even, we have $m\ge6$.  

Starting at a $\gamma^2\delta^2$, its angles determine the incident tiles $T_1, T_2,T_3, T_4$ in Figure \ref{Subfig-prism-equatorial}. Given $\AVC$ \eqref{AVC-albe2-alga2-ga2de2-m>=6}, $\gamma_1\gamma_2\cdots=\alpha\gamma^2$ determines $T_5$ while $\beta_3\beta_4\cdots=\alpha\beta^2$ determines $T_6$. Moreover, $\gamma_1\delta_4\cdots, \gamma_2\delta_3\cdots=\gamma^2\delta^2$ determine $T_7, T_8$ and $T_9, T_{10}$. The pattern repeats to the left (and to the right) until it encloses the two $m$-gons $T_5, T_6$. This uniquely determines a tiling for even $m\ge6$. For example, see Figure \ref{Subfig-prism-m=6} for the tiling with $m=6$.

\begin{figure}[h!]
\centering
\begin{subfigure}[t]{0.475\linewidth}
\centering
\begin{tikzpicture}
\tikzmath{
\r=0.8;
\n=4;
\nn=\n-1;
\th=360/\n;
\x=\r*cos(0.5*\th);
}

\raisebox{5ex} {

\foreach \a in {0,...,\nn} {
\tikzset{rotate=\a*\th}
\draw[]
	(0,0) -- (0,\r)
	(0,\r) -- (\r,\r)
	(\r,0) -- (\r,\r)
;
}

\draw[line width=1.5]
	(0,0) -- (0,\r)
;

\foreach \aa in {-1,1} {
\tikzset{xscale=\aa}

\draw[]
	(2*\r,0) -- (2*\r,-\r)
;

\draw[line width=1.5]
	(2*\r,0) -- (2*\r,\r)
	(0,0) -- (\r,0) -- (2*\r,0)
	(\r,0) -- (\r,-\r)
;
	
\foreach \bb in {-1,1} {
\tikzset{yscale=\bb}
\draw[]
	(\r,\r) -- (2*\r,\r)
;
}
}

\foreach \xs in {-2,-1,0,1,2} {
\tikzset{xshift=\xs*\r cm}
\node at (0,1.2*\r) {\small $\alpha$};
\node at (0,-1.2*\r) {\small $\alpha$};
}

\foreach \aa in {-1,1} {
\tikzset{xscale=\aa}
\node at (0.2*\r,0.75*\r) {\small $\gamma$};
\node at (0.8*\r,0.75*\r) {\small $\beta$};
\node at (1.2*\r,0.75*\r) {\small $\beta$};
\node at (1.8*\r,0.75*\r) {\small $\gamma$};
\node at (0.2*\r,0.25*\r) {\small $\delta$};
\node at (0.8*\r,0.225*\r) {\small $\gamma$};
\node at (1.2*\r,0.225*\r) {\small $\gamma$};
\node at (1.8*\r,0.225*\r) {\small $\delta$};
\node at (0.2*\r,-0.25*\r) {\small $\gamma$};
\node at (0.8*\r,-0.225*\r) {\small $\delta$};
\node at (1.2*\r,-0.225*\r) {\small $\delta$};
\node at (1.8*\r,-0.225*\r) {\small $\gamma$};
\node at (0.2*\r,-0.8*\r) {\small $\beta$};
\node at (0.8*\r,-0.8*\r) {\small $\gamma$};
\node at (1.2*\r,-0.8*\r) {\small $\gamma$};
\node at (1.8*\r,-0.8*\r) {\small $\beta$};
}

\node[inner sep=0.8,draw,shape=circle] at (-0.5*\r,0.5*\r) {\footnotesize $1$};
\node[inner sep=0.8,draw,shape=circle] at (0.5*\r,0.5*\r) {\footnotesize $2$};
\node[inner sep=0.8,draw,shape=circle] at (0.5*\r,-0.5*\r) {\footnotesize $3$};
\node[inner sep=0.8,draw,shape=circle] at (-0.5*\r,-0.5*\r) {\footnotesize $4$};

\node[inner sep=0.8,draw,shape=circle] at (0*\r,1.75*\r) {\footnotesize $5$};
\node[inner sep=0.8,draw,shape=circle] at (0*\r,-1.75*\r) {\footnotesize $6$};

\node[inner sep=0.8,draw,shape=circle] at (-1.5*\r,0.5*\r) {\footnotesize $7$};
\node[inner sep=0.8,draw,shape=circle] at (-1.5*\r,-0.5*\r) {\footnotesize $8$};
\node[inner sep=0.8,draw,shape=circle] at (1.5*\r,0.5*\r) {\footnotesize $9$};
\node[inner sep=0.25,draw,shape=circle] at (1.5*\r,-0.5*\r) {\tiny $10$};

}

\end{tikzpicture}
\caption{Squares in a prism subdivided by kites}
\label{Subfig-prism-equatorial}
\end{subfigure}
\begin{subfigure}[t]{0.475\linewidth}
\centering
\begin{tikzpicture}
\tikzmath{
\r=0.9;
\n=6;
\nn=\n-1;
\th=360/\n;
}

\foreach \a in {0,...,\nn} {
\tikzset{rotate=\a*\th}
\draw[]
	(90:\r) -- (90+\th:\r)
	(90:2.5*\r) -- (90+\th:2.5*\r) 
;

\draw[line width=1.5]
	(90:1.75*\r) -- (90+\th:1.75*\r) 
;
}

\foreach \a in {0,2,4} {
\tikzset{rotate=\a*\th}
\draw[]
	(90:\r) -- (90:1.75*\r) 
	(270:1.75*\r) -- (270:2.5*\r) 
;
\draw[line width=1.5]
	(270:\r) -- (270:1.75*\r) 
	(90:1.75*\r) -- (90:2.5*\r) 
;
}

\foreach \a in {0,...,\nn} {
\tikzset{rotate=\a*\th}
\node at (90:0.7*\r) {\small $\alpha$};
\node at (90:2.7*\r) {\small $\alpha$};
}

\foreach \a in {0,2,4} {
\tikzset{rotate=\a*\th}
\node at (0.2*\r,2.2*\r) {\small $\gamma$};
\node at (-0.2*\r,2.2*\r) {\small $\gamma$};

\node at (0.2*\r,1.85*\r) {\small $\delta$};
\node at (-0.2*\r,1.85*\r) {\small $\delta$};

\node at (0.15*\r,1.45*\r) {\small $\gamma$};
\node at (-0.15*\r,1.45*\r) {\small $\gamma$};

\node at (0.2*\r,1.1*\r) {\small $\beta$};
\node at (-0.2*\r,1.1*\r) {\small $\beta$};

\node at (0.2*\r,-1.1*\r) {\small $\gamma$};
\node at (-0.2*\r,-1.1*\r) {\small $\gamma$};

\node at (0.2*\r,-1.45*\r) {\small $\delta$};
\node at (-0.2*\r,-1.45*\r) {\small $\delta$};

\node at (0.2*\r,-1.85*\r) {\small $\gamma$};
\node at (-0.2*\r,-1.85*\r) {\small $\gamma$};

\node at (0.2*\r,-2.175*\r) {\small $\beta$};
\node at (-0.2*\r,-2.175*\r) {\small $\beta$};
}

\end{tikzpicture}
\caption{A tiling for $m=6$}
\label{Subfig-prism-m=6}
\end{subfigure}
\caption{Kite subdivision by $\gamma^2\delta^2$}
\label{Fig-prism-tilings}
\end{figure}
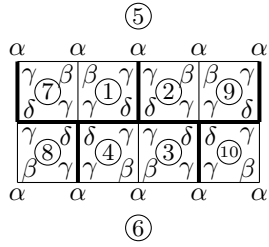
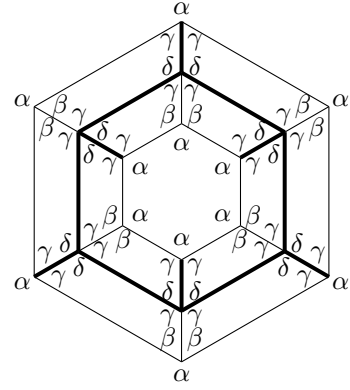

To verify the existence of the infinite family, we substitute $\beta=\gamma=\pi-\frac{1}{2}\alpha$ and $\delta=\frac{1}{2}\alpha$ into \eqref{Eq-angle-id} and get 
$\cos^2\tfrac{2}{m}\pi=(\sin\frac{1}{2}\alpha - \cos\frac{1}{2}\alpha)^2=1-\sin\alpha$. Then for each even integer $m\ge6$, there always exists a solution given by
\begin{align}
\alpha = \pi - \sin^{-1}(1-\cos^2\tfrac{2}{m}\pi) = \pi-\sin^{-1}(\sin^2\tfrac{2}{m}\pi).
\end{align}
This solution $\alpha$ always belongs to the interval $((1-\frac2m)\pi,\pi)$ because $\sin\frac{2}{m}\pi\in(0,1)$; therefore a tiling exists.
\end{subsubcase*}

\begin{subsubcase*}[$\alpha\gamma^2, \gamma^2\delta^{d\ge2}$ and $\beta_1\beta_2\cdots=\alpha^2\beta^2$] 
The vertex $\alpha^2\beta^2$ gives $\alpha+\beta=\pi$. With $\alpha>(1-\tfrac{2}{m})\pi$ from \eqref{Eq-mgon-sum}, we deduce $\beta<\frac{2}{m}\pi$. The kite angle sum then implies $2\gamma+\delta>2\pi-\beta > (2-\tfrac{2}{m})\pi$. Combining $\beta+\delta>\alpha$ from the case with $\alpha>(1-\tfrac{2}{m})\pi$ and $\beta<\tfrac{2}{m}\pi$ gives $\delta > (1-\tfrac{4}{m})\pi$. The vertex angle sum of $\gamma^2\delta^{d\ge2}$ and the lower bounds for $2\gamma+\delta,\delta$ imply $2\pi \ge 2\gamma+2\delta > (3-\frac{6}{m})\pi$, and hence $m=4,5$. 


For $m=5$, the same argument actually implies $2\gamma+3\delta>2\pi$, i.e., $\gamma^2\delta^{d\ge2}=\gamma^2\delta^2$. The vertices $\alpha\gamma^2, \alpha^2\beta^2, \gamma^2\delta^2$ and \eqref{Eq-angle-id} determine the angle values below
\begin{align*}
&\alpha=(0.62222...)\pi,&
&\beta=(0.37777...)\pi,&
&\gamma=(0.68888...)\pi,&
&\delta=(0.31111...)\pi.&
\end{align*}
The angle values then determine
\begin{align*}
\AVC = \{ \alpha\gamma^2, \alpha^2\beta^2, \gamma^2\delta^2 \}. 
\end{align*}
Resume the deduction in Figure \ref{Fig-bede-alga2-AAD-ga2ded} with $\gamma^2\delta^d=\gamma^2\delta^2$. Tiles $T_3, T_4$ are now adjacent and $T_5, T_6$ have been previously determined, as shown in Figure \ref{Subfig-bede-alga2-al2be2-m=5}. Now $\gamma_3\gamma_4\cdots=\alpha\gamma^2$ and $\beta_4\beta_5\cdots, \beta_3\beta_6\cdots=\alpha^2\beta^2$ determine $T_7, T_8, T_9$. The bottom $\alpha_7\alpha_9\cdots, \alpha_8\alpha_9\cdots = \alpha^2\beta^2$ result in two $\beta$'s in $T_{10}$, a contradiction. Hence there is no tiling for $m=5$.

For $m=4$, if both $\alpha^3, \alpha\gamma^2$ are vertices, then they imply $\alpha=\gamma=\frac{2}{3}\pi$. The former further implies $\beta=\frac{1}{3}\pi$ while the latter with $\beta\neq\delta$ and $\gamma^2\delta^{d\ge2}$ implies $\delta < \frac{1}{3}\pi$. These angle values in \eqref{Eq-angle-id} yield no solution for $\delta \in (0, \frac{1}{3}\pi)$. Now $\alpha^3$ is dismissed. Recall $2\beta>\alpha>\frac{1}{2}\pi$ from the subcase. The inequalities and $\alpha+\beta=\pi$ imply $2\beta>\alpha>\tfrac{1}{2}\pi>\beta$. Then \eqref{Eq-be,de-alga2-ga2ded-albe} is reduced to $\alpha\beta\cdots=\alpha^2\beta^2$. Resume the deduction in Figure \ref{Fig-bede-alga2-AAD-ga2ded} with $\beta_1 \vert \beta_2\cdots=\alpha^2\beta^2$ as shown in Figure \ref{Subfig-bede-alga2-al2be2-m=4}. We determine two new tiles $T_7,T_8$. Then $\alpha_7\beta_5\cdots, \alpha_8\beta_6\cdots=\alpha^2\beta^2$ imply that $\alpha_7\alpha_8\cdots$ is one of $\alpha^3, \alpha^4\cdots$ or $\alpha^2\gamma\cdots, \alpha^2\gamma^2\cdots$, but none of them is a vertex, a contradiction.

\begin{figure}[h!] 
\centering
\begin{subfigure}[t]{0.3\linewidth}
\centering
\begin{tikzpicture}

\tikzmath{
\s=5;
\r=0.625;
\th=360/4;
\x=\r*cos(0.5*\th);
}

\begin{scope}[yshift=2*\x cm]

\foreach \aaa in {0,1} {

\tikzset{yshift=\aaa*2*\x cm}

\foreach \aa in {-1,1} {

\tikzset{xscale=\aa}

\foreach \a in {0,1,2,3} {

\tikzset{xshift=\x cm, rotate=\a*\th}

\draw[]
	(0.5*\th:\r) -- (0.5*\th+\th:\r)
;

}
}
}

\draw[]
	(-2*\x,-3*\x) -- (2*\x,-3*\x)
;

\draw[line width=1.5]
	(0,\x) -- (0,-\x)
;

\foreach \aa in {-1,1} {

\tikzset{xscale=\aa}

\draw[]
	(2*\x,-1*\x) -- (4*\x,-1*\x)
	(2*\x,-1*\x) -- (2*\x,-3*\x)
	(2*\x,-3*\x) -- (4*\x,-3*\x)
	(2*\x,-3*\x) -- (2*\x,-4*\x)
;

\draw[line width=1.5]
	(2*\x,\x) -- (4*\x,\x)
	(2*\x,1*\x) -- (2*\x,3*\x) 
	(0,\x) -- (2*\x,\x)
;

\node at (0.4*\x, 2.6*\x) {\small $\beta$};
\node at (1.55*\x, 2.6*\x) {\small $\gamma$};
\node at (0.4*\x, 1.4*\x) {\small $\gamma$};
\node at (1.55*\x, 1.4*\x) {\small $\delta$};

\node at (0.4*\x, 0.6*\x) {\small $\delta$};
\node at (1.6*\x, 0.6*\x) {\small $\gamma$};
\node at (1.6*\x, -0.6*\x) {\small $\beta$};
\node at (0.35*\x, -0.625*\x) {\small $\gamma$};

\node at (2.35*\x, 0.6*\x) {\small $\gamma$};
\node at (2.35*\x, -0.625*\x) {\small $\beta$};


\node at (2.4*\x, -1.325*\x) {\small $\alpha$};
\node at (2.4*\x, -2.65*\x) {\small $\alpha$};

\node at (1.6*\x, -1.325*\x) {\small $\alpha$};
\node at (1.6*\x, -2.65*\x) {\small $\alpha$};

\node at (1.6*\x, -3.45*\x) {\small $\beta$};

}

\node at (0, -1.325*\x) {\small $\alpha$};

\node[inner sep=1,draw,shape=circle] at (-\x,2*\x) {\footnotesize $1$};
\node[inner sep=1,draw,shape=circle] at (\x,2*\x) {\footnotesize $2$};
\node[inner sep=1,draw,shape=circle] at (\x,0) {\footnotesize $3$};
\node[inner sep=1,draw,shape=circle] at (-\x,0) {\footnotesize $4$};
\node[inner sep=1,draw,shape=circle] at (-3*\x,0) {\footnotesize $5$};
\node[inner sep=1,draw,shape=circle] at (3*\x,0) {\footnotesize $6$};
\node[inner sep=1,draw,shape=circle] at (-3*\x,-2*\x) {\footnotesize $7$};
\node[inner sep=1,draw,shape=circle] at (3*\x,-2*\x) {\footnotesize $8$};
\node[inner sep=1,draw,shape=circle] at (0,-2*\x) {\footnotesize $9$};
\node[inner sep=0.25,draw,shape=circle] at (0,-3.75*\x) {\footnotesize $10$};

\end{scope}
\end{tikzpicture}
\caption{$m=5$}
\label{Subfig-bede-alga2-al2be2-m=5}
\end{subfigure}
\begin{subfigure}[t]{0.3\linewidth}
\centering
\begin{tikzpicture}

\tikzmath{
\s=5;
\r=0.625;
\th=360/4;
\x=\r*cos(0.5*\th);
}

\begin{scope}[] 

\raisebox{3ex}{

\foreach \aaa in {0,1} {

\tikzset{yshift=\aaa*2*\x cm}

\foreach \aa in {-1,1} {

\tikzset{xscale=\aa}

\foreach \a in {0,1,2,3} {

\tikzset{xshift=\x cm, rotate=\a*\th}

\draw[]
	(0.5*\th:\r) -- (0.5*\th+\th:\r)
;

}
}
}

\foreach \aa in {-1,1} {

\tikzset{xscale=\aa}

\draw[]
	(2*\x,3*\x) -- (4*\x,3*\x)
;

\draw[line width=1.5]
	(0,-\x) -- (2*\x,-\x) 
	(2*\x,1*\x) -- (2*\x,-1*\x) 
;

\node at (0.35*\x, 2.65*\x) {\small $\alpha$};
\node at (1.65*\x, 2.65*\x) {\small $\alpha$};
\node at (0.35*\x, 1.35*\x) {\small $\alpha$};
\node at (1.65*\x, 1.35*\x) {\small $\alpha$};

\node at (0.35*\x, 0.6*\x) {\small $\beta$};
\node at (1.65*\x, 0.6*\x) {\small $\gamma$};
\node at (1.65*\x, -0.6*\x) {\small $\delta$};
\node at (0.35*\x, -0.6*\x) {\small $\gamma$};

\node at (2.35*\x, 2.55*\x) {\small $\beta$};
\node at (2.35*\x, 1*\x) {\small $\gamma$};
\node at (2.35*\x, -0.65*\x) {\small $\delta$};

\node at (2*\x, 3.4*\x) {\small $\alpha / \beta$};

}

\node[inner sep=1,draw,shape=circle] at (-\x,2*\x) {\footnotesize $7$};
\node[inner sep=1,draw,shape=circle] at (\x,2*\x) {\footnotesize $8$};
\node[inner sep=1,draw,shape=circle] at (\x,0) {\footnotesize $2$};
\node[inner sep=1,draw,shape=circle] at (-\x,0) {\footnotesize $1$};
\node[inner sep=1,draw,shape=circle] at (3*\x,\x) {\footnotesize $6$};
\node[inner sep=1,draw,shape=circle] at (-3*\x,\x) {\footnotesize $5$};

}

\end{scope}

\end{tikzpicture}
\caption{$m=4$}
\label{Subfig-bede-alga2-al2be2-m=4}
\end{subfigure}
\caption{The deductions of $\alpha^2\beta^2=\vert \alpha\vert \alpha \vert \beta \vert \beta \vert$}
\label{Fig-bede-alga2-al2be2}
\end{figure}
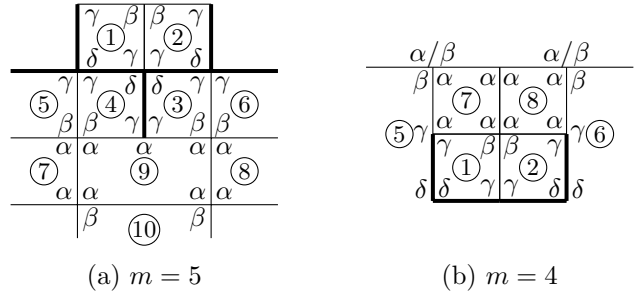
\end{subsubcase*}

\begin{subsubcase*}[$\alpha\gamma^2, \gamma^2\delta^{d\ge2}$ and $\beta_1\beta_2\cdots=\alpha^2\beta^3$] The vertex $\alpha^2\beta^3$ implies $\beta =\frac{2}{3}(\pi-\alpha)$. Recall $\beta+\delta > \alpha$ from the case, which then implies $2\pi+3\delta>5\alpha$. Then $\alpha>\frac{1}{2}\pi$ implies $3\delta>\alpha$. Comparing it with the vertex angle sum of $\alpha\gamma^2$ determines $\gamma^2\delta^{d\ge2}=\gamma^2\delta^2$. Recall $2\beta>\alpha$ from the subcase, $\alpha^2\beta^3$ and \eqref{Eq-mgon-sum} then imply $\tfrac{4}{7}\pi > \alpha > (1-\tfrac{2}{m})\pi$, and hence $m=4$.

For $m=4$, the vertex angle sums of $\alpha\gamma^2, \alpha^2\beta^3, \gamma^2\delta^2$ and \eqref{Eq-angle-id} imply 
\begin{align*}
& \alpha = (0.51664...)\pi, &
& \beta = (0.32223...)\pi, &
& \gamma = (0.74167...)\pi, &
& \delta = (0.25832...)\pi.&
\end{align*}
The angle values determine
\begin{align*}
\AVC = \{ \alpha\gamma^2, \gamma^2\delta^2,  \alpha^2\beta^3 \}.
\end{align*}
Resuming the deduction in Figure \ref{Fig-bede-alga2-AAD-ga2ded} with $\beta_1\beta_2\cdots=\alpha^2\beta^3$ and $\gamma_1\gamma_5\cdots, \gamma_2\gamma_6\cdots=\alpha\gamma^2$ uniquely determines the angle arrangement of $\beta_1\beta_2\cdots$ to be $\vert \alpha \vert \beta \vert \alpha \vert \beta \vert \beta \vert$. We illustrate the tiles, $T_7, T_8, T_9$, corresponding to $\alpha \vert \beta \vert \alpha$ in Figure \ref{Subfig-bede-AAD-albeal}. Next, $\alpha_7\gamma_8\cdots, \alpha_9\gamma_8\cdots=\alpha\gamma^2$ determine $T_{10}, T_{11}$, and hence $\delta_8\cdots=\delta^3\cdots$. It is however not admissible, a contradiction.


\begin{figure}[h!] 
\centering
\begin{subfigure}[t]{0.325\linewidth}
\centering
\begin{tikzpicture}

\raisebox{2.25ex}{
\begin{scope}[]

\tikzmath{
\r=0.85;
\gon=5;
\th=360/\gon;
\x=\r*cos(\th/2);
}

\foreach \a in {1,...,4}{

\draw[rotate=\th*\a]
	(0:0) -- (270:\r)
;
}

\foreach \a in {1,...,3}{
\draw[rotate=\th*\a]
	(270:\r) -- (270+0.5*\th:2*\x)
;
}

\foreach \a in {2,...,4}{
\draw[rotate=\th*\a]
	(270:\r) -- (270-0.5*\th:2*\x)
;
}

\draw[]
	(90-\th:2*\x) -- (90-0.5*\th:3*\x) 
	(90+\th:2*\x) -- (90+0.5*\th:3*\x) 
;

\draw[line width=1.5]
	(90:2*\x) -- (90-0.5*\th:3*\x) 
	(90:2*\x) -- (90+0.5*\th:3*\x) 
	(90-0.5*\th:\r) -- (90:2*\x)
	(90+0.5*\th:\r) -- (90:2*\x)
;

\node at (90:2.35*\x) {\small $?$};

\node at (135:1.1*\x) {\small $\alpha$}; 
\node at (188:1.1*\x) {\small $\alpha$};
\node at (166:0.3*\x) {\small $\alpha$};
\node at (162.5:1.65*\x) {\small $\alpha$};

\node at (90:0.3*\x) {\small $\beta$}; 
\node at (90:1.65*\x) {\small $\delta$};
\node at (65:1.1*\x) {\small $\gamma$};
\node at (115:1.1*\x) {\small $\gamma$};

\node at (16:0.3*\x) {\small $\alpha$}; 
\node at (18:1.65*\x) {\small $\alpha$};
\node at (45:1.1*\x) {\small $\alpha$}; 
\node at (352.5:1.1*\x) {\small $\alpha$};

\node at (152.5:1.9*\x) {\small $\beta$}; 
\node at (90+0.5*\th:1.2*\r) {\small $\gamma$};
\node at (90+0.5*\th:2.6*\x) {\small $\gamma$};
\node at (102:1.85*\x) {\small $\delta$};

\node at (27.5:1.9*\x) {\small $\beta$}; 
\node at (90-0.5*\th:1.2*\r) {\small $\gamma$};
\node at (90-0.5*\th:2.6*\x) {\small $\gamma$};
\node at (78:1.85*\x) {\small $\delta$};

\node at (275:0.25*\r) {\small $\cdots$};

\node[inner sep=1,draw,shape=circle] at (90+\th:1*\x) {\footnotesize $7$};
\node[inner sep=1,draw,shape=circle] at (90:1*\x) {\footnotesize $8$};
\node[inner sep=1,draw,shape=circle] at (90-\th:1*\x) {\footnotesize $9$};
\node[inner sep=0.25,draw,shape=circle] at (90+0.5*\th:1.6*\r) {\tiny $10$};
\node[inner sep=0.25,draw,shape=circle] at (90-0.5*\th:1.6*\r) {\tiny $11$};

\end{scope}
}
\end{tikzpicture}
\caption{$\alpha\vert\beta\vert\alpha$}
\label{Subfig-bede-AAD-albeal}
\end{subfigure}
\begin{subfigure}[t]{0.325\linewidth}
\centering
\begin{tikzpicture}

\begin{scope}[]

\tikzmath{
\r=0.625;
\th=360/4;
\x=\r*cos(0.5*\th);
}

\foreach \aaa in {0,1} {

\tikzset{yshift=\aaa*2*\x cm}

\foreach \aa in {-1,1} {

\tikzset{xscale=\aa}

\foreach \a in {0,1,2,3} {

\tikzset{xshift=\x cm, rotate=\a*\th}

\draw[]
	(0.5*\th:\r) -- (0.5*\th+\th:\r)
;

}
}
}

\draw[]
	(0,-\x) -- (0,-3*\x)
;

\draw[line width=1.5]
	(0,-\x) -- (0,\x)
;

\foreach \aa in {-1,1} {

\tikzset{xscale=\aa}

\draw[]
	(2*\x,1*\x) -- (4*\x,1*\x)
	(2*\x,-1*\x) -- (4*\x,-1*\x)
	(0,-3*\x) -- (2*\x,-3*\x)
;

\draw[line width=1.5]
	(4*\x,1*\x) -- (4*\x,-1*\x)
	(2*\x,3*\x) -- (2*\x,1*\x) 
	(2*\x,-1*\x) -- (2*\x,-3*\x)
;

\draw[line width=1.5]
	(0,\x) -- (2*\x,\x) 
	(2*\x,1*\x) -- (4*\x,1*\x)
	(0,-3*\x) -- (2*\x,-3*\x)
;

\node at (0.4*\x, 2.6*\x) {\small $\beta$};
\node at (1.6*\x, 2.6*\x) {\small $\gamma$};
\node at (0.4*\x, 1.4*\x) {\small $\gamma$};
\node at (1.6*\x, 1.4*\x) {\small $\delta$};

\node at (0.4*\x, 0.6*\x) {\small $\delta$};
\node at (1.6*\x, 0.6*\x) {\small $\gamma$};
\node at (1.6*\x, -0.6*\x) {\small $\beta$};
\node at (0.4*\x, -0.6*\x) {\small $\gamma$};

\node at (2.4*\x, 0.6*\x) {\small $\gamma$};
\node at (2.4*\x, -0.6*\x) {\small $\beta$};

\node at (0.4*\x, -1.4*\x) {\small $\beta$};
\node at (1.6*\x, -1.4*\x) {\small $\gamma$};
\node at (1.6*\x, -2.6*\x) {\small $\delta$};
\node at (0.4*\x, -2.6*\x) {\small $\gamma$};

}

\node[inner sep=0.25,draw,shape=circle] at (-\x,2*\x) {\tiny $11$};
\node[inner sep=0.25,draw,shape=circle] at (\x,2*\x) {\tiny $12$};
\node[inner sep=1,draw,shape=circle] at (\x,0) {\footnotesize $8$};
\node[inner sep=1,draw,shape=circle] at (-\x,0) {\footnotesize $7$};
\node[inner sep=0.25,draw,shape=circle] at (3*\x,0) {\tiny $10$};
\node[inner sep=1,draw,shape=circle] at (-3*\x,0) {\footnotesize $9$};

\node[inner sep=1,draw,shape=circle] at (\x,-2*\x) {\footnotesize $2$};
\node[inner sep=1,draw,shape=circle] at (-\x,-2*\x) {\footnotesize $1$};

\end{scope}

\end{tikzpicture}
\caption{$\beta^2\gamma^2$}
\label{Subfig-bede-AAD-be2ga2}
\end{subfigure}
\caption{The deductions of $\alpha \vert \beta \vert \alpha$ and $\beta^2\gamma^2$}
\label{Fig-bede-AAD-alga2-ga2de2-al2be3-be2ga2}
\end{figure}
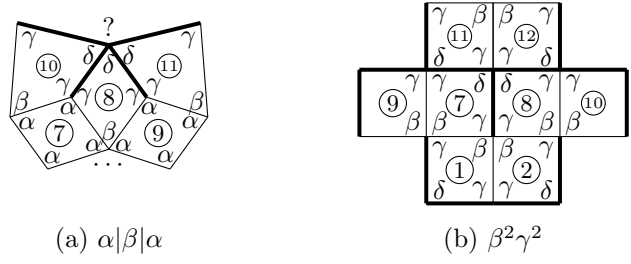
\end{subsubcase*}

\begin{subsubcase*}[$\alpha\gamma^2, \gamma^2\delta^{d\ge2}$ and $\beta_1\beta_2\cdots=\beta^2\gamma^2$] Recall $\alpha\gamma\cdots=\alpha\gamma^2$ from the case, which implies no $\alpha$ in $\beta\gamma\cdots$. Given $\beta^2\gamma^2$ and \eqref{Eq-vertex-be}, we deduce $\beta\gamma\cdots=\beta^2\gamma^2$. 

Resuming the deduction of Figure \ref{Fig-bede-alga2-AAD-ga2ded} with $\beta_1\beta_2\cdots=\beta^2\gamma^2$ determines $T_7, T_8$ in Figure \ref{Subfig-bede-AAD-be2ga2}. Then $\beta_7\gamma_1\cdots, \beta_8\gamma_2\cdots = \beta^2\gamma^2$ determine the angles in $T_9, T_{10}$. Since $\gamma\vert\gamma\cdots= \gamma^2\delta^{d\ge2}$, the vertices $\gamma_7\vert\gamma_9\cdots$, $\gamma_8\vert\gamma_{10}\cdots=\gamma^2\delta^d$ determine $T_{11}, T_{12}$. This results in $\vert \gamma_{11} \, \bvert \, \delta_7 \, \bvert \, \delta_8 \, \bvert \, \gamma_{12} \vert \cdots$, which is $\gamma^2\delta^{d\ge2}$ and can only be $\gamma^2\delta^2$ as shown. However, $\beta^2\gamma^2, \gamma^2\delta^2$ imply $\beta=\delta$, a contradiction. The deduction is independent of $m$ and hence there is no tiling in this subsubcase.
\end{subsubcase*}
\end{subcase*}

\begin{subcase*}[$\alpha\gamma^2$ and $\nexists \, \gamma^4, \gamma^2\delta^d$] Recall $\alpha\gamma\cdots=\alpha\gamma^2$ and $\gamma\delta\cdots=\gamma^2\delta^d$ and $\gamma\vert\gamma\cdots = \gamma^4, \gamma^2\delta^d$ from the case. The subcase assumption rules out $\gamma\delta\cdots, \gamma\vert\gamma \cdots$. The former reduces \eqref{Eq-vertex-de} to
\begin{align}\label{Eq-bede-alga2-vertex-de}
\delta\cdots=\delta^{d\ge3}.
\end{align}

In the absence of $\gamma\vert\gamma \cdots$, Lemma \ref{Lem-bebe-gaga-dede} rules out $\beta\vert\beta\cdots$. Hence $\beta^b,\alpha\beta^b$ are not vertices. Recall $\beta^{b}\gamma^{c}=\beta^{b\ge2}\gamma^{2}$ from the case. The same argument also rules out $\beta^{b}\gamma^{c}$. Now $\eqref{Eq-vertex-be}$ becomes $\beta\cdots=\alpha^a\beta^b$. Then $\alpha\gamma\cdots=\alpha\gamma^2$ and $\alpha>\frac{1}{2}\pi$ and no $\beta\vert\beta\cdots$ deduce
\begin{align}\label{Eq-bede-alga2-vertex-be}
\beta\cdots=\alpha^2\beta, \alpha^2\beta^2, \alpha^3\beta, \alpha^3\beta^2, \alpha^3\beta^3. 
\end{align}
Furthermore, $\alpha\gamma\cdots=\alpha\gamma^2$ and the absence of $\gamma\delta\cdots, \gamma^4\cdots$ reduce \eqref{Eq-vertex-ga} to
\begin{align}\label{Eq-bede-alga2-vertex-ga}
\gamma\cdots &= \alpha\gamma^2. 
\end{align}
The above knowledge of $\beta\cdots, \gamma\cdots, \delta\cdots$ and $\alpha>\frac{1}{2}\pi$ reduce \eqref{Eq-vertex-al} to
\begin{align}\label{Eq-bede-alga2-vertex-al}
\alpha\cdots=\alpha^3, \alpha\gamma^2, \alpha^2\beta, \alpha^2\beta^2, \alpha^3\beta, \alpha^3\beta^2, \alpha^3\beta^3. 
\end{align}
In particular, $\alpha\beta\cdots=\alpha^2\beta, \alpha^2\beta^2, \alpha^3\beta, \alpha^3\beta^2, \alpha^3\beta^3$. Among them, the pair $\alpha^2\beta^2, \alpha^3\beta$ imply $\beta=\alpha=\frac{1}{2}\pi$, contradicting $\alpha>\frac{1}{2}\pi$. The other pairs are obviously mutually exclusive. We proceed with exactly one of them being a vertex of type $\alpha\beta\cdots$ as required by Lemma \ref{Lem-albe-alga}. 

\begin{subsubcase*}[$\gamma\cdots = \alpha\gamma^2$ and $\beta\cdots=\alpha^2\beta$] The vertex angle sum of $\alpha^2\beta$ and $\frac{2}{3}\pi \ge \beta$ imply $\pi>\alpha\ge\frac{2}{3}\pi\ge\beta$. Then $\alpha\ge\frac{2}{3}\pi$ implies $\alpha^3\cdots=\alpha^3$. 

Given $\alpha\beta\cdots=\alpha^2\beta$ and $\alpha\gamma\cdots=\alpha\gamma^2$ and their unique angle arrangements (see Figure \ref{Subfig-bede-alga2-al2be-mgon-bdry}), Lemma \ref{Lem-3-div-m} implies that $m$ is a multiple of $3$.

\begin{figure}[h!] 
\centering
\begin{subfigure}[t]{0.375\linewidth}
\centering
\begin{tikzpicture}
\tikzmath{
\XS=1;
\Y=1;
}

\foreach \xs in {0,1,2} {
\tikzset{xshift=\xs*\XS cm}
\draw[]
	(0,0) -- (\XS,0)
;
}

\foreach \aa in {-1,1} {
\tikzset{xshift=1.5*\XS cm, xscale=\aa}
\draw[]
	(0.5*\XS,0) -- (0.5*\XS,\Y)
;
\draw[line width=1.5]
	(1.5*\XS,0) -- (1.5*\XS,\Y)
;

\node at (0.3*\XS,0.2*\Y) {\small $\alpha$};
\node at (0.7*\XS,0.2*\Y) {\small $\beta$};
\node at (1.3*\XS,0.2*\Y) {\small $\gamma$};
\node at (1.7*\XS,0.2*\Y) {\small $\gamma$};

\node at (0.5*\XS,-0.2*\Y) {\small $\alpha$};
\node at (1.5*\XS,-0.2*\Y) {\small $\alpha$};
}

\end{tikzpicture}
\caption{Boundary vertices of an $m$-gon}
\label{Subfig-bede-alga2-al2be-mgon-bdry}
\end{subfigure}
\begin{subfigure}[t]{0.55\linewidth}
\centering
\begin{tikzpicture}
\tikzmath{
\r=0.8;
\xs=2;
\ys=1.75;
}

\begin{scope} 

\begin{scope}[yshift=0.25 cm]
\tikzmath{
\th=360/3;
\x=\r*cos(\th/2);
\xx=0.75*\x;
}

\raisebox{-0.15cm}{
\foreach \a in {0,...,2} {
\tikzset{rotate=\a*\th}

\foreach \aa in {-1,1} {
\tikzset{xscale=\aa}
\draw[]
	(90:\xx) -- (90-0.5*\th:\r)
;
}

\draw[line width=1.5]
	(0,0) -- (90:\xx)
;

\node at (270:0.6*\r) {\tiny $\beta$};
\node at (90-0.5*\th: 0.2*\r) {\tiny $\delta$};
}

}

\end{scope}

\begin{scope}[xshift=\xs cm]
\tikzmath{
\ph=360/4;
\x=\r*cos(0.5*\ph);
\xx=0.75*\x;
}

\foreach \a in {0,...,3} {
\tikzset{rotate=\a*\ph}

\foreach \aa in {-1,1}{
\tikzset{xscale=\aa}
\draw[]
	(90:\xx) -- (90-0.5*\ph:\r)
;
}

\draw[line width=1.5]
	(0,0) -- (90:\xx)
;

\node at (45:0.7*\r) {\tiny $\beta$};
\node at (90-0.5*\ph: 0.25*\r) {\tiny $\delta$};
}

\end{scope}

\begin{scope}[xshift=2*\xs cm]
\tikzmath{
\ps=360/5;
\x = \r*cos(0.5*\ps);
\xx=0.75*\x;
}

\foreach \a in {0,...,4} {
\tikzset{rotate=\a*\ps}

\foreach \aa in {-1,1}{
\tikzset{xscale=\aa}
\draw[]
	(90:\xx) -- (90-0.5*\ps:\r)
;
}

\draw[line width=1.5]
	(0,0) -- (90:\xx)
;

\node at (90-0.5*\ps:0.7*\r) {\tiny $\beta$};
\node at (90-0.5*\ps: 0.25*\r) {\tiny $\delta$};
}

\end{scope}

\end{scope}

\begin{scope}[yshift=-\ys cm] 

\begin{scope}[yshift=0.25 cm]
\tikzmath{
\th=360/3;
\x=\r*cos(\th/2);
\xx=0.75*\x;
}

\raisebox{-0.15cm}{
\foreach \a in {0,...,2} {
\tikzset{rotate=\a*\th}

\draw[dashed]
	(270:\r) -- (270+\th:\r)
;

\node at (270:0.45*\r) {\tiny $\bar{\beta}$};
}
}
\end{scope}

\begin{scope}[xshift=\xs cm]
\tikzmath{
\ph=360/4;
\x=\r*cos(0.5*\ph);
\xx=0.75*\x;
}

\foreach \a in {0,...,3} {
\tikzset{rotate=\a*\ph}

\draw[dashed]
	(\x,\x) -- (-\x,\x)
;

\node at (45:0.6*\r) {\tiny $\bar{\beta}$};
}

\end{scope}

\begin{scope}[xshift=2*\xs cm]
\tikzmath{
\ps=360/5;
\x = \r*cos(0.5*\ps);
\xx=0.75*\x;
}

\foreach \a in {0,...,4} {
\tikzset{rotate=\a*\ps}

\draw[dashed]
	(270:\r) -- (270+\ps:\r)
;

\node at (90-0.5*\ps:0.6*\r) {\tiny $\bar{\beta}$};
}

\end{scope}

\end{scope}

\end{tikzpicture}
\caption{Neighbourhood of $\delta^d$ for $d=3,4,5$}
\label{Subfig-bede-alga2-al2be-2d-gon}
\end{subfigure}
\caption{The boundary of an $m$-gon and the boundary of an equilateral $2d$-gon given by kite subdivision of the regular polygon with angles $\bar{\beta}^d$ for $d=3,4,5$}
\label{Fig-Subdiv-Prototiles-tr-sq-pn}
\end{figure}
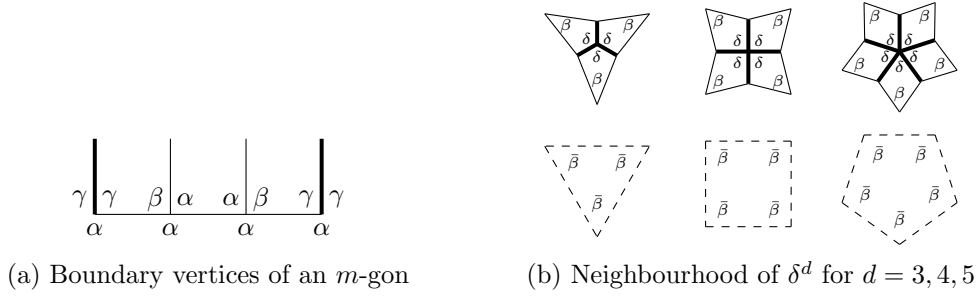

If $\alpha=\frac{2}{3}\pi$, then \eqref{Eq-mgon-sum} implies $m=4,5$, a contradiction. Hence $\alpha>\frac{2}{3}\pi$ which rules out $\alpha^3\cdots$. Now \eqref{Eq-bede-alga2-vertex-al} and \eqref{Eq-bede-alga2-vertex-be} become $\alpha\cdots=\beta\cdots=\alpha^2\beta$. Together with \eqref{Eq-bede-alga2-vertex-de}, \eqref{Eq-bede-alga2-vertex-ga}, we obtain
\begin{align}\label{Eq-AVC-al2be-alga2-ded}
\AVC = \{ \alpha^2\beta, \alpha\gamma^2, \delta^{d\ge3} \},
\end{align}
for $m \in 3\mathbb{N}$ and $m\ge6$.

We determine the tilings by determining their primal tilings. The diminished neighbourhood given by the incident kites of a $\delta^{d\ge3}$ is an equilateral $2d$-gon, centred at $\delta^{d\ge3}$ as shown in Figure \ref{Subfig-bede-alga2-al2be-2d-gon}. The $2d$-gon can be deformed into a regular $d$-gon by smoothing the $2$-path between two $\alpha^2\beta$'s (see dotted lines in Figure \ref{Subfig-bede-alga2-al2be-2d-gon}). The resulting regular $d$-gon has angles $\bar{\beta}$, while the $m$-gon becomes an $\bar{m}$-gon with angles $\bar{\alpha}$. Since every other side of an $\bar{m}$-gon is to be divided by a $y$-edge at an $\alpha\gamma^2$ as shown in Figure \ref{Subfig-bede-alga2-al2be-mgon-bdry}, we know $\bar{m}\in 2\mathbb{N}$ and hence $m=\tfrac{3}{2}\bar{m}$. Now $\AVC$ \eqref{Eq-AVC-al2be-alga2-ded} becomes $\{ \bar{\alpha}^2\bar{\beta} \}$ and the corresponding primal tilings are vertex homogeneous. Note that it is possible that $\bar{m}=d$ and $\bar{\alpha}=\bar{\beta}$.

By \cite[Proposition 2.9]{lnp} and $\AVC$ \eqref{Eq-AVC-al2be-alga2-ded} and $\bar{m}\in 2\mathbb{N}$, the vertex homogeneous candidates for the primal tilings are given by the cube with $(\bar{m},d)=(4,4)$, the truncated tetrahedron $t\mathcal{T}$ with $(\bar{m},d)=(6,3)$, the truncated cube $t\mathcal{C}$ with $(\bar{m},d)=(8,3)$, the truncated octahedron $t\mathcal{O}$ with $(\bar{m},d)=(6,4)$, the truncated dodecahedron $t\mathcal{D}$ with $(\bar{m},d)=(10,3)$, and the truncated icosahedron $t\mathcal{I}$ with $(\bar{m},d)=(6,5)$.

\begin{figure}[h!]
\centering
\begin{tikzpicture}[>=latex]
\tikzmath{
\r=0.4;
\n=4;
\nn=\n-1;
\th=360/\n;
\y=\r*sin(45);
\XS=4;
}

\begin{scope}[xshift=-0.1*\XS cm] 

\foreach \a in {0,...,\nn} {
\tikzset{rotate=\a*\th}
\draw[]
	(90-0.5*\th:1*\r) -- (0,0.8*\y)
	(90+0.5*\th:1*\r) -- (0,0.8*\y)
	(90-0.5*\th:1*\r) -- (90-0.5*\th:2.25*\r)
	(90-0.5*\th:2.25*\r) -- (0,2.5*\y)
	(90+0.5*\th:2.25*\r) -- (0,2.5*\y)
;
\draw[line width=1.5]
	(0,0) -- (0,0.8*\y)
;
\draw[arrows = {-Latex[scale=0.5]},line width=1.5]
	(0,2.5*\y) -- (90:4*\y)
;
}

\end{scope}

\draw[->]
	(0.4*\XS, 0) -- (0.6*\XS, 0)
;

\node at (0.5*\XS, 1*\r) {\scriptsize diminishing};

\begin{scope}[xshift=\XS cm] 

\foreach \a in {0,...,\nn} {
\tikzset{rotate=\a*\th}
\draw[]
	(90-0.5*\th:1*\r) -- (0,0.8*\y)
	(90+0.5*\th:1*\r) -- (0,0.8*\y)
	(90-0.5*\th:1*\r) -- (90-0.5*\th:2.25*\r)
	(90-0.5*\th:2.25*\r) -- (0,2.5*\y)
	(90+0.5*\th:2.25*\r) -- (0,2.5*\y)
;
}

\end{scope}

\node at (1.5*\XS, 1*\r) {\scriptsize smoothing};

\draw[->]
	(1.4*\XS, 0) -- (1.6*\XS, 0)
;

\begin{scope}[xshift=2*\XS cm] 

\foreach \a in {0,...,\nn} {
\tikzset{rotate=\a*\th}
\draw[dashed]
	(90-0.5*\th:1*\r) -- (90+0.5*\th:1*\r)
	(90-0.5*\th:1*\r) -- (90-0.5*\th:2.25*\r) 
	(90-0.5*\th:2.25*\r) -- (90+0.5*\th:2.25*\r)
;

}

\end{scope}

\end{tikzpicture}
\caption{\ref{Label:polar-subdiv-cube} with primal tiling determined by $(\bar{m},d)=(4,4)$ and $\{ \bar{\alpha}^2\bar{\beta} \}$ where $\bar{\alpha}=\bar{\beta}$}
\label{Fig-bede-alga2-al2be-ded-Platonic-primal}
\end{figure}
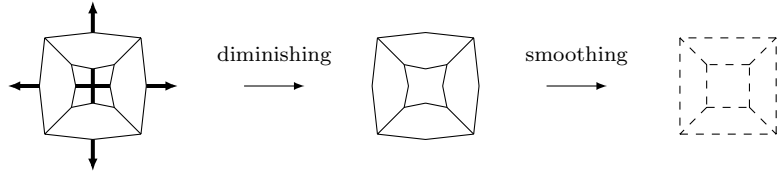

The tiling with primal tiling given by the cube is \ref{Label:polar-subdiv-cube} as shown in Figure \ref{Fig-bede-alga2-al2be-ded-Platonic-primal}. By $\AVC$ \eqref{Eq-AVC-al2be-alga2-ded}, we get $\beta=2\pi-2\alpha$ and $\gamma=\pi - \tfrac{1}{2}\alpha$ and $\delta=\tfrac{1}{2}\pi$. For $(m,d)=(6,4)$, substituting the equations into \eqref{Eq-angle-id}, \eqref{Eq-a4-cosx}, and \eqref{Eq-x2y2-cosy} determines 
\begin{align} 
&K\mathcal{C}:& &\alpha=\cos^{-1}(-\tfrac{3}{4}),& &\beta=2\pi-2\alpha, \quad \gamma=\pi - \tfrac{1}{2}\alpha, \quad \delta=\tfrac{1}{2}\pi, & \\ \notag
&&&x=\cos^{-1} \tfrac{5}{7},& &y=\cos^{-1} \tfrac{2}{\sqrt{7}}. 
\end{align}

What remain are $t\mathcal{T}, t\mathcal{C}, t\mathcal{O}, t\mathcal{D}$, $t\mathcal{I}$. We obtain the tilings via kite subdivisions and they are the \ref{Label:trunc-Platonic} tilings in Figure \ref{Fig-kite-subdiv-trunc-Platonic}. By $m=\tfrac{3}{2}\bar{m}$, we get $(m,d)=(9,3), (12,3), (9,4), (15,3)$ and $(9,5)$ in the order of the aforementioned Archimedean solids (see Table \ref{Table-md}).

\begin{table}[h!]
\centering
\def\arraystretch{1.5}
    \begin{tabular}[]{| c | c | c | c | c | c |} 
    \hline
    $(m,d)$ & $(9,3)$ & $(12,3)$ & $(9,4)$ & $(15,3)$ & $(9,5)$ \\
    \hline
    Tiling & $Kt\mathcal{T}$ & $Kt\mathcal{C}$ & $Kt\mathcal{O}$ & $Kt\mathcal{D}$ & $Kt\mathcal{I}$ \\
    \hline
    \end{tabular}
\caption{}
\label{Table-md}
\end{table}

By $\AVC$ \eqref{Eq-AVC-al2be-alga2-ded}, we get $\beta=2\pi-2\alpha$ and $\gamma=\pi - \tfrac{1}{2}\alpha$ and $\delta=\tfrac{2}{d}\pi$. Substituting the equations into \eqref{Eq-angle-id}, \eqref{Eq-a4-cosx}, and \eqref{Eq-x2y2-cosy} determines
\begin{align}\label{Eq-al2be-alga2-ded} 
&Kt\mathcal{T}:& &\alpha =2\cos^{-1}\tfrac{1}{2+4\cos\frac{2}{9}\pi},& &\beta=2\pi-2\alpha, \quad \gamma=\pi - \tfrac{1}{2}\alpha, \quad \delta=\tfrac{2}{3}\pi,& \\ \notag
&&&x = (0.183899...)\pi,& &y=(0.078507...)\pi; \\
&Kt\mathcal{C}:& &\alpha =\cos^{-1}(-\tfrac{1}{4}(2+\sqrt3)),& &\beta=2\pi-2\alpha, \quad \gamma=\pi - \tfrac{1}{2}\alpha, \quad \delta=\tfrac{2}{3}\pi,& \\ \notag
&&&x=2\tan^{-1}(\tfrac{1}{\sqrt2}(2-\sqrt3)),& &y=\tan^{-1}(\tfrac{\sqrt2}{11}(3\sqrt3-4));\\
&Kt\mathcal{O}:& &\alpha =\cos^{-1}(-\tfrac{1}{3}(1+2\cos\tfrac{2}{9}\pi)),& &\beta=2\pi-2\alpha, \quad \gamma=\pi - \tfrac{1}{2}\alpha, \quad \delta=\tfrac{1}{2}\pi,& \\ \notag
&&&x = 2\tan^{-1}(\tfrac{1}{\sqrt3}\tan\tfrac{1}{9}\pi),& &y=(0.098747...)\pi;\\
&Kt\mathcal{D}:& &\alpha =\cos^{-1}\tfrac{-5+\tfrac{1}{\sqrt5}-\textstyle\sqrt{6(1+\frac{1}{\sqrt5})}}{8},& &\beta=2\pi-2\alpha, \quad \gamma=\pi - \tfrac{1}{2}\alpha, \quad \delta=\tfrac{2}{3}\pi,& \\ \notag
&&&x =(0.222571...)\pi,& &y=(0.028283...)\pi; \\
&Kt\mathcal{I}:& &\alpha =2\cos^{-1}\tfrac{1+\sqrt5}{4+8\cos\frac{2}{9}\pi},& &\beta=2\pi-2\alpha, \quad \gamma=\pi - \tfrac{1}{2}\alpha, \quad \delta=\tfrac{2}{5}\pi,& \\ \notag
&&&x = (0.082219...)\pi,& &y=(0.084759...)\pi.
\end{align}
We remark that the exact formulae for $x$ and $y$ in some of them are too lengthy to be listed. The formulae can be obtained by a computer algebra system.
\end{subsubcase*}

\begin{subsubcase*}[$\gamma\cdots = \alpha\gamma^2$ and $\beta\cdots=\alpha^3\beta^2$] 
The assumption reduces \eqref{Eq-bede-alga2-vertex-al} to $\alpha\cdots=\alpha\gamma^2, \alpha^3\beta^2$. In particular, $\alpha^2\cdots=\alpha^3\beta^2$, which implies $2\pi-3\alpha=2\beta$. With $\beta+\delta>\alpha$ from the case, it implies $\delta > \frac{5}{2}\alpha-\pi$. Meanwhile, $2\pi-3\alpha=2\beta>0$ and $\alpha>(1-\frac{2}{m})\pi$ deduce $m=4,5$. 

For $m=5$, the inequality \eqref{Eq-mgon-sum} implies $\alpha>\frac{3}{5}\pi$, and then $\delta > \frac{5}{2}\alpha-\pi$ implies $\delta > \frac{1}{2}\pi$. Hence \eqref{Eq-bede-alga2-vertex-de} becomes $\delta\cdots=\delta^3$. However, the vertex angle sums of $\alpha\gamma^2, \alpha^3\beta^2, \delta^3$ and $m=5$ yield no solution in \eqref{Eq-angle-id} for $\alpha \in (\tfrac{3}{5}\pi,\tfrac{2}{3}\pi)$. Therefore there is no tiling. 

For $m=4$, without $\beta\vert\beta\cdots$ the vertex $\alpha^3\beta^2$ has a unique angle arrangement $\vert \alpha\vert \alpha \vert \beta \vert \alpha \vert \beta \vert$ which determines tiles $T_1, ..., T_5$ in Figure \ref{Subfig-bede-alga2-al3be2-AAD-alalbealbe}. The top $\alpha_1\alpha_2\cdots=\alpha^3\beta^2$ determines the angles in $T_6, T_7, T_8$; while $\alpha_2\gamma_3\cdots=\alpha\gamma^2$ determines the angles in $T_9$. It results in $\alpha_2\beta_9\gamma_7\cdots$, which is not a vertex, a contradiction. Therefore there is no tiling. 

\begin{figure}[h!] 
\centering
\begin{subfigure}[t]{0.3\linewidth}
\centering
\begin{tikzpicture}

\tikzmath{
\s=4;
\r=0.85;
}

\begin{scope}

\tikzmath{
\ph=360/5;
\x=\r*cos(\ph/2);
}

\foreach \a in {0,...,4} {

\tikzset{rotate=\a*\ph}

\draw[]
	(0,0) -- (90:\r) 
	(90:\r) -- (90-0.5*\ph:2*\x)
	(90:\r) -- (90+0.5*\ph:2*\x)
;

}

\foreach \a in {-1,1} {

\tikzset{rotate=\a*\ph}

\draw[line width=1.5]
	(270:2*\x) -- (270-0.5*\ph:\r) 
	(270:2*\x) -- (270+0.5*\ph:\r) 
;

}

\draw[]
	(90:\r) -- (90-0.2*\ph:2*\r)
	(90:\r) -- (90+0.2*\ph:2*\r)
	(90-0.5*\ph:2*\x) -- (90-0.6*\ph:2*\r)
;

\draw[line width=1.5]
	(90-0.5*\ph:2*\x) -- (90-0.45*\ph:2.2*\r)
;

\foreach \a in {0,2,3} {

\tikzset{rotate=\a*\ph}

\node at (270:0.3*\r) {\small $\alpha$};
\node at (270:1.3*\r) {\small $\alpha$};
\node at (295:0.85*\r) {\small $\alpha$};
\node at (245:0.85*\r) {\small $\alpha$};

}

\foreach \a in {-1,1} {

\tikzset{rotate=\a*\ph}

\node at (270:0.35*\r) {\small $\beta$};
\node at (270:1.3*\r) {\small $\delta$};
\node at (295:0.85*\r) {\small $\gamma$};
\node at (245:0.85*\r) {\small $\gamma$};

}

\node at (90:1.35*\r) {\small $\alpha$};
\node at (90-0.215*\ph:1.3*\r) {\small $\beta$};
\node at (90+0.215*\ph:1.3*\r) {\small $\beta$};

\node at (90-0.4*\ph:2.05*\x) {\small $\gamma$};
\node at (90-0.625*\ph:2*\x) {\small $\beta$};
\node at (90-\ph:1.15*\r) {\small $\gamma$};

\node at (90-0.525*\ph:2.65*\x) {\small $?$};

\node[inner sep=1,draw,shape=circle] at (90+0.5*\ph:\x) {\footnotesize $1$};
\node[inner sep=1,draw,shape=circle] at (90-0.5*\ph:\x) {\footnotesize $2$};
\node[inner sep=1,draw,shape=circle] at (90-1.5*\ph:\x) {\footnotesize $3$};
\node[inner sep=1,draw,shape=circle] at (90-2.5*\ph:\x) {\footnotesize $4$};
\node[inner sep=1,draw,shape=circle] at (90+1.5*\ph:\x) {\footnotesize $5$};
\node[inner sep=1,draw,shape=circle] at (90-1*\ph:2*\x) {\footnotesize $9$};

\node[inner sep=1,draw,shape=circle] at (90:2.25*\x) {\footnotesize $6$};
\node[inner sep=1,draw,shape=circle] at (90-0.32*\ph:2.4*\x) {\footnotesize $7$};
\node[inner sep=1,draw,shape=circle] at (90+0.32*\ph:2.4*\x) {\footnotesize $8$};

\end{scope}
\end{tikzpicture}
\caption{}
\label{Subfig-bede-alga2-al3be2-AAD-alalbealbe}
\end{subfigure}
\begin{subfigure}[t]{0.3\linewidth}
\centering
\begin{tikzpicture}

\tikzmath{
\r=0.85;
}

\raisebox{4ex}{
\begin{scope}[]

\tikzmath{
\r=0.75;
\ps=360/6;
}

\foreach \p in {2,3,4,5} {

\draw[rotate=\p*\ps]
	(90-1*\ps:\r) -- (90:\r)
;

}

\foreach \p in {-1,0,1} {

\draw[rotate=\p*\ps]
	(270:\r) -- (270:2*\r)
;

\node at (282.5+\p*\ps:1.15*\r) {\small $\alpha$};
\node at (257.5+\p*\ps:1.15*\r) {\small $\alpha$};

}

\node at (270:0.7*\r) {\small $\alpha$};
\node at (270-\ps:0.7*\r) {\small $\alpha$};
\node at (270+\ps:0.7*\r) {\small $\alpha$};

\node[inner sep=1,draw,shape=circle] at (0,0) {\footnotesize $1$};
\node[inner sep=1,draw,shape=circle] at (270+0.5*\ps:1.75*\r) {\footnotesize $2$};
\node[inner sep=1,draw,shape=circle] at (270-0.5*\ps:1.75*\r) {\footnotesize $3$};

\end{scope}
}

\end{tikzpicture}
\caption{}
\label{Subfig-bede-alga2-al3be2-AAD-alal}
\end{subfigure}
\begin{subfigure}[t]{0.3\linewidth}
\centering
\begin{tikzpicture}

\tikzmath{
\r=0.85;
}

\begin{scope}[yscale=-1]
\tikzmath{
\th=360/3;
\x=\r*cos(0.5*\th);
}

\foreach \a in {0,1,2} {

\draw[rotate=\th*\a]
	(0:0) -- (90:\r) 
	(90:\r) -- (90-0.5*\th:2*\x)
	(90:\r) -- (90+0.5*\th:2*\x)
	(90-0.5*\th:2*\x) -- (90-0.5*\th:3*\x)
;

\node at (270+\th*\a:0.25*\r) {\small $\alpha$};
\node at (270+\th*\a:0.75*\r) {\small $\alpha$};
\node at (270+0.4*\th+\th*\a:0.75*\r) {\small $\alpha$};
\node at (270-0.4*\th+\th*\a:0.75*\r) {\small $\alpha$};

\node at (90+\th*\a:1.15*\r) {\small $\alpha$};
\node at (90+\th*\a:1.55*\r) {\small $\alpha$};
\node at (40+\th*\a:1.1*\r) {\small $\alpha$};
\node at (140+\th*\a:1.1*\r) {\small $\alpha$};

}

\end{scope}

\end{tikzpicture}
\caption{}
\label{Subfig-bede-alga2-al3be2-AAD-alal-cube}
\end{subfigure}
\caption{The deductions of $\alpha^3\beta^2=\vert \alpha\vert \alpha \vert \beta \vert \alpha \vert \beta \vert$ and $\alpha\vert\alpha\cdots=\alpha^3$}
\label{Fig-bede-alga2-al3be2-AAD-alalbealbe-AAD-alal}
\end{figure}
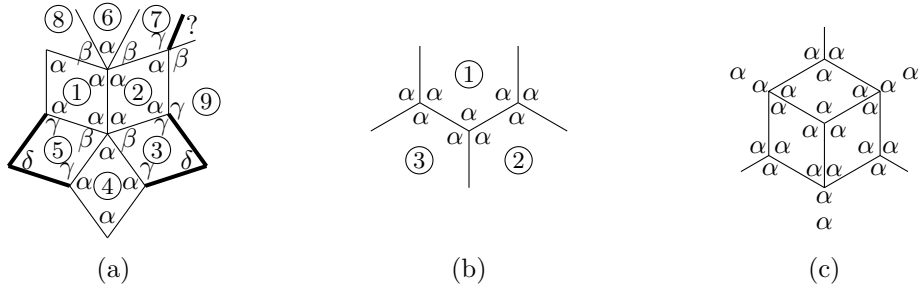
\end{subsubcase*}

\begin{subsubcase*}[$\gamma\cdots = \alpha\gamma^2$ and $\beta\cdots=\alpha^2\beta^2$] 
The assumption reduces \eqref{Eq-bede-alga2-vertex-al} to $\alpha\cdots=\alpha^3, \alpha\gamma^2, \alpha^2\beta^2$ and $\beta\cdots=\alpha^2\beta^2$. The absence of $\beta\vert\beta\cdots$ from the subcase implies that $\alpha^2\beta^2$ has a unique angle arrangement $\vert \alpha \vert \beta \vert \alpha \vert \beta \vert$. Then $\alpha\vert\alpha\cdots$ can only be $\alpha^3$. The three incident tiles at an $\alpha^3$ are illustrated in Figure \ref{Subfig-bede-alga2-al3be2-AAD-alal}. Each adjacent vertex $\alpha\vert\alpha\cdots$ is also $\alpha^3$. The deduction repeats at each $\alpha\vert\alpha\cdots$ and results in a monohedral tiling by regular $m$-gons. The tiling is either given by the cube (Figure \ref{Subfig-bede-alga2-al3be2-AAD-alal-cube}) or by the regular dodecahedron. 

Excluding $\alpha^3$ we now have $\alpha\cdots= \alpha\gamma^2, \alpha^2\beta^2$. With subsubcase assumption and \eqref{Eq-bede-alga2-vertex-de},
\begin{align}\label{Eq-AVC-alga2-al2be2-ded}
\AVC = \{ \alpha\gamma^2, \alpha^2\beta^2, \delta^d \}.
\end{align} 
Given $\alpha\beta\cdots=\alpha^2\beta^2$ and $\alpha\gamma\cdots=\alpha\gamma^2$ and their unique angle arrangements, Lemma \ref{Lem-even-m} implies that $m\ge4$ is even.

For $m=4$, a pair $\delta\,\bvert\,\delta$ at a $\delta^d$ have incident tiles $T_1,T_2$ illustrated in Figure \ref{Fig-a4-a2b2-Tiling-alga2-al2be2-ded}. The unique angle arrangements of the vertices in \eqref{Eq-bede-alga2-vertex-de} then determine $T_3,T_4,T_5$. The same argument repeats for each pair $\delta\,\bvert\,\delta$ until it returns to the starting pair and determines the \ref{Label:EMT-Herschel} family. The shaded tiles $T_1,T_3,T_5$ form a time zone. The minimal member of the family has three time zones and has the Herschel graph as its underlying graph.


\begin{figure}[h!] 
\centering
\begin{tikzpicture}

\tikzmath{
\s=1;
\r=0.6;
\th=360/4;
\x=\r*cos(0.5*\th);
\R = sqrt(\x^2+(3*\x)^2);
\aR = acos(3*\x/\R);
\tz=3;
\tzz=\tz-1;
}

\fill[gray!25]
	(\th:2*\r) -- (\th:\r) -- (0:\r) -- (2*\r, -\r) -- (2*\r, -2*\r) -- (4*\r, -2*\r) -- (4*\r, -\r) -- (3*\r, 0) -- (2*\r, \r) -- (2*\r, 2*\r) -- cycle
;

\foreach \a in {0,...,\tz} {

\tikzset{shift={(2*\a*\r,0)}}

\draw[]
	(\th:\r) -- (\th:2*\r)
	(\th:\r) -- (0:\r)
;

\draw[shift={(0:\r)}]
	(0:0) -- (-0.5*\th:2*\x) 
	(-0.5*\th:2*\x) -- ([shift={(-0.5*\th:2*\x)}]270:\r)
;

}

\foreach \a in {0,...,\tzz} {

\tikzset{shift={(2*\a*\r,0)}}

\draw[shift={(0:\r)}]
	(0:0) -- (0.5*\th:2*\x)
	(-0.5*\th:2*\x) -- ([shift={(-0.5*\th:2*\x)}]0.5*\th:2*\x)
;

}

\foreach \a in {0,...,\tz} {

\tikzset{shift={(2*\a*\r,0)}}

\draw[line width=1.5]
	(\th:\r) -- (\th:2*\r)
;

\draw[line width=1.5, shift={(0:\r)}]
	(-0.5*\th:2*\x) -- ([shift={(-0.5*\th:2*\x)}]270:\r)
;

}

\foreach \a in {0,...,\tzz} {

\tikzset{shift={(2*\a*\r,0)}}

\node at ([shift={(\r,0)}]\th:1.75*\r) {\small $\delta$};
\node at ([shift={(\r,0)}]\th:0.4*\r) {\small $\beta$};
\node at ([shift={(\r,0)}]0.65*\r,1.1*\r) {\small $\gamma$};
\node at ([shift={(\r,0)}]-0.65*\r,1.1*\r) {\small $\gamma$};

}

\foreach \a in {0,...,\tzz} {

\tikzset{shift={(2*\a*\r,0)}}

\foreach \c in {0,...,3} {

\node at ([shift={(2*\r,0)}]\th+\c*\th:0.6*\r) {\small $\alpha$};

}
}

\foreach \a in {0,...,\tzz} {

\tikzset{shift={(2*\a*\r,0)}}

\node at ([shift={(3*\r,0)}]3*\th:1.75*\r) {\small $\delta$};
\node at ([shift={(3*\r,0)}]3*\th:0.4*\r) {\small $\beta$};
\node at ([shift={(3*\r,0)}]0.65*\r,-1.1*\r) {\small $\gamma$};
\node at ([shift={(3*\r,0)}]-0.65*\r,-1.1*\r) {\small $\gamma$};

}

\node at (2*\tz*\r+3*\r, 0) {\Large $\cdots$};

\node[inner sep=1,draw,shape=circle] at (\r,1.5*\x) {\footnotesize $1$};
\node[inner sep=1,draw,shape=circle] at (3*\r,1.5*\x) {\footnotesize $2$};
\node[inner sep=1,draw,shape=circle] at (2*\r,0*\x) {\footnotesize $3$};
\node[inner sep=1,draw,shape=circle] at (4*\r,0*\x) {\footnotesize $4$};
\node[inner sep=1,draw,shape=circle] at (3*\r,-1.5*\x) {\footnotesize $5$};

\end{tikzpicture}
\caption{The \ref{Label:EMT-Herschel} tilings with $\AVC = \{ \alpha\gamma^2, \alpha^2\beta^2, \delta^d \}$, a time zone (shaded) consists of three tiles, one equatorial square and two polar kites}
\label{Fig-a4-a2b2-Tiling-alga2-al2be2-ded}
\end{figure}
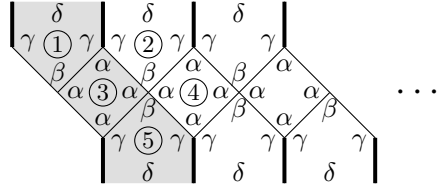

The vertex angle sums from $\AVC$ \eqref{Eq-AVC-alga2-al2be2-ded} imply
\begin{align}\label{Eq-alga2-al2be2-ded}
\beta = \pi - \alpha, \quad
\gamma = \pi - \tfrac{1}{2}\alpha, \quad
\delta = \tfrac{2}{d}\pi.
\end{align}
Substituting the above equations into \eqref{Eq-angle-id} gives
\begin{align*}
\cos \tfrac{1}{2}\alpha \sin \tfrac{1}{2}\alpha = \sin^2\tfrac{1}{2}\alpha( \cos \tfrac{1}{d}\pi + \cos \tfrac{1}{2}\pi ).
\end{align*}
For $\alpha \in (\tfrac{1}{2}\pi, \pi)$, we know $\sin \tfrac{1}{2}\alpha \neq 0$. Then the above deduces
\begin{align*}
\cos \tfrac{1}{d}\pi = \frac{ \cos \tfrac{1}{2}\alpha }{\sin \tfrac{1}{2}\alpha } = \cot \tfrac{1}{2}\alpha.
\end{align*}
Hence 
for each positive integer $d\ge3$, there exists a solution $\alpha\in(\frac12\pi,\pi)$ given by
\begin{align}\label{Eq-alga2-al2be2-alarccot}
\alpha = \pi-2 \tan^{-1} (\cos \tfrac{1}{d}\pi),
\end{align}
and so
the tiling exists: the kite has the angles with values given by \eqref{Eq-alga2-al2be2-ded} and the square has the angle value $\alpha$ given by \eqref{Eq-alga2-al2be2-alarccot}.

In Figure \ref{Fig-a4-a2b2-Tiling-alga2-al2be2-ded}, the $\alpha^2\beta^2$'s are located along the equator. The distance between the $\alpha^2\beta^2, \delta^d$ in the same kite is given by diagonal of length $\tfrac{1}{2}\pi$ through the angles $\beta, \delta$. Then the triangle (as a half of a kite, Figure \ref{Fig-a4-a2b2-kite}) has edges $x, y, \tfrac{1}{2}\pi$ and opposite angles $\tfrac{1}{2}\delta, \tfrac{1}{2}\beta, \tfrac{1}{2}\pi$. By the spherical sine law, we get $\sin x =\frac{ \sin \frac{1}{2}\delta }{ \sin \gamma}$. By the spherical cosine law, we further obtain
\begin{align*}
\cos y &=  \frac{ \cos \frac{1}{2}\beta \sin \frac{1}{2}\delta }{ \sin \gamma}.
\end{align*} 
Substituting \eqref{Eq-alga2-al2be2-ded} into the above and \eqref{Eq-x2y2-cosx}, the edge lengths are given by
\begin{align}
x = \cos^{-1} \cos^2 \tfrac{1}{d}\pi, \qquad
y = \cos^{-1} \sin \tfrac{1}{d}\pi=\tfrac{\pi}{2}-\tfrac{\pi}{d}.
\end{align}

For even $m\ge6$, the vertices in $\AVC$ \eqref{Eq-AVC-alga2-al2be2-ded} imply $\beta=\pi-\alpha$, $2\gamma=2\pi-\alpha$ and $\delta=\tfrac{2}{d}\pi$. Combined with the kite angle sum and $\alpha>(1-\frac{2}{m})\pi$, we get $\tfrac{2}{d}\pi=\delta > (1-\tfrac{4}{m})\pi$, which implies $d<6$. Hence $\delta^d=\delta^3, \delta^4, \delta^5$. The vertex $\delta^3$ implies $\delta=\frac{2}{3}\pi > (1-\tfrac{4}{m})\pi$ and $m=6,8,10$. Then the solutions for the angles in \eqref{Eq-angle-id} are determined below; Similarly, the vertices $\delta^4$ and $\delta^5$ determine $m=6$ and the solutions for the angles below.

For $m=6$, we have
\begin{align}
&\delta^3:&
&\alpha=\cos^{-1}(-\tfrac{4}{5}),&
&\beta=\pi-\alpha, \qquad
\gamma=\pi-\tfrac{1}{2}\alpha, \quad
&\delta=\tfrac{2}{3}\pi,& \\ \notag
&&
&x=\cos^{-1}\tfrac{2}{3},&
&y=\tan^{-1}\tfrac{\sqrt2}{5}. \\
&\delta^4:&
&\alpha=\cos^{-1}(-\tfrac{7}{11}),&
&\beta=\pi-\alpha, \qquad
\gamma=\pi-\tfrac{1}{2}\alpha, \quad
&\delta=\tfrac{1}{2}\pi,& \\ \notag
&&
&x=\cos^{-1} \tfrac{5}{6},&
&y=\sin^{-1}\tfrac{1}{3}. \\ 
&\delta^5:&
&\alpha=2\tan^{-1}\tfrac{3}{2}(\sqrt5-1),&
&\beta=\pi - \alpha, \qquad
\gamma=\pi-\tfrac{1}{2}\alpha, \quad
&\delta=\tfrac{2}{5}\pi,& \\ \notag
&&
&x=\cos^{-1} \tfrac{1}{12}(9+\sqrt{5}),&
&y=\tan^{-1}\tfrac{1}{38}(9+\sqrt5).
\end{align}

For $m=8$, we have
\begin{align}
&\alpha=2\tan^{-1}(2+\sqrt2),&
&\beta=\pi-\alpha, \qquad 
\gamma=\pi-\tfrac{1}{2}\alpha, \qquad
\delta=\tfrac{2}{3}\pi,&\\ \notag
&x=\cos^{-1} \tfrac{1}{4}(2+\sqrt{2}),&
&y=\cos^{-1}\tfrac{1}{\sqrt6}(1+\sqrt2).
\end{align}

For $m=10$, we have
\begin{align}
&\alpha=2\tan^{-1}\tfrac{1}{2}(5+\sqrt5),&
&\beta=\pi-\alpha, \qquad
\gamma=\pi-\tfrac{1}{2}\alpha, \quad
\delta=\tfrac{2}{3}\pi,&\\ \notag
&x=\cos^{-1}\tfrac{1}{10}(5+2\sqrt{5}),&
&y=\cos^{-1} \tfrac{1}{30} ( 15\sqrt{3} + \sqrt{15} ).
\end{align}

Therefore the vertices for each $m$ are listed as follows,
\begin{align} 
\label{Eq-AVC-alga2-al2be2-de3}
&m=6,8,10,\quad \AVC = \{  \alpha\gamma^2, \alpha^2\beta^2, \delta^3 \}; \\
\label{Eq-AVC-alga2-al2be2-de4,5}
&m=6, \ \quad\qquad \AVC = \{  \alpha\gamma^2, \alpha^2\beta^2, \delta^{d=4,5} \}.
\end{align}
The construction is straightforward. The diminished neighbourhood of $\delta^d$ is a $2d$-gon which can be deformed to a regular $d$-gon with angles $\bar{\beta}^b$, as shown in Figure \ref{Subfig-bede-alga2-al2be-2d-gon}. In view of this, 
the original $m$-gon becomes an $\bar{m}$-gon with angles $\bar{\alpha}^{\bar{m}}$, where $\bar{m}=\frac{m}{2}$. Now, $\AVC$s \eqref{Eq-AVC-alga2-al2be2-de3} and \eqref{Eq-AVC-alga2-al2be2-de4,5} are reduced to $\AVC=\{ \bar{\alpha}^2\bar{\beta}^2 \}$, where $\bar{\alpha}^2\bar{\beta}^2=\vert \bar{\alpha} \vert \bar{\beta} \vert \bar{\alpha} \vert \bar{\beta} \vert $. Then using the underlying prototiles, one can obtain the primal graphs in Figure \ref{Fig-Rect-Platonic-Tilings-alga2-al2be2-ded} and hence the \ref{Label:rect-Platonic} tilings (Figure \ref{Fig-Platonic-Archimedean}) given by the highlighted kite subdivisions.

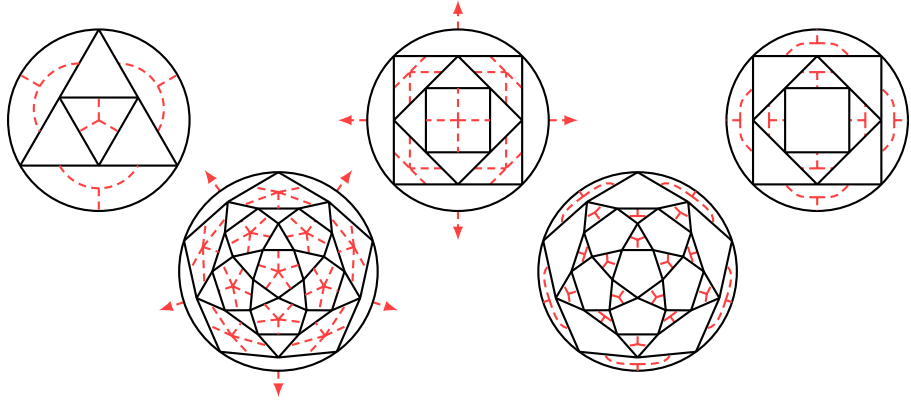
\begin{figure}[h!] 
\centering
\begin{tikzpicture}[>=latex]

\tikzmath{
\xs=4.75;
\ys=2;
}

\begin{scope}[thick] 

\begin{scope}[]

\tikzmath{
\th=360/3;
\r=0.6;
\x=\r*cos(\th/2);
}

\foreach \a in {0,...,2} {

\tikzset{rotate=\a*\th}

\draw[]
	(270:\r) -- (270+\th:\r)
	(90:2*\r) -- (90+\th:2*\r)
;

\draw[red!75, densely dashed]
	(0,0) -- (90:\x)
	(90-0.5*\th:1.5*\r) -- (90-0.5*\th:2*\r)
;

\arcThroughThreePoints[red!75, densely dashed]{$(90-0.5*\th:1*\r) !1/2! (90-1*\th:2*\r)$}{90-0.5*\th:1.5*\r}{$(90:2*\r) !1/2! (90-0.5*\th:1*\r)$};

} 

\draw[] (0,0) circle (2*\r);

\end{scope}

\begin{scope}[xshift=\xs cm]

\tikzmath{
\ph=360/4;
\r=0.6;
\x=\r*cos(0.5*\ph);
}

\foreach \a in {0,...,3} {
\tikzset{rotate=\a*\ph}

\draw[red!75, densely dashed]
	($(0,2*\x) !1/2! (2*\x,2*\x)$) -- (1.5*\x, 1.5*\x)
	($(2*\x,0) !1/2! (2*\x,2*\x)$) -- (1.5*\x, 1.5*\x)
	($(0,2*\x) !1/2! (\x,\x)$) -- (1.5*\x, 1.5*\x)
	($(2*\x,0) !1/2! (\x,\x)$) -- (1.5*\x, 1.5*\x)
;
}

\foreach \a in {0,...,3} {

\tikzset{rotate=\a*\ph}

\draw[]
	(\x,\x) -- (-\x,\x)
	(2*\x,0) -- (0,2*\x)
	(2*\x,2*\x) -- (-2*\x,2*\x)
;

\draw[red!75, densely dashed]
	(0,0) -- (0,\x)
;

\draw[->, red!75, densely dashed]
	(0,2*\r) -- (0,2.65*\r)
;

}

\draw[] (0,0) circle (2*\r);

\end{scope}

\begin{scope}[xshift=2*\xs cm]

\tikzmath{
\ph=360/4;
\r=0.6;
\x=\r*cos(0.5*\ph);
}

\draw[] (0,0) circle (2*\r);

\foreach \a in {0,...,3} {
\tikzset{rotate=\a*\ph}

\draw[red!75, densely dashed]
	(0,2.4*\x) -- (0,2*\r)
	(0,2.4*\x) to[out=0, in=135, distance=0.25cm] (1*\x,2*\x) 
	(0,2.4*\x) to[out=180, in=45, distance=0.25cm] (-1*\x,2*\x) 
;

\draw[red!75, densely dashed]
	(0,1.5*\x) -- (0,\x)
	(0,1.5*\x) -- ($(0,2*\x) !1/2! (\x,\x)$)
	(0,1.5*\x) -- ($(0,2*\x) !1/2! (-\x,\x)$)
;
}

\foreach \a in {0,...,3} {

\tikzset{rotate=\a*\ph}

\draw[]
	(\x,\x) -- (-\x,\x)
	(2*\x,0) -- (0,2*\x)
	(2*\x,2*\x) -- (-2*\x,2*\x)
;

}

\end{scope}

\end{scope}

\begin{scope}[yshift=-\ys cm, xshift=0.5*\xs cm, thick] 

\tikzmath{
\s=4;
}

\begin{scope}[]

\tikzmath{
\ps=360/5;
\r=0.35;
\x = \r*cos(0.5*\ps);
}

\foreach \a in {0,...,4} {
\tikzset{rotate=\a*\ps}

\draw[]
	(270:\r) -- (270+\ps:\r)
	(270:\r) -- (270-0.5*\ps:\x+\r)
	(270:\r) -- (270+0.5*\ps:\x+\r)
	(270-0.25*\ps:2.5*\r) -- (270-0.5*\ps:\x+\r)
	(270+0.25*\ps:2.5*\r) -- (270+0.5*\ps:\x+\r)
 	(270-0.25*\ps:2.5*\r) -- (270+0.25*\ps:2.5*\r)
	(90-0.25*\ps:2.5*\r) -- (90+0.25*\ps:2.5*\r) 
	(90-0.25*\ps:2.5*\r) -- (90-0.5*\ps:3.25*\r) 
	(90+0.25*\ps:2.5*\r) -- (90+0.5*\ps:3.25*\r)
	(90:3.75*\r) -- (90-0.5*\ps:3.25*\r) 
	(90:3.75*\r) -- (90+0.5*\ps:3.25*\r) 
;

\draw[red!75, densely dashed]
	(0,0) -- (0,\x)
	(90-0.5*\ps:\x+\r) -- ($(90-0.5*\ps:1*\r) !1/2! (90-\ps:\x+\r)$)
	(90-0.5*\ps:\x+\r) -- ($(90-0.5*\ps:1*\r) !1/2! (90:\x+\r)$)
	(90-0.5*\ps:\x+\r) -- ($(90:\x+\r) !1/2! (90-0.25*\ps:2.5*\r)$)
	(90-0.5*\ps:\x+\r) -- ($(90-\ps:\x+\r) !1/2! (90-0.75*\ps:2.5*\r)$)
	(90-0.5*\ps:\x+\r) -- ($(90-0.25*\ps:2.5*\r) !1/2! (90-0.75*\ps:2.5*\r)$)
	(90:3*\r) -- ($(90-0.25*\ps:2.5*\r) !1/2! (90+0.25*\ps:2.5*\r)$)
	(90:3*\r) -- ($(90-0.25*\ps:2.5*\r) !1/2! (90-0.5*\ps:3.25*\r)$)
	(90:3*\r) -- ($(90+0.25*\ps:2.5*\r) !1/2! (90+0.5*\ps:3.25*\r)$)
	(90:3*\r) -- ($(90:3.75*\r) !1/2! (90-0.5*\ps:3.25*\r)$)
	(90:3*\r) -- ($(90:3.75*\r) !1/2! (90+0.5*\ps:3.25*\r)$)
;

\draw[->, red!75, densely dashed]
	(90-0.5*\ps:3.75*\r) -- (90-0.5*\ps:4.75*\r)
;

}

\draw[] (0,0) circle (3.75*\r);

\end{scope}

\begin{scope}[xshift=1*\xs cm]

\tikzmath{
\ps=360/5;
\r=0.35;
\x = \r*cos(0.5*\ps);
}

\draw[] (0,0) circle (3.75*\r);

\foreach \a in {0,...,4} {
\tikzset{rotate=\a*\ps}

\draw[red!75, densely dashed]
	(90:1.2*\r) -- (90:\x)
	(90:1.2*\r) -- ($(90-0.5*\ps:\r) !1/2! (90:\x+\r)$)
	(90:1.2*\r) -- ($(90+0.5*\ps:\r) !1/2! (90:\x+\r)$)
	(90:2.59*\x) -- ($(90-0.25*\ps:2.5*\r) !1/2! (90:\x+\r)$)
	(90:2.59*\x) -- ($(90+0.25*\ps:2.5*\r) !1/2! (90:\x+\r)$)
	(90:2.59*\x) -- ($(90-0.25*\ps:2.5*\r) !1/2! (90+0.25*\ps:2.5*\r)$)
	(90-0.5*\ps:2.75*\r) -- ($(90-0.25*\ps:2.5*\r) !1/2! (90-0.75*\ps:2.5*\r)$)
	(90-0.5*\ps:2.75*\r) -- ($(90-0.25*\ps:2.5*\r) !1/2! (90-0.5*\ps:3.25*\r)$)
	(90-0.5*\ps:2.75*\r) -- ($(90-0.75*\ps:2.5*\r) !1/2! (90-0.5*\ps:3.25*\r)$)
	(270:3.5*\r) -- (270:3.75*\r) 
	(270:3.5*\r) to[out=180, in=280, distance=0.15 cm] ($(270:3.5*\r) !3/5! (270-0.5*\ps:3.5*\r)$)
	(270:3.5*\r) to[out=0, in=260, distance=0.15 cm] ($(270:3.5*\r) !3/5! (270+0.5*\ps:3.5*\r)$)
;
}

\foreach \a in {0,...,4} {
\tikzset{rotate=\a*\ps}

\draw[]
	(270:\r) -- (270+\ps:\r)
	(270:\r) -- (270-0.5*\ps:\x+\r)
	(270:\r) -- (270+0.5*\ps:\x+\r)
	(270-0.25*\ps:2.5*\r) -- (270-0.5*\ps:\x+\r)
	(270+0.25*\ps:2.5*\r) -- (270+0.5*\ps:\x+\r)
 	(270-0.25*\ps:2.5*\r) -- (270+0.25*\ps:2.5*\r)
	(90-0.25*\ps:2.5*\r) -- (90+0.25*\ps:2.5*\r) 
	(90-0.25*\ps:2.5*\r) -- (90-0.5*\ps:3.25*\r) 
	(90+0.25*\ps:2.5*\r) -- (90+0.5*\ps:3.25*\r)
	(90:3.75*\r) -- (90-0.5*\ps:3.25*\r) 
	(90:3.75*\r) -- (90+0.5*\ps:3.25*\r) 
;
}

\end{scope}

\end{scope}[] 

\end{tikzpicture}
\caption{The family of \ref{Label:rect-Platonic} tilings with $\AVC\equiv \{ \alpha\gamma^2, \alpha^2\beta^2, \delta^d \}$ determined in two steps -- step 1. construction by the underlying prototiles, a regular $d$-gon $\beta^d$ and a regular $\bar{m}$-gon $\alpha^{\bar{m}}$, and reduced $\AVC = \{ \vert \alpha \vert \beta \vert \alpha \vert \beta \vert \}$; step 2. kite subdivisions (marked by dashed lines) on one prototile throughout a tiling in step 1}
\label{Fig-Rect-Platonic-Tilings-alga2-al2be2-ded}
\end{figure}
\end{subsubcase*}

\begin{subsubcase*}[$\gamma\cdots=\alpha\gamma^2$ and $\beta\cdots=\alpha^3\beta$] The assumption reduces \eqref{Eq-bede-alga2-vertex-al} to $\alpha\cdots=\alpha\gamma^2, \alpha^3\beta$. Given $\beta\cdots=\alpha^3\beta$ and $\gamma\cdots=\alpha\gamma^2$ and $\delta\cdots=\delta^{d\ge3}$ from \eqref{Eq-bede-alga2-vertex-de}, the vertex types are
\begin{align}\label{Eq-AVC-alga2-al3be-ded}
\AVC = \{ \alpha\gamma^2, \alpha^3\beta, \delta^d \}.
\end{align}
The vertex $\alpha^3\beta$ and \eqref{Eq-mgon-sum} imply $\frac{2}{3}\pi>\alpha>(1-\tfrac{2}{m})\pi$, and hence $m=4,5$.  

For $m=5$, an $\alpha^3\beta$ vertex has incident tiles $T_1, T_2, T_3, T_4$ as illustrated in Figure \ref{Fig-bede-alga2-al3be-m=5-AAD}. Then $\alpha_1\gamma_4\cdots, \alpha_3\gamma_4\cdots=\alpha\gamma^2$ determine $T_5, T_6$. By $\alpha_1\beta_5\cdots, \alpha_3\beta_6\cdots$, we then determine $T_7, T_8$. The adjacent vertices $\alpha_1\alpha_2\cdots, \alpha_1\alpha_7\cdots$ can only be $\alpha^3\beta$, this determine $T_9, T_{10}$. By symmetry, we also determine $T_{11}, T_{12}$. Then $\alpha_2\gamma_{10}\cdots, \alpha_2\gamma_{12}\cdots=\alpha\gamma^2$ result in adjacent $\gamma$'s in $T_{13}$, a contradiction. Therefore there is no tiling for $m=5$.

\begin{figure}[h!] 
\centering
\begin{tikzpicture}

\tikzmath{
\s=6;
\r=0.7;
\ph=360/5;
}

\begin{scope}

\foreach \p in {0,...,4} {

\tikzset{rotate=\p*\ph}

\draw[]
	(270:\r) -- (270+\ph:\r)
;

\node at (270:0.65*\r) {\small $\alpha$};
	
}

\foreach \aa in {-1,1} {

\tikzset{xscale=\aa}

\draw[] 
	(0,-\r) -- (0.75*\r, -1.75*\r)
	(0.75*\r, -1.75*\r) -- (1.75*\r, -1.75*\r)
	(1.75*\r, -1.75*\r) -- (2*\r, -0.75*\r)
	(270+\ph:\r) -- (2*\r, -0.75*\r)
	(2*\r, -0.75*\r) -- (3*\r, -0.5*\r)
	(3*\r, -0.5*\r) -- (3.75*\r, -1.25*\r)
	(3.75*\r, -1.25*\r) -- (3*\r, -2*\r)
	(1.75*\r, -1.75*\r) -- (3*\r, -2*\r)
	(270+\ph:\r) -- (1.75*\r, 0.75*\r)
	(1.75*\r, 0.75*\r) -- (2.5*\r, 0.75*\r)
	(2.5*\r, 0.75*\r) -- (2.75*\r, 0*\r)
	(2*\r, -0.75*\r) -- (2.75*\r, 0*\r)
;

\draw[line width=1.5]
	(1.75*\r, 0.75*\r) -- (1.25*\r, 1.75*\r) 
	(270+2*\ph:\r) -- (1.25*\r, 1.75*\r) 
	(0.75*\r, -1.75*\r) -- (0, -3*\r)
	(0, -3*\r) -- (1.25*\r, -3*\r)
;

\draw[]
	(1.25*\r, -3*\r) -- (1.75*\r, -1.75*\r)
;

\node at (0.4*\r,-1*\r) {\small $\alpha$};
\node at (1*\r,-0.6*\r) {\small $\alpha$};
\node at (0.85*\r,-1.5*\r) {\small $\alpha$};
\node at (1.7*\r,-0.9*\r) {\small $\alpha$};
\node at (1.55*\r,-1.5*\r) {\small $\alpha$};

\node at (1.45*\r,-2*\r) {\small $\beta$};
\node at (0.95*\r,-2*\r) {\small $\gamma$};
\node at (1.15*\r,-2.8*\r) {\small $\gamma$};
\node at (0.45*\r,-2.75*\r) {\small $\delta$};

\node at (1.9*\r,-2*\r) {\small $\alpha$};

\node at (2.2*\r,-0.95*\r) {\small $\alpha$};
\node at (2.85*\r,-0.75*\r) {\small $\alpha$};
\node at (3.4*\r,-1.25*\r) {\small $\alpha$};
\node at (2.9*\r,-1.75*\r) {\small $\alpha$};
\node at (2.1*\r,-1.55*\r) {\small $\alpha$};

\node at (1.95*\r,-0.45*\r) {\small $\alpha$};
\node at (1.35*\r,-0.2*\r) {\small $\alpha$};
\node at (1.85*\r,0.5*\r) {\small $\alpha$};
\node at (2.35*\r,0.5*\r) {\small $\alpha$};
\node at (2.45*\r,0.1*\r) {\small $\alpha$};

\node at (1*\r,0.1*\r) {\small $\beta$};
\node at (0.8*\r,0.75*\r) {\small $\gamma$};
\node at (1.5*\r,0.75*\r) {\small $\gamma$};
\node at (1.2*\r,1.25*\r) {\small $\delta$};

\node at (0.5*\r,1.1*\r) {\small $\gamma$};

}

\node at (0,-1.4*\r) {\small $\beta$};
\node at (-0.4*\r,-1.8*\r) {\small $\gamma$};
\node at (0.4*\r,-1.8*\r) {\small $\gamma$};
\node at (0,-2.5*\r) {\small $\delta$};

\node[inner sep=1,draw,shape=circle] at (-1.1*\r,-1.1*\r) {\footnotesize $1$};
\node[inner sep=1,draw,shape=circle] at (0,0) {\footnotesize $2$};
\node[inner sep=1,draw,shape=circle] at (1.1*\r,-1.1*\r) {\footnotesize $3$};
\node[inner sep=1,draw,shape=circle] at (0,-2*\r) {\footnotesize $4$};

\node[inner sep=1,draw,shape=circle] at (-1*\r,-2.4*\r) {\footnotesize $5$};
\node[inner sep=1,draw,shape=circle] at (1*\r,-2.4*\r) {\footnotesize $6$};

\node[inner sep=1,draw,shape=circle] at (-2.75*\r,-1.25*\r) {\footnotesize $7$};
\node[inner sep=1,draw,shape=circle] at (2.75*\r,-1.25*\r) {\footnotesize $8$};

\node[inner sep=1,draw,shape=circle] at (-1.9*\r,0) {\footnotesize $9$};
\node[inner sep=0.25,draw,shape=circle] at (1.9*\r,0) {\footnotesize $11$};

\node[inner sep=0.25,draw,shape=circle] at (-1.125*\r,0.6*\r) {\footnotesize $10$};
\node[inner sep=0.25,draw,shape=circle] at (1.125*\r,0.6*\r) {\footnotesize $12$};
\node[inner sep=0.25,draw,shape=circle] at (0,1.5*\r) {\footnotesize $13$};

\end{scope}

\end{tikzpicture}
\caption{The deduction of $\alpha^3\beta$}
\label{Fig-bede-alga2-al3be-m=5-AAD}
\end{figure}
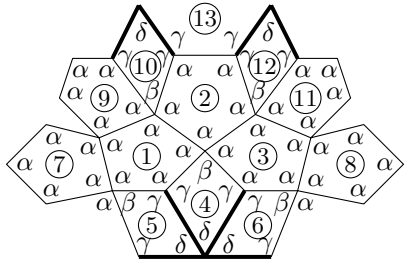

For $m=4$, three consecutive $\delta \, \bvert \, \delta \, \bvert \, \delta$ at a $\delta^d$ vertex determine the three incident tiles $T_1, T_2, T_3$ in Figure \ref{Fig-a4-a2b2-Tiling-alga2-al3be-ded}. Then $\gamma_1\gamma_2\cdots, \gamma_2\gamma_3=\alpha\gamma^2$ determine $T_4, T_5$. Next, $\alpha_4\beta_2\alpha_5\cdots, \alpha_5\beta_3\cdots=\alpha^3\beta$ determine $T_6, T_7$. Hence $\alpha_5\alpha_6\alpha_7\cdots=\alpha^3\beta$ determines $T_8$. The same argument repeats and determines the family of \ref{Label:EMT-Antiprisms} tilings. A time zone formed ($T_1,T_4,T_6, T_8$) is shaded in Figure \ref{Fig-a4-a2b2-Tiling-alga2-al3be-ded}. The family of \ref{Label:EMT-Antiprisms} tilings are generated by gluing copies of the time zone. 


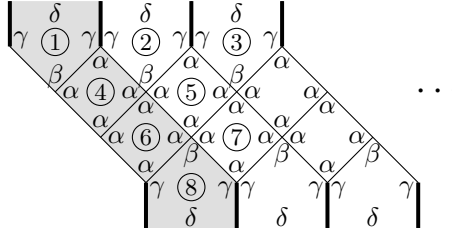
\begin{figure}[h!] 
\centering
\begin{tikzpicture}
\tikzmath{
\s=1;
\r=0.6;
\th=360/4;
\x=\r*cos(0.5*\th);
\R = sqrt(\x^2+(3*\x)^2);
\aR = acos(3*\x/\R);
\tz=3;
\tzz=\tz-1;
}

\fill[gray!25]
	(-1.5*\r,2.5*\r) --  (1.5*\th:3*\x)  -- (3.5*\th:3*\x) -- (1.5*\r, -2.5*\r) -- (1.5*\r, -2.5*\r) 
	-- (3.5*\r, -2.5*\r) -- (3.5*\r, -1.5*\r) -- (0.5*\r, 1.5*\r) -- (0.5*\r, 1.5*\r) -- (0.5*\r, 2.5*\r) -- cycle
;

\foreach \a in {0,...,\tz} {

\tikzset{shift={(2*\a*\r,0)}}

\draw[]
	(1.5*\th:\x) -- (1.5*\th:3*\x)
	(1.5*\th:3*\x) -- (-1.5*\r, 2.5*\r)
	(1.5*\th:\x) -- (3.5*\th:\x)
	(3.5*\th:\x) -- (3.5*\th:3*\x)
	(3.5*\th:3*\x) -- (1.5*\r, -2.5*\r)
;

}

\foreach \a in {0,...,\tzz} {

\tikzset{shift={(2*\a*\r,0)}}

\draw[]
	(-0.5*\r, 0.5*\r) -- (0.5*\r, 1.5*\r)
	(0.5*\r, -0.5*\r) -- (1.5*\r, 0.5*\r)
	(1.5*\r, -1.5*\r) -- (2.5*\r, -0.5*\r)
;

}

\foreach \a in {0,...,\tz} {

\tikzset{shift={(2*\a*\r,0)}}

\draw[line width=1.5]
	(1.5*\th:3*\x) -- (-1.5*\r, 2.5*\r)
	(3.5*\th:3*\x) -- (1.5*\r, -2.5*\r)
;

}

\foreach \a in {0,...,\tzz} {

\tikzset{shift={(2*\a*\r,0)}}

\node at (-0.5*\r,2.25*\r) {\small $\delta$};
\node at (-0.5*\r,0.85*\r) {\small $\beta$};
\node at (-1.25*\r,1.65*\r) {\small $\gamma$};
\node at (0.25*\r,1.65*\r) {\small $\gamma$};

\node at (2.5*\r, -2.25*\r) {\small $\delta$};
\node at (2.5*\r,-0.9*\r) {\small $\beta$};
\node at (1.75*\r,-1.65*\r) {\small $\gamma$};
\node at (3.25*\r,-1.65*\r) {\small $\gamma$};

}

\foreach \a in {0,...,\tzz} {

\tikzset{shift={(2*\a*\r,0)}}

\foreach \c in {0,...,3} {

\node at ([shift={(0.5*\r,0.5*\r)}]\th+\c*\th:0.65*\r) {\small $\alpha$};
\node at ([shift={(1.5*\r,-0.5*\r)}]\th+\c*\th:0.65*\r) {\small $\alpha$};

}
}

\node[inner sep=1,draw,shape=circle] at (-0.5*\r,1.6*\r) {\footnotesize $1$};
\node[inner sep=1,draw,shape=circle] at (1.5*\r,1.6*\r) {\footnotesize $2$};
\node[inner sep=1,draw,shape=circle] at (3.5*\r,1.6*\r) {\footnotesize $3$};

\node[inner sep=1,draw,shape=circle] at (0.5*\r,0.5*\r) {\footnotesize $4$};
\node[inner sep=1,draw,shape=circle] at (2.5*\r,0.5*\r) {\footnotesize $5$};
\node[inner sep=1,draw,shape=circle] at (1.5*\r,-0.5*\r) {\footnotesize $6$};
\node[inner sep=1,draw,shape=circle] at (3.5*\r,-0.5*\r) {\footnotesize $7$};
\node[inner sep=1,draw,shape=circle] at (2.5*\r,-1.6*\r) {\footnotesize $8$};

\node at (2*\tz*\r+2*\r, 0.5*\r) {\Large $\cdots$};

\end{tikzpicture}
\caption{The \ref{Label:EMT-Antiprisms} tiling with $\AVC \equiv \{ \alpha\gamma^2, \alpha^3\beta, \delta^d \}$, a (shaded) time zone consists of two equatorial squares and two polar kites}
\label{Fig-a4-a2b2-Tiling-alga2-al3be-ded}
\end{figure}

It remains to show the existence of a tiling for each $d\ge4$. By $\AVC$ \eqref{Eq-AVC-alga2-al3be-ded}, we have
\begin{align} \label{Eq-alga2-al3be-ded}
\beta = 2\pi - 3\alpha, \quad
\gamma=\pi-\tfrac{1}{2}\alpha, \quad
\delta = \tfrac{2}{d}\pi. 
\end{align}
Substituting the above into \eqref{Eq-angle-id} deduces
\begin{align*}
\sin \tfrac{3}{2}\alpha  \sin\tfrac{1}{2}\alpha   = \sin^2 \tfrac{1}{2}\alpha ( \cos \tfrac{1}{d}\pi + \cos \alpha).
\end{align*}
As $\sin \tfrac{1}{2}\alpha \neq 0$ for $\alpha \in (\tfrac{1}{2}\pi,\tfrac{2}{3}\pi)$, trigonometric identities deduce
\begin{align*}
\cos \tfrac{1}{d}\pi = \frac{ \sin \tfrac{3}{2}\alpha }{ \sin \tfrac{1}{2}\alpha  } - \cos \alpha = \sin\alpha\cot\tfrac{1}{2}\alpha = 2\cos^2\tfrac{1}{2}\alpha = 1 + \cos \alpha.
\end{align*}
Hence 
for every positive integer $d\ge4$, there exists a solution $\alpha\in(\tfrac{1}{2}\pi, \tfrac{2}{3}\pi)$ 
given by
\begin{align}\label{Eq-alga2-al3be-alarccos}
\alpha=\cos^{-1}(\cos \tfrac{1}{d}\pi -1),
\end{align}
and so
the tiling exists: the kite has angles with values from \eqref{Eq-alga2-al3be-ded} and the square has angles with values given in \eqref{Eq-alga2-al3be-alarccos}. Furthermore, \eqref{Eq-a4-cosx} and \eqref{Eq-x2y2-cosy} determine edge lengths $x,y$,
\begin{align*}
&x = \cos^{-1} \frac{\cos \frac{1}{d}\pi }{ 2 - \cos \frac{1}{d}\pi },&
&y 
=\cos^{-1}(3\sqrt{\cos x}\tan\tfrac{1}{2d}\pi).&
\end{align*}
\end{subsubcase*}

\begin{subsubcase*}[$\gamma\cdots=\alpha\gamma^2$ and $\beta\cdots=\alpha^3\beta^3$] The assumption reduces \eqref{Eq-bede-alga2-vertex-al} to $\alpha\cdots=\alpha\gamma^2, \alpha^3\beta^3$. The latter implies $\beta=\frac{2}{3}\pi- \alpha$. With $\beta+\delta>\alpha$ from the case and $\alpha>(1-\tfrac{2}{m})\pi$, it further implies $\delta > 2\alpha - \frac{2}{3}\pi > 4(\frac{1}{3}-\frac{1}{m})\pi$. By $\frac{2}{3}\pi- \alpha=\beta>0$ and $\alpha>(1-\tfrac{2}{m})\pi$, we deduce $m=4,5$.


For $m=5$, the inequality $\delta > 4(\frac{1}{3}-\frac{1}{m})$ deduces $\delta>\frac{8}{15}\pi$, and hence \eqref{Eq-bede-alga2-vertex-de} becomes $\delta^d=\delta^3$. However, $\alpha\gamma^2, \alpha^3\beta^3, \delta^3$ has no solution in \eqref{Eq-angle-id} for $\alpha \in (\tfrac{3}{5}\pi,\tfrac{2}{3}\pi)$, and yield no tiling. 

For $m=4$, a similar argument reduces \eqref{Eq-bede-alga2-vertex-de} to $\delta^d=\delta^3, \delta^4, \delta^5$. The angle sums of $\alpha\gamma^2, \alpha^3\beta^3, \delta^3$ yield no solution in \eqref{Eq-angle-id} for $\alpha \in (\frac{1}{2}\pi, \frac{2}{3}\pi)$. Meanwhile, one of $\delta^4, \delta^5$ and $\alpha\gamma^2, \alpha^3\beta^3$ in \eqref{Eq-angle-id}, \eqref{Eq-a4-cosx} and \eqref{Eq-x2y2-cosy} determine
\begin{align}
&\delta^4:& 
&\alpha=\cos^{-1}(-\tfrac{1}{5}),&
&\beta=\tfrac{2}{3}\pi-\alpha, \ 
\gamma=\pi - \tfrac{1}{2}\alpha, \ 
\delta=\tfrac{1}{2}\pi,& \\ \notag
&&
&x=\cos^{-1}\tfrac{2}{3},&
&y=\tan^{-1}(3-2\sqrt2);& \\
&\delta^5:& 
&\alpha=2\tan^{-1}\tfrac{1}{2}(\sqrt{15}-\sqrt3),&
&\beta=\tfrac{2}{3}\pi-\alpha,
\gamma=\pi - \tfrac{1}{2}\alpha, \ \
\delta=\tfrac{2}{5}\pi,& \\ \notag
&&
&x= \cos^{-1} \tfrac{1}{6} (3 + \sqrt{5}),&  
&y=\tan^{-1}\tfrac{1}{4}(3-\sqrt 5).&
\end{align}
We remark that the solution with $y>x$ is rejected because of the spherical sine law on Figure \ref{Fig-a4-a2b2-kite}. With $\gamma\cdots=\alpha\gamma^2$ and $\delta\cdots=\delta^4, \delta^5$, the respective angle values determine
\begin{align}
\label{Eq-AVC-alga2-al3be3-de4}
\AVC &= \{ \alpha\gamma^2, \alpha^3\beta^3, \delta^4 \}; \\
\label{Eq-AVC-alga2-al3be3-de5}
\AVC &= \{ \alpha\gamma^2, \alpha^3\beta^3, \delta^5 \}.
\end{align}
In the absence of $\beta\vert\beta\cdots$, every vertex in the above $\AVC$s has a unique angle arrangement. Therefore, it is straightforward to construct the two tilings---\ref{Label:multigr-Cube-Dodeca} in Figure \ref{Fig-Platonic-multigraph-Cube-Dodeca}---directly from the $\AVC$s or diminishing the kite subdivision (dashed lines) in Figure \ref{Fig-a4-a2b2-Tilings-alga2-de4/de5-al3be3}.

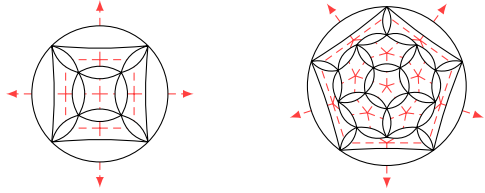
\begin{figure}[h!] 
\centering
\begin{tikzpicture}[>=latex]

\tikzmath{
\xs=2.5;
}

\begin{scope}[]
\tikzmath{
\r=0.4;
\th=360/4;
\x=\r*cos(0.5*\th);
}

\foreach \a in {0,...,3} {
\tikzset{rotate=\a*\th}

\draw[red!75, densely dashed]
	(0,0) -- (0,\x)
	(0,1.6*\x) -- (0,\x)
	(0,1.6*\x) -- (0,2.5*\x)
	(0,1.6*\x) -- (0.5*\th: 1.6*\r)
	(0,1.6*\x) -- (1.5*\th: 1.6*\r)
;
\draw[->, red!75, densely dashed]
	(90:2.25*\r) -- (90:3.1*\r)
;
}

\foreach \a in {0,...,3} {
\tikzset{rotate=\a*\th}

\fill[white]
	(0.5*\th:\r) to[out=150, in=30] (1.5*\th:\r) to[out=330, in=210] (0.5*\th:\r) 
	(0.5*\th:\r) to[out=75, in=195] (0.5*\th:2.25*\r) to[out=255, in=15] (0.5*\th:\r) 
	(0.5*\th:2.25*\r) to[out=185, in=355] (1.5*\th:2.25*\r) to[out=45, in=135] (0.5*\th:2.25*\r)
;
}

\foreach \a in {0,...,3} {
\tikzset{rotate=\a*\th}

\draw[]
	(0.5*\th:\r) to[out=150, in=30] (1.5*\th:\r)
	(0.5*\th:\r) to[out=210, in=-30] (1.5*\th:\r)
	(0.5*\th:\r) to[out=75, in=195] (0.5*\th:2.25*\r)
	(0.5*\th:\r) to[out=15, in=255] (0.5*\th:2.25*\r)
	(0.5*\th:2.25*\r) to[out=185, in=355] (1.5*\th:2.25*\r)
;
}

\draw[] (0,0) circle (2.25*\r);

\end{scope}

\node at (0,-0.5*\xs) {};
\end{tikzpicture} \hspace{1cm}
\begin{tikzpicture}[>=latex]

\tikzmath{
\xs=2.5;
}

\raisebox{0.75ex}{

\begin{scope}[]

\tikzmath{
\r=0.3;
\ph=360/5;
\X=\r*cos(0.5*\ph);
}

\foreach \a in {0,...,4} {
\tikzset{rotate=\a*\ph}

\draw[red!75, densely dashed]
	(0,0) -- (0,\X)
	(0,1.8*\X) -- (0,\X)
	(0,1.8*\X) -- (90-0.5*\ph:1.5*\r)
	(0,1.8*\X) -- (90+0.5*\ph:1.5*\r)
	(0,1.8*\X) -- (90-0.25*\ph:2.2*\r)
	(0,1.8*\X) -- (90+0.25*\ph:2.2*\r)
	(270:2.5*\r) -- (270+0.25*\ph:2.2*\r) 
	(270:2.5*\r) -- (270-0.25*\ph:2.2*\r) 
	(270:2.5*\r) -- (270+0.5*\ph:3.1*\r) 
	(270:2.5*\r) -- (270-0.5*\ph:3.1*\r) 
	(270:2.5*\r) -- (270:3*\r)
;

\draw[red!75, densely dashed, ->]
	(90-0.5*\ph:3.5*\r) -- (90-0.5*\ph:4.5*\r)
;
}

\foreach \a in {0,...,4} {
\tikzset{rotate=\a*\ph}

\fill[white]
	(90-0.5*\ph:\r) to[out=150, in=30] (90+0.5*\ph:\r) to[out=-30, in=210] (90-0.5*\ph:\r)
	(270:\r) to[out=300, in=60] (270:2.25*\r) to[out=120, in=240] (270:\r) 
	(90:2.25*\r) to[out=60, in=300] (90:3.5*\r) to[out=240, in=120] (90:2.25*\r) 
	(90:3.5*\r) to[out=-40, in=180-\ph+40]  (90-\ph:3.5*\r) arc (90-\ph:90:3.5*\r)
;

\foreach \aa in {-1,1} {
\tikzset{xscale=\aa}

\fill[white]
	(270+2*\ph:2.25*\r) to[out=185, in=320] (90:2.25*\r) to[out=20, in=130] (270+2*\ph:2.25*\r) 
;

}
}

\foreach \a in {0,...,4} {
\tikzset{rotate=\a*\ph}

\draw[]
	(90-0.5*\ph:\r) to[out=150, in=30] (90+0.5*\ph:\r)
	(90-0.5*\ph:\r) to[out=210, in=-30] (90+0.5*\ph:\r)
	(270:\r) to[out=300, in=60] (270:2.25*\r)	
	(270:\r) to[out=240, in=120] (270:2.25*\r)
	(90:2.25*\r) to[out=60, in=300] (90:3.5*\r)
	(90:2.25*\r) to[out=120, in=240] (90:3.5*\r)
	(90:3.5*\r) to[out=-40, in=180-\ph+40]  (90-\ph:3.5*\r)
;

\foreach \aa in {-1,1} {
\tikzset{xscale=\aa}

\draw[]
	(270+2*\ph:2.25*\r) to[out=185, in=320] (90:2.25*\r)
	(270+2*\ph:2.25*\r) to[out=130, in=20] (90:2.25*\r)
;

}
}

\draw[] (0,0) circle (3.5*\r);

\end{scope}

}

\end{tikzpicture}
\caption{The \ref{Label:multigr-Cube-Dodeca} tilings with $\AVC \equiv \{ \alpha\gamma^2, \alpha^3\beta^3, \delta^{d=4,5} \}$ determined in two steps -- step 1. construction by the underlying prototiles, a regular $d$-gon with angles $\beta^d$ and a schematic $2$-gon with angles $\alpha^2$, and $\AVC=\{ \vert \alpha\vert\beta \vert \alpha \vert \beta \vert \alpha \vert \beta \vert \}$; step 2. kite subdivision (indicated by the dashed lines) on the regular $d$-gons}
\label{Fig-a4-a2b2-Tilings-alga2-de4/de5-al3be3}
\end{figure}


\end{subsubcase*}
\end{subcase*}
\end{case*}

\begin{case*}[$\alpha^2\gamma^2$] The kite angle sum and the vertex angle sum of $\alpha^2\gamma^2$ and $\alpha>\frac{1}{2}\pi$ imply $\beta+\delta>2\alpha>\pi$. The vertex angle sum of $\alpha^2\gamma^2$ and $\alpha>\frac{1}{2}\pi$ also implies $\alpha + \gamma =\pi$ and $\alpha>\frac{1}{2}\pi > \gamma$. Recall $2\gamma>\frac{2}{3}\pi$ from the beginning of the proof, the inequalities deduce $\frac{2}{3}\pi>\alpha>\frac{1}{2}\pi > \gamma > \frac{1}{3}\pi$. In particular, $\frac{2}{3}\pi > \alpha > (1-\tfrac{2}{m})\pi$ deduce $m=4,5$.

Assume $\beta>\delta$. Then $\frac{2}{3}\pi\ge\beta$ and $\beta+\delta>2\alpha>\pi$ imply $\frac{2}{3}\pi\ge\beta>\alpha>\frac{1}{2}\pi$. If $\alpha^a\beta^b$ is a vertex, then $\frac{2}{3}\pi\ge\beta>\alpha$ imply that $\alpha^a\beta^b$ has degree $\ge4$, while $\beta>\alpha>\frac{1}{2}\pi$ imply that $\alpha^a\beta^b$ has degree $<4$, a contradiction. Hence $\alpha^a\beta^b$ is not a vertex. Meanwhile, \eqref{Eq-be,de-vertex-alga} implies $\alpha\beta\gamma\cdots=\alpha\beta^b\gamma^2$. Then $\beta>\alpha>\frac{1}{2}\pi$ and $2\gamma>\frac{2}{3}\pi$ determine $\alpha\beta^b\gamma^2=\alpha\beta\gamma^2$. However, $\alpha^2\gamma^2$ and $\alpha\beta\gamma^2$ imply $\alpha=\beta$, a contradiction. Hence $\alpha\beta\gamma\cdots$ is also not a vertex. Therefore, in reference to \eqref{Eq-vertex-albe}, we know that $\alpha\beta\cdots$ is not a vertex, contradicting Lemma \ref{Lem-albe-alga}.

Now we know $ \delta > \beta$. Then $\beta+\delta>2\alpha>\pi$ imply $\delta>\alpha>\frac{1}{2}\pi$. Now we have $2\gamma > \frac{2}{3}\pi \ge \delta>\alpha>\frac{1}{2}\pi$. By $ \delta>\frac{1}{2}\pi$, we get $\delta^d=\delta^3$. Meanwhile, $\pi>2\gamma> \frac{2}{3}\pi \ge \delta>\alpha>\frac{1}{2}\pi$ reduce  \eqref{Eq-be,de-vertex-gade} to $\gamma\delta\cdots = \gamma^2\delta^2, \alpha\gamma^2\delta$, and \eqref{Eq-vertex-de} becomes $\delta\cdots=\delta^3, \gamma^2\delta^2, \alpha\gamma^2\delta$. However, $\alpha^2\gamma^2$ and one of $\gamma^2\delta^2, \alpha\gamma^2\delta$ imply $\delta=\alpha$, a contradiction. Therefore $\delta\cdots=\delta^3$, which is now a vertex.

The vertex $\delta^3$ and $ \delta > \beta$ imply $\delta=\frac{2}{3}\pi > \beta$. Given $\delta\cdots=\delta^3$, it reduces \eqref{Eq-be,de-vertex-gaxga} to $\gamma \vert \gamma \cdots = \gamma^4, \alpha\gamma^4$. However, $\frac{1}{2}\pi>\gamma$ rules out $\gamma^4$, while $\alpha^2\gamma^2, \delta^3, \alpha\gamma^4$ imply $\alpha=\delta=\frac{2}{3}\pi$, contradicting $\delta>\alpha$. Hence $\gamma \vert \gamma \cdots$ is not a vertex and Lemma \ref{Lem-bebe-gaga-dede} implies that $\beta\vert\beta\cdots$ is also not a vertex.

From \eqref{Eq-vertex-albe} and \eqref{Eq-be,de-vertex-alga}, we deduce $\alpha\beta\cdots=\alpha^a\beta^b, \alpha\beta^b\gamma^2$. The kite angle sum and $\delta=\frac{2}{3}\pi$ and $\frac{1}{2}\pi>\gamma$ imply $\beta>\frac{1}{3}\pi$. Then $\beta>\frac{1}{3}\pi$ and $2\beta, 2\gamma>\frac{2}{3}\pi > \alpha>\frac{1}{2}\pi$ and $\alpha^2\gamma^2$ imply $\alpha\beta^b\gamma^2=\alpha\beta\gamma^2$. Meanwhile, $\frac{2}{3}\pi>\alpha>\frac{1}{2}\pi$ and $\frac{2}{3}\pi>\beta>\frac{1}{3}\pi$ and the absence of $\beta\vert\beta\cdots$ imply $\alpha^a\beta^b=\alpha^2\beta^2, \alpha^3\beta$. Hence $\alpha\beta\cdots=\alpha^2\beta^2, \alpha^3\beta, \alpha\beta\gamma^2$. Lemma \ref{Lem-albe-alga} implies that one of them is a vertex.

For $m=5$, the vertex angle sums of $\alpha^2\gamma^2, \delta^3$ and with one of $\alpha^2\beta^2, \alpha^3\beta$ yield no solution in \eqref{Eq-angle-id} for $\alpha\in(\frac{3}{5}\pi,\frac{2}{3}\pi)$, whereas with $\alpha\beta\gamma^2$ we get
\begin{align*}
&\alpha = (0.61901...)\pi,& 
&\beta = (0.61901...)\pi,&
&\gamma = (0.38098...)\pi,&
&\delta= \tfrac{2}{3}\pi.&
\end{align*}
The angle values and the absence of $\beta\vert\beta\cdots$ determine the vertices below
\begin{align*}
\AVC = \{ \delta^3, \alpha^2\gamma^2, \alpha\beta\gamma^2 \},
\end{align*}
which contradicts Counting Lemma on $\beta, \gamma$. Hence there is no tiling. 

For $m=4$, the vertex angle sums of $\alpha^2\gamma^2, \delta^3$ and one of $\alpha^2\beta^2, \alpha^3\beta, \alpha\beta\gamma^2$ and \eqref{Eq-angle-id} imply the following solutions for $\alpha\in(\frac{1}{2}\pi,\frac{2}{3}\pi)$:
\begin{align*}
&\alpha^2\beta^2:&
&\alpha=(0.52789...)\pi,&
&\beta =(0.47210...)\pi,&
&\gamma=(0.47210...)\pi,&
&\delta=\tfrac{2}{3}\pi;& \\
&\alpha^3\beta:&
&\alpha=(0.52133...)\pi,&
&\beta=(0.43598...)\pi,&
&\gamma=(0.47866...)\pi,&
&\delta=\tfrac{2}{3}\pi;& \\
&\alpha\beta\gamma^2:&
&\alpha=(0.53911...)\pi,&
&\beta=(0.53911...)\pi,&
&\gamma=(0.46088...)\pi,&
&\delta=\tfrac{2}{3}\pi.&
\end{align*}
The respective angle values with $\alpha^2\beta^2, \alpha\beta\gamma^2$ and without $\beta\vert\beta\cdots$ determine,
\begin{align*}
\AVC &= \{ \delta^3, \alpha^2\gamma^2, \alpha^2\beta^2 \}; \\
\AVC &= \{ \delta^3, \alpha^2\gamma^2, \alpha\beta\gamma^2\}.
\end{align*}
The second $\AVC$ (the same as the one with $m=5$) contradicts Counting Lemma on $\beta, \gamma$. In the first $\AVC$, we have $\alpha\beta\cdots=\alpha^2\beta^2$, and without $\beta\vert\beta\cdots$ the vertex has a unique angle arrangement $\vert \alpha \vert \beta \vert \alpha \vert \beta \vert$. Meanwhile, $\alpha^2\gamma^2 = \vert \alpha \vert \alpha \vert \gamma \, \bvert \, \gamma \vert$. Hence the same deduction of $\alpha^2\gamma^2$ in Figure \ref{Fig-be-AAD-albe2-al2ga2} leads to two $\beta$'s in the same tile, a contradiction. 

The corresponding angle values with $\alpha^3\beta$ and without $\beta\vert\beta\cdots$ determine 
\begin{align}\label{Eq-AVC-de3-al3be-al2ga2}
\AVC = \{ \delta^3, \alpha^3\beta, \alpha^2\gamma^2 \}.
\end{align}
Denoting $\beta$ by $\bullet$ and $\gamma^2$ by $\_$, the above $\AVC$ becomes $\{ \alpha^2\_, \alpha^3\bullet \}$. The same argument on Figures \ref{Fig-be-hex-de3} and \ref{Fig-be-Tilings-sq-hex-al2be-al3ga2} determines the same two tilings by squares and hexagons with $\AVC\equiv\{ \alpha^2\_, \alpha^3\bullet \}$. The two \ref{Label:quad-subdiv-thick-Octa} tilings (Figure \ref{Fig-Platonic-thicken-Octa}) with $\AVC$ \eqref{Eq-AVC-de3-al3be-al2ga2} are obtained by kite subdivision on the hexagons, from their centres to each $\_$, see in Figure \ref{Fig-be,de-Tilings-sq-hex-de3-al3be-al2ga2}. The subdivisions in Figure \ref{Fig-be,de-Tilings-sq-hex-de3-al3be-al2ga2} and those in Figure \ref{Fig-be-Tilings-sq-hex-al2be-al3ga2} are related via flips.

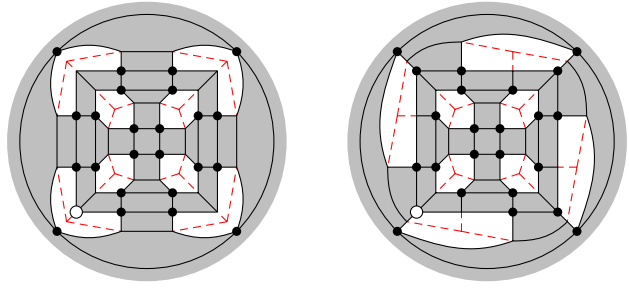
\begin{figure}[h!]
\centering
\begin{tikzpicture}

\tikzmath{
\s=1;
\ys=7;
\r=0.24;
\rr=0.25*\r;
\ps=360/3;
\X=\r*cos(0.5*\ps);
\ph=360/6;
\l=\r*cos(0.5*\ph);
\th=360/4;
\x=\r*cos(\th/2);
}

\begin{scope}[]

\fill[gray!50, scale=1.1] (0,0) circle (7*\r);

\fill[white] (0,0) circle (7*\r);

\fill[gray!50]
	(\x, \x) -- (-\x, \x) -- (-\x, -\x) -- (\x, -\x) -- cycle
;

\foreach \a in {0,1,2,3} {

\tikzset{rotate=\a*\th}

\fill[gray!50]
	(\x, \x) -- (\x, 3*\x) -- (2*\x, 4*\x) -- (2*\x, 7*\x) to[out=20,in=160] (7*\x, 7*\x) arc (45:135:7*\r) 
	to[out=20,in=160] (-2*\x, 7*\x) -- (-2*\x, 4*\x) -- (-\x, 3*\x) -- (-\x, \x) -- cycle
;

\fill[gray!50]
	(2*\x, 4*\x) -- (2*\x, 5.5*\x) -- (5.5*\x, 5.5*\x) 
	-- (5.5*\x, 2*\x) -- (4*\x, 2*\x) -- (4*\x, 4*\x) -- cycle
;

\draw[red, densely dashed]
	(2.5*\x, 2.5*\x) -- (\x, 3*\x)
	(2.5*\x, 2.5*\x) -- (3*\x, \x)
	(2.5*\x, 2.5*\x) -- (4*\x, 4*\x)
	(6.25*\x, 6.25*\x) -- (5.5*\x, 5.5*\x)
	(6.25*\x, 6.25*\x) to[] (2*\x, 7*\x)
	(6.25*\x, 6.25*\x) to[] (7*\x, 2*\x)
;

}


\foreach \aa in {0,1,2,3} {

\tikzset{shift={(\aa*\th:2*\x)}}

\foreach \a in {0,1,2,3} {

\tikzset{rotate=\a*\th}

\draw[]
	(0.5*\th:\r) -- (1.5*\th:\r)
;

}
}

\foreach \a in {0,1,2,3} {

\tikzset{rotate=\a*\th}

\draw[]
	(\x,3*\x) -- (2*\x, 4*\x)
	(3*\x,\x) -- (4*\x, 2*\x)
	(2*\x, 4*\x) -- (4*\x, 4*\x)
	(4*\x,2*\x) -- (4*\x, 4*\x)
	(2*\x, 4*\x) -- (-2*\x, 4*\x)
	(2*\x, 4*\x) -- (2*\x, 5.5*\x)
	(-2*\x, 4*\x) -- (-2*\x, 5.5*\x)
	(-2*\x, 5.5*\x) -- (2*\x, 5.5*\x)
	(4*\x, 4*\x) -- (5.5*\x, 5.5*\x)
	(2*\x, 5.5*\x) -- (5.5*\x, 5.5*\x)
	(-2*\x, 5.5*\x) -- (-5.5*\x, 5.5*\x)
	(2*\x, 5.5*\x) -- (2*\x, 7*\x)
	(2*\x, 7*\x) to[out=20,in=160] (7*\x, 7*\x)
	(5.5*\x, 2*\x) -- (7*\x,2*\x)
	(7*\x,2*\x) to[out=70,in=290] (7*\x, 7*\x)
	(-2*\x, 7*\x) -- (2*\x, 7*\x)
;

\fill (\x,\x) circle (\rr); 
\fill (2*\x, 4*\x) circle (\rr);
\fill (4*\x, 2*\x) circle (\rr);
\fill (2*\x, 5.5*\x) circle (\rr); 
\fill (-2*\x, 5.5*\x) circle (\rr); 
\fill (7*\x, 7*\x) circle (\rr); 

}

\draw (0,0) circle (7*\r);

\draw[fill=white] (-5.5*\x, -5.5*\x) circle (1.35*\rr);

\end{scope}

\begin{scope}[xshift=4.5*\s cm]

\fill[gray!50, scale=1.1] (0,0) circle (7*\r);

\fill[white] (0,0) circle (7*\r);

\fill[gray!50]
	(\x, \x) -- (-\x, \x) -- (-\x, -\x) -- (\x, -\x) -- cycle
;

\foreach \a in {0,1,2,3} {

\tikzset{rotate=\a*\th}

\fill[gray!50]
	(-2*\x, 5.5*\x) -- (-2*\x, 7.75*\x) to[out=20, in=160] (7*\x, 7*\x) arc (45:135:7*\r) -- (-5.5*\x, 5.5*\x) -- cycle
;

\fill[gray!50]
	(\x, \x) -- (\x, 3*\x) -- (2*\x, 4*\x) -- (2*\x, 5.5*\x) 
	-- (-2*\x, 5.5*\x) -- (-2*\x, 4*\x) -- (-\x, 3*\x) -- (-\x, \x) -- cycle
;

\fill[gray!50]
	(2*\x, 4*\x) -- (2*\x, 5.5*\x) -- (5.5*\x, 5.5*\x) 
	-- (5.5*\x, 2*\x) -- (4*\x, 2*\x) -- (4*\x, 4*\x)   -- cycle
;

\draw[red, densely dashed]
	(2.5*\x, 2.5*\x) -- (\x, 3*\x)
	(2.5*\x, 2.5*\x) -- (3*\x, \x)
	(2.5*\x, 2.5*\x) -- (4*\x, 4*\x)
	(2*\x, 7*\x) -- (6.25*\x, 6.25*\x)
	(2*\x, 7*\x) -- (-2*\x, 7.75*\x) 
	(2*\x, 7*\x) -- (2*\x, 5.5*\x)
;

}


\foreach \aa in {0,1,2,3} {

\tikzset{shift={(\aa*\th:2*\x)}}

\foreach \a in {0,1,2,3} {

\tikzset{rotate=\a*\th}

\draw[]
	(0.5*\th:\r) -- (1.5*\th:\r)
;

}
}

\foreach \a in {0,1,2,3} {

\tikzset{rotate=\a*\th}

\draw[]
	(\x,3*\x) -- (2*\x, 4*\x)
	(3*\x,\x) -- (4*\x, 2*\x)
	(2*\x, 4*\x) -- (4*\x, 4*\x)
	(4*\x,2*\x) -- (4*\x, 4*\x)
	(2*\x, 4*\x) -- (-2*\x, 4*\x)
	(2*\x, 4*\x) -- (2*\x, 5.5*\x)
	(-2*\x, 4*\x) -- (-2*\x, 5.5*\x)
	(-2*\x, 5.5*\x) -- (2*\x, 5.5*\x)
	(4*\x, 4*\x) -- (5.5*\x, 5.5*\x)
	(2*\x, 5.5*\x) -- (5.5*\x, 5.5*\x)
	(-2*\x, 5.5*\x) -- (-5.5*\x, 5.5*\x)
	(5.5*\x, 5.5*\x) -- (7*\x, 7*\x)
	(-2*\x, 5.5*\x) -- (-2*\x, 7.75*\x)
	(-2*\x, 7.75*\x) to[out=180, in=45] (-6.25*\x, 6.25*\x)
	(-2*\x, 7.75*\x) to[out=20, in=160] (7*\x, 7*\x)
;

\fill (\x,\x) circle (\rr); 
\fill (2*\x, 4*\x) circle (\rr);
\fill (4*\x, 2*\x) circle (\rr);

\fill (-2*\x, 5.5*\x) circle (\rr); 
\fill (5.5*\x, 5.5*\x) circle (\rr); 
\fill (7*\x,7*\x) circle (\rr); 

}

\draw (0,0) circle (7*\r);

\draw[fill=white] (-5.5*\x, -5.5*\x) circle (1.35*\rr);

\end{scope}

\end{tikzpicture}
\caption{The tilings by squares and hexagons with $\AVC = \{ \alpha^2\_, \alpha^3\bullet \}$, via kite subdivision of the hexagons, determine \ref{Label:quad-subdiv-thick-Octa} tilings with $\AVC = \{ \delta^3, \alpha^3\beta, \alpha^2\gamma^2 \}$}
\label{Fig-be,de-Tilings-sq-hex-de3-al3be-al2ga2}
\end{figure}

The angles and edges are determined as follows. The vertices in $\AVC$ \eqref{Eq-AVC-de3-al3be-al2ga2} imply 
\begin{align*}
&\beta = 2\pi - 3\alpha,&
&\gamma = \pi - \alpha,&
&\delta= \tfrac{2}{3}\pi.&
\end{align*}
Substituting them into \eqref{Eq-angle-id} determines
\begin{align}
&\alpha=\cos^{-1}\tfrac{1}{4}(\sqrt3-2),&
&\beta=2\cos^{-1}\tfrac{1}{8}(1+3\sqrt3),& \\ \notag
&\gamma=\cos^{-1}\tfrac{1}{4}(2-\sqrt3),&
&\delta = \tfrac{2}{3}\pi.&
\end{align}
Furthermore, the angle values and \eqref{Eq-a4-cosx} and \eqref{Eq-x2y2-cosy} imply
\begin{align}
&x=2\tan^{-1} \tfrac{1}{2\sqrt2}(\sqrt{3}-1),&
&y=\tan^{-1}(2\sqrt2-\sqrt6).&
\end{align}
The unique solution for $y \in (0,\pi)$ is due to $\beta<\delta$, which implies $y>x$.
\end{case*}

\begin{case*}[$\alpha\gamma^4$] The vertex angle sum of $\alpha\gamma^4$ and $\alpha>\frac{1}{2}\pi$ imply $2\gamma<\frac{3}{4}\pi$. Then the kite angle sum implies $\beta+\delta > \frac{5}{4}\pi$. Combined with $\frac{2}{3}\pi \ge \beta, \delta$, we further get $\beta, \delta > \frac{7}{12}\pi > \frac{1}{2}\pi$. Recall $2\gamma>\frac{2}{3}\pi$ from the beginning of the proof. Then $\alpha>\frac{1}{2}\pi$ and $\alpha\gamma^4$ imply $\frac{2}{3}\pi>\alpha$. Hence \eqref{Eq-mgon-sum} deduces $m=4,5$.

Given $\gamma\delta\cdots=\gamma^2\delta^{d\ge2}, \alpha\gamma^2\delta^{d}$ from \eqref{Eq-be,de-vertex-gade}, the vertex angle sum of $\alpha\gamma^4$ and $\alpha>\frac{1}{2}\pi$ and $2\gamma>\frac{2}{3}\pi \ge \delta >\frac{1}{2}\pi$ rule out $\alpha\gamma^2\delta^{d}$. Hence $\delta\cdots=\delta^3, \gamma^2\delta^2$.

The vertex $\alpha\gamma^4$ has a unique angle arrangement $\vert \alpha \vert \gamma \, \bvert \, \gamma \vert \gamma \, \bvert \, \gamma \vert$. Then Lemma \ref{Lem-bebe-gaga-dede} implies that $\beta\vert\beta \cdots$ is a vertex. Now \eqref{Eq-vertex-be} and \eqref{Eq-be,de-vertex-alga} imply $\beta\vert\beta \cdots=\beta^{b\ge3}, \alpha^a\beta^{b\ge2}, \beta^{b\ge2}\gamma^c, \alpha\beta^{b\ge2}\gamma^2$. Next, $\alpha,\beta>\frac{1}{2}\pi$ and $2\gamma>\frac{2}{3}\pi$ deduce $\beta^b=\beta^3$ and $\beta^{b\ge2}\gamma^c=\beta^2\gamma^2$, and dismiss $\alpha\beta^{b\ge2}\gamma^2$. The inequalities $\frac{2}{3}\pi>\alpha$ and $\frac{2}{3}\pi\ge\beta$ imply that $\alpha^a\beta^b$ has degree $\ge4$; while $\alpha,\beta>\frac{1}{2}\pi$ imply that $\alpha^a\beta^b$ has degree $<4$. Then $\alpha^a\beta^b$ is also not a vertex. Hence $\beta\vert\beta \cdots=\beta^3, \beta^2\gamma^2$.

If $\beta^3$ is a $\beta\vert\beta \cdots$ vertex, then $\beta=\frac{2}{3}\pi$, and $\beta\neq\delta$ rules out $\delta^3$, resulting in $\delta\cdots = \gamma^2\delta^2$. For $m=4,5$, the vertex angle sums of $\alpha\gamma^4, \beta^3, \gamma^2\delta^2$ and \eqref{Eq-angle-id} imply
\begin{align*}
&m=4:&
&\alpha=(0.51032...)\pi,&
&\beta = \tfrac{2}{3}\pi,&
&\gamma=(0.37241...)\pi,&
&\delta=(0.62758...)\pi;& \\
&m=5:&
&\alpha=(0.60585...)\pi,&
&\beta=\tfrac{2}{3}\pi,&
&\gamma=(0.34853...)\pi,&
&\delta=(0.65146...)\pi.& 
\end{align*}
The angle values determine the vertices below
\begin{align*}
\AVC = \{  \beta^3, \gamma^2\delta^2, \alpha\gamma^4  \},
\end{align*}
where $\alpha\beta\cdots$ is not a vertex, contradicting Lemma \ref{Lem-albe-alga}.

If $\beta^2\gamma^2$ is a $\beta\vert\beta \cdots$ vertex, then $\beta\neq\delta$ excludes $\gamma^2\delta^2$, resulting in $\delta\cdots =\delta^3$. For $m=4,5$, the vertex angle sums of $\alpha\gamma^4, \beta^2\gamma^2, \delta^3$ and \eqref{Eq-angle-id} imply
\begin{align*}
&m=4:&
&\alpha=(0.51114...)\pi,&
&\beta=(0.62778...)\pi,&
&\gamma=(0.37221...)\pi,&
&\delta=\tfrac{2}{3}\pi;& \\
&m=5:&
&\alpha=(0.60601...)\pi,&
&\beta=(0.65150...)\pi,&
&\gamma=(0.34849...)\pi,&
&\delta=\tfrac{2}{3}\pi.& 
\end{align*}
The angle values determine the vertices below
\begin{align*}
\AVC = \{ \delta^3, \beta^2\gamma^2, \alpha\gamma^4 \},
\end{align*}
where $\alpha\beta\cdots$ is again not a vertex, contradicting Lemma \ref{Lem-albe-alga}.

Therefore there is no tiling for this case.
\end{case*}

\begin{case*}[$\alpha\beta^b\gamma^2$] Based on the previous discussion, we may assume that $\alpha\gamma^2$, $\alpha^2\gamma^2, \alpha\gamma^4$ are not vertices. Then \eqref{Eq-be,de-vertex-alga} becomes $\alpha\gamma\cdots=\alpha\beta^b\gamma^2, \alpha\gamma^2\delta^d$.

The kite angle sum and the vertex angle sum of $\alpha\beta^b\gamma^2$ imply $\delta>\alpha+(b-1)\beta$. Given $b\ge 1$, it implies $\frac{2}{3}\pi \ge \delta>\beta$ and $\frac{2}{3}\pi \ge \delta > \alpha >\frac{1}{2}\pi$. Then $\frac{2}{3}\pi  > \alpha > (1-\tfrac{2}{m})\pi$ deduce $m=4,5$.


By \eqref{Eq-be,de-vertex-gade}, we know $\gamma\delta\cdots=\gamma^2\delta^d, \alpha\gamma^2\delta^d$. Then \eqref{Eq-vertex-de} becomes $\delta\cdots=\delta^d, \gamma^2\delta^{d\ge2}, \alpha\gamma^2\delta^d$. By $\delta > \alpha >\frac{1}{2}\pi$ and $2\gamma>\frac{2}{3}\pi$ and Lemma \ref{Lem-bega2-ga2de}, we further deduce $\delta^d=\delta^3$ and $\gamma^2\delta^d= \gamma^2\delta^2$ and $\alpha\gamma^2\delta^d=\alpha\gamma^2\delta$. Hence $\delta\cdots =\delta^3, \gamma^2\delta^2, \alpha\gamma^2\delta$. Moreover, the assumption of $\alpha\beta^b\gamma^2$ and Counting Lemma imply that one of $\delta^3, \gamma^2\delta^2$ is a vertex. Indeed, if $\delta\cdots=\alpha\gamma^2\delta$, then Counting Lemma on $\gamma, \delta$ implies $\gamma\cdots=\alpha\gamma^2\delta$, contradicting $\alpha\beta^b\gamma^2$ being a vertex.

The exact form of $\alpha\beta^b\gamma^2$ will be determined in the subcases: $2\gamma\le\pi$ and $2\gamma>\pi$.


\begin{subcase*}[$2\gamma\le\pi$] The kite angle sum and $\pi\ge2\gamma$ and $\frac{2}{3}\pi \ge \delta$ imply $\beta+\pi + \frac{2}{3}\pi \ge \beta+2\gamma+\delta >2\pi$, which implies $\beta>\frac{1}{3}\pi$. Combined with $\alpha>\frac{1}{2}\pi$, it gives $\alpha+\beta>\frac{5}{6}\pi > \frac{2}{3}\pi \ge \delta$. Then the kite angle sum implies $\alpha+2\beta+2\gamma > \beta + 2\gamma + \delta > 2\pi$. Hence $\alpha\beta^b\gamma^2=\alpha\beta\gamma^2$. Given $\beta\neq\delta$ and $\alpha\gamma^2\delta^d=\alpha\gamma^2\delta$, it further rules out $\alpha\gamma^2\delta$. Therefore \eqref{Eq-be,de-vertex-alga} becomes $\alpha\gamma\cdots=\alpha\beta\gamma^2$, and $\delta\cdots=\delta^3, \gamma^2\delta^2$. We proceed with one of them being a vertex.


\begin{subsubcase*}[$\alpha\beta\gamma^2, \delta^3$] The vertex angle sum of $\delta^3$ and $\delta>\alpha, \beta$ imply $\delta=\frac{2}{3}\pi > \alpha, \beta$. Then $\beta^b, \alpha^a\beta^b$ have degrees $\ge4$. By $\beta>\frac{1}{3}\pi$, we get $\beta^b=\beta^4, \beta^5$. By $\beta>\frac{1}{3}\pi$ and $\alpha>\frac{1}{2}\pi$, we also have $\alpha^a\beta^b = \alpha^3\beta, \alpha^2\beta^2, \alpha\beta^3, \alpha\beta^4$. Moreover, $3\beta>\pi\ge2\gamma$ and $\alpha\beta\gamma^2$ imply $\alpha+4\beta>\alpha+\beta+2\gamma=2\pi$, ruling out $\alpha\beta^4$. Then $\alpha^a\beta^b=\alpha^3\beta, \alpha^2\beta^2, \alpha\beta^3$. By $2\beta>\frac{2}{3}\pi > \alpha$ and $\alpha\beta\gamma^2$, we get $3\beta+2\gamma> \alpha+\beta+2\gamma= 2\pi$, which implies $\beta^b\gamma^2=\beta^2\gamma^2$. Therefore \eqref{Eq-vertex-be} becomes $\beta\cdots= \beta^4, \beta^5, \alpha\beta^3, \alpha^2\beta^2, \alpha^3\beta, \beta^2\gamma^2, \alpha\beta\gamma^2$.


For $m=5$, the vertex angle sums of $\alpha\beta\gamma^2, \delta^3$ with one of $\beta^4, \beta^5, \alpha\beta^3, \alpha^2\beta^2, \beta^2\gamma^2$ in \eqref{Eq-angle-id} yield solutions for the angles below, while with $\alpha^3\beta$ there is no suitable solutions.
\begin{align*}
&\beta^4:&
&\alpha=(0.62050...)\pi,&
&\beta=\tfrac{1}{2}\pi,&
&\gamma=(0.43974...)\pi,&
&\delta=\tfrac{2}{3}\pi;& \\
&\beta^5:&
&\alpha=(0.62307...)\pi,&
&\beta=\tfrac{2}{5}\pi,&
&\gamma=(0.48846...)\pi,&
&\delta=\tfrac{2}{3}\pi;& \\
&\alpha\beta^3:&
&\alpha=(0.62137...)\pi,&
&\beta=(0.45954...)\pi,&
&\gamma=(0.45954...)\pi,&
&\delta=\tfrac{2}{3}\pi;& \\
&\alpha^2\beta^2:&
&\alpha=(0.62392...)\pi,&
&\beta=(0.37607...)\pi,&
&\gamma=\tfrac{1}{2}\pi,&
&\delta=\tfrac{2}{3}\pi;& \\
&\beta^2\gamma^2:&
&\alpha=(0.61901...)\pi,&
&\beta=(0.61901...)\pi,&
&\gamma=(0.38098...)\pi,&
&\delta=\tfrac{2}{3}\pi.&
\end{align*}

For $m=4$, the vertex angle sums of $\alpha\beta\gamma^2, \delta^3$ and one of $\beta^4, \beta^5, \alpha\beta^3, \alpha^2\beta^2$, $\alpha^3\beta, \beta^2\gamma^2$ and \eqref{Eq-angle-id} imply
\begin{align*}
&\beta^4:&
&\alpha=(0.54088...)\pi,&
&\beta=\tfrac{1}{2}\pi,&
&\gamma = (0.47955...)\pi,&
&\delta = \tfrac{2}{3}\pi;& \\
&\beta^5:&
&\alpha=(0.54769...)\pi,&
&\beta=\tfrac{2}{5}\pi,&
&\gamma=(0.52615...)\pi,&
&\delta=\tfrac{2}{3}\pi;&\\
&\alpha\beta^3:&
&\alpha=(0.54162...)\pi,&
&\beta=(0.48612...)\pi,&
&\gamma=(0.48612...)\pi,&
&\delta = \tfrac{2}{3}\pi;& \\
&\alpha^2\beta^2:&
&\alpha=(0.54339..)\pi,&
&\beta=(0.45660...)\pi,&
&\gamma=\tfrac{1}{2}\pi,&
&\delta=\tfrac{2}{3}\pi;&\\
&\alpha^3\beta:&
&\alpha=(0.55421...)\pi,&
&\beta=(0.33735...)\pi,&
&\gamma=(0.55421...)\pi,&
&\delta=\tfrac{2}{3}\pi;&\\
&\beta^2\gamma^2:&
&\alpha=(0.53911...)\pi,&
&\beta=(0.53911...)\pi,&
&\gamma=(0.46088...)\pi,&
&\delta=\tfrac{2}{3}\pi.&
\end{align*}
Without $\alpha^2\gamma^2$ (the case assumption), the above angle values determine the $\AVC$'s below. Other than the one with $\alpha^2\beta^2$, they will be treated collectively.
\begin{align*}
&m=4, 5,& \AVC = \, &\{ \delta^3, \alpha\beta\gamma^2, \beta^4 \}, \{ \delta^3, \alpha\beta\gamma^2, \beta^5 \}, \{ \delta^3, \alpha\beta\gamma^2, \alpha\beta^3 \},& \\ 
&&&\{ \delta^3, \alpha\beta\gamma^2, \alpha^2\beta^2, \gamma^4 \}, \{ \delta^3, \alpha\beta\gamma^2, \beta^2\gamma^2 \};& \\
&m=4,& \AVC = \, &\{ \delta^3, \alpha\beta\gamma^2, \alpha^3\beta \}.&
\end{align*}

Among $\{ \delta^3, \alpha\beta\gamma^2, \beta^4 \}, \{ \delta^3, \alpha\beta\gamma^2, \beta^5 \}, \{ \delta^3, \alpha\beta\gamma^2, \alpha\beta^3 \}$ and $\{ \delta^3$, $\alpha\beta\gamma^2, \beta^2\gamma^2 \}$, we see $\beta\vert\beta \cdots=\beta^4, \beta^5, \alpha\beta^3, \beta^2\gamma^2$ respectively, while the absence of $\gamma\vert\gamma\cdots$ contradicts Lemma \ref{Lem-bebe-gaga-dede}. 

It remains to discuss $\AVC = \{ \delta^3, \alpha\beta\gamma^2, \alpha^2\beta^2, \gamma^4 \}$, where $\alpha^2\cdots=\beta^2\cdots=\alpha^2\beta^2$ and $\alpha\gamma\cdots=\alpha\beta\gamma^2$. Assume that $\gamma^4$ appears as a vertex. It has a unique angle arrangement $\vert \gamma \, \bvert \, \gamma \vert \gamma \, \bvert \, \gamma \vert$ and Lemma \ref{Lem-bebe-gaga-dede} implies that $\beta\vert\beta\cdots$ is also a vertex. Hence an $\alpha^2\beta^2$ has an angle arrangement $\vert \alpha \vert \alpha \vert \beta \vert \beta \vert$ which determines tiles $T_1, T_2, T_3, T_4$ in both pictures of Figure \ref{Fig-bede-albebga2-AAD-de3-al2be2} for $m=4,5$. For $m=4$, the bottom $\alpha_1\vert\alpha_2\cdots$ in Figure \ref{Subfig-bede-albebga2-AAD-de3-al2be2-m=4} are $\alpha^2\cdots=\alpha^2\beta^2$ and determines $T_5, T_6$. Since $\alpha_2\gamma_3\cdots, \alpha_2\gamma_5\cdots=\alpha\beta\gamma^2$, their unique angle arrangements result in two $\beta$'s in $T_7$, a contradiction. The deduction argument for $m=5$ is analogous and results in $\bvert \, \gamma_8 \vert \alpha_2 \vert \gamma_9 \, \bvert \, \cdots$ in Figure \ref{Subfig-bede-albebga2-AAD-de3-al2be2-m=5}, which is not a vertex, also a contradiction. Therefore both $\alpha^2\beta^2, \gamma^4$ are not vertices. 

\begin{figure}[h!] 
\centering
\begin{subfigure}[t]{0.4\linewidth}
\centering
\begin{tikzpicture}[>=latex]

\tikzmath{
\s=5;
\r=0.625;
\th=360/4;
\ph=360/5;
}

\begin{scope}[]

\tikzmath{
\x=\r*cos(0.5*\th);
\R = sqrt(\x^2+(3*\x)^2);
\aR = acos(3*\x/\R);
}

\foreach \aaa in {-1, 0, 1} {

\tikzset{shift={(0, \aaa*2*\x)}}

\foreach \aa in {-1, 1} {

\tikzset{shift={(\aa*\x, 0)}}

\foreach \a in {0,...,3} {

\draw[rotate=\th*\a]
	(0.5*\th:\r) -- (1.5*\th:\r)
;

}
}
}

\foreach \bb in {-1, 1} {

\tikzset{shift={(0, \bb*2*\x)}, yscale=\bb}

\foreach \b in {-1, 1} {

\tikzset{shift={(\b*\x, 0)}, xscale=\b}

\draw[line width=1.5]
	(0.5*\th:\r) -- (1.5*\th:\r)
	(0.5*\th:\r) -- (3.5*\th:\r)
;

}
}

\foreach \aa in {-1, 1} {

\tikzset{shift={(\aa*\x, 0)}}

\foreach \a in {0,...,3} {

\node at (0.5*\th+\a*\th: 0.65*\r) {\small $\alpha$};

}

}

\draw[]
	(2*\x, \x) -- (3*\x, \x)
	(2*\x, -\x) -- (3*\x, -\x)
;

\foreach \aaa in {-1, 1} {

\tikzset{shift={(0, \aaa*2*\x)}, yscale=\aaa}

\foreach \aa in {-1, 1} {

\tikzset{shift={(\aa*\x, 0)}, xscale=\aa}

\foreach \a in {1,3} {

\node at (0.5*\th+\a*\th: 0.6*\r) {\small $\gamma$};

}

\node at (2.5*\th: 0.6*\r) {\small $\beta$};
\node at (0.5*\th: 0.6*\r) {\small $\delta$};

}
}

\foreach \c in {-1, 1} {

\node at (2.4*\x, \c*0.6*\x) {\small $\beta$}; 

\node at (2.4*\x, \c*1.4*\x) {\small $\gamma$}; 

}

\node[inner sep=1,draw,shape=circle] at (-\x,0) {\footnotesize $1$};
\node[inner sep=1,draw,shape=circle] at (\x,0) {\footnotesize $2$};
\node[inner sep=1,draw,shape=circle] at (\x,2*\x) {\footnotesize $3$};
\node[inner sep=1,draw,shape=circle] at (-\x,2*\x) {\footnotesize $4$};
\node[inner sep=1,draw,shape=circle] at (\x,-2*\x) {\footnotesize $5$};
\node[inner sep=1,draw,shape=circle] at (-\x,-2*\x) {\footnotesize $6$};
\node[inner sep=1,draw,shape=circle] at (3*\x,0) {\footnotesize $7$};

\end{scope}

\end{tikzpicture}
\caption{$m=4$}
\label{Subfig-bede-albebga2-AAD-de3-al2be2-m=4}
\end{subfigure}
\begin{subfigure}[t]{0.4\linewidth}
\centering
\begin{tikzpicture}[>=latex]

\tikzmath{
\s=5;
\r=0.625;
\th=360/4;
\ph=360/5;
}

\begin{scope}[]

\tikzmath{
\x=\r*cos(\ph/2);
\y=\r*sin(\ph/2);
}

\foreach \aa in {-1,1} {

\tikzset{xshift=\aa*\x cm, xscale=-\aa}

\foreach \a in {0,...,4} {

\draw[rotate=\a*\ph]
	(0.5*\ph:\r) -- (-0.5*\ph:\r)
;

\node at (0.5*\ph+\a*\ph:0.65*\r) {\small $\alpha$};

}
}

\foreach \aa in {-1,1} {

\draw[yscale=\aa, line width=1.5]
	(\x+\r,0) -- (\x+2*\r, \y) 
;

\foreach \bb in {-1,1} {

\draw[yscale=\bb]
	(0,\y) -- (0,\y+1.25*\r)
	([shift={(\x,0)}]\ph:\r) -- (35:2.5*\r) 
;

\draw[xscale=\aa, yscale=\bb, line width=1.5]
	(0,\y+1.25*\r) -- (1.25*\r,2*\y+1.25*\r)
	([shift={(\x,0)}]\ph:\r) -- (1.25*\r,2*\y+1.25*\r)
;

}
}

\foreach \aa in {-1,1} {

\foreach \bb in {-1,1} {

\tikzset{xscale=\aa, yscale=\bb}

\node at (0.3*\x,1.15*\x) {\small $\beta$};
\node at (0.25*\x,2*\x) {\small $\gamma$};
\node at (1.125*\x,1.4*\x) {\small $\gamma$};
\node at (1.15*\x,2.45*\x) {\small $\delta$};


}

\tikzset{yscale=\aa}

\node at (1.85*\x,1*\x) {\small $\beta$};
\node at (2.25*\x,0.45*\x) {\small $\gamma$};

}

\node at (1.75*\x,1.75*\x) {\small $\gamma$};
\node at (3*\x,0) {\small $?$};

\node[inner sep=1,draw,shape=circle] at (-\x,0) {\footnotesize $1$};
\node[inner sep=1,draw,shape=circle] at (\x,0) {\footnotesize $2$};
\node[inner sep=1,draw,shape=circle] at (0.75*\x,1.8*\x) {\footnotesize $3$};
\node[inner sep=1,draw,shape=circle] at (-0.75*\x,1.8*\x) {\footnotesize $4$};
\node[inner sep=1,draw,shape=circle] at (0.75*\x,-1.8*\x) {\footnotesize $5$};
\node[inner sep=1,draw,shape=circle] at (-0.75*\x,-1.8*\x) {\footnotesize $6$};

\node[inner sep=1,draw,shape=circle] at (2.5*\x,\x) {\small $8$};
\node[inner sep=1,draw,shape=circle] at (2.5*\x,-\x) {\small $9$};

\end{scope}

\end{tikzpicture}
\caption{$m=5$}
\label{Subfig-bede-albebga2-AAD-de3-al2be2-m=5}
\end{subfigure}
\caption{The deduction of $\alpha^2\beta^2=\vert\alpha\vert\alpha\vert\beta\vert\beta\vert$}
\label{Fig-bede-albebga2-AAD-de3-al2be2}
\end{figure}
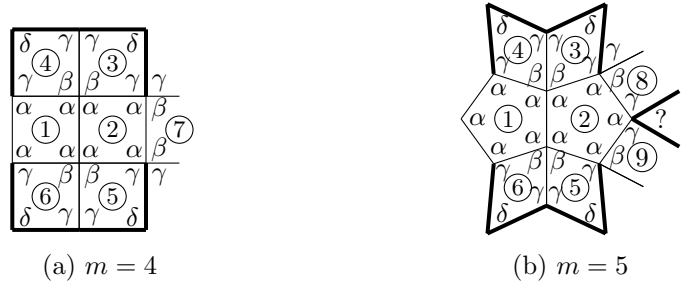

In summary, none of $\beta^4, \beta^5, \alpha\beta^3, \alpha^2\beta^2, \alpha^3\beta, \beta^2\gamma^2$ is a vertex, and \eqref{Eq-vertex-be} becomes $\beta\cdots=\alpha\beta\gamma^2$. Then Counting Lemma on $\beta,\gamma$ deduces $\gamma\cdots=\alpha\beta\gamma^2$ and the same lemma on $\gamma,\delta$ further implies $\delta\cdots=\delta^3$. By $\frac{2}{3}\pi > \alpha > \frac{1}{2}\pi$ and the above, we conclude $\alpha\cdots=\alpha\beta\gamma^2$ and obtain the same $\AVC$ \eqref{AVC-de3-albega2}. The tilings are the canonical seeds of the \ref{Label:trunc-Octa}, \ref{Label:trunc-Icosa} families.
\end{subsubcase*}

\begin{subsubcase*}[$\alpha\beta\gamma^2, \gamma^2\delta^2$] 

The vertex $\gamma^2\delta^2$ has a unique angle arrangement $\bvert \, \delta \, \bvert \, \gamma \vert \gamma \, \bvert \, \delta \, \bvert$. Then Lemma \ref{Lem-bebe-gaga-dede} implies that $\beta\vert\beta\cdots$ is also a vertex.

Recall from the case that $\frac{2}{3}\pi \ge \delta > \alpha >\frac{1}{2}\pi$. Then $\gamma^2\delta^2$ implies $\gamma+\delta=\pi$ and $\gamma<\frac{1}{2}\pi$. Combined with the kite angle sum, we further get $\beta>\frac{1}{2}\pi$. By $\frac{2}{3}>\alpha>\frac{1}{2}\pi$ and $\frac{2}{3}\pi \ge \beta > \frac{1}{2}\pi$, we rule out $\alpha^a\beta^b$. By $\beta>\frac{1}{2}\pi$, we deduce $\beta^b=\beta^3$. Combined with $\alpha\beta\gamma^2$ and \eqref{Eq-vertex-be}, we get $\beta^2\cdots=\beta^3$. Hence $\beta\vert\beta\cdots=\beta^3$, which is a vertex.

By the vertex angle sums of $\alpha\beta\gamma^2, \gamma^2\delta^2, \beta^3$ and \eqref{Eq-angle-id}, we determine
\begin{align*}
&\alpha=(0.51810...)\pi,&
&\beta=\tfrac{2}{3}\pi,&
&\gamma=(0.40761...)\pi,&
&\delta=(0.59238...)\pi.& 
\end{align*}
The angle values determine
\begin{align*}
\AVC = \{ \beta^3, \gamma^2\delta^2, \alpha\beta\gamma^2 \}.
\end{align*}
The unique angle arrangement of a $\beta^3$ determines tiles $T_1, T_2, T_3$ in Figure \ref{Fig-be,de-AAD-albega2-ga2de2-be3}. Then $\gamma_2\gamma_3\cdots, \gamma_1\gamma_3\cdots=\gamma^2\delta^2$ result in $\vert \gamma_4 \, \bvert \, \delta_3 \, \bvert \, \gamma_5 \vert \cdots$, which is not admissible, a contradiction.

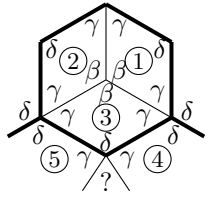
\begin{figure}[h!] 
\centering
\begin{tikzpicture}

\tikzmath{
\r=1;
\th=360/3;
\x=\r*cos(0.5*\th);
\ph=360/6;
}

\foreach \a in {0,1,2} {
\tikzset{rotate=\th*\a}
\draw[]
	(0:0) -- (90:\r) 
;
}

\foreach \aa in {-1,1} {
\tikzset{xscale=\aa}
\draw[]
	%
	%
	(270:\r) -- (270-0.1*\th:1.5*\r)
;
\draw[line width=1.5]
;
}

\foreach \b in {0,...,5} {
\tikzset{rotate=\b*\ph}
\draw[line width=1.5]
	(90:\r) -- (90+\ph:\r)
;
}

\foreach \b in {1,2} {
\tikzset{rotate=2*\b*\ph}
\draw[line width=1.5]
	(90:\r) -- (90:\r+\x)
;
}

\foreach \a in {0,1,2} {
\tikzset{rotate=\th*\a}
\node at ([shift={(0,-\x)}]90:0.6*\x) {\small $\beta$}; 
\node at ([shift={(0,-\x)}]270:0.65*\x) {\small $\delta$};
\node at ([shift={(0,-\x)}]0:1*\x) {\small $\gamma$};
\node at ([shift={(0,-\x)}]180:1*\x) {\small $\gamma$};
}

\foreach \a in {1,2} {
\tikzset{rotate=\th*\a}
\node at (81:1.15*\r) {\small $\delta$};
\node at (99:1.15*\r) {\small $\delta$};
}

\foreach \aa in {-1,1} {
\tikzset{xscale=\aa}


\node at (270-0.125*\th:1.1*\r) {\small $\gamma$}; 
}

\node at (270:1.35*\r) {\small $?$};



\node[inner sep=1,draw,shape=circle] at (90-0.5*\th:0.5*\r) {\footnotesize $1$};
\node[inner sep=1,draw,shape=circle] at (90+0.5*\th:0.5*\r) {\footnotesize $2$};
\node[inner sep=1,draw,shape=circle] at (270:0.5*\r) {\footnotesize $3$};
\node[inner sep=1,draw,shape=circle] at (90-2.45*\ph:1.25*\r) {\footnotesize $4$};
\node[inner sep=1,draw,shape=circle] at (90+2.45*\ph:1.25*\r) {\footnotesize $5$};

\end{tikzpicture}
\caption{The deduction of $\beta^3$}
\label{Fig-be,de-AAD-albega2-ga2de2-be3}
\end{figure}
\end{subsubcase*}
\end{subcase*}

\begin{subcase*}[$2\gamma>\pi$] Recall $\alpha, \delta > \frac{1}{2}\pi$ from the case. Then $\alpha, \delta > \frac{1}{2}\pi$ and $2\gamma>\pi$ dismiss $\gamma^4, \alpha\gamma^4,\alpha\gamma^2\delta^d$. Hence \eqref{Eq-be,de-vertex-alga} becomes $\alpha\gamma\cdots=\alpha\beta^b\gamma^2$, and \eqref{Eq-be,de-vertex-gaxga} becomes $\gamma \vert \gamma\cdots=\gamma^2\delta^{d\ge2}$, and \eqref{Eq-vertex-de} becomes $\delta\cdots =\delta^3$. Hence $\delta^3$ is a vertex. Given $\beta\neq\delta$, it implies $\delta=\tfrac{2}{3}\pi > \beta$. Moreover, $\delta=\tfrac{2}{3}\pi$ and $2\gamma>\pi$ dismiss $\gamma^2\delta^{d\ge2}$. Now $\gamma \vert \gamma\cdots$ is not a vertex, and Lemma \ref{Lem-bebe-gaga-dede} implies that $\beta\vert\beta\cdots$ is also not a vertex. Without $\gamma^2\delta^{d\ge2}, \alpha\gamma^2\delta^d$, list \eqref{Eq-be,de-vertex-gade} rules out $\gamma\delta\cdots$. 

The absence of $\beta\vert\beta\cdots$ implies that $\beta^b, \beta^{b\ge2}\gamma^2$ are not vertices and $\alpha\beta^b\gamma^2 = \alpha\beta\gamma^2, \alpha\beta^2\gamma^2$.

Meanwhile, $\frac{2}{3}\pi > \alpha, \beta$ imply that $\alpha^a\beta^b$ has degree $\ge4$. Then $\alpha>\frac{1}{2}\pi$ and no $\beta\vert\beta\cdots$ deduce $\alpha^a\beta^b=\alpha^2\beta^2, \alpha^3\beta, \alpha^3\beta^2, \alpha^3\beta^3$. Hence \eqref{Eq-vertex-be} becomes $\beta\cdots=\alpha^2\beta^2, \alpha^3\beta, \alpha^3\beta^2, \alpha^3\beta^3, \alpha\beta\gamma^2, \alpha\beta^2\gamma^2$.

The absence of $\gamma\vert\gamma\cdots$ implies that $\gamma^c$ is not a vertex. The knowledge of $\beta\cdots$, $\delta\cdots$ and $\alpha\gamma\cdots$ above determine $\gamma\cdots=\alpha\beta\gamma^2, \alpha\beta^2\gamma^2$. These two vertices are mutually exclusive.

If $\alpha\beta^2\gamma^2$ is a vertex, then $\gamma\cdots=\alpha\beta^2\gamma^2$. The kite angle sum and $\alpha\beta^2\gamma^2$ imply $\frac{2}{3}\pi=\delta > \alpha+\beta$. We have $2\pi>3\alpha+3\beta$, which dismiss $\alpha^2\beta^2, \alpha^3\beta, \alpha^3\beta^2, \alpha^3\beta^3$. Hence $\beta\cdots=\gamma\cdots=\alpha\beta^2\gamma^2$, contradicting Counting Lemma on $\beta, \gamma$.

Therefore $\gamma\cdots=\alpha\beta\gamma^2$. Counting Lemma on $\beta, \gamma$ implies that $\beta\cdots=\alpha\beta\gamma^2$. Combined with $\delta\cdots=\delta^3$ and $\delta=\frac{2}{3}\pi>\alpha>\frac{1}{2}\pi$ and the above, we get $\alpha\cdots=\alpha\beta\gamma^2$. We again obtain $\AVC$ \eqref{AVC-de3-albega2}, which results in the canonical seeds of the \ref{Label:trunc-Octa}, \ref{Label:trunc-Icosa} families.
\end{subcase*}
\end{case*}

\begin{case*}[$\alpha\gamma^2\delta^d$] Based on the previous discussion, we may assume $\alpha\gamma\cdots=\alpha\gamma^2\delta^d$. The kite angle sum and $\alpha\gamma^2\delta^d$ imply $\beta>\alpha+(d-1) \delta$. For $d \ge 1$, it implies $\frac{2}{3}\pi \ge \beta > \alpha >\frac{1}{2}\pi$. Then $\alpha^a\beta^b$ is not a vertex. Without $\alpha^a\beta^b, \alpha\beta^b\gamma^2$, list \eqref{Eq-vertex-be} becomes $\beta\cdots=\beta^b, \beta^b\gamma^c$. Then $\alpha\beta\cdots$ is not a vertex, contradicting Lemma \ref{Lem-albe-alga}. Therefore there is no tiling.  \qedhere
\end{case*}
\end{proof}

\section{A glimpse into a bigger picture}
\label{Sec-Discussion}

In the case when $m=3$ (i.e., equilateral triangles), the truncated tetrahedron also gives rise to a finite family via kite subdivision of its hexagons. In light of the discussion of the truncated octahedron and the truncated icosahedron (see \eqref{AVC-albe2-de3-albega2-alga4}, Proposition \ref{Prop-be}) we consider
\begin{align}\label{AVC-trunc-albega2-de3-albe2-alga4}
\AVC = \{ \alpha\beta\gamma^2, \delta^d, \alpha\beta^2, \alpha\gamma^4 \}.
\end{align}
For the truncated tetrahedron, we take $d=3$ and use the vertex angle sums to derive
\begin{align*}
\beta = \pi  - \tfrac{1}{2}\alpha, \quad
\gamma = \tfrac{1}{2}\pi  - \tfrac{1}{4}\alpha, \quad
\delta = \tfrac{2}{3}\pi,
\end{align*}
and $\tfrac{1}{3}\pi < \alpha <\tfrac{2}{3}\pi$. Substituting the above into \eqref{Eq-angle-id}, \eqref{Eq-a4-cosx} and \eqref{Eq-x2y2-cosy} determines the angles and edge lengths
\begin{align}
&\alpha=\cos^{-1}\tfrac{7}{18},&
&\beta=\cos^{-1}(-\tfrac{5}{6}),&
&\gamma =\cos^{-1}\tfrac{1}{2\sqrt3},& \\ \notag 
&x = \cos^{-1}\tfrac{7}{11},&
&y =\tan^{-1}\tfrac{2}{3}\sqrt{6}.&
\end{align}
With the angle values and edge lengths, one may obtain tilings in Figure \ref{Fig-kt-subdiv-trunc-tetrahedron}. 

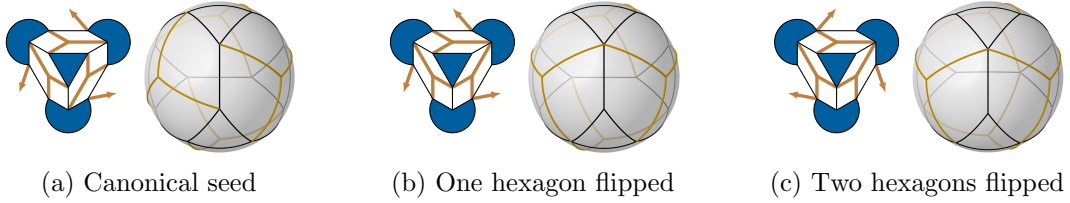
\begin{figure}[h!]
\centering
\begin{subfigure}[t]{0.325\linewidth}
\centering
\begin{tikzpicture}[>=latex]
\tikzmath{
\s=2;
\r=0.3;
\th=360/3;
\x=\r*cos(30);
\y=\r*sin(30);
\XS=2;
}

\raisebox{1ex}{

\foreach \a in {0,1,2} {
\tikzset{rotate=\a*\th}
\draw[HGold, line width=1.25]
	(0, 2*\y) -- (90+0.5*\th:\r)
	(0, 2*\y) -- (90-0.5*\th:2*\r) 
	(0, 2*\y) -- (-\x, 3*\y)
;
\draw[arrows = {-Latex[scale=0.45]}, HGold, line width=1.25]
	(\x, 3*\y) -- (0, 5*\y)
;
}

\foreach \a in {0,1,2} {
\tikzset{rotate=\a*\th}

\fill[teal!75!blue]
	(2*\x, 0) arc (270:270+2*\th:\r) -- (90-0.5*\th:2*\r) -- cycle
;
}

\fill[teal!75!blue]
	(90-0.5*\th:\r) -- (90-1.5*\th:\r) -- (90-2.5*\th:\r) -- cycle;
;

\foreach \a in {0,1,2} {
\tikzset{rotate=\a*\th}
\draw[]
	(90-0.5*\th:\r) -- (90+0.5*\th:\r)
	(90-0.5*\th:\r) -- (90-0.5*\th:2*\r) 
	(90-0.5*\th:2*\r) -- (\x, 3*\y)
	(90+0.5*\th:2*\r) -- (-\x, 3*\y)
	(-\x, 3*\y) -- (\x, 3*\y)
	(2*\x, 0) arc (270:270+2*\th:\r)
;
}

}

\begin{scope}[] 
\node [inner sep=0] (image) at (\XS,0) 
            {\includegraphics[height=2cm]{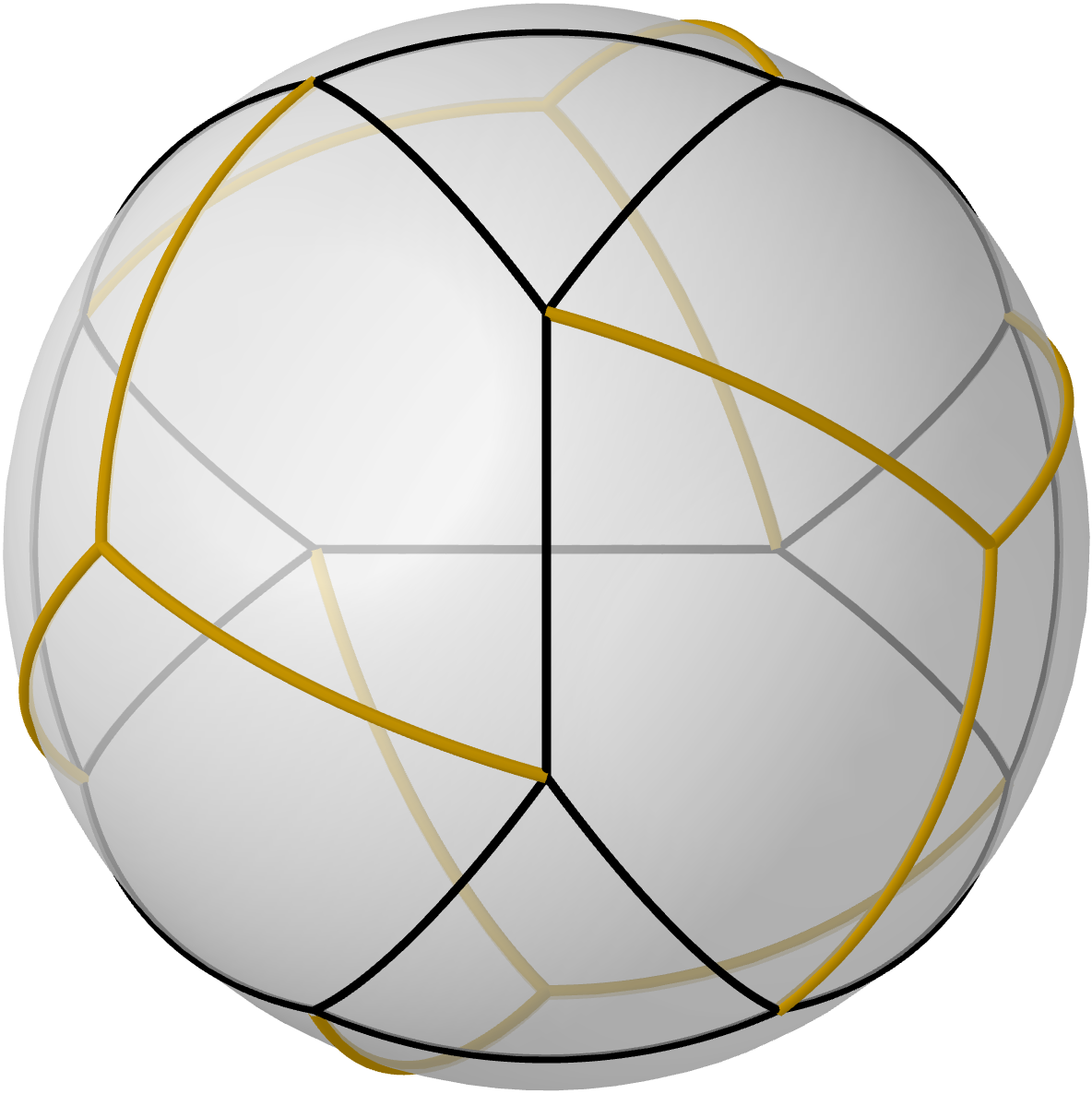}};
\end{scope}

\end{tikzpicture}
\caption{Canonical seed}
\label{Subfig-kt-subdiv-trunc-tetrahedron-seed}
\end{subfigure}
\begin{subfigure}[t]{0.325\linewidth}
\centering
\begin{tikzpicture}[>=latex]
\tikzmath{
\s=2;
\r=0.3;
\th=360/3;
\x=\r*cos(30);
\y=\r*sin(30);
\XS=2;
}

\raisebox{1ex}{

\foreach \a in {0,1} {
\tikzset{rotate=\a*\th}
\draw[HGold, line width=1.25]
	(0, 2*\y) -- (90+0.5*\th:\r)
	(0, 2*\y) -- (90-0.5*\th:2*\r) 
	(0, 2*\y) -- (-\x, 3*\y)
;
}

\draw[HGold, line width=1.25]
	(90-\th:2*\y) -- (90-0.5*\th:2*\r)
	(90-\th:2*\y) -- (270:\r) 
	(90-\th:2*\y) -- ([rotate=-1.5*\th]-\x, 3*\y)
;

\foreach \a in {0,1,2} {
\tikzset{rotate=\a*\th}
\draw[arrows = {-Latex[scale=0.45]}, HGold, line width=1.25]
	(\x, 3*\y) -- (0, 5*\y)
;
}

\foreach \a in {0,1,2} {
\tikzset{rotate=\a*\th}
\fill[teal!75!blue]
	(2*\x, 0) arc (270:270+2*\th:\r) -- (90-0.5*\th:2*\r) -- cycle
;
}

\fill[teal!75!blue]
	(90-0.5*\th:\r) -- (90-1.5*\th:\r) -- (90-2.5*\th:\r) -- cycle;
;

\foreach \a in {0,1,2} {
\tikzset{rotate=\a*\th}
\draw[]
	(90-0.5*\th:\r) -- (90+0.5*\th:\r)
	(90-0.5*\th:\r) -- (90-0.5*\th:2*\r) 
	(90-0.5*\th:2*\r) -- (\x, 3*\y)
	(90+0.5*\th:2*\r) -- (-\x, 3*\y)
	(-\x, 3*\y) -- (\x, 3*\y)
	(2*\x, 0) arc (270:270+2*\th:\r)
;
}
}

\begin{scope}[] 
\node [inner sep=0] (image) at (\XS,0) 
            {\includegraphics[height=2cm]{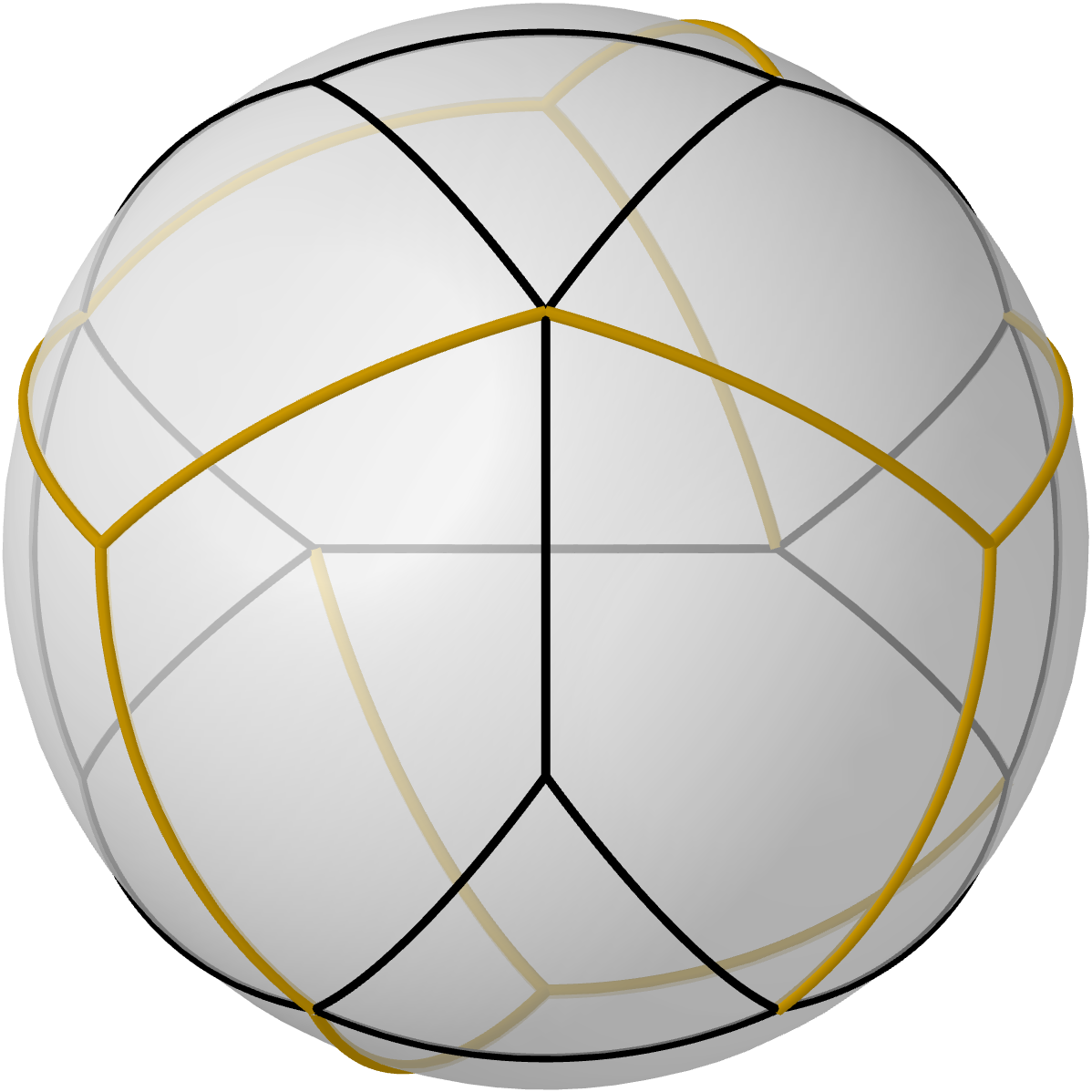}};
\end{scope}

\end{tikzpicture}
\caption{One hexagon flipped}
\label{Subfig-kt-subdiv-trunc-tetrahedron-F1}
\end{subfigure}
\begin{subfigure}[t]{0.325\linewidth}
\centering
\begin{tikzpicture}[>=latex]
\tikzmath{
\s=2;
\r=0.3;
\th=360/3;
\x=\r*cos(30);
\y=\r*sin(30);
\XS=2;
}

\raisebox{1ex}{

\foreach \a in {0,1} {
\tikzset{rotate=\a*\th}
\draw[HGold, line width=1.25]
	(0, 2*\y) -- (90+0.5*\th:\r)
	(0, 2*\y) -- (90-0.5*\th:2*\r) 
	(0, 2*\y) -- (-\x, 3*\y)
;
}

\draw[HGold, line width=1.25]
	(90-\th:2*\y) -- (90-0.5*\th:2*\r)
	(90-\th:2*\y) -- (270:\r) 
	(90-\th:2*\y) -- ([rotate=-1.5*\th]-\x, 3*\y)
;

\foreach \a in {0,1,2} {
\tikzset{rotate=\a*\th}
\draw[arrows = {-Latex[scale=0.45]}, HGold, line width=1.25]
	(-\x, 3*\y) -- (0, 5*\y)
;
}

\foreach \a in {0,1,2} {
\tikzset{rotate=\a*\th}

\fill[teal!75!blue]
	(2*\x, 0) arc (270:270+2*\th:\r) -- (90-0.5*\th:2*\r) -- cycle
;
}

\fill[teal!75!blue]
	(90-0.5*\th:\r) -- (90-1.5*\th:\r) -- (90-2.5*\th:\r) -- cycle;
;

\foreach \a in {0,1,2} {
\tikzset{rotate=\a*\th}
\draw[]
	(90-0.5*\th:\r) -- (90+0.5*\th:\r)
	(90-0.5*\th:\r) -- (90-0.5*\th:2*\r) 
	(90-0.5*\th:2*\r) -- (\x, 3*\y)
	(90+0.5*\th:2*\r) -- (-\x, 3*\y)
	(-\x, 3*\y) -- (\x, 3*\y)
	(2*\x, 0) arc (270:270+2*\th:\r)
;
}
}

\begin{scope}[] 
\node [inner sep=0] (image) at (\XS,0) 
            {\includegraphics[height=2cm]{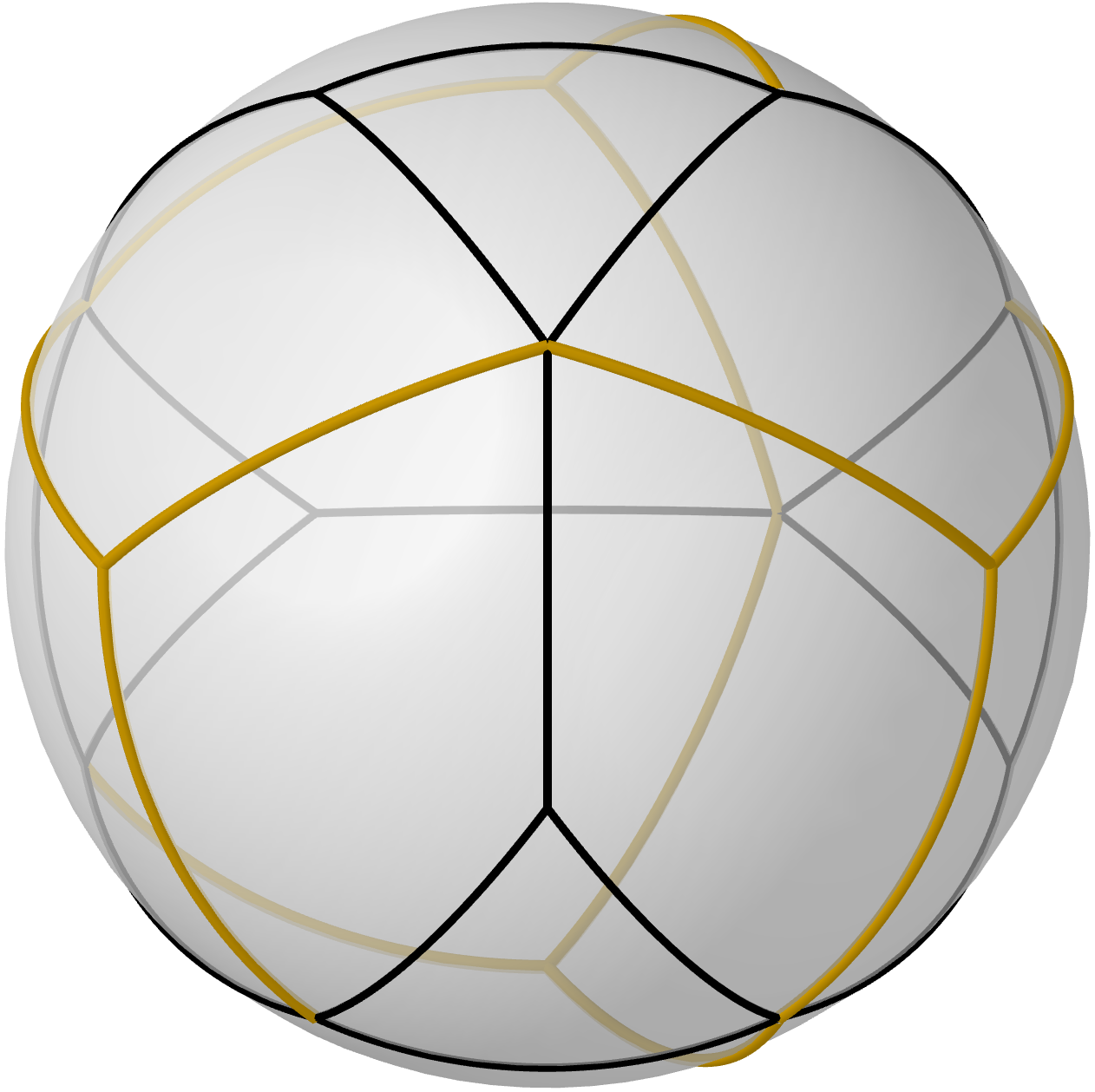}};
\end{scope}

\end{tikzpicture}
\caption{Two hexagons flipped}
\label{Subfig-kt-subdiv-trunc-tetrahedron-F2}
\end{subfigure}
\caption{The tilings via kite subdivision of the truncated tetrahedron}
\label{Fig-kt-subdiv-trunc-tetrahedron}
\end{figure}

\begin{prop} A finite family of three tilings are given by a kite subdivision of the truncated tetrahedron and their flips. 
\end{prop}

The three tilings are those in Figure \ref{Fig-kt-subdiv-trunc-tetrahedron}: the canonical seed in Figure \ref{Subfig-kt-subdiv-trunc-tetrahedron-seed}, one subdivision flipped in Figure \ref{Subfig-kt-subdiv-trunc-tetrahedron-F1}, and two subdivisions flipped in Figure \ref{Subfig-kt-subdiv-trunc-tetrahedron-F2}.

\begin{proof} The vertices (including their numbers) of the three tilings are 
\begin{align*}
\AVC = \, &\{ 4\delta^3, 12\alpha\beta\gamma^2 \}, \\
\AVC = \, &\{ 4\delta^3, 6\alpha\beta\gamma^2, 3\alpha\beta^2, 3\alpha\gamma^4 \}, \\
\AVC = \, &\{ 4\delta^3, 4\alpha\beta\gamma^2, 4\alpha\beta^2, 4\alpha\gamma^4 \}.
\end{align*}

It suffices to show that they are all of the tilings. Obviously, one may flip a combination of subdivided hexagon(s) and check for duplicates via a computer. 

By the same construction for the tilings via the truncated octahedron and the truncated icosahedron, $\AVC = \{ 4\delta^3, 12\alpha\beta\gamma^2 \}$ uniquely determines the canonical seed. 

Next, we assume that there is an $\alpha\beta^2$. It is straightforward to determine its incident tiles $T_1, T_2, T_3$ as shown in Figure \ref{Fig-trunc-tetrahedron-albe2-nbhd}. From $\AVC$ \eqref{AVC-trunc-albega2-de3-albe2-alga4}, we know that $\delta_2, \delta_3\cdots=\delta^3$ which determine $T_4, T_5$ and $T_6,T_7$ respectively. Now $\gamma_2\gamma_3\gamma_4\gamma_6\cdots=\alpha\gamma^4$ determines $T_8$.

Notice that the adjacent vertices $\alpha_1\gamma_2\gamma_5\cdots, \alpha_1\gamma_3\gamma_7\cdots$ can be $\alpha\beta\gamma^2$ or $\alpha\gamma^4$, but exactly one of them can be the former and the other one must be the latter. One can easily check that any such combination does not uniquely determine the whole tiling. Likewise for $\alpha_8\beta_4\cdots, \alpha_8\beta_6\cdots$, which can be $\alpha\beta^2$ or $\alpha\beta\gamma^2$, but exactly one of them can be the former and the other one must be the latter. 

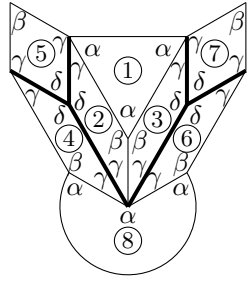
\begin{figure}[h!]
\centering
\begin{tikzpicture}[>=latex]
\tikzmath{
\r=0.9;
\th=360/3;
\x=\r*cos(30);
\y=\r*sin(30);
}

\foreach \a in {0,1,2} {
\tikzset{rotate=\a*\th}
\draw[]
	(90-0.5*\th:\r) -- (90+0.5*\th:\r)
	(270:\r) -- (270:2*\r)
;
}

\foreach \a in {1,2} {
\tikzset{rotate=\a*\th}
\draw[]
	(90-0.5*\th:2*\r) -- (\x, 3*\y)
	(90+0.5*\th:2*\r) -- (-\x, 3*\y)
	(-\x, 3*\y) -- (\x, 3*\y)
;
}

\foreach \aa in {-1,1} {
\tikzset{xscale=\aa}
\draw[line width=1.5]
	(90-\th:2*\y) -- (90-0.5*\th:\r)
	(90-\th:2*\y) -- (270:2*\r) 
	(90-\th:2*\y) -- ([rotate=-0.5*\th]\x, 3*\y)
;
}

\draw[rotate=-\th]
	(2*\x, 0) arc (270:270+2*\th:\r)
;

\node at (270:2.15*\r) {\small $\alpha$};

\foreach \a in {0,1,2} {
\tikzset{rotate=\a*\th}
\node at (270:0.6*\r) {\small $\alpha$};
}

\foreach \aa in {-1,1} {
\tikzset{xscale=\aa}
\node at (0.85*\x, 0*\y) {\small $\gamma$};
\node at (1.6*\r, 0.2*\y) {\small $\gamma$};
\node at (1.6*\r, 1.4*\y) {\small $\beta$};

\node at (1.15*\x, 0.75*\y) {\small $\gamma$};

\node at (1.05*\r, -0.3*\y) {\small $\delta$};
\node at (0.7*\r, -0.9*\y) {\small $\delta$};
\node at (1*\r, -1.3*\y) {\small $\delta$};

\node at (0.2*\x, -1.05*\r) {\small $\beta$};
\node at (0.2*\x, -1.45*\r) {\small $\gamma$};

\node at (0.9*\x, -1.375*\r) {\small $\beta$};
\node at (1.4*\r, -0.8*\y) {\small $\gamma$};
\node at (0.45*\x, -1.625*\r) {\small $\gamma$};

\node at (0.9*\x, -1.75*\r) {\small $\alpha$};
}

\node[inner sep=1,draw,shape=circle] at (0,0) {\footnotesize $1$};
\node[inner sep=1,draw,shape=circle] at (-0.5*\x,-0.72*\r) {\footnotesize $2$};
\node[inner sep=1,draw,shape=circle] at (0.5*\x,-0.72*\r) {\footnotesize $3$};
\node[inner sep=1,draw,shape=circle] at (-1*\x,-1*\r) {\footnotesize $4$};
\node[inner sep=1,draw,shape=circle] at (-1.275*\r,0.5*\y) {\footnotesize $5$};
\node[inner sep=1,draw,shape=circle] at (1*\x,-1*\r) {\footnotesize $6$};
\node[inner sep=1,draw,shape=circle] at (1.275*\r,0.5*\y) {\footnotesize $7$};

\node[inner sep=1,draw,shape=circle] at (0,-2.5*\r) {\footnotesize $8$};

\end{tikzpicture}
\caption{The neighbourhood of $\alpha\beta^2$}
\label{Fig-trunc-tetrahedron-albe2-nbhd}
\end{figure}

Up to mirror symmetry, it suffices to consider two combinations of the four vertices below,
\begin{align*}
(\alpha_1\gamma_2\gamma_5\cdots, \alpha_1\gamma_3\gamma_7\cdots, \alpha_8\beta_4\cdots, \alpha_8\beta_6\cdots) = (\alpha\beta^2, \alpha\gamma^4, \alpha\beta\gamma^2, \alpha\beta^2), (\alpha\beta^2, \alpha\gamma^4, \alpha\beta^2, \alpha\beta\gamma^2).
\end{align*}
The former deduces the tiling in Figure \ref{Subfig-kt-subdiv-trunc-tetrahedron-F1} while the latter deduces the tiling in Figure \ref{Subfig-kt-subdiv-trunc-tetrahedron-F2}. Therefore we have all the tilings.  
\end{proof}

In fact, without using the truncated tetrahedron, what we proved above is that $\AVC$ \eqref{AVC-trunc-albega2-de3-albe2-alga4} gives rise to the three tilings in the finite family. 

Finite families of tilings can also be derived from the truncated cube $t\mathcal{C}$ by taking $d=3$ in $\AVC$ \eqref{AVC-trunc-albega2-de3-albe2-alga4} and the truncated dodecahedron $t\mathcal{D}$ by taking $d=5$. Alternatively, the values for the angles and the edge length $x$ can be found in \cite[Table 3]{lnp}. 
\begin{align}
t\mathcal{C}: \quad &\alpha =\cos^{-1}\tfrac{1}{4}(2\sqrt2-1),& &\beta =\cos^{-1}(-\tfrac{1}{4}(2+\sqrt2)), \\ \notag
&\gamma =\tfrac{1}{2}\cos^{-1}(-\tfrac{1}{4}(2+\sqrt2)),& &\delta=\tfrac{1}{2}\pi,& \\ \notag
&x=\tan^{-1}\tfrac{1}{7}(6\sqrt2-4),& &y=\tfrac{1}{2}\cos^{-1}\tfrac{1}{17}(4\sqrt2-7).& \\
t\mathcal{D}: \quad &\alpha =2\cos^{-1}\tfrac{4+\sqrt5}{5+\sqrt5},& &\beta=\cos^{-1}(-\tfrac{4+\sqrt5}{5+\sqrt5}),& \\ \notag
&\gamma =\tfrac{1}{2}\cos^{-1}(-\tfrac{4+\sqrt5}{5+\sqrt5}),& &\delta=\tfrac{2}{5}\pi,& \\ \notag
&x=\tan^{-1}\tfrac{2}{3}(5-2\sqrt5),& &y=\tfrac{1}{2}\cos^{-1}\tfrac{8-3\sqrt5}{5\sqrt5-8}.&
\end{align}

An immediate consequence of the above discussion is as follows.
\begin{prop}[Proposition \ref{Prop-Platonic-finite-families}]\label{Prop-tP-families} The truncated Platonic solids give rise to finite families of dihedral tilings via kite subdivisions and their flip modifications. Their canonical seeds are given in Figure \ref{Fig-trunc-platonic-seeds}.
\end{prop}

\begin{figure}[h!]
\centering
\begin{tikzpicture}[>=latex]
\tikzmath{
\s=3;
\r=0.3;
\XS=3.5;
\YS=3;
}

\begin{scope}[] 
\node [inner sep=0] (image) at (0,0) 
            {\includegraphics[height=2cm]{dih-kt-subdiv-hex-trunc-tetrahedron-seed}};
\node at (0,-0.5*\s) {\small Kite-subdivided $t\mathcal{T}$};
\end{scope}

\begin{scope}[] 
\node [inner sep=0] (image) at (\XS,0) 
            {\includegraphics[height=2cm]{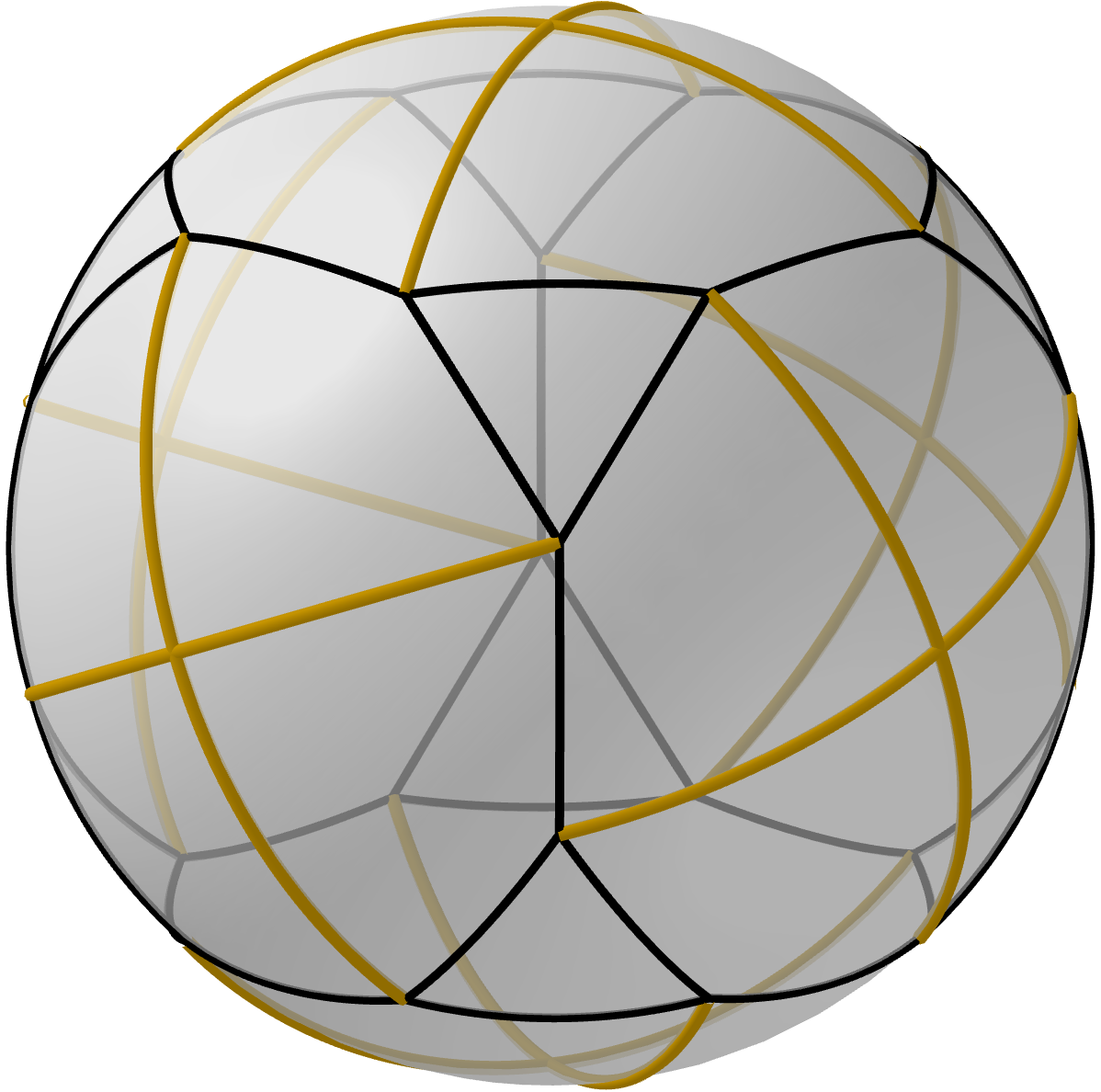}};
\node at (\XS,-0.5*\s) {\small Kite-subdivided $t\mathcal{C}$};
\end{scope}

\begin{scope}[] 
\node [inner sep=0] (image) at (2*\XS,0) 
            {\includegraphics[height=2cm]{dih-trunc-octahedron}};
\node at (2*\XS,-0.5*\s) {\small Kite-subdivided $t\mathcal{O}$};
\end{scope}

\begin{scope}[] 
\node [inner sep=0] (image) at (0.5*\XS,-\YS) 
            {\includegraphics[height=2cm]{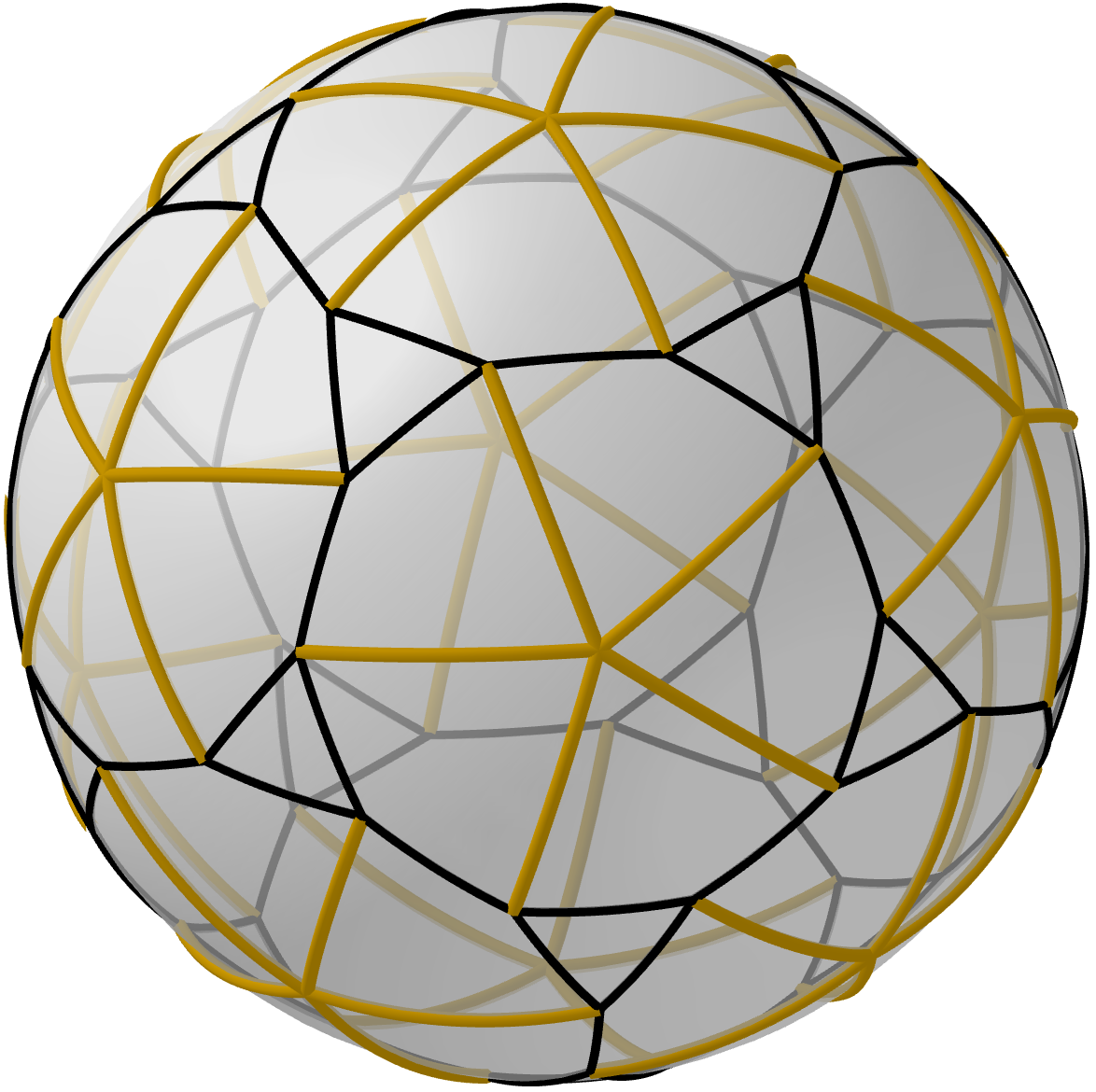}};
\node at (0.5*\XS,-0.5*\s-\YS) {\small Kite-subdivided $t\mathcal{D}$};
\end{scope}

\begin{scope}[] 
\node [inner sep=0] (image) at (1.5*\XS,-\YS) 
            {\includegraphics[height=2cm]{dih-trunc-icosahedron}};
\node at (1.5*\XS,-0.5*\s-\YS) {\small Kite-subdivided $t\mathcal{I}$};
\end{scope}

\end{tikzpicture}
\caption{Canonical seeds for the families by kite-subdivided truncated Platonic solids}
\label{Fig-trunc-platonic-seeds}
\end{figure}

Further studies on tilings having a kite and an equilateral triangle as prototiles will appear in our future work.

\section{Acknowledgement}

The third author expresses his sincere gratitude towards Tom\'a\v{s} Kaiser, Roman Nedela, and Min Yan for fruitful discussions. 

\section*{Funding}

The first author was supported by JSPS Kakenhi grant 25K21149.
\noindent The third author was supported by Academic Career in Pilsen 2024, 2025 and 2026.


\begin{thebibliography}{99}

\bibitem{aehj}
C.~Adams, C.~Edgar, P.~Hollander and L.~Jacoby,
\newblock The non-edge-to-edge tilings of the sphere by regular polygons.
\newblock {\em Discrete \& Computational Geometry}, 72:1029--1085, 2024.


\bibitem{awy}
Y.~Akama, E.~X.~Wang, M.~Yan,
\newblock Tilings of the sphere by congruent pentagons III: edge combination $a^5$,
\newblock {\it Adv. in Math.} {\bf 394} (2022), \#107881.

\bibitem{ay}
Y.~Akama, M.~Yan,
\newblock On deformed dodecahedron tiling,
\newblock {\it Australas. J. of Comb.} {\bf 85(1)} (2023), 1--14.

\bibitem{bj}
D.~Barnette, E.~Jucovi\v{c},
\newblock Hamiltonian circuits on $3$-polytopes,
\newblock {\em Journal of Combinatorial Theory}, 9:54--59, 1970.


\bibitem{cl}
H.~M.~Cheung, H.~P.~Luk,
\newblock Dihedral tilings of the sphere by regular polygons and quadrilaterals II: regular polygons with high gonality and rhombi,
\newblock {\em preprint}, 2023, {\tt arXiv:2403.07014}.

\bibitem{cly}
H.~M.~Cheung, H.~P.~Luk, M.~Yan,
\newblock Tilings of the sphere by congruent quadrilaterals or triangles,
\newblock {\it preprint}, 2022, {\tt arXiv:2204.02736}.

\bibitem{cly2}
H.~M.~Cheung, H.~P.~Luk, M.~Yan,
\newblock Tilings of the sphere by congruent pentagons IV: edge combination $a^4b$,
\newblock {\it preprint}, 2023, {\tt arXiv:2307.11453}.


\bibitem{desi}
M.~Deza, M.~D.~Sikiri\'c,
\newblock {\it Geometry of Chemical Graphs: Polycycles and Two-faced Maps}
\newblock Cambridge University Press, 2008.

\bibitem{emm}
J.~A.~Ellis-Monaghan, I.~Moffatt, 
\newblock Chapter 1.1.4 Ribbon Graphs, 
\newblock {\em Graphs on Surfaces: Dualities, Polynomials, and Knots}, SpringerBriefs in Mathematics, Springer, pp. 5--7, 2013.

\bibitem{gj}
B.~Gr\"unbaum, N.~W.~Johnson,
\newblock The faces of a regular-faced polyhedron,
\newblock {\em Journal of the London Mathematical Society} 1 (1965) 1:577--586.

\bibitem{joh}
N.~Johnson, 
\newblock Convex solids with regular faces, 
\newblock {\em Canadian Journal of Mathematics} 18 (1966), 169--200.

\bibitem{gru}
B.~Gr\"unbaum,
\newblock On polyhedra in $\mathbb{E}^3$ having all faces congruent,
\newblock {\em Bull. Research Council Israel} {\bf 8F} (1960), 215--218.

\bibitem{hig}
Y.~Higuchi,
\newblock Combinatorial curvature for planar graphs,
\newblock {\it J. Graph Theory} {\bf 38(4)} (2001), 220--229.

\bibitem{gsy}
H.~H.~Gao, N.~Shi, M.~Yan,
\newblock Spherical tiling by $12$ congruent pentagons,
\newblock {\it J. Comb. Theory Ser. A} {\bf 120(4)} (2013), 744--776.

\bibitem{lc}
H.~P.~Luk, H.~M.~Cheung, 
\newblock Rational angles and tilings of the sphere by congruent quadrilaterals,
\newblock {\it Ann. Comb.} {\bf 28} (2024), 485--527.

\bibitem{lnp}
H.~P.~Luk, R.~Nedela, C.~Purcell,
\newblock Tiling the sphere with regular polygons,
\newblock {\em preprint}, 2025, {\tt arXiv:2512.05577}.

\bibitem{luk1}
H.~P.~Luk,
\newblock Dihedral tilings of the sphere by regular polygons and quadrilaterals I: squares and rhombi,
\newblock {\em Combinatorial Theory} {\bf 5(3)} (2025) {\tt \url{doi.org/10.5070/C65365567}}. 

\bibitem{luk2}
H.~P.~Luk,
\newblock Dihedral tilings of the sphere by regular polygons and quadrilaterals: quadrilaterals with equal opposite edges,
\newblock {\em preprint}, 2023, {\tt arXiv:2403.05938}.

\bibitem{luk3}
H.~P. Luk,
\newblock A parity phenomenon of spherical tilings.
\newblock {\em Archiv der Mathematik} (2026) {\tt \url{doi.org/10.1007/s00013-026-02233-2}}.

\bibitem{luk4}
H.~P.~Luk,
\newblock Odd tilings of the sphere and graphs of Herschel type,
\newblock {\em preprint}, 2024.

\bibitem{rob}
S.~A.~Robertson, 
\newblock Isometric folding of Riemannian manifolds, 
\newblock {\em Proc. R. Soc. Edinb.} {\bf 79} (1977), 275--284.

\bibitem{som}
D.~M.~Y.~Sommerville,
\newblock Division of space by congruent triangles and tetrahedra,
\newblock {\it Proc. Royal Soc. Edinburgh} {\bf 43} (1923), 85--116.

\bibitem{ua}
Y.~Ueno, Y.~Agaoka,
\newblock Classification of tilings of the $2$-dimensional sphere by congruent triangles,
\newblock {\it Hiroshima Math. J.} {\bf 32(3)} (2002), 463--540.

\bibitem{ua2}
Y.~Ueno, Y.~Agaoka, 
\newblock Examples of spherical tilings by congruent quadrangles, 
\newblock {\it Math. Inform. Sci., Fac. Integrated Arts Sci., Hiroshima Univ.} {\bf Ser. IV 27} (2001), 135--144.

\bibitem{wy}
E.~X.~Wang, M.~Yan,
\newblock Tilings of sphere by congruent pentagons I: edge combinations $a^2b^2c$ and $a^3bc$, 
\newblock {\it Adv. in Math.} {\bf 394} (2022), \#107866.

\bibitem{wy2}
E.~X.~Wang, M.~Yan,
\newblock Tilings of sphere by congruent pentagons II: edge combination $a^3b^2$,
\newblock {\it Adv. in Math.} {\bf 394} (2022), \#107867.

\bibitem{zal} 
V.~A.~Zalgaller, 
\newblock Convex polyhedra with regular faces,
\newblock {\em Zap. Naucn. Sem. Leningrad. Otdel. Mat. Inst. Steklov. (LOMI)} 2, 220 (1967).

\end{thebibliography}
\end{document}